\documentclass[10pt]{amsart}

\usepackage{mathtools}
\mathtoolsset{showonlyrefs}

\usepackage{graphicx}
\usepackage{amsmath,amssymb,amsthm}
\usepackage[english]{babel}

\usepackage[utf8]{inputenc}

\usepackage{amsfonts}
\usepackage{amsmath}
\usepackage{latexsym}
\usepackage{amscd}
\usepackage{shadow}
\usepackage{a4wide}
\usepackage{endnotes}
\usepackage{amsopn}
\usepackage{url}
\usepackage{dsfont}
\usepackage{bbold}
\DeclareSymbolFontAlphabet{\mathbb}{AMSb}
\DeclareMathAlphabet{\mathbbold}{U}{bbold}{m}{n}

\usepackage{centernot}

\usepackage{enumitem}

\usepackage{tensor}

\usepackage[foot]{amsaddr}

\usepackage{hyperref}
\hypersetup{colorlinks=true}

\usepackage[margin=2cm,bottom=2.5cm,top=2.5cm]{geometry}

\numberwithin{figure}{section}
\theoremstyle{plain}
\newtheorem{thm}{Theorem}[section]
\newtheorem{theoreme}{Theorem}[section]
\newtheorem{prop}[thm]{Proposition}

\newtheorem{cor}[thm]{Corollary}
\newtheorem{lemma}[thm]{Lemma}
\numberwithin{equation}{section}

\theoremstyle{remark}

\newtheorem{rmq}[thm]{Remark}
\newtheorem{rem}[thm]{Remark}

\theoremstyle{definition}
\newtheorem{dfn}[thm]{Definition}

\newcommand{\eps}{\varepsilon}

\newcommand{\ceps}{\epsilon}

\newcommand{\hatun}[1]{\overset{\lower.9em\hbox{${\scriptscriptstyle 1 \wedge}$}}{#1}}
\newcommand{\hatdeux}[1]{\overset{\lower.9em\hbox{${\scriptscriptstyle \!\!2 \wedge}$}}{#1}}

\renewcommand\Re{\mathrm{Re}\,} \renewcommand\Im{\mathrm{Im}\,}
\newcommand\R{{\mathbb R}} \newcommand\N{{\mathbb N}}
\newcommand\Z{{\mathbb Z}}

\newcommand\supp{{\mathrm{supp}\hskip 0.05cm}}

\renewcommand{\Im}{  \text{Im}   }
\renewcommand{\Re}{  \text{Re}   }

\def\ra{\rightarrow}

 \def\cdotv{\raise 2pt\hbox{,}}

\renewcommand\Re{\mathrm{Re}\,} \renewcommand\Im{\mathrm{Im}\,}

\makeatletter
\def\@tvsp{\mathchoice{{}\mkern-4.5mu}{{}\mkern-4.5mu}{{}\mkern-2.5mu}{}}
\def\ltrivert{\left|\@tvsp\left|\@tvsp\left|}
\def\rtrivert{\right|\@tvsp\right|\@tvsp\right|}
\makeatother

\def\ra{\rightarrow}

 \def\cdotv{\raise 2pt\hbox{,}}

\newcommand{\Grond}{\mathcal{G}}
\newcommand{\Prond}{\mathcal{P}}

\newcommand{\Ab}{\mathbb{A}}
\newcommand{\Af}{\mathfrak{A}}
\newcommand{\Bf}{\mathfrak{B}}
\newcommand{\tf}{\mathfrak{S}}
\newcommand{\xx}{\mathbbold{E}}
\newcommand{\yy}{\mathbbold{S}}

\newcommand{\tprime}{\varrho}\newcommand{\Lp}{\ell}
\newcommand{\Lo}{\mathfrak{L}}

\newcommand{\canonchi}{\boldsymbol{\chi}}
\newcommand{\cutoffchi}{\chi}

\usepackage{tikz}

\newcommand{\swallowtaildrawing}[1]{  \begin{tikzpicture}[
    x=#1pt,y=#1pt,baseline=-0.15ex,rotate=-90]
    \draw[line width=0.45pt,line cap=round,line join=round]
      plot[domain=-1.07:1.07,samples=241,variable=\t]
      ({\t*\t*\t-\t},
       {0.75*\t*\t*\t*\t-0.5*\t*\t});
  \end{tikzpicture}}
\newcommand{\swallowtail}{\mathord{\mathchoice
  {\swallowtaildrawing{12}}{\swallowtaildrawing{12}}
  {\swallowtaildrawing{8.5}}{\swallowtaildrawing{6}}}}

\begin{document}

\title[Dispersion for the Schrödinger equation inside strictly convex domains]{Dispersion for the Schrödinger equation inside strictly
  convex domains: the general case}

\author{Oana Ivanovici$^{1}$}
\address{$^{1}$Sorbonne Université, CNRS, Laboratoire Jacques-Louis Lions, LJLL, F-75005 Paris, France}
\email{oana.ivanovici@math.cnrs.fr}
\author{Fabrice Planchon$^{2,3}$}
\address{$^{2}$Sorbonne Universit\'e, CNRS, IMJ-PRG F-75005 Paris, France}\address{$^{3}$Institut universitaire de France (IUF), France}
  \email{fabrice.planchon@sorbonne-universite.fr}

      \thanks{
    \\
   The authors were partially supported by ERC grant ANADEL 757 996.}
\date{\today}

\begin{abstract}
{We consider a general smooth bounded strictly convex domain $\Omega\subset\mathbb{R}^d$ of dimension $d\geq2$ and describe dispersion for the semiclassical Schrödinger equation with Dirichlet boundary condition. Our results hold more generally on compact smooth Riemannian manifolds with smooth strictly convex boundary.} More specifically, we construct a sharp local in (semiclassical) time parametrix and then proceed to obtain dispersion estimates: our fixed-time decay rate for the Green function {exhibits a loss of $1/4$ in the exponent of $h/t$ with respect to the boundaryless case}. The loss is sharp for dispersion, as shown in \cite{doiSF21} in the case of a model convex domain. {On compact three-dimensional manifolds with smooth strictly convex boundary, the resulting spectrally localized Strichartz estimates yield global well-posedness in the energy space for the defocusing cubic nonlinear Schr\"odinger equation, thus matching the corresponding result for boundaryless manifolds due to \cite{bgt04}.  Moreover, we extend bounds on the time growth of higher Sobolev norms from \cite{PTV} to our setting.}

\end{abstract}

\maketitle
\par\medskip\noindent

\section{Introduction}
Let us consider the Schr\"odinger equation on a manifold $(\Omega,g)$, with a strictly convex boundary $\partial
\Omega$ (a precise definition of strict convexity will be provided later): 
\begin{equation}\label{schrod}
-i\partial_t v+\Delta_g v=\kappa |v|^{2} v, \quad v|_{t=0}=v_0,\quad  v|_{\mathbb{R}\times\partial\Omega}=0,
\end{equation}
where $\Delta_g$ denotes the Laplace operator with Dirichlet boundary condition, and $\kappa=0$ (linear equation) or $\kappa=\pm 1$ (defocusing or focusing nonlinear cubic equation).

In \cite{doiSF21} the first author recently solved the cubic nonlinear Schrödinger equation \eqref{schrod} on a {\it model} convex domain (called the Friedlander domain), matching known results on generic compact manifolds without boundary. It required obtaining dispersive bounds for the linear semiclassical Schrödinger flow, from which, using the standard $TT^*$ argument, Strichartz-type estimates follow. Dispersion for the linear Schr\"odinger flow in $\mathbb{R}^{d}$ reads as:
\begin{equation}\label{dispschrodrd}
\|e^{\pm it\Delta_{\mathbb{R}^d}}\|_{L^1(\mathbb{R}^d)\rightarrow L^{\infty}(\mathbb{R}^d)}\leq C(d)t^{-d/2}, \text{ for all } t\neq 0.
\end{equation}
Indeed, \eqref{dispschrodrd}, together with the unitary property of the flow on $L^2(\mathbb{R}^d)$ yield all known Strichartz estimates; the endpoint cases are more delicate (see \cite{KeTa98}, \cite{Fo05}, \cite{Vi07}).
On any boundaryless Riemannian manifold $(\Omega,g)$ one may follow the same path, replacing the exact formula by a parametrix, constructed locally within a small ball, thanks to finite speed of propagation for waves or in semiclassical time for Schr\"odinger flow, namely, on short-time, wavelength-sized intervals (i.e., their length is the inverse of the frequency), allowing for almost finite speed of propagation. By time rescaling, dispersion for the semiclassical Schr\"odinger equation in the Euclidean space reads, with $\psi \in C_{0}^{\infty}$ being a smooth cutoff to localize frequencies,
\begin{equation}\label{sclSchrod}
\sup\Big| \psi(hD_{t})e^{\pm i th \Delta_{\mathbb{R}^d}}\Big|
\leq \frac{C(d)}{h^{d}} \min \Big(1, ({h\over t})^{\frac{d}{2}}\Big) \text{ for all } 0<|t|\lesssim 1\,.
\end{equation}
While for $\Omega =\mathbb{R}^d$, dispersive properties of \eqref{schrod} are well understood, the study of Schr\"odinger-type dispersive equations on manifolds (with curved geometry or variable metrics) has proved particularly elusive, due to the presence of multiple reflected waves. In \cite{bgt04}, these linear estimates were used, together with a classical argument due to Yudovitch, to obtain global well-posedness for the defocusing cubic NLS on a generic 3D compact manifold without boundary. In \cite{doiSF21}, the goal was to  match this result in our context, with a model convex boundary. For $d\geq 2$, the so-called Friedlander model domain is the half-space
$\Omega_d=\{(x,y)|x>0, y\in\mathbb{R}^{d-1}\}$ with the metric $g_F$ associated with the following Laplace operator
\begin{equation}
  \label{eq:LapF}
\Delta_F=\partial^2_x+\sum_{j}\partial_{y_j}^2+ x\sum_{j,k}q_{j,k}\partial_{y_j}\partial_{y_k}\,,
\end{equation}
where $q_{j,k} $ are constants and $q(\theta)=\sum_{j,k}q_{j,k}\theta_j \theta_k$ is a positive definite
quadratic form. Note that $q$ is not, in general, invariant by rotations and we cannot reduce to the radial case in $y$. One may see $\Delta_{F}$ as the Laplace operator in geodesic normal coordinates near the boundary, but with all coefficients $q_{j,k}(x,y)$ frozen at their values on the boundary. Strict convexity of $\Omega_d$ with the metric associated from $\Delta_F$ is equivalent to ellipticity of $\sum_{j,k}q_{j,k}\partial_{y_j}\partial_{y_k}$. When $q_{j,j}=1$ and $q_{j\neq k}=0$ (i.e. when $q(\theta)=|\theta|^2$)
the domain $(\Omega_d,g_F)$ is a first-order approximation of the unit disk in polar coordinates $(r,\theta)$: $r=1-\frac x2$, $\theta=y$.

{
In the present work, we deal with the Schrödinger equation inside a general smooth bounded strictly convex domain in Euclidean space. More generally, we work on a compact smooth Riemannian manifold $(\Omega,g)$ of dimension $d\geq2$, with nonempty smooth strictly convex boundary. Here, strict convexity means that the second fundamental form of $\partial\Omega$, defined by $\mathrm{II}(X,Y)=g(\nabla_X\nu,Y)$ with $\nu$ the outward unit normal, is positive definite. This setting includes smooth bounded Euclidean domains whose boundary has strictly positive principal curvatures. More generally, a smooth convex domain in Euclidean space whose principal curvatures admit a uniform positive lower bound is necessarily bounded (it has compact closure), in every dimension $d\geq2$. This implication fails for general Riemannian metrics. We assume compactness to avoid additional geometric and analytic issues at infinity; extensions under suitable uniform geometric assumptions are beyond the scope of this paper.
}

In the linear part of the paper, $t$ denotes semiclassical time.
With the sign convention in \eqref{SE}, the propagators in semiclassical
and physical time are \( \mathcal U_h(t)=e^{-ith\Delta_g}\), $ e^{-it_{\mathrm{phys}}\Delta_g}$,  $t_{\mathrm{phys}}=ht$.
Thus $v_h(t)=u_h(ht)$, where
$(i\partial_{t_{\mathrm{phys}}}-\Delta_g)u_h=0$ and the initial data
are the same. For later convenience we set $D_t=-i\partial_t$. In particular, along these
homogeneous flows, $
 \psi(hD_t)v_h=\psi(-h^2\Delta_g)v_h$, $\psi(h^2D_{t_{\mathrm{phys}}})u_h=\psi(-h^2\Delta_g)u_h$. We will always specify which of the two time variables is used.
Our first result reads as follows: 
\begin{thm}\label{thmdisperS}
 Let $\psi \in C_{0}^{\infty}([1/2,3/2])$ such that $0\leq \psi\leq 1$. There exist
$C>0$, $T_{0}>0$, $1>h_0>0$ and $1\geq a_{0}>0$ such that, uniformly in $a\in ]0,a_{0}]$, $h\in (0,h_0)$ and  { $|t|\in(0,T_0]$}, the solution $v_h$ to
\begin{equation} \label{SE} 
\left\{ \begin{array}{l}
   (ih\partial_t- h^2\Delta_g) v_h=0,  \;\; \text{in } \Omega
\\ v_h|_{t=0}=v_{h,0}, 
 v_h|_{\partial\Omega}=0, \end{array} \right.
 \end{equation}
with $v_{h,0}=\psi(-h^2 \Delta_{g})\delta_a$, $\delta_a$ being any Dirac mass at distance $a$ from $\partial\Omega$, satisfies

\begin{equation}\label{dispco}
 \|\psi(-h^2\Delta_g)v_h(t)\|_{L^\infty(\Omega)}
 \leq \frac{C}{h^d}\min\left(1,
 \left(\frac h{|t|}\right)^{\frac{d-1}{2}}
 \left[\left(\frac h{|t|}\right)^{1/4}
       +\left(\frac h{|t|}\right)^{1/2}+h^{1/3}\right]\right).
\end{equation}
Up to changing $C$, the right-hand side is also
$C h^{-d}\min\{1,(h/|t|)^{d/2-1/4}\}$.
In physical time, the solution $u_h(t_{\mathrm{phys}})=v_h(t_{\mathrm{phys}}/h)$,
with the same initial datum, satisfies, for $0<|t_{\mathrm{phys}}|\leq hT_0$,
\begin{equation}\label{eq:dispersion-physical-time}
\begin{aligned}
 \|\psi(-h^2\Delta_g)u_h(t_{\mathrm{phys}})\|_{L^\infty(\Omega)}
 &\leq \frac{C}{h^d}
 \min\left(1,\left(\frac{h^2}{|t_{\mathrm{phys}}|}\right)^{d/2-1/4}\right)\\
 &=C\,
 \underbrace{\min\{h^{-d},|t_{\mathrm{phys}}|^{-d/2}\}}_{\text{frequency-localized free dispersion}}
 \underbrace{\max\left(1,\left(\frac{|t_{\mathrm{phys}}|}{h^2}\right)^{1/4}\right)}_{\text{loss}}.
\end{aligned}
\end{equation}
The Friedlander model shows that the decay exponent $d/2-1/4$ cannot be improved in general; see \cite[Theorem~1]{doiSF21}. In that model the kernel has size comparable to {$a^{1/4}h^{-d}\left(\frac{h}{|t|}\right)^{(d-1)/2+1/4}$} for $\sqrt a<|t|\leq T_0$ and $|t|h^{1/3}\ll a\leq a_0$.
\end{thm}
\begin{rmq}
If the source lies at distance $a>a_0$ from the boundary, we use the strictly interior estimate. After decreasing $T_0$ if necessary, the spectrally localized kernel is bounded by
\[
 Ch^{-d}\min\left(1,\left(\frac{h}{|t|}\right)^{d/2}\right),
 \qquad 0<|t|\leq T_0.
\]
These sources lie in strictly interior coordinate patches. In Section~\ref{sec:strictly-interior-contribution} we reduce to a boundaryless problem and apply the semiclassical parametrix of \cite{bgt04}. Thus the bound in Theorem~\ref{thmdisperS} holds for all source points.
\end{rmq}
\begin{rmq}
  Optimality concerns the decay exponent and its loss with respect to the free dispersion; it does not assert saturation of \eqref{dispco} uniformly in $a$ throughout the entire time interval.
\end{rmq}
\begin{rmq}
{After spectral localization, the kernel is $O(h^\infty)$ between points
that cannot be connected by a reflected ray in the time interval under
consideration. This follows by nonstationary phase.}  Consequently, the
compactness assumption may be replaced by uniform geometric and symbol
bounds on the finite family of coordinate patches met by the flow, together
with a uniform positive lower bound for the second fundamental form in a (tubular) neighborhood of the
boundary.
\end{rmq}

{Theorem \ref{thmdisperS} gives the dispersion estimate in both semiclassical and physical time. The physical-time formulation can be compared to \eqref{sclSchrod}; in semiclassical time, the loss is $1/4$ in the exponent of $h/t$.} {We may informally relate this loss} to the presence of caustics in arbitrarily
small times if $a$ is small. Moreover, one of the key features in Theorem \ref{thmdisperS} is that $T_{0}$ depends only on the geometry of $\partial\Omega$ and the metric $g$ : the dispersive bound holds {\it uniformly} with respect to both the source point and its distance $a$ to the boundary and the frequency $1/h$. In fact, say for $a=h^{\nu}$, $\nu>0$, there are at most $1/\sqrt a =h^{-\nu/2}$ reflections, and caustics between them, as we will see later ; so in the high-frequency regime $h\ra 0$, we have to deal with an increasingly large number of caustics, even to travel a small distance over a small time $T_{0}$. These caustics occur because optical rays are no longer
diverging in the normal direction, which yields less
dispersion than in the $\mathbb{R}^d$ case. 
{In the Friedlander model, tracking these caustics yields the optimality of the decay exponent in the regime specified in Theorem~\ref{thmdisperS}.}

\begin{thm}\label{thmStrichartz}
\leavevmode

Let $d\geq2$, $2\leq q,r\leq\infty$, and assume
\[
 \frac1q\leq\left(\frac d2-\frac14\right)
                    \left(\frac12-\frac1r\right),\qquad
 s=\frac d2-\frac2q-\frac dr.
\]
There exist $T_0>0$ and $C>0$, independent of $h\in(0,h_0)$ and
$v_{h,0}\in L^2(\Omega)$, such that the solution to \eqref{SE} satisfies
\begin{equation}\label{eq:strichartz-semiclassical}
 \|\psi(-h^2\Delta_g)v_h\|_{L_t^q([-T_0,T_0];L^r(\Omega))}
 \leq C h^{-s-1/q}\|v_{h,0}\|_{L^2(\Omega)}.
\end{equation}
Equivalently, for $u_h(t_{\mathrm{phys}})=
e^{-it_{\mathrm{phys}}\Delta_g}v_{h,0}$,
\begin{equation}\label{eq:strichartz-physical-time}
 \|\psi(-h^2\Delta_g)u_h\|_{L_{t_{\mathrm{phys}}}^q([-hT_0,hT_0];L^r(\Omega))}
 \leq C h^{-s}\|v_{h,0}\|_{L^2(\Omega)}.
\end{equation}
Along the corresponding homogeneous solutions, the spectral
cutoff can also be written as $\psi(hD_t)$ or
$\psi(h^2D_{t_{\mathrm{phys}}})$.
\end{thm}
Indeed, the change of variables $t_{\mathrm{phys}}=ht$ gives
\[
 \|\psi(-h^2\Delta_g)v_h\|_{L_t^q([-T_0,T_0];L^r)}
 =h^{-1/q}\|\psi(-h^2\Delta_g)u_h\|_{L_{t_{\mathrm{phys}}}^q([-hT_0,hT_0];L^r)}.
\]
Here $1/q=0$ when $q=\infty$. Thus the power in physical time on
intervals of length $O(h)$ is $h^{-s}$, while on a fixed interval
of semiclassical time it is $h^{-s-1/q}$.
The proof of Theorem \ref{thmStrichartz} follows from Theorem \ref{thmdisperS} using the classical $TT^*$ argument and the endpoint argument of Keel-Tao \cite{KeTa98} for $q=2$ when $d\geq 3$. 
{ The index $s$ in the physical-time estimate on intervals of length $O(h)$ illustrates the loss relative to free Strichartz estimates: for $d=2$}, for which the (forbidden) endpoint $(2,\infty)$ with $s=0$ is replaced by $(8/3,\infty)$ with $s=1/4$. This improves \cite{bss08} where for $d=2$, one has $(3,\infty)$. More generally, \cite{bss08} obtains $(2,\infty)$ as an endpoint for $d\geq 3$, e.g. $s=d/2-1$, whereas we have $(2,2(2d-1)/(2d-5))$ as our endpoint pair, with $s=1/(2d-1)$.  For $d=3$, our endpoint pair is $(2,10)$: that $10<+\infty$ allows us to adapt the argument from \cite{bgt04} and obtain well-posedness for the cubic equation.

In Theorem \ref{thmdisperS}, we mainly focus on constructing a sharp parametrix for the Schrodinger equation \eqref{SE} in semiclassical time, providing optimal bounds on the amplitude of the wave, including at a discrete set of swallowtail-type caustics whose number becomes arbitrarily large when the source gets closer and closer to the boundary. 
While a natural outcome of this parametrix is optimal dispersion bounds, we believe that such a sharp parametrix will prove useful for a broad range of applications beyond dispersive effects or localization of eigenfunctions, including sharp quantitative versions of propagation of singularities that are of importance in control theory.

In what follows, we provide applications of Theorem \ref{thmStrichartz} to the three-dimensional nonlinear Schrödinger equation, and let $\Omega$ be a three-dimensional compact manifold with smooth, strictly convex boundary. 
\begin{thm}\label{gwp}
Let $d=3$ and $v_{0}\in H^{1}_{0}(\Omega)$. There exists a unique global-in-time solution $v\in C_{t}(H^{1}_{0}(\Omega))$ to \eqref{schrod} with $\kappa=1$ (defocusing equation), and its energy is conserved along the flow.  { For $\kappa=-1$, the same global conclusion holds provided $\|v_0\|_{L^2(\Omega)}\|\nabla_gv_0\|_{L^2(\Omega)}<c_\Omega$, for a sufficiently small constant $c_\Omega>0$ depending only on the geometry.}
\end{thm}
Moreover, as in the boundaryless case, preservation of regularity holds and { one may adapt the argument} of \cite{PTV} to obtain exponential growth for the $H^{m}$ norm of the solution, where $m\in \mathbb{N}$, $m>1$. Here we state and prove such growth for $m=2,4$ as higher regularity would require dealing with suitable compatibility conditions for the data, which are beyond the scope of this paper.
\begin{thm}\label{growthHs}
  Let $d=3$ and $v_{0}\in H^{1}_{0}(\Omega)\cap H^{m}(\Omega)$ with $m=2,4$, $\Delta_{g} v_{0} \in H^{1}_{0}(\Omega)$ if $m=4$. Then the solution $v$ from Theorem \ref{gwp} is $C_{t}(H^{m}(\Omega))$, and its norm grows at most exponentially: there exists $C=C(m,\|v_{0}\|_{H^{m}(\Omega)})$ such that,
  \begin{equation}
    \label{eq:sobolev-growth}
      \| v(t,\cdot)\|_{H^{m}(\Omega)}\leq C \exp( C |t|)\,.
    \end{equation}
  \end{thm}

We now briefly discuss linear Strichartz estimates and their optimality. In \cite{doiSF21}, we proved Theorems \ref{thmdisperS} and \ref{thmStrichartz}  when the domain $\Omega$ is the so-called Friedlander model domain introduced above (that is $\Omega_d$ together with the metric associated from \eqref{eq:LapF}). Moreover, we showed that the estimates \eqref{dispco} are sharp in the regime $a>h^{1/3}t$, $t\geq \sqrt{a}$ and this holds for every such $t$. However, the fact that the dispersive bounds hold with a sharp loss does not necessarily imply that the Strichartz estimates from Theorem \ref{thmStrichartz} cannot be improved. In fact, even if the loss of $1/4$ holds for all $t$ such that $a>h^{1/3}t$, the locus of points in space where this happens is so narrow that, by integration in time, one could obtain much better Strichartz bounds than those provided by the $TT^*$ argument. This has been mentioned in
 \cite{doi25} for dimensions $d\geq 2$ and proved in $d=1$ (when the analysis is simpler due to the lack of tangential variables). We conjecture that Strichartz estimates hold with only $1/6$ loss in this regime, (possibly by following the approach in \cite[the proof of Thm.4 ]{doi25}). While it is well known that 
 a loss of at least $\frac 16$ derivatives in Strichartz estimates must hold for {\eqref{SE} on the Friedlander model domain (obtained by taking the initial data to be a gallery mode; see \cite[Section~2]{doi})}, it is not clear whether this result should be sharp. In \cite{doi25} the first author studied the case of the one-dimensional semiclassical Schrödinger equation on the half line $x>0$ and operator $\partial^2_x+x$ (which is the one-dimensional equivalent of the Friedlander domain $(\Omega_d, \Delta_F)$) : it turns out that, throughout the regime $a\leq h^{1/3}t$ the dispersive bounds exhibit a loss smaller than $1/4$, which makes it possible to improve the corresponding Strichartz bounds. The method used in \cite{doi25} involves exponential sum bounds and Van der Corput type derivative-tests and could also be adapted to our $d\geq 2$ general situation. However, as mentioned in \cite{doi25}, in order to approach a $1/6$ loss one needs refined bounds on exponential sums which are out of reach for the time being. Nonetheless, as we shall see in the proof of Theorem \ref{gwp} (and as already shown in \cite{doiSF21} in the model situation), the $1/4$ loss in the Strichartz bounds in Theorem \ref{thmStrichartz}  will be sufficient to obtain global well-posedness for the cubic NLS.\\

We conclude this introduction with a brief overview of the content in the following sections.
\begin{itemize}[leftmargin=5.5mm]
\item The second section 
  {reduces the globally localized dispersive estimate to the glancing region; the other pieces and the passage from local parametrices to the global propagator are dealt with in Appendix~\ref{app:localization-reduction}.}
We then proceed with building our parametrix for the Schrödinger propagator in the glancing regime, which is the key tool to prove Theorem \ref{thmdisperS}. 
  This step follows closely the construction from \cite{ILLP} in the case of the wave equation: we recall its main steps, in particular the construction of quasimodes for small initial distance $a<h^{2/3-\ceps}$. To obtain a solution for the Schrödinger flow we rely on the parametrix for the Helmholtz equation constructed in \cite{meta} (which in turn relies crucially on \cite{mel76}). Using the Airy-Poisson formula that we introduced in \cite{ilpCE}, we then obtain a parametrix, both as a "spectral'' sum and its counterpart after Poisson summation, as a sum of wave packets between two consecutive reflections. This construction holds for any $a$ and $h$. 

\item The third section first addresses dispersion estimates when $a>h^{2/3-\ceps}$ : both formulas available to express the parametrix are used, first the one in terms of reflected waves when, in addition, $a>(ht)^{1/2}$ and then the spectral sum when $a\leq (ht)^{1/2}$. When $a\sim (ht)^{1/2}$ both methods provide the same loss.
The analysis of the oscillatory integrals follows the spirit of \cite{ILLP} but departs from it on several counts. For all $a>(ht)^{1/2}$, all the wave packets interfere in time and one has to carefully add all their contributions (recall that for wave equation in the same regime, only a finite number of waves  would overlap).
As only the waves that leave at $t=0$ from $x=a$ within a small cone of directions of aperture $\sqrt{a}$ may provide $1/4$ loss, we separate this "tangential" case, dealt with in Section \ref{sectang}, from the "transverse" case, {dealt with in} Section \ref{sectransv}. In particular, in the tangential case, we state three main results (Propositions \ref{dispNgrand}, \ref{dispNpetitloin}, \ref{dispNpetitpres}) which provide refined estimates for each integral in the sum defining the Green function in this regime, depending on the number of reflections and the position of the spatial variables: these results complete the proof of Theorem \ref{thmdisperS} for "not too small" $a$. When $a<(ht)^{1/2}$ we work with the spectral sum as we did in \cite{doiSF21}.
  \item In the fourth section, for $a\leq h^{2/3-\ceps}$, we use the construction of quasimodes to obtain a function $g_{h,a}$ such that $J(g_{h,a})$ equals the initial data, and then we proceed as in the case $a>h^{2/3-\ceps}$ using a mix of dispersion estimates on each gallery mode, the spectral sum, and Poisson summation to obtain a $1/6$ loss below $h^{2/3}$, essentially due to the propagation and accumulation of cusps.
  
\item In the fifth section we deal with the nonlinear applications and prove first Theorem \ref{gwp}, then Theorem \ref{growthHs}.

\item {In Appendix~\ref{app:phase-equivalence} we return to the equivalence of the glancing phase functions. The proofs of the interior, hyperbolic and elliptic estimates, together with the spectral matching of the local parametrix and the global kernel, are given in Appendix~\ref{app:localization-reduction}.} {Appendix~\ref{app:refined-quartic-estimate} provides the refined degenerate stationary phase estimates used in the dispersive analysis.}
\end{itemize}

In the remainder of the paper, $A\lesssim B$ means that there exists a constant $C$ such that $A\leq CB$. This constant may change from line to line but is independent of all parameters; we make it explicit when needed. Similarly, $A\sim B$ means both $A\lesssim B$ and $B\lesssim A$.

\section{A parametrix construction}
{
We first recall in Section~\ref{secMelcoord} the glancing normal form
and the results of \cite[Section~2]{ILLP}, keeping the same
notations: microlocal strict convexity, the generating function for
Melrose's canonical transformation $\canonchi_M$, and the parametrix
for the Helmholtz equation. Before constructing our Schrödinger
parametrix, we explain in Section~\ref{sec:global-microlocal-partition}
how the global spectral cutoff is reduced, by a finite spatial and
microlocal partition, to the glancing region. The bounds for the
interior, hyperbolic and elliptic terms are stated in
Sections~\ref{sec:strictly-interior-contribution}--\ref{sec:elliptic-contribution},
and we compare the local and global kernels in
Section~\ref{sec:boundary-spectral-matching}; the proofs are in
Appendix~\ref{app:localization-reduction}. We then return to the
construction. Sections~\ref{sec:generating-function-properties}--\ref{sec:airy-preliminaries}
recall the properties of $\Gamma$, the wave parametrix and the
Airy identities. As in \cite{doi08}, where the semiclassical
Schrödinger equation was dealt with {\it outside} a strictly convex
obstacle, Kanaï's transform gives us a parametrix for \eqref{SE}
near glancing. We construct it in
Sections~\ref{sec:parametrix-large-distance} and~\ref{sec:parametrix-small-distance}
for the two ranges of the source distance $a$. The local estimates
are summed in Section~\ref{sec:completion-linear-proof}.}

\subsection{Melrose's classification theorem for glancing hypersurfaces}\label{secMelcoord}
\leavevmode

{ We first describe the local glancing construction.  Its place in the global proof will be specified in Sections~\ref{sec:global-microlocal-partition} and~\ref{sec:strictly-interior-contribution}.  Near a boundary point, we choose}
geodesic normal coordinates $(x,y)$ on $\Omega$, with $x>0$ on $\Omega$, $y\in\mathbb{R}^{d-1}$ such that $\partial\Omega=\{(0,y), y\in \mathbb{R}^{d-1}\}$; local coordinates on $\Omega\times\mathbb{R}_t$ are then $(x,y,t)$. Local coordinates on the base induce local coordinates on the cotangent bundle, namely $(x,y,t,\xi,\eta,\tau)$ on $T^*(\Omega\times\mathbb{R}_t)$. The corresponding local coordinates on the boundary are $(y,t,\eta,\tau)$. In these coordinates,
the Laplacian $\Delta_g$ becomes (\cite[III,Appendix C]{Hormander})
\begin{equation}
  \label{eq:laptilde}
 \Delta=e_{g}^{-1} \Delta_g e_{g}=\partial^2_x+R(x,y,\partial_y)\,, \text{ up to conjugation by a non vanishing smooth factor } e_{g}(x,y).
\end{equation}
Since the boundary is everywhere strictly (geodesically) convex, for every point $(0,y_0)\in \partial \Omega$ there exists $\kappa_0=(0,y_0,\xi_0,\eta_0)\in T^*\Omega$ where the boundary is microlocally strictly convex, i.e. such that  there exists a bicharacteristic passing through $(0,y_0,\xi_0,\eta_0)$ that intersects $\partial\Omega$ tangentially having exactly second-order contact  with the boundary and remaining in the complement of $\partial\overline{\Omega}$, e.g. ($\{.,.\}$ denoting the usual Poisson bracket)
\[
\tau^2=R(0,y_0,\eta_0), \quad \{\xi^2+R(x,y,\eta),x\}|_{\kappa_0}=2\xi_0=0,
\]
\[ \{\{\xi^2+R(x,y,\eta),x\},\xi^2+R(x,y,\eta)\}|_{\kappa_0}=2\partial_xR(0,y_0,\eta_0)>0\,.
\]
 We work near such a gliding point and assume (without loss of generality) that $y_0=0$, hence $\kappa_0=(0,0;0,\eta_0)$. On the boundary
and for $(0,y)$ near $(0,0)$, the metric reads $\xi^2+\sum_{j,k=1}^{d-1}R^{j,k}(0,y)\eta_j\eta_k$, and, using again \cite[III,Appendix C]{Hormander}, we assume moreover that {$R^{j,k}(0,0)=\delta_{j,k}$}, and define
\begin{equation}
  \label{eq:R01}
  R_0(y,\partial_y)  :=R(0,y,\partial_y)=\sum_{j}\partial_{y_j}^2+O(|y|)\,,\quad   R_1(y,\partial_y)  :=\partial_x R(0,y,\partial_y)=\sum_{j,k}R_1^{j,k}(y)\partial_{y_j}\partial_{y_k}\,.
\end{equation}
Strict convexity for $\partial\Omega$ is equivalent to $R_1$ being elliptic (its symbol is  positive definite). Define our model Laplacian $\Delta_{M}$ and (Fourier) multipliers $\tau, q$
\begin{equation}
  \label{eq:LapM}
\Delta_M=\partial^2_x+\sum_{j}\partial_{y_j}^2+ x\sum_{j,k}R_1^{j,k}(0)\partial^{2}_{y_jy_k}\,,\, q(\eta)=\sum_{j,k}R_1^{j,k}(0)\eta_j \eta_k\,,\,  \tau(\omega,\eta)=\sqrt{ |\eta|^2+\omega q(\eta)^\frac 2 3}\,.
\end{equation}
With $q_{j,k}:=R^{j,k}_1(0)$, $\Delta_M$ coincides with the model operator $\Delta_F$ from \eqref{eq:LapF}, that had been used in \cite{doiSF21}. In what follows, we use the notation $\Delta_M$ as in \cite{ILLP} (because the coefficients depend on $R$).
Later we will use various symbols $p(x,y,\eta,\omega,\sigma)$ that will be defined in a conic neighborhood of the set
\begin{equation}
  \label{eq:N0}
 N_0=\{x=0\,,\,\,y=0\,,\,\,\omega=0\,,\,\,\sigma=0\,,\,\,\eta\in \R^{d-1}\setminus\{0\}\}\,,
\end{equation}
with homogeneity defined as $ \lambda(x,y,\eta,\omega,\sigma)\rightarrow (x,y, \lambda \eta, \lambda^{2/3}\omega, \lambda^{1/3}\sigma)$.
A symbol of degree $m$ will be a smooth function defined in a conic neighborhood of $N_0$
with asymptotic expansion $p= \sum_{k\geq 0}p_{m-k/3}(x,y,\eta,\omega,\sigma)$,
with $p_{m-k/3}$ homogeneous of degree $m-k/3$.

The following result, essential for our purposes, is due to Melrose and Taylor \cite{meta, mel76} (here we only recall the statement ; see the remarks following \cite[Theorem 2.1]{ILLP} for explanations and references).

\begin{theoreme}\label{thmMelrose}[Melrose-Taylor, Zworski]
{Let $\tau(\omega,\eta)$ be defined in \eqref{eq:LapM}.
There exist a neighborhood $U$ of $(x,y)=(0,0)$, phase functions
$\psi,\zeta$, symbols $p_0,p_1$ and a function $e_0$, all depending
on $(x,y,\eta,\omega)$, such that}
  \begin{itemize}[leftmargin=5.5mm]
  \item $\psi$ is homogeneous of degree $1$, $\zeta=\omega-x|\eta|^{2/3} e_0(x,y,\eta/|\eta|,\omega/|\eta|^{2/3})$ is homogeneous of degree $2/3$ (e.g. $e_{0}$ is homogeneous of degree $0$), $p_0$, $p_1$ are symbols of degree $0$;
    \item the phase functions $\psi$ and $\zeta$ are solutions to the following eikonal equations 
    \begin{equation} \label{systeikeq} 
<\nabla_{(x,y)}\psi,\nabla_{(x,y)}\psi>+\zeta<\nabla_{(x,y)}\zeta,\nabla_{(x,y)}\zeta>=\tau^2(\omega,\eta),\quad    <\nabla_{(x,y)}\psi,\nabla_{(x,y)}\zeta>=0.
 \end{equation}
 Here $<.,.>$ is the symmetric bilinear form obtained by polarization of the principal symbol $\xi^2+R(x,y,\eta)$ of the operator $\Delta$ (which is a second order homogeneous polynomial).
\item the function $G(x,y,\eta,\omega)$ below is a quasimode in $U$:
  \begin{gather}
    \label{eq:defG}
        G(x,y,\eta,\omega)=e^{i\psi}(p_0 Ai(-\zeta)+ ix p_1 q^{-1/6}(\eta) Ai'(-\zeta))\\
  \label{eq:eqG}
  -\Delta G=\tau^2 G+O_{C^\infty}(\tau^{-\infty})\,.
\end{gather}
Moreover, $p_0$, $e_0$ are elliptic and $e_0>0$ near any $(0,0,\eta,0)$ with $\eta\in \R^{d-1}\setminus\{0\}$.
  \end{itemize}
\end{theoreme}
Melrose's classification theorem for glancing
hypersurfaces (see \cite{mel76}) applies, in the non-homogeneous setting, locally near any point in the set $\Sigma_0=\{(X,Y,\Xi,\Theta)\,,\,\, X=0\,,\,\,Y=0\,,\,\,
\Xi=0\,,\,\,|\Theta|=1\}$.
Therefore, there exists a canonical transformation $\canonchi_M$ such that, near $\Sigma_0$
\begin{equation}
  \label{eq:melrose}
  \canonchi_M(\{X=0,\Xi^2+|\Theta|^2+Xq(\Theta)=1\})=\{x=0,\xi^2+R(x,y,\eta)=1\}\,.
\end{equation}
The crucial fact that such a canonical transformation $\canonchi_{M}$ may actually be defined in the whole neighborhood of $\Sigma_{0}$ then follows from the transversality of the Hamiltonian flow with respect to $\Sigma_{0}$. The following proposition, proved in \cite[Prop.2.1]{ILLP}, will be essential for us.

\begin{prop}\label{lemgamma}
The generating function for $\canonchi_M$ writes $(x-X)\Xi+(y-Y)\Theta+\Gamma(x,y,\Xi,\Theta)$, where $\Gamma(0,y,\Xi,\Theta)$ is independent of $\Xi$ (as $\canonchi_M(\{X=0\})=\{x=0\}$) and
\begin{equation}
  \label{eq:GAB}
\Gamma(x,y,\Xi,\Theta)=B_{\Gamma}(y,\Theta)+xA_{\Gamma}(x,y,\Xi,\Theta).
\end{equation}
The transformation $\canonchi_M$ is then generated by the following relations:
\begin{equation} \label{genchi} 
\left\{ \begin{array}{l}
X=x+x\frac{\partial A_{\Gamma}}{\partial \Xi}(x,y,\Xi,\Theta)\,,\,\,Y=y+\frac{\partial B_{\Gamma}}{\partial \Theta}(y,\Theta)+x\frac{\partial A_{\Gamma}}{\partial \Theta}(x,y,\Xi,\Theta)\\ 
\xi=\Xi+A_{\Gamma}(x,y,\Xi,\Theta)+x\frac{\partial A_{\Gamma}}{\partial x}(x,y,\Xi,\Theta)\,,\,\,
\eta=\Theta +\frac{\partial B_{\Gamma}}{\partial y}(y,\Theta)+x\frac{\partial A_{\Gamma}}{\partial y}(x,y,\Xi,\Theta)\,.
\end{array} \right.
 \end{equation}
There exists an elliptic symbol $p(x,y,\eta,\omega,\sigma)$ of degree $0$,
supported near $N_0$ (from \eqref{eq:N0}) and
\begin{equation}\label{eq:Gosc}
G(x,y,\eta,\omega):=\frac{1}{2\pi} e^{-i\tau B_{\Gamma}(0,\eta/\tau)}\int e^{i(y\eta+s^3/3+s(xq^{1/3}(\eta)-\omega)+\tau\Gamma(x,y,s q^{1/3}(\eta)/\tau,\eta/\tau))}p(x,y,\eta,\omega,s)ds
\end{equation}
such that Theorem \ref{thmMelrose} holds with this $G$.
\end{prop}
\begin{rmq}\label{rmqvarphigamma}
Set $\varphi_{\Gamma}(x,y,\Xi,\Theta):=x\Xi+y\Theta+\Gamma(x,y,\Xi,\Theta)$, then the canonical transformation $\chi_M$ is such that $\chi_M(\partial_{\Xi}\varphi_{\Gamma},\nabla_{\Theta}\varphi_{\Gamma},\Xi,\Theta)=(x,y,\partial_x\varphi_{\Gamma},\nabla_y\varphi_{\Gamma})$.
\end{rmq}
\subsection{Global spectral localization and a finite microlocal partition}
\label{sec:global-microlocal-partition}

We first relate the local construction near glancing to the globally
spectrally localized propagator in Theorem~\ref{thmdisperS}. Let
$\widetilde\psi\in C^\infty_0((0,+\infty))$ such that
$\widetilde\psi=1$ on a neighborhood of $\supp\psi$, and set
\begin{equation}\label{eq:global-spectral-cutoffs}
 \Pi_h=\psi(-h^2\Delta_g),\qquad
 \widetilde\Pi_h=\widetilde\psi(-h^2\Delta_g).
\end{equation}
Thus $\Pi_h\widetilde\Pi_h=\widetilde\Pi_h\Pi_h=\Pi_h$.  We denote by
$\mathcal U_h(t)$ the exact Dirichlet propagator associated with
\eqref{SE}; the operator whose kernel has to be estimated is
\begin{equation}\label{eq:global-localized-kernel}
 \mathcal K_h(t)=\Pi_h\mathcal U_h(t)\Pi_h.
\end{equation}
The auxiliary cutoff $\widetilde\Pi_h$ acts as the identity on
the spectral support of $\Pi_h$; we use it to keep the same frequency scale after
spatial and microlocal localization, up to negligible remainders.
Here $\Pi_h$ is a single spectral cutoff away from zero. In
Section~\ref{sec:defoc-cubic-nonl} we will define the Littlewood--Paley
operators $\Delta_j$, using the same global Dirichlet functional
calculus with $h=2^{-j}$ and cutoffs forming a partition of unity.

We will use the following local comparison for both interior and boundary
patches; its proof is in Appendix~\ref{app:common-localization}.
\begin{lemma}[Short-time local comparison]\label{lem:boundary-locality}
  Let $\Delta_{\mathrm{loc}}$
  be the Laplace–Beltrami operator on a smooth compact manifold whose metric agrees with g on a enlarged coordinate neighborhood.
  For an interior patch we pick such manifold to have no boundary; for a boundary patch we retain the original boundary and Dirichlet condition throughout that neighborhood. Let $\chi_j,\chi_k$ be smooth spatial cutoffs supported in the common neighborhood, at   a fixed positive distance from the region where the two geometries may differ.
For every fixed $f\in C^\infty_0((0,+\infty))$ and sufficiently small
$T_0>0$,
\begin{equation}\label{eq:boundary-locality}
 \chi_k\bigl[\mathcal U_h(t)f(-h^2\Delta_g)
 -e^{-ith\Delta_{\mathrm{loc}}}f(-h^2\Delta_{\mathrm{loc}})\bigr]\chi_j
 =O_{C^\infty}(h^\infty),\qquad |t|\le T_0.
\end{equation}
We identify the kernels, including their smooth density factors, in the
common coordinates. The remainder is uniform in both spatial variables,
including the physical boundary, and after any fixed number of time
derivatives. The time restriction depends on the spectral support and
the distance to the boundary of the patch; it is independent of
the distance to the physical boundary. All constants may be chosen uniformly over a fixed finite atlas.
\end{lemma}
We use this lemma with $f=\psi^2$ for $\mathcal K_h$, and with the
auxiliary spectral cutoffs when needed. At $t=0$ it also identifies the
localized spectral multipliers.
We choose a finite smooth partition of unity such that
\begin{equation}\label{eq:finite-spatial-partition}
 1=\sum_{j\in\mathcal J_{\mathrm{int}}}
       (\chi_j^{\mathrm{int}})^2
   +\sum_{k\in\mathcal J_{\partial}}
       (\chi_k^{\partial})^2.
\end{equation}
The cutoffs $\chi_j^{\mathrm{int}}$ are supported in coordinate
patches compactly contained in $\Omega$, while the
$\chi_k^{\partial}$ are supported in boundary normal-coordinate
patches of the form used above.
The latter patches cover a fixed tubular neighborhood of $\partial\Omega$.  After reducing
$T_0$, if necessary, the interior patches may be enlarged so that every
bicharacteristic segment of length at most $T_0$ that is relevant to a fixed
interior term remains in a boundaryless coordinate neighborhood.  By nonstationary phase, terms whose spatial supports
{cannot be joined by such a segment in this time have an
$O(h^\infty)$ kernel, together with all the fixed-order derivatives used below.} This gives the short-time locality
we need for the Schrödinger equation.

We first apply the physical partition in both kernel variables, giving
boundary pieces $\chi_k^\partial\mathcal K_h(t)\chi_j^\partial$,
and then microlocalize them. The global spectral multipliers stay
inside each piece: no spatial cutoff is commuted with them.
We next refine each boundary term in phase space.  We use the same notation
for $R_0(y,\partial_y)$ in \eqref{eq:R01} and for its principal symbol
$R_0(y,\eta)$.  In the space-time cotangent coordinates introduced above,
$\tau$ is the wave-frequency variable. On a fixed compact neighborhood of the spectral support
selected by $\widetilde\psi$, we choose
$\delta>0$ and symbols $b_{\mathrm{hyp}}$, $b_{\mathrm{gl}}$ and
$b_{\mathrm{ell}}$, smooth in $(\tau,y,\eta)$ and even in $\tau$, such that
 $b_{\mathrm{hyp}}+b_{\mathrm{gl}}+b_{\mathrm{ell}}=1$,
with
\begin{align}
 \supp b_{\mathrm{hyp}}
   &\subset\{\tau^2-R_0(y,\eta)>\delta\},
   \label{eq:hyperbolic-region}\\
 \supp b_{\mathrm{gl}}
   &\subset\{|\tau^2-R_0(y,\eta)|<2\delta\},
   \label{eq:glancing-region}\\
 \supp b_{\mathrm{ell}}
   &\subset\{\tau^2-R_0(y,\eta)<-\delta\}.
   \label{eq:elliptic-region}
\end{align}
The glancing symbol $b_{\mathrm{gl}}$ is divided into
finitely many conic neighborhoods of points
\begin{equation}\label{eq:finite-glancing-cover}
 (\tau;x,y;\xi,\eta)=(\tau_0;0,y_0;0,\eta_0),\qquad
 R_0(y_0,\eta_0)=\tau_0^2,
\end{equation}
small enough that the canonical transformation $\canonchi_M$, its generating
function $\Gamma$, and the Fourier integral operator $J$ constructed below
are available with uniform symbol bounds.
Since the spectral cutoff restricts $h\sqrt{-\Delta_g}$ to a fixed compact interval,
 only finitely many such
neighborhoods are needed, independently of $h$.  {Here $\tau$ is the semiclassical wave-frequency parameter. For the semiclassical Schrödinger flow $\mathcal U_h(t)=\exp(-ith\Delta_g)$, the corresponding time-frequency variable, associated with $hD_t$ with $D_t=-i\partial_t$, is $\tau^2$. {We therefore use $b_{\mathrm{reg}}(\sqrt{\cdot},y,\eta)$ when quantizing in $(t,y)$, with $x$ fixed, on the positive spectral window under consideration.} One may also use a finite subdivision of the spectral interval as in Section~\ref{sec:elliptic-contribution}. Throughout this reduction, time-frequency localization refers to Fourier multipliers in $t$, whereas a time cutoff means multiplication by a smooth function of $t$.
On the homogeneous flow, $\psi(hD_t)\mathcal U_h(t)=\mathcal U_h(t)\psi(-h^2\Delta_g)$ identifies time-frequency localization with spectral localization. The realization by Fourier multipliers in time is applied after physical localization, since these multipliers commute with time-independent spatial cutoffs. This does not justify applying the spatial functional calculus to the already spatially cut-off solution.}
The identities in Appendix~\ref{app:boundary-spectral-matching}
will relate this scalar spectral cutoff to the local parametrix.

We microlocalize only on the observation side. Write
$\chi_j$ for any member of the physical partition, and let
$\mathcal J=\mathcal J_{\mathrm{int}}\sqcup\mathcal J_\partial$.
The exact physical reconstruction is
\begin{equation}\label{eq:global-operator-partition}
 \mathcal K_h(t)=\sum_{k,j\in\mathcal J}
 \chi_k\,[\chi_k\mathcal K_h(t)\chi_j]\,\chi_j.
\end{equation}
For each boundary observation patch, quantize the preceding symbols as
{operators in $(t,y)$, acting at fixed $x$. We choose larger cutoffs equal to
one on the supports of the partition. With a finite spectral subdivision,
this gives the families}
$\{\mathcal Q_{k,\ell,h}^{\mathrm{reg}}\}_{\ell}$,
$\mathrm{reg}\in\{\mathrm{hyp},\mathrm{gl},\mathrm{ell}\}$.
We use these glancing operators in
Lemma~\ref{lem:direct-glancing-matching}. The symbols are chosen to sum
to one on {an enlarged compact set of time and tangential frequencies}
containing the characteristic set of the spectrally localized kernel.
{Quantizing their sum in the same convention gives the identity there.
When composing with the auxiliary cutoffs, the symbol composition
corrections can be chosen within the larger supports, giving the
reconstruction below modulo $O(h^\infty)$.} The complementary tangential
frequencies are elliptic and are removed by
Lemma~\ref{lem:elliptic-spectral-removal}, proved in
Appendix~\ref{app:elliptic-contribution}, also at large
parameter-elliptic tangential frequencies. {Time frequencies outside}
the enlarged spectral window are absent before time truncation and give
$O(h^\infty)$ on the smaller time interval after truncation.
Consequently, for boundary indices $k,j$,
\begin{equation}\label{eq:global-kernel-partition}
 \chi_k^\partial\mathcal K_h(t)\chi_j^\partial
 =\sum_{\mathrm{reg}\in\{\mathrm{hyp},\mathrm{gl},\mathrm{ell}\}}
   \sum_\ell\mathcal Q_{k,\ell,h}^{\mathrm{reg}}
        [\chi_k^\partial\mathcal K_h(\cdot)\chi_j^\partial](t)
   +O(h^\infty)\,,
\end{equation}
and the remainder is uniform in both kernel variables. In the elliptic
{terms, after subdividing the spectral interval, each piece consists of a
tangential operator followed by a scalar time-frequency multiplier $f_\ell(hD_t)$.}
Since these multipliers commute with physical cutoffs, their action
on the homogeneous flow is exactly $f_\ell(-h^2\Delta_g)$ before
physical localization. Lemma~\ref{lem:elliptic-spectral-removal}, with
the physical cutoffs included in its nested supports, therefore applies.
This is not a commutation of a spatial cutoff with a spectral multiplier.

The sums in \eqref{eq:global-operator-partition} and
\eqref{eq:global-kernel-partition} are finite. Pairs of patches which
cannot be joined in the chosen time contribute $O(h^\infty)$ and may
be kept in the sum.
For the remaining pairs, we choose enlarged common coordinate neighborhoods
before fixing $T_0$, and use Lemma~\ref{lem:boundary-locality}.
\subsection{The strictly interior contribution}
\label{sec:strictly-interior-contribution}
For the finite partition of Section~\ref{sec:global-microlocal-partition},
the interior pieces satisfy, after possibly decreasing $T_0$,
\begin{equation}\label{eq:interior-piece-dispersion}
 \big\|\mathcal K_h(t)(\chi_j^{\mathrm{int}})^2
       \big\|_{L^1(\Omega)\to L^\infty(\Omega)}
 \leq C h^{-d}
       \min\left\{1,\left(\frac h{|t|}\right)^{d/2}\right\},
 \qquad 0<|t|\leq T_0 .
\end{equation}
The same bound holds with the interior cut-off on the observation side,
since $\mathcal K_h(t)^*=\mathcal K_h(-t)$. Thus the terms for which the source or observation point lies in the interior satisfy the free
dispersive bound, uniformly over our finite atlas. We are left with
terms with both source and observation points near the boundary.
The reduction to the boundaryless estimates of
\cite[Proposition~2.1, Corollary~2.4, and Lemmas~2.5 and~2.7]{bgt04}
is proved in Appendix~\ref{app:strictly-interior-contribution}.
Throughout these reductions, a smoothing $O(h^\infty)$ remainder means
that its kernel and all derivatives of any fixed order are $O(h^N)$
for every $N$.

\subsection{The uniformly hyperbolic boundary contribution}
\label{sec:hyperbolic-contribution}
In the region \eqref{eq:hyperbolic-region}, the fixed separation from
glancing allows us to use the classical geometrical-optics construction of
\cite[Section~1.3]{meta}; see also \cite{bss08,bss12} for different parametrices'constructions, from which our estimates could be extracted with appropriate care.
\begin{prop}\label{prop:hyperbolic-dispersion}
{
Let $\mathcal K_h^{\mathrm{hyp}}$ be a hyperbolic piece of
\eqref{eq:global-kernel-partition}, with rays remaining uniformly
transverse to the boundary at every reflection. For sufficiently
narrow spectral windows and sufficiently small $T_0$, it satisfies
}
\begin{equation}\label{eq:hyperbolic-kernel-bound}
 \|\mathcal K_h^{\mathrm{hyp}}(t)\|_{L^1(\Omega)\to L^\infty(\Omega)}
 \leq C h^{-d}
     \min\left\{1,\left(\frac h{|t|}\right)^{d/2}\right\},
 \qquad 0<|t|\leq T_0.
\end{equation}
The constant is uniform in the source and observation points, irrespective of their distances to the boundary.
\end{prop}
We prove this bound in Appendix~\ref{app:hyperbolic-contribution}, using the Melrose--Taylor parametrix of \cite{meta}.

\subsection{The elliptic boundary contribution}
\label{sec:elliptic-contribution}
{
In each boundary patch $k$, cover the compact elliptic support
by finitely many products of spectral intervals and tangential
neighborhoods. Let $f_\ell$ denote the smooth spectral cutoffs
and $\mathcal A_{k,\ell,h}^{\mathrm{ell}}$ the associated properly
supported semiclassical pseudodifferential cutoffs in $y$,
with $x$ as a parameter. We choose their nested supports so that
$R(x,y,\eta)-\tau^2\geq\delta/2$ on the enlarged supports whenever
$\tau^2\in\supp f_\ell$. These tangential operators preserve
the Dirichlet trace. Boundary elliptic regularity then gives
the following estimate.
}
\begin{lemma}\label{lem:elliptic-spectral-removal}
For the nested cutoffs just described, and every $M,N\geq0$,
\begin{equation}\label{eq:elliptic-spectral-smoothing}
 \big\|\mathcal A_{k,\ell,h}^{\mathrm{ell}}
       f_\ell(-h^2\Delta_g)\big\|_{L^2(\Omega)\to H^M(\Omega)}
       \leq C_{M,N}h^N .
\end{equation}
Moreover, the kernels of
$\mathcal A_{k,\ell,h}^{\mathrm{ell}}f_\ell(-h^2\Delta_g)
\mathcal U_h(t)\Pi_h$ and of the corresponding source-side operators
are $O(h^\infty)$ in every $C^m(\overline\Omega\times\overline\Omega)$
norm, uniformly for $|t|\leq T_0$.
\end{lemma}
We give the construction, the references and a sketch of proof for this elliptic estimate
in Appendix~\ref{app:elliptic-contribution}.

\subsection{Spectral matching of the boundary parametrices}
\label{sec:boundary-spectral-matching}
Lemma~\ref{lem:boundary-locality} identifies the globally localized kernel
on the boundary patches with that of a local
Dirichlet problem, uniformly up to the physical boundary. We next
compare this local kernel with our glancing parametrix.
For each glancing piece, we choose {a properly supported operator in the tangential variables and in time} $\mathcal Q_{\nu,h}^{\mathrm{gl}}$, acting at fixed $x$,
with compact frequency support and uniform symbol bounds up to $x=0$.
{
We subdivide the spectral window into finitely many pieces.
On each piece, we use a fixed positive multiple of $h$ to place the frequency
support  of $\mathcal Q_{\nu,h}^{\mathrm{gl}}$ where the multipliers of Definition~\ref{dfnparametrix}
equal one. We adjust the semiclassical time inversely, so that
the propagator and the physical time are unchanged.
This normalization is used only for the local comparison;
returning to the original parameters changes the estimates
only by fixed constants.
We choose additional frequency cutoffs equal to one on a neighborhood
of this support and supported where these multipliers equal one.

For the spatial and time localizations, we choose the source and
observation supports and the time interval within the region
of validity of the parametrix comparison, with a fixed margin
to its spatial and temporal limits. This imposes no separation
from the physical boundary $x=0$.}
\begin{lemma}\label{lem:direct-glancing-matching}
{
In each boundary patch, the kernel of the corresponding glancing
piece in \eqref{eq:global-kernel-partition} agrees, modulo a smoothing
$O(h^\infty)$ remainder, with the kernel obtained by applying
$\mathcal Q_{\nu,h}^{\mathrm{gl}}\psi^2(hD_t)$ to the local parametrix
of Definition~\ref{dfnparametrix}, physically localized in both
spatial variables. The remainder is uniform in the source and
observation points, including at the boundary, and remains
$O(h^\infty)$ after any fixed number of observation and time
derivatives. The local comparison follows from
Propositions~\ref{propsommefinie} and~\ref{propparama<}.
}
\end{lemma}
Moreover, these cutoffs preserve a pointwise bound of the form
$Ch^{-d}(1+|t|/h)^{-d/2+1/4}$, with uniform constants.
Both the lemma and such property are proven in Appendix~\ref{app:boundary-spectral-matching}.
We may therefore concentrate on the local glancing parametrix and return to the full propagator in
Section~\ref{sec:completion-linear-proof}.

\subsection{Properties of the generating function $\Gamma$}\label{sec:generating-function-properties}
A detailed construction of the function $\Gamma$ of Proposition \ref{lemgamma} has been provided  in \cite[Appendix 6.2]{ILLP}. We briefly review its properties below.

Set $\Theta=\varrho\omega$ with $\varrho =|\Theta|$ near $1$ and $\omega=\Theta/|\Theta|$. Functions $A_{\Gamma}, B_{\Gamma}$ are to be defined near the glancing set $\mathcal{GL}=\{x=0,\Xi=0,\varrho-1=0\}$ and for $(y,\omega)$ near $\{0\}\times \mathbb{S}^{d-2}$. We work with formal Taylor expansions $F$ near $\mathcal{GL}$ such that $F=\sum_{a,b,c} f_{a,b,c}(y,\omega)X^a(\varrho-1)^b\xi^c$.
We attribute a degree to each factor $x, \varrho-1, \Xi$: a monomial of the form $x^a(\varrho-1)^b\Xi^c$ is homogeneous of degree $k$ if and only if $c+2(a+b)=k$. 
For such a  formal series $F(x,y,\Xi,\varrho,\omega)$, defined near $\mathcal{GL}$, we write $F= \sum_{k\geq 0} F_k$, where $F_k$ is homogeneous of degree $k$; we also write $F\in \mathcal{H}_{\geq j}$ if and only if $F=\sum_{k\geq j} F_k$. Therefore, $F_0=f_0(y,\omega)$, $F_1=\Xi f_1(y,\omega)$, $
F_2=X f_2^0(y,\omega)+(\varrho-1)f_2^1(y\omega)+\Xi^2f_2^2(y,\omega)$, and so on. Replacing $X,\xi,\eta$ by their formulas \eqref{genchi} (as functions of $(x,y,\Xi,\Theta)$) and using that from \eqref{eq:melrose} 
\begin{equation}\label{eqiffchi}
\xi^2+R(x,y,\eta)=1\quad \text{ if and only if }\quad \Xi^2+|\Theta|^2+Xq(\Theta)=1,
\end{equation}
(where we observe that there is no $Y$ in the second equation), we get
\begin{equation}\label{formalseriesBA}
B_{\Gamma}= \sum_{j\geq 0} (\varrho-1)^j B_{2j}(y,\omega)\,,\quad
A_{\Gamma}= \sum_{k\geq 1}A_k\,.
\end{equation}
Using \eqref{genchi} and $\xi|_{\mathcal{GL}}=0$, we have $A_{0}=0$. We also have $\xi(x,y,\Xi,\Theta)\in \mathcal{H}_{\geq 1}$ and $X(x,y,\Xi,\Theta)\in \mathcal{H}_{\geq 2}$. Moreover, from the proof of Melrose's theorem \cite{mel76}, if formal {series} of the form \eqref{formalseriesBA} satisfy \eqref{genchi} and \eqref{eqiffchi}, then there exist $C^{\infty}$ functions $A_{\Gamma}, B_{\Gamma}$ with the same Taylor expansion near $\mathcal{GL}$ satisfying \eqref{genchi} and \eqref{eqiffchi}. We now summarize the results of \cite[Lemma 6.3, 6.4, Prop.6.5]{ILLP}.
\begin{prop}\label{propimpformgamma}
The phase function $\Gamma(x,\cdot)=B_{\Gamma}+x A_{\Gamma}$ is such that, near the glancing set $\mathcal{GL}$,
\begin{equation}\label{AGam}
\left\{ \begin{array}{l}
A_{\Gamma}(x,y,\Xi,\Theta)= \Xi\Lp(y,\omega)+\alpha_{\Gamma}(y,\omega)x+{\mu_\Gamma}(y,\omega)(\Xi^2+|\Theta|^2-1)+\mathcal{H}_{j\geq3},\\
B_{\Gamma}(y,{\Theta})=B_0(y,\omega)+(\varrho-1)B_2(y,\omega)+\mathcal{H}_{j\geq3},
 \end{array} \right.
\end{equation}
where $\omega=\Theta/|\Theta|$, $\varrho=|\Theta|$. The functions $B_0$, $B_{2}$ and $\Lp$ are obtained explicitly using only homogeneous contributions up to order $2$ and are defined as follows. Firstly,
\begin{equation}\label{eqB0}
B_0(y,\omega)=\phi(y,\omega)-y\cdot\omega\,,
\end{equation}
where, for $\eta\in \mathbb{R}^{d-1}\setminus 0$, $\phi(y,\eta)$ is the unique function, homogeneous of degree $1$ in $\eta$, solving the eikonal equation $R_0(y,\nabla_y\phi)=|\eta|^2$, with $\phi_{|y\cdot\eta=0}=0$ (see \cite[Lemma 6.4, (6.25)]{ILLP}). Secondly, 
\begin{equation}\label{eqorder2}
1+\Lp(y,\omega)=\Big(\frac{R_1(y,\omega+\nabla_yB_0(y,\omega))}{q(\omega)}\Big)^{1/3},
\end{equation} 
and as $R_1(0,\omega)=q(\omega)$ and $\nabla_yB_0(0,\omega)=0$, then $\Lp(0,\omega)=0$.
The function $B_2(y,\omega)$ satisfies 
\begin{equation}\label{eqB2tr}
(\omega+\nabla_yB_2)\cdot \nabla_{\eta}R_0(y,\omega+\nabla_yB_0(y,\omega))=\frac{2R_1(y,\omega+\nabla_yB_0(y,\omega))}{q(\omega)(1+\Lp)}\,,
\end{equation}
which is a linear transport equation and we can take $B_{2|y.\omega=0}=0$.
Moreover, $\Lp(0,\omega)=0$, $B_0(0,\omega)=0$, $\nabla_yB_0(0,\omega)=0$, $B_2(0,\omega)=0$ and $\nabla_yB_2(0,\omega)=0$. {For nonzero tangential frequencies we use the degree-one extension
$B_0(y,\theta)=|\theta|B_0(y,\theta/|\theta|)$.
The coefficients $B_{2j}$, $j\ge1$, are extended with degree zero:
$B_{2j}(y,\theta)=B_{2j}(y,\theta/|\theta|)$.
In \eqref{formalseriesBA} all coefficients are evaluated on the unit
sphere; their subscripts refer to weighted Taylor order.} The functions {$\alpha_\Gamma$ and ${\mu_\Gamma}$} are related by
\begin{equation}\label{formalsystem}
{\mu_\Gamma}(y,\omega)=\frac{\alpha_\Gamma(y,\omega)}{q(\omega)(1+\Lp)}-\frac{\nabla_{\eta}R_0(y,\omega+\nabla_yB_0)}{4q(\omega)}\nabla_y\Big(\frac{1}{1+\Lp}\Big).
\end{equation}
\end{prop}
\begin{rmq}
That the formal expansion is not uniquely defined reflects that the
group of canonical transformations which preserves the model $\{X=0, \Xi^2+\vert\Theta\vert^2+Xq(\Theta)=1\}$ is not trivial.
\end{rmq}

We start with a representative of $\Gamma$ for which $\alpha_\Gamma$
need not vanish. To normalize this coefficient, we compose on the model
side with the Hamilton flow generated by
$F=X_M(p_M-1)b(Y_M,\Theta)$, where
$p_M=\Xi^2+|\Theta|^2+X_Mq(\Theta)$ and $b$ is a function to be chosen.
This flow preserves $\{X_M=0\}$ and $\{p_M=1\}$ and fixes their
intersection. Its action on the coefficients is
$\alpha_\Gamma^{\rm new}=\alpha_\Gamma+q(1+\Lp)^2b$ and
$\mu_\Gamma^{\rm new}=\mu_\Gamma+(1+\Lp)b$,
while the remaining terms have higher weighted degree
(see the Appendix in \cite{ILLP}). Since $q(1+\Lp)^2>0$,
choosing $b=-\alpha_\Gamma/[q(1+\Lp)^2]$, with $\alpha_\Gamma$
the coefficient before this transformation, gives
$\alpha_\Gamma^{\rm new}=0$. This does not change $B_\Gamma$ or $\Lp$.
We drop the superscript $\rm new$ and keep the same notations for
the transformed coefficients.

The third-order coefficients are described in
Appendix~\ref{app:canonical-normalization}. With the normalization
$\alpha_\Gamma=0$, the expansion, before the next transformation, is
\begin{equation}\label{eq:normalized-AGamma}
 A_\Gamma=\Lp\Xi+\mu_\Gamma(\Xi^2+\rho^2-1)
 +\alpha_1x\Xi+\beta_1(\rho-1)\Xi+\mu_1\Xi^3+R_A,
 \qquad R_A\in\mathcal H_{\ge4},\quad \rho=|\Theta|.
\end{equation}
Here $\alpha_1,\beta_1,\mu_1$ are functions of
$(y,\Theta/|\Theta|)$; in particular $\beta_1$ is the coefficient of
$(\rho-1)\Xi$. We now use the residual canonical transformation
specified in Appendix~\ref{app:canonical-normalization} to make the
new coefficient $\beta_1$ vanish. Keeping the same notations after
this transformation, we henceforth use
\eqref{eq:normalized-AGamma} with $\beta_1=0$.
The coefficient $\alpha_1$ need not vanish.
All subsequent expansions refer to this representative of $\Gamma$.
\subsection{Definition of a parametrix}\label{sec:parametrix-definition}
In what follows, we work in the coordinates of section \ref{secMelcoord}, hence with $\Omega=\Omega_d=\mathbb{R}_+\times \mathbb{R}^{d-1}$ and $g$ the metric associated from the Laplace operator \eqref{eq:laptilde}.
{We work with a smooth auxiliary Dirichlet problem that agrees with the
original problem in an enlarged boundary patch.} By Lemma~\ref{lem:boundary-locality}, its spectrally localized
kernel agrees with the global kernel on the smaller patch, up to
$O(h^\infty)$ uniformly up to the physical boundary and on the
corresponding time interval. We compare the parametrix below with
this local kernel.
Let $a_0>0$ be small and $a\in (0,a_0]$; denote by $\Grond_{S}(t,x,y,a)$ the
Green function for the Schrödinger equation with Dirichlet boundary condition, and $\delta_{(a,0)}:=\delta_{x=a,y=0}$ the source point,
\begin{equation}
\label{schrodeq}
(  i\partial_t -\Delta) \Grond_S=0 \,,\text{ for }\, x>0\,,
 \, \Grond_S|_{x=0}=0, \Grond_S|_{t=0}= \delta_{(a,0)}.
\end{equation}
We will frequently need smooth cutoff functions $\varkappa\geq 0$ in $ C_{0}^{\infty}(\R^{m})$ with $m=1$ or with $m=d-1$. For $m=1$, $\varkappa$ will be such that $\varkappa=1$ near $1$, $\varkappa=0$ outside a small neighborhood of $1$, and for $m=d-1$, $\varkappa$ will be radial and such that ${\varkappa}=1$ near $\mathbb{S}^{m-1}$, ${\varkappa}=0$ outside a small neighborhood of $\mathbb{S}^{m-1}$. We will abuse notation and retain $\varkappa$ as a generic notation, independently of the value of $m$ (which will be clear from context) as well as the size of the (small) support of $\varkappa$, which we assume from now on to be smaller than $0<\ceps_{0}<1/100$.
\begin{dfn}\label{dfnparametrix}
Let $h\in (0,1]$. A function $\mathcal{P}_{S,h,a}(t,x,y)$ is a parametrix for \eqref{schrodeq} if there exist $a_0>0$, $r>0$ and a neighborhood $V$ of $(t,y)=(0,0)$ such that for all $\alpha$ one has  
 \[
\sup_{0<a\leq a_0}\sup_{0<x\leq r}\sup_{(t,y)\in V} \Big|\partial^{\alpha}_{t,x,y}({\varkappa}(hD_{t}){\varkappa}(hD_{y})(\mathcal{P}_{S,h,a}-\Grond_S(h\cdot,\cdot,a))\Big| \in O(h^{\infty}).
 \]
\end{dfn}
\begin{rmq}\label{rmqPcal}
In \cite{ILLP} we denoted $\Grond$ the Green function for the wave equation with Dirichlet condition
\begin{equation}
\label{waveeq}
(  \partial_t^2 -\Delta) \Grond=0 \,,\text{ for }\, x>0\,,
 \, \Grond_{|x=0}=0, \Grond|_{t=0}= \delta_{(a,0)} \text{ and } \partial_t \Grond |_{t=0}=0.
\end{equation}
We had $\Grond=\cos(t\sqrt{|\Delta|})$, but worked with the half wave propagator $e^{\pm it\sqrt{|\Delta|}}$, from which we obtained $\Grond$ and $\partial_{t} \Grond$. A parametrix $\mathcal{P}_{h,a}$ for \eqref{waveeq} was then defined in a very similar way (see \cite[Definition 2.6]{ILLP}).
\end{rmq}
\begin{rmq}
\par\noindent{ 
For a solution of the wave equation, {time-frequency localization}
is a spectral localization for the Dirichlet Laplacian. For the
semiclassical Schr\"odinger flow, the precise identity is
\eqref{eq:temporal-spectral-matching}, {with semiclassical time-frequency variable $\tau^2$, associated with $hD_t$}.
The cutoff $\varkappa(hD_y)$ localizes tangential frequencies in the boundary coordinates; it is distinct from the global spectral localization. We insert larger microlocal cutoffs {equal to one along the rays issuing from the chosen frequency support for the short times under consideration.} The comparison with the globally spectrally localized propagator is established in Section~\(\ref{sec:boundary-spectral-matching}\);
 no commutation of
$\varkappa(hD_y)$ with $\Delta_g$ is required.}
\end{rmq}
We rescale
$\omega=\frac{\alpha}{h^{2/3}}$, $\eta=\frac{\theta}{h}$, $s=\frac{\sigma}{h^{1/3}}$ in \eqref{eq:Gosc} (defining $G$), hence $\tau(\alpha,\theta)= h\tau(\omega,\eta)$. 
Also let $\cutoffchi^{\flat}\in C^{\infty}(\mathbb{R})$ such that $\cutoffchi^{\flat}=1$ on $(-\infty,1]$ and $\cutoffchi^{\flat}=0$ on $[2,\infty)$, and $\cutoffchi^{\sharp}=1-\cutoffchi^{\flat}$. We let
  \begin{equation}    \label{eq:Phi}
    \Phi(x,y,\theta,\alpha,\sigma)=y\cdot \theta+\sigma^3/3+\sigma(xq^{1/3}(\theta)-\alpha)     +\tau(\alpha,\theta)
    \Gamma(x,y,\sigma q^{1/3}(\theta)/\tau(\alpha,\theta),\theta/\tau(\alpha,\theta))\,,
  \end{equation}and as our change of variables $s=\frac{\sigma}{h^{1/3}}$ provides a factor $h^{-1/3}$, we set
\begin{equation}
  \label{eq:qh}
  p_h(x,y,\theta,\alpha,\sigma):=h^{-1/3}p(x,y,\theta/h,\alpha/h^{2/3},\sigma/h^{1/3})\varkappa(\theta)\varkappa(\tau(\alpha,\theta))\cutoffchi^{\sharp}(\alpha/h^{2/3})\,,
\end{equation}
where the relevance of all cutoff functions will reveal itself later on. We get \begin{equation}
  \label{eq:12}
\cutoffchi^{\sharp}(\alpha/h^{2/3})  \varkappa(\theta)\varkappa(\tau(\alpha,\theta))G(x,y,\theta/h,\alpha/h^{2/3})=\frac{1}{2\pi}\int e^{\frac ih  \Phi(x,y,\theta,\alpha,\sigma)}  p_h(x,y,\theta,\alpha,\sigma)d\sigma.
\end{equation}
We now define an operator acting on smooth $f(y',\rho)$, with $\hat f$ its Fourier transform in all variables, 
\begin{equation}
  \label{eq:J}
    J(f)(x,y)=\int G(x,y,\eta,\omega)\cutoffchi^{\sharp}(\omega)q(\eta)^{1/6}\varkappa(h\eta)\varkappa(h\tau(\omega,\eta))\hat{f}(\eta,\omega/h^{1/3})d\eta d\omega\,.
  \end{equation}
  After rescaling and substitution of \eqref{eq:12} in \eqref{eq:J},
\begin{equation}
  J(f)(x,y) 
   = \frac 1 {2\pi h^d} \int e^{\frac ih (\Phi(x,y,\theta,\alpha,\sigma)-y'\cdot\theta-{\tprime}\alpha)}p_h(x,y,\theta,\alpha,\sigma)q(\theta)^{1/6}
  f(y',{\tprime}) \,dy'd{\tprime}d\theta d\alpha d\sigma\,.
\end{equation}
\begin{lemma} (\cite[Lemma 2.8]{ILLP})
The operator $J$ is well defined from tempered distributions $\mathcal{S}'_{y',\tprime}$ into smooth functions of $(x,y)$ near $(0,0)$. In the semiclassical setting with $h$ as small parameter, $J$ is a semiclassical Fourier integral operator associated to a
canonical transformation $\canonchi_J$, defined near the set $\{y'=0,\tprime=0, |\theta|=1,\alpha=0 \}$ and such that
 $ \canonchi_J(y'=0,{\tprime}=0,|\theta|=1,\alpha=0)=\{y=0,x=0,|\theta|=1,\xi=0\}$.

Moreover, $J$ is elliptic on this set and, microlocally near this set, an intertwining relation holds,
\begin{equation}
  \label{eq:LapJ}
  -h^2 \Delta J(f)= J(\tau^2(hD_{{\tprime}},hD_{y'}) f).
\end{equation}
\end{lemma}
As the symbol $p_h$ is smooth and compactly supported in $(\theta,\alpha,\sigma)$, $J$ is easily extended to $\mathcal{S}'_{y',\tprime}$. The Lemma then follows from Theorem \ref{thmMelrose} ($p_0$ is elliptic and $x p_1$ vanishes on $\partial\Omega$.)  

\subsection{Some useful results on Airy functions}\label{sec:airy-preliminaries} 
We now digress and present a variation on the Poisson summation
formula, the "Airy-Poisson summation formula".
For $z\in\mathbb{C}$ we set 
\begin{equation}
  \label{eq:Apm}
  A_\pm(z)=e^{\mp i\pi/3} Ai(e^{\mp i\pi/3} z)\,,\,\,\text{ then } \,
Ai(-z)=A_+(z)+A_-(z)\, \text{ and } \overline{A_+(z)}=A_-(\overline{z}).
\end{equation}
The following two lemmas are proved in \cite[Lemmas 1 and 3]{ilpCE} :
\begin{lemma}\label{lemL}
Define, for $\omega \in \R$, the function $  L(\omega)=\pi+i\log \frac{A_-(\omega)}{A_+(\omega)}$: $L$ is an analytic, real valued, strictly increasing function with $L(0)=\pi/3$, $\lim_{\omega\rightarrow -\infty} L(\omega)=0$, and, for $\omega\geq 1$,
\begin{equation}
  \label{eq:propL}
    L(\omega)=\frac 4 3 \omega^{\frac 3 2} +\frac{\pi}{2} -B_L(\omega^{\frac 3
    2})\,,\quad
  B_L(u)= \sum_{k=1}^\infty b_k u^{-k}\,,\,\, (b_k)_{k}\in\R\,,\,\,
  b_1>0\,.
\end{equation}
Moreover, with $\{-\omega_k\}_{k\geq 1}$ being zeros of the Airy function in decreasing order,
\begin{equation}
  \label{eq:propL2}
  Ai(-\omega_k)=0 \iff L(\omega_k)=2\pi k \text{ and }
  L'(\omega_k)= 2\pi \int_0^\infty Ai^2(x-\omega_k) \,dx\,.
\end{equation}
\end{lemma}
\begin{lemma}
  In $\mathcal{D}'(\R_\omega)$, one has
  \begin{equation}
    \label{eq:AiryPoisson}
        \sum_{N\in \Z} e^{-i NL(\omega)}= 2\pi \sum_{k\in \N^*} \frac 1
    {L'(\omega_k)} \delta(\omega-\omega_k)\,.
  \end{equation}
\end{lemma}
\begin{lemma}\label{lem:airy-wronskian}
For real $\omega$, with primes denoting derivatives in $\omega$,
\begin{equation}\label{eq:airy-wronskian-Lprime}
 A_+'(\omega)A_-(\omega)-A_+(\omega)A_-'(\omega)
 =\frac{i}{2\pi},\qquad
 L'(\omega)=\frac{1}{2\pi|A_+(\omega)|^2}>0.
\end{equation}
There exist $c,C>0$ such that, for $\omega\ge1$,
\begin{equation}\label{eq:Lprime-symbol-bounds}
 c\sqrt\omega\le L'(\omega)\le C\sqrt\omega,\qquad
 |\partial_\omega^j L'(\omega)|\le C_j\omega^{1/2-j},
 \quad j\ge1.
\end{equation}
\end{lemma}
\begin{proof}
The classical connection formula
$A_\pm(\omega)=(Ai(-\omega)\mp iBi(-\omega))/2$
and $Ai(z)Bi'(z)-Ai'(z)Bi(z)=1/\pi$
(see \cite[pp.~393, 414]{Olver97}) give the Wronskian identity in \eqref{eq:airy-wronskian-Lprime}.
Differentiating the definition of $L$ gives the second identity.
The differentiable Airy expansions of \cite[pp.~392--414]{Olver97} (or Lemma \ref{lemL})
give $L'(\omega)=2\sqrt\omega(1+O(\omega^{-3}))$ as
$\omega\to+\infty$, with symbol bounds for the remainder.
This proves \eqref{eq:Lprime-symbol-bounds} for large $\omega$;
positivity and smoothness give the same bounds on the remaining
compact interval, up to adjusting constants.
\end{proof}
Let us define, for $\omega\in \R$, and without loss of generality, an arbitrary choice of $+$ sign for the time propagator,
\begin{equation}
  \label{eq:Kequiv}
  K_{S,\omega}(f)(t,x,y)=\int e^{iht\tau^2(\omega,\eta)} G(x,y,\eta,\omega)\cutoffchi^{\sharp}(\omega) q^{1/6}(\eta)
\varkappa(h\eta) \varkappa(h\tau(\omega,\eta)) \hat f(\eta,\frac{\omega}{h^{1/3}}) \, d\eta\,.
\end{equation}
Then, $J(f)(x,y)=\int_\R K_{S,\omega}(f)(0,x,y)\,d\omega$.
Let $\eta=\frac \theta h$, $q^{\frac 16}(\eta)=h^{-\frac 1 3}q^{\frac 1 6}(\theta)$ and $\alpha=h^{\frac 2 3}\omega$, then
\begin{equation}
  \label{eq:Kom}
 K_{S,\omega}(f)(t,x,y)=\frac{h^{-\frac13}}{2\pi h^{d-1}} \int e^{\frac i
    h(t\tau^2(\alpha,\theta)+\Phi(x,y,\theta,\alpha,\sigma)-y'\cdot \theta-{\tprime} \alpha )}
  p_h  q^{\frac 1 6}(\theta)f(y',{\tprime}) \,dy'd{\tprime}d\theta d\sigma\,.
\end{equation}
Both cutoff functions $\varkappa$ in $K_{S,\omega}(f)$ relate to localization operators in Definition \ref{dfnparametrix}. Moreover, $K_{S,\omega}(f)$ is a suitable test function in $\omega$ (smooth with support in $|\omega|\leq \ceps_0h^{-2/3}$.) Using \eqref{eq:AiryPoisson},\begin{equation}\label{formulafAP}
 \langle \sum_{N\in \Z} e^{-iN L(\omega)}
    , K_{S,\omega}(f)(t,x,y)\rangle_{\omega}=2\pi  \sum_{k\in \N^*} \frac 1
      {L'(\omega_k)}  K_{S,\omega_k}(f)(t,x,y)\,.
\end{equation}
The $N=0$ term in the sum over $N$ is $J(f)$. Moreover, at $x=0$ the RHS vanishes as $G(0,y,\eta,\omega_k)=0$ (recall \eqref{eq:defG} and $\zeta(x,y,\eta,\omega)|_{x=0}=\omega$). By design of $G$ and $K_{S,\omega}$ and using \eqref{eq:eqG}, $(ih\partial_{t}-h^2\Delta)K_{S,\omega}(f)$ is supported in $|\omega|\leq \ceps_0h^{-2/3}$ and such that, for small $r_0$ and $a_0$ and for all (large) $M\in \mathbb{N}$,
\begin{equation}\label{eq:Kom1}
\sup_{|a|<a_0}\sup_{|(t,x,y)|<r_0}\sup_{\omega}\Big|\nabla^{\alpha}_{t,x,y,\omega}(ih\partial_{t}-h^2\Delta)K_{S,\omega}(f)\Big|\leq C_{M,\alpha}h^M\,.
\end{equation}
Up to the infinite sum issue (that will be addressed later), in order to obtain a parametrix in the sense of Definition \ref{dfnparametrix}, it remains to consider finding a suitable function $f$ that will recover the data at $t=0$ in \eqref{formulafAP}. This has been done successfully in \cite{ILLP} for the wave equation, where $K_{S,\omega}$ is replaced by $K_{\omega}$ given by
\begin{equation}
  \label{eq:Kequivwave}
  K_{\omega}(f)(t,x,y)=\int e^{it\tau(\omega,\eta)} G(x,y,\eta,\omega)\cutoffchi^{\sharp}(\omega) q^{1/6}(\eta)
\varkappa(h\eta) \varkappa(h\tau(\omega,\eta)) \hat f(\eta,\frac{\omega}{h^{1/3}}) \, d\eta\,.
\end{equation}
To obtain $K_{S,\omega}$ that satisfies the semiclassical Schrödinger equation we then used the Kanaï transform.

The coarse hyperbolic contribution is controlled by
Proposition~\ref{prop:hyperbolic-dispersion}, and the refined elliptic
pieces are removed by Lemma~\ref{lem:elliptic-spectral-removal}.
Following the localization order and matching procedure in
Section~\ref{sec:boundary-spectral-matching}, we use the existing
local initial-data constructions, starting with
$\chi_0(hD_x)\varkappa(hD_y)\delta_{(a,0)}$, as in \cite{ILLP}.
Here $\varkappa$ selects a tangential frequency annulus and
$\chi_0\in C^\infty_0$ is supported near zero.  This is a local
representative, not an asserted equality with the globally localized
datum of Theorem~\ref{thmdisperS}.  In particular, the assertion of
transversality applies to the part separated from the \emph{invariant}
glancing set within the spectral support selected by
$\widetilde\psi$; it does not replace
the spectral matching step.  The local datum itself is given by
\begin{equation}\label{datachi0var}
  \chi_0(hD_{x})\varkappa(hD_{y})\delta_{(a,0)}=\int e^{i((x-a)\xi+y\cdot\eta)}\chi_0(h\xi)\varkappa(h\eta)d\xi d\eta
  =\frac{1}{h^d}\widehat{\chi}_0\Big(\frac{x-a}{h}\Big)\widehat{\varkappa}\Big(\frac{y}{h}\Big)\,.
\end{equation}
Therefore, at $x=0$, this data is $O(h^{\infty})$ provided $a\geq h^{1-\ceps}$ for some $\ceps>0$. In \cite{ILLP}, in the case of the Friedlander model operator $\Delta_M$ and for small $a\lesssim h^{2/3}$, we took advantage of the known, explicit, spectral resolution of $-\Delta_{M}$ (see \cite[Lemma 2]{doiSF21}) in order to consider initial data $\chi_0(-h^2\Delta_M)\varkappa(hD_y)\delta_{(a,0)}$ that can be further expressed as a sum of eigenfunctions $\sum_{k} e_k(x,\eta)e_k(a,\eta)$ that vanish on the boundary. By contrast, in the general case, we can only work with the quasimodes constructed in \cite[Section 2.6]{ILLP} for $k\ll h^{-1/4}$ ; this is a source of significant difficulties for very small $a\lesssim h^{2/3}$ (notice that, in the model case dealt with in \cite{doiSF21}, the dispersive estimates followed immediately in this regime using basic properties of gallery modes $e_k$).

\subsection{Parametrix construction for $a\geq h^{\frac 23-\ceps}$, $\ceps>0$}\label{sec:parametrix-large-distance}
An initial data \eqref{datachi0var} is $O(h^{\infty})$ on the boundary for any $\chi_0$, compactly supported near $0$. Let $\chi_0\in C^{\infty}_0(-2\ceps_0,2\ceps_0)$ with $\chi_{0\big|[-\ceps_0,\ceps_0]}=1$.
\begin{lemma}(\cite[Lemma 2.12]{ILLP})
Let $a_0>0$, $r_0>0$ be small enough. For all $a\in [h^{\frac 23-\ceps},a_0]$, there exists a smooth function $g_{h,a}$ such that $\varkappa(hD_{y'})g_{h,a}=g_{h,a}$ and
  \begin{equation}
    \label{eq:Jdirac}
    J(g_{h,a})(x,y)=\chi_0(hD_{x})\varkappa(hD_{y})\delta_{(a,0)} + O_{C^{\infty}(|(x,y)|\leq r_0)}(h^{\infty})\,,
  \end{equation}
where the remainder is  $O_{C^{\infty}(|(x,y)|\leq r_0)}(h^{\infty})$ uniformly in $a$ and where $J$ is defined in \eqref{eq:J}.
\end{lemma}
\begin{lemma}(\cite[Lemma 2.13]{ILLP})
\label{lemp}
There exists a smooth phase function $\psi_a({\tprime},\theta')$ and a symbol $r_h(\tprime,\theta')$ of degree $1/3$, supported near $\tprime =0$, $|\theta'|=1$, of the form $r_h(\tprime,\theta')=h^{1/3}\sum_{k\geq 0}r_{h,k}h^k$ with $r_{h,0}(0,\theta')\neq 0$ for $|\theta'|=1$, such that the function $g_{h,a}$ solves \eqref{eq:Jdirac}, with
  \begin{align}
    \label{eq:gosc}
    g_{h,a}(y',{\tprime}) & = h^{-d}\int e^{\frac i h (\psi_a({\tprime},\theta')+y'\cdot \theta')}
    r_h(\tprime, \theta')q^{-1/6}(\theta') \,d\theta' \\
\label{eq:phia}
  \psi_a({\tprime},\theta') & ={\tprime}^3/3+a({\tprime}q^{1/3}(\theta')+O({\tprime}^3))+O(a^2)\,.
\end{align}
Moreover, $\psi_a$ is the critical value of the phase ${\tprime}\alpha-\Phi(a,0,\theta',\alpha,s)$ at critical points in $(\alpha,s)$.
\end{lemma}

\begin{rmq}
The operator $J$ defined in \eqref{eq:J} can be microlocally inverted. Its inverse, denoted $J^{-1}$, may be written in the form
\[
J^{-1}(F)(y',\tprime)=h^{-d-1}\int e^{\frac ih(-\Phi(x,y,\theta',\alpha',s)+y'\cdot\theta'+\tprime\alpha')}q_h(x,y,\theta',\alpha',s)q^{-1/6}(\theta')F(x,y)dxdyd\theta' d\alpha' ds,
\]
where $q_h(x,y,\theta',\alpha',s)$ is a symbol of order $1/3$, $q_h= h^{1/3}\sum_{k\geq 0} h^kq_{h,k}$, supported near $\{x=0,y=0,\alpha'=0,s=0\}$ and elliptic on this set; $J^{-1}$ is such that $J^{-1}\circ J(f)=f\text{ modulo } O(h^{\infty})$ (see the proof of \cite[Lemma 2.12]{ILLP}).  
Defining $g_{h,a}:=J^{-1}(\chi_0(hD_x)\varkappa(hD_y)\delta_{(a,0)})\,$, we obtain 
\begin{equation}\label{ghaF}
\hat{g}_{h,a}(\theta/h,\alpha/h)=h^{-1}q^{-1/6}(\theta)\int e^{-\frac ih \Phi(a,0,\theta,\alpha,s)}\chi(s) \tilde q_h(\theta,\alpha,s) ds,
\end{equation}
where cutoff $\chi(s)\in C^{\infty}_0((-2r,2r))$ is equal to $1$ on $[-r,r]$ for $r\sim \sqrt{\ceps_0}+ a_0$ and where $\tilde q_h$ is an asymptotic expansion with small parameter $h$ and main contribution $q_h(a,0, \theta,\alpha,s)\chi_0(\partial_x\Phi(a,0,\theta,\alpha,s))\chi(\nabla_y\Phi(a,0,\theta,\alpha,s))$.
\end{rmq}
\begin{dfn}
    Let $g_{h,a}$ be defined in \eqref{eq:gosc}: using \eqref{eq:AiryPoisson}, we define $\Prond_{S,h,a}$ equivalently as
    \begin{align}
      \label{eq:Prond}
      \Prond_{S,h,a}(t,x,y) & = \langle \sum_{N\in \Z} e^{-iN L(\omega)}
    , K_{S,\omega}(g_{h,a})(t,x,y)\rangle_{\omega}\\
      \label{eq:Prond2}
      \Prond_{S,h,a}(t,x,y) & =2\pi  \sum_{k\in \N^*} \frac 1
      {L'(\omega_k)}  K_{S,\omega_k}(g_{h,a})(t,x,y)\,.
    \end{align}
    \end{dfn}
Using both localizations in $\eta$ and $\tau(\omega,\eta)$, \eqref{eq:Prond2} may be reduced to a finite sum over $k\lesssim h^{-1}$: recall that support considerations on $K_{S,\omega}$ (as a function of $\omega$) provide $|\omega|\leq \ceps_0 h^{-2/3}$. After Airy-Poisson summation, this translates into $|\omega_{k}|\leq \ceps_0 h^{-2/3}$. 
As $\omega_k\sim C k^{2/3}$, we can introduce a cutoff in the sum over $k$, $\cutoffchi^{\flat}_{\epsilon_{0}}(h^{2/3}\omega_{k}):={\cutoffchi^{\flat}}(h^{2/3}\omega_k/\ceps_{0})$. Then, \eqref{eq:Prond2} rewrites as a finite sum,
 \begin{equation}
      \label{eq:Prond2cut}
      \Prond_{S,h,a}(t,x,y):=2\pi  \sum_{k\in \N^*} \frac{{\cutoffchi^{\flat}_{\epsilon_{0}}}(h^{2/3}\omega_k        )}
      {L'(\omega_k)}  K_{S,\omega_k}(g_{h,a})(t,x,y)\,.
    \end{equation}
From $\omega_1> 2.33$, we remark that ${\cutoffchi^{\sharp}(\omega)}$ was introduced in $K_{\omega}$ in connection with the sum over $k\geq 1$.     The cutoff ${\cutoffchi^{\sharp}}(\omega_k)$ is of no use in \eqref{eq:Prond2cut} but will help in estimating how many indices $N$ contribute significantly in the sum \eqref{eq:Prond}.     Again with \eqref{eq:AiryPoisson}, we also have
 \begin{equation}
      \label{eq:Prondcut}
      \Prond_{S,h,a}(t,x,y)= \langle \sum_{N\in \Z} e^{-iN L(\omega)}
    , {\cutoffchi^{\flat}}(h^{2/3}\omega/\ceps_{0})K_{S,\omega}(g_{h,a})(t,x,y)\rangle_{\omega}\,.
    \end{equation}
As in \cite{ILLP}, the sum over $N$ converges in $\mathcal{D}'_{\omega}$ and ${\cutoffchi^{\flat}}(h^{2/3}\omega/\ceps_0)K_{S,\omega}(g_{h,a})(t,x,y)$ is smooth in $(t,x,y)$ in a neighborhood of $(0,0,0)$ and smooth and compactly supported in $\omega$. By Lemma \ref{lem:airy-wronskian}, $L'(\omega)\gtrsim1$ for $\omega\ge1$. Recall we have a cutoff $\cutoffchi^{\sharp}(\omega)$ in \eqref{eq:Kequiv}; it gives $\omega\ge1$, or $\alpha\ge h^{2/3}$.
We also record how smooth errors are summed. If $r_h(t,x,y,\omega)$
is supported in $1\le\omega\lesssim h^{-2/3}$ and is
$O_{C^\infty}(h^\infty)$, integration by parts using
$\partial_\omega e^{-iNL(\omega)}=-iNL'(\omega)e^{-iNL(\omega)}$
 gives, for every $M$ and multi-index $\beta$,
\begin{equation}\label{eq:summable-smooth-errors}
 \left|\partial_{t,x,y}^{\beta}
 \int e^{-iNL(\omega)}r_h(t,x,y,\omega)\,d\omega\right|
 \le C_{M,\beta}h^M(1+|N|)^{-2}.
\end{equation}
The support length and derivative costs are polynomial in $h^{-1}$
and are absorbed by increasing the order of the remainder.
Thus smooth errors may be summed over all $N$, without first
truncating the sum. In particular, we may replace $g_{h,a}$
modulo smooth $O(h^\infty)$ errors and use \eqref{ghaF}.
Substitution in \eqref{eq:Kom} yields
\begin{multline}\label{Kghauseful}
K_{S,\omega}(g_{h,a})(t,x,y)=\frac{h^{2/3}}{2\pi h^{d+1}}\int e^{\frac ih (t\tau^2(h^{2/3}\omega,\theta)+\Phi(x,y,\theta,h^{2/3}\omega,\sigma)-\Phi(a,0,\theta,h^{2/3}\omega,s))}\\
\times p_h(x,y,\theta,h^{2/3}\omega,\sigma)\chi(s)\tilde q_h(\theta,h^{2/3}\omega,s)dsd\theta d\sigma\,.
\end{multline}
We set $ \Prond_{S,h,a}(t,x,y)=\sum_{N\in \Z} W_N(t,x,y)$, where $W_{N}$ is defined as
\begin{align}
  \label{eq:newVN}
  \quad W_N(t,x,y) := & \int e^{-iNL(\omega)}{\cutoffchi^{\flat}}(h^{2/3}\omega/\ceps_0 )K_{S,\omega}(g_{h,a})(t,x,y)d\omega\\
 = & \begin{multlined}[t] \frac 1 {2\pi h^{d+1}} \int e^{\frac i h (
    t\tau^2(\alpha,\theta)+\Phi(x,y,\theta,\alpha,\sigma)-\Phi(a,0,\theta,\alpha,s)-Nh L(h^{-2/3}\alpha)
    )} \\
\quad\quad\quad{}\times {\cutoffchi^{\flat}}(\alpha/\ceps_0)\chi(s)p_h(x,y,\theta,\alpha,\sigma)\tilde q_h (\theta,\alpha,s) ds \,d\theta  d{\sigma}d\alpha\,.\end{multlined}
\end{align}
The symbol ${\cutoffchi^{\flat}}(\alpha/\ceps_0) \chi(s)p_h \tilde q_h$ of $W_N$  is the same for every $N$,  is of degree $0$ and is given by an asymptotic expansion with small parameter $h$ and main term equal to $1$ (indeed, since $\tilde q_h$ has been obtained by inverting $J$, whose symbol is $p_h$.) Note that we do not have a finite sum over $N$: convergence should be understood in the distributional sense. The cutoff in $\alpha$ is redundant but we will leave it there to emphasize compact support in $\alpha$. Our main result in this section is the following proposition:
\begin{prop}\label{propsommefinie}
Let $a\in (h^{\frac{2}{3}-\ceps},a_0]$ with small $a_0,\ceps>0$.
\begin{itemize}[leftmargin=5.5mm]
\item
For $|t|\lesssim 1$, $\Prond_{S,h,a}$ is essentially a finite sum in $N$ at any given time,
\begin{equation}
  \label{eq:newProndcut}
  \Prond_{S,h,a}(t,x,y)=\sum_{|N|\lesssim {1+|t| a^{-1/2}}} W_N(t,x,y)+O_{C^{\infty}}(h^{\infty})\,. 
  \end{equation}
Moreover we can introduce a cutoff ${\cutoffchi}^{\sharp}(4\alpha/a)$ in the definition of $W_N$ without changing its main contribution modulo $O(h^{\infty})$ terms. 
\item At $t=0$, we have $\mathcal{P}_{S,h,a}(0,x,y)=\chi_0(hD_{x}){\varkappa}(hD_y)\delta_{(a,0)}+O_{C^\infty}(h^\infty)$.
  \item $\Prond_{S,h,a}$ is a parametrix in the sense of the Definition \ref{dfnparametrix}.
\end{itemize}
\end{prop}
\begin{rem}\label{rmqalpha>a}
Recall the cutoff $\cutoffchi^{\sharp}(\alpha/h^{2/3})$ is already present in $K_{S,\omega}$ by \eqref{eq:Kequiv} and \eqref{eq:qh}. It only restricts to $\alpha\ge h^{2/3}$. The additional cutoff $\cutoffchi^{\sharp}(4\alpha/a)$ in the proposition is used to restrict to $\alpha\gtrsim a$. The last statement in the first part of Proposition \ref{propsommefinie} translates into the contribution of the integrals defining $V_N$ being irrelevant for small values $\alpha\leq a/2$: for $a>h^{2/3 -\ceps}$ we can further restrict to $\alpha>a/2$. This follows immediately from the expression of $G(x,y,\theta/h,\omega)$ appearing in the definition of $K_{S,\omega}(g_{h,a})$ (recall \eqref{eq:Kequiv}): using \eqref{eq:defG}, $G$ reads as a sum of Airy functions computed at $-\zeta=x|\eta|^{2/3}e_0(x,y,\eta,\omega)-\omega$ with an elliptic $e_0$, close to $1$. These Airy functions are exponentially decreasing for $-\zeta>0$; hence, if $a>h^{2/3-\ceps}$, $\eta=\theta/h$ with $|\theta|\sim 1$ and $\omega=h^{-2/3}\alpha$, we must have $a\lesssim \alpha$ since otherwise the contribution from $G$ is $O(h^{\infty})$. The rescaled parameter is $\alpha^{3/2}/h$, which is large when
$\alpha\gg h^{2/3}$. The stronger bound
$\alpha\gtrsim h^{2/3-\ceps}$ ensures this in the present construction.

\end{rem}
\begin{proof}
 From Theorem \ref{thmMelrose}, $G(0,y,\eta,\omega_k)=e^{i\psi(0,y,\eta,\omega_k)}
  p_0Ai(-\omega_k)$, which immediately yields $G(0,y,\eta,\omega_k)=0$. Therefore, from
\eqref{eq:Prond2cut} being a finite sum, we get $\Prond_{S,h,a}(t,x,y)_{|\partial\Omega} =0$,
that is, the Dirichlet boundary condition holds for $\Prond_{S,h,a}$.
By \eqref{eq:Kom1} and \eqref{eq:summable-smooth-errors},
$(ih\partial_t-h^2\Delta)\Prond_{S,h,a}
=O_{C^\infty(W)}(h^\infty)$ uniformly in $a<a_0$,
including at the boundary.
By the forthcoming Lemma \ref{sommeNfinie}, the sum over the tail
$|N|\ge Ch^{-1/3}$ is $O_{C^\infty}(h^\infty)$.
Its hypotheses hold for $g_{h,a}$ by \eqref{eq:gosc}:
the $\tprime$ support is fixed and each smooth seminorm
is bounded by a fixed power of $h^{-1}$, uniformly in $a$.
This preliminary truncation uses only the cutoff
$\cutoffchi^{\sharp}(\alpha/h^{2/3})$ already in $K_{S,\omega}$.
We then insert $\cutoffchi^{\sharp}(4\alpha/a)$ as explained
in Remark \ref{rmqalpha>a}; the errors on the remaining
$O(h^{-1/3})$ indices still sum to $O(h^\infty)$.
Both statements on $\mathcal{P}_{S,h,a}$ are independent of the particular choice of the function $g_{h,a}$ such that \eqref{eq:Jdirac} holds. It remains to check that, with our choice of $g_{h,a}$ given in \eqref{eq:Jdirac}, $\mathcal{P}_{S,h,a}(0,x,y)$ is the correct initial value :  this has been proved in \cite[Prop.2.15]{ILLP} for the wave equation and remains true as, at $t=0$, there is no difference between $\mathcal{P}_{h,a}$ (of \cite[Definition 2.6]{ILLP}) and $\mathcal{P}_{S,h,a}$ (see Remark \ref{rmqPcal}).
It remains to obtain the sharper time-dependent truncation \eqref{eq:newProndcut}, using the full phase in \eqref{eq:newVN}, and to check that at $t=0$, all the oscillatory integrals $W_N(0,x,y)$ for $|N|\geq 1$ provide a $O(h^{\infty})$ contribution, while $W_0(0,x,y)=J(g_{h,a})(x,y)$ which, by design, is our initial data. The fact that the number of $N$ is finite will allow us to deduce that $\mathcal{P}_{S,h,a}(0,x,y)=W_0(0,x,y)=J(g_{h,a})(x,y)=\chi_0(hD_x)\varkappa(hD_y)\delta_{(a,0)}$ and to conclude.
\end{proof}
\begin{lemma}
  \label{sommeNfinie}
  Let $I$ be a fixed compact interval and $f_h$ a smooth family
  of functions of $(y',\tprime)$, supported in $\tprime\in I$. Assume that for
  every multi-index $\beta$ there exist $C_\beta,M_\beta>0$ such that
  $\|\partial_{y',\tprime}^{\beta}f_h\|_{L^\infty}
  \le C_\beta h^{-M_\beta}$.
  For $K_{S,\omega}(f_h)$ defined by \eqref{eq:Kom}, there exists
  $C>0$, depending only on $I$, the fixed local phase and cutoffs,
  and the fixed observation neighborhood and time interval, such that
\begin{equation}\label{eq:suminN}
  \langle \sum_{|N|\ge C h^{-1/3}}e^{-iN L(\omega)}
    , {\cutoffchi^{\flat}}(h^{2/3}\omega/\ceps_{0})K_{S,\omega}(f_h)(t,x,y)\rangle_{\omega} = O_{C^\infty}(h^{\infty})\,.
\end{equation}
\end{lemma}
\begin{rmq}
  The threshold constant $C$ is independent of $h$, the source distance
$a$, and the seminorm constants $C_\beta,M_\beta$ of $f_h$.
The constants in the remainder may depend on these seminorm bounds.
\end{rmq}
\begin{proof}
We consider the sum over all $N$. Integrations by parts in $\theta$
make the $y'$ integral absolutely convergent, at a polynomial cost in
$h^{-1}$ under the assumed seminorm bounds. All phases in the sum
are linear in $t$ and $N$ and given by
\[
t\tau^2(\alpha,\theta)+\Phi(x,y,\theta,\alpha,\sigma)-N h L(\alpha h^{-2/3}) -y'\theta-\tprime \alpha,
\]
with large parameter $1/h$ as a factor, where, using \eqref{eq:propL}, 
$hNL(\alpha h^{-2/3})=hN\frac{\pi}{2}+\frac 43 N\alpha^{3/2} -Nh B_L(\alpha^{3/2}/h)$.
Integration variables are $\sigma,\alpha$, $\theta$ and also $y'$ and $\tprime$ ; only stationary points with respect to $\alpha$ and $\sigma$ will be required for the sum in $N$ to be finite.
Critical points in $\alpha$ are such that
\begin{equation}\label{eq:critalphaNN}
2t\tau \partial_{\alpha}\tau +\partial_{\alpha}\Phi(x,y,\theta,\alpha,\sigma)=\tprime+h^{1/3}NL'(\alpha h^{-2/3})\,,
\end{equation}
while those with respect to $\sigma$ are such that $\partial_{\sigma}\Phi(x,y,\theta,\alpha,\sigma)=0$.
We keep the Airy--Poisson derivative in its exact form, including
when $\alpha h^{-2/3}$ is bounded.
We recall that
\[
\Phi(x,y,\theta,\alpha,\sigma)=y\cdot \theta+\sigma^3/3+\sigma(xq^{1/3}(\theta)-\alpha)+\tau(\alpha,\theta)\Gamma(x,y,\sigma q^{1/3}(\theta)/\tau(\alpha,\theta),\theta/\tau(\alpha,\theta)),
\]
where $\Gamma(x,y,\Xi,\Theta)=B_{\Gamma}(y,\Theta)+xA_{\Gamma}(x,y,\Xi,\Theta)$ are given in  Proposition \ref{propimpformgamma}. We start with computing derivatives of $\tau\Gamma(x,y,\sigma sq^{1/3}(\theta)/\tau,\theta/\tau)$ which depend on $\alpha$ through $\tau$ only.
Take $\Theta=\theta/\tau(\alpha,\theta)$, $\tau=\tau(\alpha,\theta)$, $\omega=\theta/|\theta|$, then
\[
\tau B_{\Gamma}(y,\theta/\tau)=\tau (B_0(y,\omega)+(|\theta|/\tau-1)B_2(y,\omega)+...+(|\theta|/\tau-1)^jB_{2j}(y,\omega)+...)
\]
and, writing $\rho=|\theta|/\tau$, $\theta/\tau=\rho \frac{\theta}{|\theta|}$, yields (see the proof of \cite[Lemma 2.17]{ILLP} for details)
\begin{equation}\label{partialBGamma}
\partial_{\alpha}\Big(\tau B_{\Gamma}(y,\theta/\tau)\Big)
=\partial_{\alpha}\tau\Big[ B_0(y,\theta/|\theta|)-B_2(y,\theta/|\theta|)-\Big(\frac{|\theta|^2}{\tau^2}-1\Big)B_4(y,\theta/|\theta|)+\mathcal{H}_{j\geq 4}\Big].
\end{equation}
Using \eqref{eq:normalized-AGamma}, with $\Xi=sq^{1/3}(\theta)/\tau$
and $\rho=|\theta|/\tau$, we obtain at fixed $(x,y,s,\theta)$
\begin{equation}\label{partialAGamma}
\partial_\alpha(\tau A_\Gamma)
=\partial_\alpha\tau\bigl[-\mu_\Gamma(\Xi^2+\rho^2+1)
 -2\mu_1\Xi^3+(1-\Xi\partial_\Xi-\rho\partial_\rho)R_A\bigr].
\end{equation}
The $\alpha_1x\Xi$ term gives zero in this identity. The last remainder
has weighted order at least two, since $\partial_\rho$ can lower the
order by two. In particular, it is $O(\gamma)$ when $x=O(\gamma)$,
$\Xi=O(\sqrt\gamma)$ and $\rho-1=O(\gamma)$.
We have $\partial_{\alpha}\tau(\alpha,\theta)=\frac{q^{2/3}(\theta)}{2\tau(\alpha,\theta)}$ and using \eqref{partialBGamma}, \eqref{partialAGamma} (which comes with a factor $x\in \mathcal{H}_{j\geq 2}$), \eqref{eq:critalphaNN} becomes
\begin{multline}\label{critalphaNNf}
t q^{2/3}(\theta)+ \partial_{\alpha}\tau\Big[B_0(y,\theta/|\theta|)-B_2(y,\theta/|\theta|)-\Big(\frac{|\theta|^2}{\tau^2}-1\Big)B_4(y,\theta/|\theta|)+\mathcal{H}_{j\geq 4}\Big]-\sigma\\
{}+x\partial_\alpha\bigl(\tau A_\Gamma(x,y,\sigma q^{1/3}(\theta)/\tau,\theta/\tau)\bigr)=\tprime+h^{1/3}NL'(\alpha h^{-2/3})\,.
\end{multline}
Using now \eqref{partialAGamma} and that $B_{\Gamma}(y,\theta/\tau)$ does not depend on $\sigma$, $\partial_{\sigma}\Phi(\cdots)=0$ becomes
\begin{equation}\label{sigmacrit}
\sigma^2+xq^{1/3}(\theta)-\alpha+\tau(\alpha,\theta)x\partial_{\sigma}A_{\Gamma}(x,y,\sigma q^{1/3}(\theta)/\tau(\alpha,\theta),\theta/\tau(\alpha,\theta))=0.
\end{equation}
Recalling \eqref{AGam}, let $\Xi=\frac{\sigma q^{1/3}(\theta)}{\tau(\alpha,\theta)}$ and $\Theta=\frac{\theta}{\tau(\alpha,\theta)}$.
As $A_{\Gamma}(x,y,\Xi,\Theta)$ depends on $\sigma$ only through $\Xi$ and {
$\partial_\Xi A_\Gamma=\Lp+2\mu_\Gamma\Xi+\alpha_1x
+3\mu_1\Xi^2+\partial_\Xi R_A$, with
$\partial_\Xi R_A\in\mathcal H_{\ge3}$}, then
\begin{equation}\label{partialsigmaAGamma}
\partial _{\sigma}A_{\Gamma}(x,y,\sigma q^{1/3}(\theta)/\tau,\theta/\tau)=\frac{q^{1/3}(\theta)}{\tau(\alpha,\theta)}\Big(\Lp(y,\theta/|\theta|)+2\frac{\sigma q^{1/3}(\theta)}{\tau(\alpha,\theta)}{\mu_\Gamma(y,\theta/|\theta|)}+\mathcal{H}_{j\geq 2}\Big),
\end{equation}
where $\mathcal{H}_{j\geq 2}$ contains weights $x$, $\frac{\sigma^2q^{2/3}(\theta)}{\tau^2}$ and $\frac{|\theta|}{\tau}-1=\frac{|\theta|^2-\tau^2}{\tau(|\theta|+\tau)}=-\frac{\alpha q^{2/3}(\theta)}{\tau (|\theta|+\tau)}$.
Using \eqref{sigmacrit} and \eqref{partialsigmaAGamma},
\begin{equation}\label{critsigma}
\sigma^2+xq^{1/3}(\theta)\Big(1+\Lp(y,\theta/|\theta|)+2\frac{\sigma q^{1/3}(\theta)}{\tau(\alpha,\theta)}{\mu_\Gamma(y,\theta/|\theta|)}+\mathcal{H}_{j\geq 2}\Big)=\alpha\,.
\end{equation}
We recall that $\Lp$ was chosen after \eqref{eqorder2} and depends on the curvature at the boundary near $y=0$:
\[
1+\Lp(y,\omega)=\Big(R_1(y,\omega+\partial_y B_0(y,\omega))/q(\omega)\Big)^{1/3},  \quad R_1(0,\omega)=q(\omega),
\]  
where from \eqref{eqB0} we have $\omega+\partial_yB_0(y,\omega)=\partial_{y}\phi(y,\omega)=\omega(1+O(y))$. By finite speed of propagation of the semiclassical Schrödinger flow, for bounded time $|t|$ we must have $|y|$ bounded (see also Lemma \ref{derriere}); hence, there exists $T_0<1$ sufficiently small such that if $|t|\leq T_0$ then 
\(\Big(R_1(y,\omega+\partial_y B_0(y,\omega))/q(\omega)\Big)^{1/3}>1/2\).
We recall that $\alpha\leq \ceps_{0}$, small. We have $B_0=O(|y|^2)$, $B_2=O(|y|^2)$, $B_{2j}=O(|y|)$, $\forall j\geq 2$ and $\frac{|\theta|}{\tau}-1=\frac{|\theta|^2-\tau^2}{\tau(|\theta|+\tau)}=-\frac{\alpha q^{2/3}(\theta)}{\tau (|\theta|+\tau)}=O(\alpha)$, therefore the coefficient of $\partial_{\alpha}\tau$ in \eqref{partialBGamma} is like  {$B_0(y,\theta/|\theta|)-B_2(y,\theta/|\theta|)+O(\alpha y)=y(O(y)+O(\alpha)).$} For $|t|\leq T_0$, taking $T_0$ smaller if necessary, we assume $|B_0(y,\theta/|\theta|)-B_2(y,\theta/|\theta|)|=O(y^2)\leq |t|/8$. As $y \lesssim t$ by finite speed of propagation, taking $\ceps_0$ smaller if necessary, we assume $O(\alpha y)\leq |t|/8$ for $\alpha\leq \ceps_0$. Therefore, the contribution from $\partial_{\alpha}(\tau B_{\Gamma}(y,\theta/\tau))/\partial_{\alpha}\tau$ in \eqref{critalphaNNf} is $O(|t|/4)$ and the coefficient of $\partial_{\alpha}\tau$ in \eqref{critalphaNNf} behaves like $t\leq T_0<1$. As $\tprime$ is bounded ($f_h$ has fixed compact support in $\tprime$), it remains to bound $\sigma$ in \eqref{critalphaNNf}. Going back to \eqref{critsigma}, using that $0\leq x<1$ and that the terms in $\mathcal{H}_{j\geq 2}$ come with the factors $\sigma^2, x,\alpha$, it follows that
\[
(\sigma+x\mu_\Gamma q^{2/3}(\theta)/\tau)^2+x q^{1/3}(1+\Lp +\mathcal{H}_{j\geq 2})=\alpha+ x^2\mu_\Gamma^2 q^{4/3}(\theta)/\tau^2,
\]
which implies that $\sigma^2$ is bounded. For $|\sigma|$ sufficiently
large, integrations by parts in $\sigma$ give a smooth
$O(h^\infty)$ error before multiplication by $e^{-iNL(\omega)}$;
\eqref{eq:summable-smooth-errors} justifies its summation over $N$.
The term $x\partial_\alpha(\tau A_\Gamma)$ is uniformly bounded.
On the support, $\omega=\alpha h^{-2/3}\ge1$.
Lemma \ref{lem:airy-wronskian} gives the bounds
\eqref{eq:Lprime-symbol-bounds} uniformly on this support.
Using $h^{1/3}NL'(\alpha h^{-2/3})$ exactly in
\eqref{critalphaNNf}, stationarity therefore requires
\begin{equation}\label{nrN}
 |N|\sqrt\alpha\lesssim 1+|t|+|\tprime|+|\sigma|.
\end{equation}
The right-hand side of \eqref{nrN} is bounded in terms of $I$
and the fixed local phase, supports and time interval.
Since $\alpha\ge h^{2/3}$, we choose $C$ using only these bounds.
Thus, for $|N|\ge C h^{-1/3}$,
the absolute value of the $\alpha$ derivative of the phase is
bounded below by $c|N|\sqrt\alpha$.
Each integration by parts in $\alpha$ gains
$h/(|N|\sqrt\alpha)$ and costs at most $h^{-2/3}$ from
symbol derivatives or derivatives of the reciprocal phase derivative.
The resulting factor is
$h^{1/3}/(|N|\sqrt\alpha)\lesssim |N|^{-1}$.
After $M>1$ integrations, summation gives
\[
 \sum_{|N|\ge C h^{-1/3}}|N|^{-M}
 \lesssim_M h^{(M-1)/3}.
\]
Increasing $M$ absorbs the fixed powers of $h^{-1}$ in
\eqref{eq:Kom}, \eqref{eq:qh}, and those from any fixed number
of $(t,x,y)$ derivatives. This proves \eqref{eq:suminN}.
We only used $\alpha\gtrsim h^{2/3}$, with no additional factor $h^{-\ceps}$.
The preliminary truncation therefore remains valid down to that scale, independently of $a$. The stronger restriction used for stationary
phase in Section \ref{secta>} is a separate condition.

The same argument applies to $K_{S,\omega,\gamma}$ defined in
\eqref{defKomegakgamma}, with a threshold constant depending also on
the fixed dyadic cutoff, but independent of $\gamma$ and $a$.
On its support $\alpha\sim\gamma$;
for $\gamma\gtrsim h^{2/3}$, it gives
\begin{equation}\label{eq:suminN-gamma}
 \Big\langle\sum_{|N|\ge C\gamma^{-1/2}}e^{-iNL(\omega)},
 {\cutoffchi^{\flat}}(h^{2/3}\omega/\ceps_0)
 K_{S,\omega,\gamma}(f_h)(t,x,y)\Big\rangle_\omega
 =O_{C^\infty}(h^\infty).
\end{equation}
Indeed, each integration by parts gains $h/(|N|\gamma^{3/2})$,
including the cost $\gamma^{-1}$ of a symbol derivative.
After $M>1$ integrations, the sum is bounded by
$C_M(h/\gamma)^M\gamma^{-1/2}\lesssim_M h^{(M-1)/3}$,
before the fixed polynomial costs are absorbed. This truncation
is independent of a source distance $a$; $h^{-1/3}$ is the uniform
threshold obtained from $\gamma\gtrsim h^{2/3}$.
\end{proof}

\subsection{Parametrix construction for $a\leq h^{2/3-\ceps}$, $\ceps>0$}\label{sec:parametrix-small-distance} In this section, $a$ is very small and we follow the construction of a suitable $g_{h,a}$ of \cite[Section 2.3]{ILLP} to recover the initial data.

\subsubsection{Spectral theory for $-\Delta_M$: the initial data in terms of model gallery modes}

Let $-\Delta_M$ be the Friedlander model operator introduced in \eqref{eq:LapM},
and recall $q$ is a positive definite quadratic form.
Taking the Fourier transform in the $y$ variable, the operator $-\Delta_{M}$ becomes
$-\partial^2_x+|\eta|^2+xq(\eta)$. For  $\eta\neq 0$, this operator is a positive self-adjoint operator 
on $L^2(\mathbb{R}_+)$, with compact resolvent. The next Lemma is proved in \cite{ilpCE} (with $q(\eta)=|\eta|^2$, but only using $q(\eta)\neq 0$):

\begin{lemma}\label{lemorthog}(see \cite[Lemma 2]{ilpCE})
There exists a Hilbert basis of $L^{2}(\mathbb{R}_{+})$ where $\{e_k(x,\eta)\}_{k\geq 0}$ are eigenfunctions of $-\partial^2_x+|\eta|^2+xq(\eta)$, with eigenvalues $\lambda_k(\eta)=|\eta|^2+\omega_k q(\eta)^{2/3}=\tau^2(\omega_k,\eta)$. These eigenfunctions are translated and rescaled Airy functions:
\begin{equation}\label{eig_k}
 e_k(x,\eta)=\frac{\sqrt{2\pi} q(\eta)^{1/6}}{\sqrt{L'(\omega_k)}}
Ai\Big(xq(\eta)^{1/3}-\omega_k\Big),
\end{equation}
where $L'(\omega_k)$ (from \eqref{eq:propL2}) normalizes $\|e_k(.,\eta)\|_{L^2(\mathbb{R}_+)}=1$.
\end{lemma}
For $a>0$, the Dirac distribution $\delta_{x=a}$ on $\mathbb{R}_+$ may be decomposed as  $\delta_{x=a}=\sum_{k\geq 1} e_k(x,\eta)e_k(a,\eta)$. Eigenfunctions of 
$-\partial^2_{x}+xq(\eta)$ are $e_k(x,\eta)$ with eigenvalue $\lambda_k(\eta)-|\eta|^2=\tau^2(\omega_k,\eta)-|\eta|^2=\omega_k q^{2/3}(\eta)$,  and for cutoffs $\varkappa$, $\chi_0$ to be chosen, the following spectral decomposition holds
\begin{equation}\label{diracmoddeltaM<} 
\chi_0(-h^2\partial^2_{x}+xq(h\partial_y))\varkappa(hD_y)\delta_{(a,0)}=\sum_{k\geq 1}\int
e^{iy\cdot \eta} \\ 
\chi_0\bigl(h^{\frac 23}\omega_kq^{\frac 23}(h\eta)\bigr) \varkappa(h\eta)e_k(x,\eta)e_k(a,\eta)d\eta\,.
\end{equation}
With $\chi_0\in C^{\infty}_0(-2\ceps_0, 2\ceps_0)$, $\chi_{0}=1$ on $[-\ceps_0,\ceps_0]$, \eqref{diracmoddeltaM<} is a finite sum with $O(\ceps_0/h)$ terms. 
Setting $\theta=h\eta$, from support considerations, taking the support of $\chi_0$ smaller if necessary, we can assume that $\chi^{\flat}(\alpha/\ceps_{0})\chi_0(\alpha q^{2/3}(\theta))\varkappa(\theta)=\chi_0(\alpha q^{2/3}(\theta))\varkappa(\theta)$,
where $\alpha=h^{2/3}\omega$ and $\chi^{\flat}(h^{2/3}\omega/\ceps_{0})$ is the cutoff introduced in \eqref{eq:Prond2cut} (which restricts the support of $K_{S,\omega}$ to values $\omega\leq \ceps_0/h^{2/3}$).

As in \cite[Section 2.3]{ILLP}, for $a\leq h^{2/3-\ceps}$, the easiest way to define initial data is to choose the left-hand side term in \eqref{diracmoddeltaM<}, which does vanish on the boundary ($e_k(0,\eta)=0$ for every $k\geq 1$.) Using both \eqref{eq:Prondcut} and \eqref{eq:Prond2cut}, it remains to obtain a smooth function $g_{h,a}$ such that $\Prond_{S,h,a}$, as defined in \eqref{eq:Prond2cut}, satisfies
\begin{equation}\label{dataamic}
\Prond_{S,h,a}(0,x,y)=\chi_0(-h^2\partial^2_{x}+xq(h\partial_y))\varkappa(hD_y)\delta_{(a,0)}+O_{C^{\infty}}(h^{\infty}).
\end{equation}
As, at $t=0$, the parametrix for the semiclassical Schrödinger equation and the one for the wave equation coincide, we briefly recall the construction of \cite[Section 2.3]{ILLP}. For some $K_{\ceps}$ (that can be freely chosen) such that $h^{-\ceps}\ll h^{-2\ceps}\lesssim K_{\ceps}\lesssim h^{-1/4+\ceps}\ll h^{-1/4}$, define $\mathcal{E}_{M}$ and rewrite the sum over $k$ in \eqref{diracmoddeltaM<} as follows
\begin{gather}
\label{defEmathcalM}\,\,\quad
\begin{multlined}[t]
\mathcal{E}_M(x,y,a,\omega_k):=
\int e^{iy\eta} \chi_0(h^{2/3}\omega_kq^{2/3}(h\eta)) \varkappa(h\eta)e_k(x,\eta)e_k(a,\eta)d\eta,\\
\end{multlined}\\
  \label{dataEMEM}
\sum_{k\geq1}\cutoffchi^{\flat}\Big(\frac{\omega_k}{\omega_{K_{\ceps}}}\Big)\mathcal{E}_M(x,y,a,\omega_k)+\sum_{k\geq 1}\chi^{\flat}(h^{2/3}\omega_k/\ceps_{0})\cutoffchi^{\sharp}\Big(\frac{\omega_k}{\omega_{K_{\ceps}}}\Big)\mathcal{E}_M(x,y,a,\omega_k)\,,
\end{gather}
where we recall that $\cutoffchi^{\flat}\in C^{\infty}(\mathbb{R})$ is such that $\cutoffchi^{\flat}=1$ on $(-\infty,1]$ and $\cutoffchi^{\flat}=0$ on $[2,\infty)$, and $\cutoffchi^{\sharp}=1-\cutoffchi^{\flat}$. 
\begin{prop}\label{propdataapetitpetit1}(\cite[Prop.2.19]{ILLP})
For all $a\in (0,h^{2/3-\ceps}]$ there exists a smooth function $g_{h,a,1}$ such that
\begin{equation}\label{trucmuche}
\sum_{k\geq 1}\frac{2\pi}{L'(\omega_k)}\cutoffchi^{\flat}(h^{2/3}\omega_k/\ceps_{0})K_{S,\omega_k}(g_{h,a,1})(0,x,y)=\sum_{k\geq1}\cutoffchi^{\flat}\Big( \frac{\omega_k}{\omega_{K_{\ceps}}}\Big)\mathcal{E}_M(x,y,a,\omega_k).
\end{equation}
\end{prop}
The proof of Proposition \ref{propdataapetitpetit1} is the same as that of \cite[Prop.2.19]{ILLP}) (provided in \cite[Section 2.3.3]{ILLP}) as, at $t=0$, the operators $K_{S,\omega_k}$ and $K_{\omega_k}$ (introduced in \eqref{eq:Kequivwave} in the case of the wave equation) coincide. 

\begin{rmq}
The general idea is that, for sufficiently small $a,x$ and $h^{2/3}\omega_k$, the phase term $\Gamma$ is so small that $e^{\frac ih \tau\Gamma}$ can be incorporated into of the symbol ; therefore, a parametrix for the semiclassical Schrödinger equation with Laplace operator given by \eqref{eq:laptilde} should behave like a suitable "perturbation" of the parametrix for the wave equation with $-\Delta_M$, at least very close to $t=0$. The condition on $\omega_k\leq \omega_{K_{\ceps}}$ with $K_{\ceps}\lesssim h^{-1/4+\ceps}$ ensures the necessary smallness of all the contributions depending on $h^{2/3}\omega_k$ in $\Gamma$.
\end{rmq}

\begin{prop}\label{propdataapetitpetit2}(\cite[Prop. 2.20]{ILLP})
  For all $a\in (0,h^{2/3-\ceps}]$, there exists a smooth function $g_{h,a,2}$ such that
  \begin{gather}
\label{eq:g1}\begin{multlined}[t]
\langle \sum_{N\in \Z} e^{-iN L(\omega)}
    , {\cutoffchi^{\flat}}(h^{2/3}\omega/\ceps_{0}){\cutoffchi^{\sharp}}(h^{2\ceps}\omega)K_{S,\omega}(g_{h,a,2})(0,x,y)\rangle_{\omega}\\
\quad\quad\quad\quad    {}=\sum_{k\geq 1}\cutoffchi^{\sharp}\Big(\frac{\omega_k}{\omega_{K_{\ceps}}}\Big)\chi^{\flat}(h^{2/3}\omega_k/\ceps_{0})\mathcal{E}_M(x,y,a,\omega_k)+O(h^{\infty}),
\end{multlined}\\
\label{eq:g12}
\langle \sum_{N\in \Z} e^{-iN L(\omega)}
    , {\cutoffchi^{\flat}}(h^{2/3}\omega/\ceps_{0}){\cutoffchi^{\sharp}}(\omega){\cutoffchi^{\flat}}( h^{2\ceps}\omega)K_{S,\omega}(g_{h,a,2})(0,x,y)\rangle_{\omega}=O(h^{\infty}).
\end{gather}
\end{prop}
We introduce a new cutoff ${\cutoffchi^{\sharp}}(h^{2\ceps}\omega)$ in \eqref{eq:g1} and removed ${\cutoffchi^{\sharp}}(\omega)$ (which is supported where $\omega\geq 2$ and identically $1$ on the support of ${\cutoffchi}^{\sharp}(h^{2\ceps}\omega)$.) The proof of Proposition \ref{propdataapetitpetit2} is provided in \cite[Section 2.3.4]{ILLP}.
This cutoff restricts to
$\alpha=h^{2/3}\omega\gtrsim h^{2/3-2\ceps}$.
On a dyadic piece $\alpha\sim\gamma$, it therefore gives
$\gamma^{3/2}/h\gtrsim h^{-3\ceps}$,
independently of the size of $a$.
The sum of \eqref{eq:g1} and \eqref{eq:g12} yields the second term in \eqref{dataEMEM}. Finally, we have
\begin{prop}\label{propparama<}
For all $a\in (0,h^{2/3-\ceps}]$, set $g_{h,a}:=g_{h,a,1}+g_{h,a,2}$, Then $\Prond_{S,h,a}$ is a parametrix in the sense of Definition \ref{dfnparametrix} and \eqref{dataamic} holds.
\end{prop}
Proposition \ref{propparama<} follows from Propositions \ref{propdataapetitpetit1} and \ref{propdataapetitpetit2} using $K_{S,\omega}(g_{h,a,1}+g_{h,a,2})=K_{S,\omega}(g_{h,a,1})+K_{S,\omega}(g_{h,a,2})$ together with \eqref{eq:Prond2}.
\subsubsection{Idea of the proof of Proposition \ref{propdataapetitpetit1}}\label{sectproofs}
Proposition \ref{propdataapetitpetit1} reduces to constructing a smooth function $g_{h,a,1}$ such that, for all $k$ such that $\chi^{\flat}(\omega_k/\omega_{K_{\ceps}})\neq 0$, 
\begin{equation}\label{eqKE}
\frac{2\pi}{L'(\omega_k)}K_{S,\omega_k}(g_{h,a,1})(0,x,y)=\mathcal{E}_M(x,y,a,\omega_k) +O(h^{\infty}).
\end{equation}
For a function $g(y',\tprime)$, let $\hatdeux{g}(y',\alpha/h)$ be its Fourier transform with respect to $\tprime$ at $\alpha/h$. 
Using the definition of $K_{S,\omega}(g)$, we only need to construct $\hatdeux{g}_{h,a,1}(.,\omega/h^{1/3})$ for $\omega\in\{\omega_k, 1\leq k\ll h^{-1/4}\}$. We briefly recall several notations from \cite{ILLP} that will be useful in order to state \cite[Lemma 2.31]{ILLP}, which allows us to construct $\hatdeux{g}_{h,a,1}$.
\begin{dfn}
Let $G$ be defined in \eqref{eq:Gosc}.
We set
\begin{equation}\label{defeomkgen}
e(x,y,\eta,\omega)=\frac{\sqrt{2\pi} q(\eta)^{1/6}}{\sqrt{L'(\omega)}}e^{-iy\eta}G(x,y,\eta,\omega).
\end{equation}
Replacing $G(\cdots)$ with $G_M(\cdots):=e^{i y\eta}Ai(xq(\eta)^{1/3}-\omega)$  in \eqref{defeomkgen} yields \eqref{eig_k} instead of $e(x,y,\eta,\omega)$.
\end{dfn}
\begin{dfn}
Let $\varkappa$ be as in Definition \eqref{dfnparametrix}. For $g\in L^2(\mathbb{R}^{d-1})$, we define
\begin{equation}\label{defFk}
F_{\omega_k}(g)(x,y): =\frac{1}{(2\pi )^{d-1}}\int e^{i(y-y')\eta}e(x,y,\eta,\omega_k)(\varkappa(hD_{y'})g)(y')d\eta dy'\,. \end{equation}
\end{dfn}
\begin{dfn}\label{blabla}
Let $\tilde \varkappa\in C^{\infty}_0(\mathbb{R}^{d-1})$ be such that $\tilde\varkappa=1$ on the support of $\varkappa$ and vanish outside a neighborhood of $\mathbb{S}^{d-2}$. Also let $\chi\in C^{\infty}_0$ be a smooth cutoff supported in the ball centered at $0$ with radius $1/16$ of $\mathbb{R}^{d-1}$.  
For $f\in L^2(\mathbb{R}^{d-1})$ we define an operator $\Lo$ as
\begin{equation}\label{def:L}
\Lo(f)(y):=\int e^{i(y-y')\eta-i|\eta|B_0(y',\eta/|\eta|)}\tilde\varkappa(h\eta)\chi(y')f(y')dy'd\eta,
\end{equation}
{where $B_0(y,\eta/|\eta|)$ is the first coefficient in
\eqref{formalseriesBA} and is homogeneous of degree zero in $\eta$.
The phase term $|\eta|B_0(y,\eta/|\eta|)=B_0(y,\eta)$ has degree one.}\end{dfn}
To define $g_{h,a,1}$, we need to "invert" $F_{\omega_k}$, which requires estimating derivatives of $e(x,y,\eta,\omega_k)$ with respect to $(y,\eta)$ : in the Friedlander model the corresponding mode $e_k(x,\eta)$ in \eqref{eig_k} does not depend  on $y$, but here, differentiating with respect to $y$ yields a large factor due to $\tau B_0$ in the phase $\tau \Gamma(x,y,\sigma q^{1/3}(\eta)/\tau,\eta/\tau)$. 
In order to get rid of the {homogeneous term of degree one $|\eta|B_0(y,\eta/|\eta|)$} in the phase function of $e(x,y,\eta,\omega_k)$, we set $\tilde F_{\omega_k}:=F_{\omega_k}\circ \Lo$ to obtain a new operator whose phase function does not include the contribution $|\eta|B_0(y,\eta/|\eta|)$; we then prove that, at least for $k\ll h^{-1/4}$, these operators can be inverted. {The remaining contribution
$(\tau-|\eta|)(B_0-B_2)(y,\eta/|\eta|)$ and the higher radial terms
are retained in the phase.}

We compute $\tilde F_{\omega_k}=F_{\omega_k}\circ \Lo(f)$ for $f\in L^2(\mathbb{R}^{d-1})$, setting $\tilde{e}(x,y,\eta,\omega_k)=e^{-i|\eta|B_0(y,\eta/|\eta|)}e(x,y,\eta,\omega_k)$:
\begin{align*}
F_{\omega_k}\circ \Lo(f) (x,y) 
&=\frac{1}{(2\pi )^{d-1}} \int_{\eta} e^{i y\cdot \eta}e(x,y,\eta,\omega_k)\varkappa(h\eta)\widehat{\Lo(f)}(\eta)d\eta \\
&=\frac{1}{(2\pi )^{d-1}} \int_{y'}\int_{\eta} e^{i (y-y')\cdot\eta+i|\eta|(B_0(y,\eta/|\eta|)-B_0(y',\eta/|\eta|))}\tilde{e}(x,y,\eta,\omega_k)\varkappa(h\eta)\chi(y')f(y')d\eta dy'.
\end{align*}
The following lemma, proved in \cite[Lemma 2.28]{ILLP}, shows that $\tilde e$ behaves as a symbol:
\begin{lemma}\label{lemestimderivekk}
Uniformly for $h^{2/3}\omega_k\ll 1$, we have:
\begin{equation}\label{estimderivekk}
\|\partial^{\beta}_{(y,\theta)}\tilde{e}(.,y,\theta/h,\omega_k)\|_{L^2(x\geq 0)}\lesssim \big(\omega_k/h^{1/3}\big)^{|\beta|}\,.
\end{equation}
\end{lemma}
Using Lemma \ref{lemestimderivekk}, \cite[Prop. 2.25]{ILLP} proves the following result, which allows us to invert $\tilde F_{\omega_k}^*\circ \tilde F_{\omega_k}$ at least for $k\leq K_{\ceps}$ and thus to define $g_{h,a,1}$.
\begin{prop}\label{propmatrix}(\cite[Prop.2.27]{ILLP})
Let $0<\mu<1/4$; the operators $\tilde F_{\omega_k}^*\circ \tilde F_{\omega_k}:L^2(\mathbb{R}^{d-1})\rightarrow L^2(\mathbb{R}^{d-1})$ are pseudodifferential operators that are uniformly elliptic with respect to $1\leq k\leq h^{-\mu}$.
\end{prop} 
\begin{lemma}\label{lemfkhatgk}(\cite[Lemma 2.31]{ILLP})
Let $h^{-2\ceps}\lesssim K_{\ceps}\lesssim h^{-1/4+\ceps}$, $\ceps>0$ small. For $1\leq k\leq K_{\ceps}$, define
\[
f_{\omega_k}(y):=\Big(\tilde F^*_{\omega_k}\circ \tilde F_{\omega_k}\Big)^{-1}\Big(\tilde F^*_{\omega_k}(\mathcal{E}_M(.,\omega_k))\Big)(y)\,,
\] 
then define $g_{h,a,1}$ as follows: $\forall k$ such that $\chi^{\flat}(\omega_k/\omega_{K_{\ceps}})\neq 0$,
$\hatdeux{g}_{h,a,1}(.,\omega_k/h^{1/3}) :=\frac{\sqrt{L'(\omega_k)}}{\sqrt{2\pi}}\Lo(f_{\omega_k})$,
and $\forall k$ such that $\chi^{\flat}(\omega_k/\omega_{K_{\ceps}})=0$,
$\hatdeux{g}_{h,a,1}(.,\omega_k/h^{1/3}) :=0$. Then \eqref{eqKE} holds for $g_{h,a,1}$.
\end{lemma}
We obtain the explicit form of $\hatdeux{g}_{h,a,1}$ as a corollary of Lemma \ref{lemfkhatgk} :
\begin{cor}\label{lemgomegak}(\cite[Corollary 2.32]{ILLP})
We keep the notations from the proof of Lemma \ref{lemfkhatgk}.
Let 
\begin{equation}\label{Io}
I_a(\eta,\omega_k):=\int_{x,y} e^{-iy\cdot \eta}\overline{e(x,y,\eta,\omega_k)}\mathcal{E}_M(x,y,a,\omega_k)dxdy.
\end{equation}
For $1\leq k\leq K_{\ceps}$, the (Fourier transform in the second variable of the) function $g_{h,a,1}$ defined in Lemma \ref{lemfkhatgk} reads as
\[
\hatdeux{g}_{h,a,1}(y',\omega_k/h^{1/3})=\frac{\sqrt{L'(\omega_k)}}{\sqrt{2\pi}}\int e^{iy'\cdot\eta}\varkappa(h\eta)r(\eta,\omega_k)I_a(\eta,\omega_k)d\eta,
\]
where $r(.,\omega_k)$ is an elliptic symbol of order $0$.
\end{cor}

\subsubsection{Idea of the proof of Proposition \ref{propdataapetitpetit2}}\label{proofprop2}
Our goal is to obtain $g_{h,a,2}$ such that \eqref{eq:g1} holds. {As observed in} \cite[Section 2.3.4]{ILLP} the sum in the second line of \eqref{eq:g1}, involving $\mathcal{E}_M$ for large $\omega_k$, equals, using \eqref{eq:AiryPoisson},
\begin{equation}\label{EMomegamare}
\langle \sum_{N\in \Z} e^{-iN L(\omega)}
    , {\cutoffchi^{\flat}}(h^{2/3}\omega/\ceps_{0})\cutoffchi^{\sharp}\Big(\frac{ \omega}{\omega_{K_{\ceps}}}\Big)\frac{L'(\omega)}{2\pi}\mathcal{E}_M(x,y,a,\omega))\rangle_{\omega},
\end{equation}
and nonstationary phase (with respect to $\omega$) easily applies for all $|N|\geq 2$ in the second line of \eqref{EMomegamare}, providing an $O(h^{\infty})$ contribution. Therefore, it remains to obtain $g_{h,a,2}$ such that
\begin{multline}\label{tosolve}
\int {\cutoffchi^{\flat}}(h^{2/3}\omega)\cutoffchi^{\sharp}( h^{2\ceps} \omega)\Big(1+\sum_{\pm} e^{\pm i L(\omega)}\Big)K_\omega(g_{h,a,2})(0,x,y)d\omega\\+\sum_{|N|\geq 2}\int e^{- iN L(\omega)}{\cutoffchi^{\flat}}(h^{2/3}\omega/\ceps_{0})\cutoffchi^{\sharp}( h^{\ceps}\omega)K_\omega(g_{h,a,2})(0,x,y)d\omega \\
=\int  {\cutoffchi^{\flat}}(h^{2/3}\omega/\ceps_{0})\cutoffchi^{\sharp}\Big(\frac{ \omega}{\omega_{K_{\ceps}}}\Big)\Big(1+\sum_{\pm} e^{\pm i L(\omega)}\Big)\frac{L'(\omega)}{2\pi}\mathcal{E}_M(x,y,a,\omega)d\omega+O(h^{\infty}),
\end{multline}
where $\mathcal{E}_M$ is given in \eqref{defEmathcalM} as the integral of a product of two Airy functions $e^{iy\eta}\mathrm{Ai}(-\zeta_M(x,\eta,\omega))\mathrm{Ai}(-\zeta_M(a,\eta,\omega))$, where $\zeta_M=\omega-xq^{1/3}(\eta)$. Using Lemma \ref{lemL},
$
1+\sum_{\pm}e^{\pm iL(\omega)}=1-\Big(\frac{A_+}{A_-}\Big)(\omega)-\Big(\frac{A_-}{A_+}\Big)(\omega)
$. As $h^{2/3}\omega_{K_{\ceps}}\sim (hK_{\ceps})^{2/3}\gtrsim (h^{1-2\ceps})^{2/3}\gg h^{2/3-\ceps}\geq a$, it follows that $h^{2/3}\omega$ is much larger than $a$ on the support of the symbol of the integral in the last line of \eqref{tosolve}, and $Ai(-\zeta_M(a,\eta,\omega))=\sum_{\pm}A_{\pm}(\zeta_M(a,\eta,\omega))$. 
\begin{lemma}
In the integral defining $\mathcal{E}_M(.,\omega)$, the standard stationary phase theorem in $\xi$ applies. Moreover, for the phase corresponding to $N=0$ in the second line of \eqref{EMomegamare}, we have
\[
\phi_{M,\pm,\mp}(x,y,\eta,\omega):= y\eta\pm\frac 23 \Big((\omega-xq^{1/3}(\eta))^{3/2}-(\omega-aq^{1/3}(\eta))^{3/2}\Big).
\]
In the same way, the phases corresponding to $N=\pm1 $ in the second line of \eqref{EMomegamare} are
\[
\phi_{M,\pm,\pm}(x,y,\eta,\omega)\mp\frac 43 \omega^{3/2}:= y\eta\pm \frac 23\Big((\omega-xq^{1/3}(\eta))^{3/2}+(\omega-aq^{1/3}(\eta))^{3/2}-2\omega^{3/2}\Big).
\]
Moreover, for $x>2h^{2/3-\ceps}$, the integral in the second line of \eqref{tosolve} is $O(h^{\infty})$.
\end{lemma}
After simplifications, the integral in the last line in \eqref{tosolve} becomes
\[
\int  {\cutoffchi^{\flat}}(h^{2/3}\omega/\eps_{0})\cutoffchi^{\sharp}\Big(\frac{ \omega}{\omega_{K_{\ceps}}}\Big)\Big(1+\sum_{\pm} e^{\pm i L(\omega)}\Big)\frac{L'(\omega)}{2\pi}\mathcal{E}_M(x,y,a,\omega)d\omega=E_{M,+}(x,y,a)+E_{M,-}(x,y,a),
\]
where we have set, modulo $O(h^{\infty})$,
\begin{multline}\label{EM1}
E_{M,+}(x,y,a):=\int  {\cutoffchi^{\flat}}(h^{2/3}\omega/\eps_{0})\cutoffchi^{\sharp}({\omega}/{\omega_{K_{\ceps}}}) \int q(\eta)^{1/6}\varkappa(h\eta)\chi_0(h^{2/3}\omega q^{2/3}(h\eta))\\
\times e^{iy\eta}\Big(A_+(\zeta_M(x,\eta,\omega))-\Big(\frac{A_+}{A_-}\Big)(\omega)A_-(\zeta_M(x,\eta,\omega))\Big)A_-(\zeta_M(a,\eta,\omega))d\eta d\omega,
\end{multline}
\begin{multline}\label{EM2}
E_{M,-}(x,y,a):=\int  {\cutoffchi}^{\flat}(h^{2/3}\omega/\eps_{0})\cutoffchi^{\sharp}({\omega}/{\omega_{K_{\ceps}}})\int q(\eta)^{1/6}\varkappa(h\eta)\chi_0(h^{2/3}\omega q^{2/3}(h\eta))\\
\times e^{iy\eta}\Big(A_-(\zeta_M(x,\eta,\omega))-\Big(\frac{A_-}{A_+}\Big)(\omega)A_+(\zeta_M(x,\eta,\omega))\Big)A_+(\zeta_M(a,\eta,\omega))d\eta d\omega,
\end{multline}
where the phase functions of the Airy terms in the second line of \eqref{EM1} are $\phi_{M,+,-}$ and $\phi_{M,-,-}+L(\omega)$, while the phase functions of the Airy terms in the second line of \eqref{EM2} are $\phi_{M,-,+}$ and $\phi_{M,+,+}-L(\omega)$. Moreover, for $x>2h^{2/3-\ceps}$, $E_{M,\pm}(x,y,a)=O(h^{\infty})$ and $E_{M,\pm}(0,y,a)=0$. This means that we can introduce a smooth cutoff $\chi_1( x/h^{2/3-\ceps})$ with $\chi_1\in C^{\infty}_0$ equal to $1$ on $[-1,1]$ and equal to $0$ for $x\geq 2h^{2/3-\ceps}$ such that $E_{M,\pm}(x,y,a)=\chi_1(x/h^{2/3-\ceps})E_{M,\pm}(x,y,a)+O(h^{\infty})$, and therefore we need to construct $g_{h,a,2}$ such that \eqref{tosolve} holds with the last line replaced by $\chi_1(x/h^{2/3-\ceps})(E_{M,+}(x,y,a)+E_{M,-}(x,y,a))$ (instead of $E_{M,+}(x,y,a)+E_{M,-}(x,y,a)$). 

Going back to \eqref{tosolve}, we observe that the symbol of its left hand side is supported in $\omega\gtrsim h^{-2\ceps}$, while the right hand side is essentially supported where $x,a\lesssim h^{2/3-\ceps}$. For such values of $x$ and $\omega$ we have $\zeta(x,y,\eta,\omega)=\omega-x|\eta|^{2/3}e_0(x,y,\eta/|\eta|,\omega_k/|\eta|^{2/3})\geq \omega/2$; using \eqref{eq:defG} we write
\[
G(x,y,\eta,\omega)=e^{i\psi}\sum_{\pm}\Big(p_0A_{\pm}(\zeta)+i x |\eta|^{-1/3}p_1A'_{\pm}(\zeta)\Big)=:G_{\pm}(x,y,\eta,\omega).
\]
We eventually obtain $g_{h,a,2}$ as in \cite[Prop. 2.34]{ILLP} that we recall below :
\begin{prop}\label{propapetitg2}
There exists smooth functions $g_{h,a,2,\pm}$ such that, with $g_{h,a,2}:=\sum_{\pm} g_{h,a,2,\pm}$,
\begin{multline}\label{g2-}
\int \cutoffchi^{\flat}(h^{2/3}\omega/\ceps_{0})\cutoffchi^{\sharp}( h^{2\ceps}\omega)\Big(G_+(x,y,\eta,\omega)-\Big(\frac{A_+}{A_-}\Big)(\omega)G_-(x,y,\eta,\omega)\Big)\\
q(\eta)^{1/6}\varkappa(h\eta)\varkappa(h\tau(\omega,\eta))\hat{g}_{h,a,2,-}(\eta,\omega/h^{1/3})d\eta d\omega
=E_{M,+}(x,y,a)+O(h^{\infty}),
\end{multline}
\begin{multline}\label{g2+}
\int\cutoffchi^{\flat}(h^{2/3}\omega/\ceps_{0})\cutoffchi^{\sharp}( h^{2\ceps}\omega) \Big(G_-(x,y,\eta,\omega)-\Big(\frac{A_-}{A_+}\Big)(\omega)G_+(x,y,\eta,\omega)\Big)\\
\times q(\eta)^{1/6}\varkappa(h\eta)\varkappa(h\tau(\omega,\eta))\hat{g}_{h,a,2,+}(\eta,\omega/h^{1/3})d\eta d\omega
=E_{M,-}(x,y,a)+O(h^{\infty}).
\end{multline}
\begin{multline}\label{g2+0}
\int\cutoffchi^{\flat}(h^{2/3}\omega/\ceps_{0})\cutoffchi^{\sharp}( h^{2\ceps}\omega) \Big(G_+(x,y,\eta,\omega)-\Big(\frac{A_+}{A_-}\Big)(\omega)G_-(x,y,\eta,\omega)\Big)\\
q(\eta)^{1/6}\varkappa(h\eta)\varkappa(h\tau(\omega,\eta))\hat{g}_{h,a,2,+}(\eta,\omega/h^{1/3})d\eta d\omega
=O(h^{\infty}),
\end{multline}
\begin{multline}\label{g2-0}
\int\cutoffchi^{\flat}(h^{2/3}\omega/\ceps_{0})\cutoffchi^{\sharp}( h^{2\ceps}\omega) 
\Big(G_-(x,y,\eta,\omega)-\Big(\frac{A_-}{A_+}\Big)(\omega)G_+(x,y,\eta,\omega)\Big)\\
\times q(\eta)^{1/6}\varkappa(h\eta)\varkappa(h\tau(\omega,\eta))\hat{g}_{h,a,2,-}(\eta,\omega/h^{1/3})d\eta d\omega
=O(h^{\infty}),
\end{multline}
\begin{equation}\label{toshowsecond}
  \sum_{|N|\geq 2 }\int e^{- iN L(\omega)}\int\cutoffchi^{\flat}(h^{2/3}\omega/\ceps_{0})\cutoffchi^{\sharp}( h^{2\ceps}\omega)
  K_\omega(g_{h,a,2})(0,x,y)d\omega =O(h^{\infty}).
\end{equation}
\end{prop}

\section{Dispersion estimates when $a\geq h^{2/3-\ceps}$}
\label{sec:dispersion-estimates}
We use the parametrix construction from \eqref{eq:Prond} and \eqref{eq:Prond2} to obtain the following dispersion estimates, restricting to positive times for simplicity.
{Before the local partition, we exchange the two kernel variables if necessary
so that the source has the larger distance to the boundary. We therefore
work with $0\le x\le a$; the reconstruction for both variables is
explained in Section~\ref{sec:completion-linear-proof}.}
\begin{thm}\label{dispintermediaire}
There exist $a_{0}$, $c$, $C$, $\ceps$ such that for all $|(t,x,y,h,a)|<a_{0}$, one has
\begin{itemize}
\item for $t\leq c \sqrt a$ and $0\leq x\leq a$
  \begin{equation}
    \label{eq:1}
|  \Prond_{S,h,a}(t,x,y)|\leq C h^{-d}  \inf \left(1, (\frac ht)^{\frac{d-1} 2}\right) \,;
  \end{equation}
\item For $a\geq \max\{h^{2/3-\ceps}, (ht)^{1/2}\}$, $t>c\sqrt a$, $\lambda=a^{3/2}/h$ and {$0<x\leq a$}
\begin{equation}
    \label{eq:dispersion-reflections}
|  \Prond_{S,h,a}(t,x,y)|\leq C h^{-d}  \left(\frac h t\right)^{\frac{d-1} 2} \times 
  \left\{ \begin{array}{l} 
 \left(  \frac{ht}{a} \right)^{\frac 1 4}\, \quad \text{ if } t/\sqrt{a}>\lambda^{1/3} ;\\
 \left(  \left(\frac{ah}{t} \right)^{\frac 1 4} +h^{\frac 1 3}\right), \quad \text{ if } t/\sqrt{a}\leq \lambda^{1/3};\,
 \end{array} \right.
  \end{equation}
\item For $a\lesssim \max\{ h^{2/3-\ceps},(ht)^{1/2}\}$
,  $t>c\sqrt a$    
 and {$0<x\leq a$}
\begin{equation}
    \label{eq:dispersion-quasimodes}
|  \Prond_{S,h,a}(t,x,y)|\leq C h^{-d}  \left(\frac h t\right)^{\frac{d-1} 2} \left((ht)^{1/4} +h^{\frac 1 3-2\ceps}\right)
 \,.
  \end{equation}
\end{itemize}
\end{thm}

The first displayed bound is only a weaker short-time bound and is not
used to deduce the final Schr\"odinger decay. For the unreflected
contribution we use the stronger exponent $d/2$ in
Proposition~\ref{propfreeN=0}; the reflected contributions and the
small-time transition are included in the summation in
Section~\ref{sec:completion-linear-proof}.
For related short-time localization and Strichartz estimates in a more
general boundary geometry, see \cite{bss08}.

The second estimate, \eqref{eq:dispersion-reflections}, is proved using the parametrix as a sum over reflected waves. The first term is due to swallowtail singularities (and always larger than the corresponding factor in the free dispersion) and the second term is due to the presence of cusps appearing after each swallowtail singularity, between two consecutive reflections; notice that here we use the parametrix construction in an extended region $a\geq h^{2/3-\ceps}$. The construction of a parametrix is available for all $a\geq h^{2/3-\ceps}$ (as in \cite{ILLP}), but only adding the assumption $a\gtrsim (ht)^{1/2}$ allows us to obtain the bounds \eqref{eq:dispersion-reflections} using the sum over reflections.

The third estimate, \eqref{eq:dispersion-quasimodes}, will be proved using the parametrix as a sum over quasimodes, and we postpone its proof to section \ref{anciennelast}, where we deal with decay of such quasimodes.

The estimates \eqref{eq:dispersion-reflections} follow from Propositions \ref{proptang} and \ref{proptransv} in Section \ref{secta>}. When $a\sim (ht)^{1/2}$, the estimates {\eqref{eq:dispersion-reflections}} and \eqref{eq:dispersion-quasimodes} yield the same bound
\[
|\Prond_{S,h,a}(t,x,y)|\leq C h^{-d}\Big(\frac ht\Big)^{(d-1)/2}(ht)^{1/4}.
\]
Indeed, in this regime,
$\frac{t}{\sqrt a}\mathrel{\sim}\frac{a^{3/2}}{h}=\lambda\gg \lambda^{1/3}$,
and hence the bounds in \eqref{eq:dispersion-quasimodes} apply with $a\sim (ht)^{1/2}$. These bounds have recently been improved in this regime in the one-dimensional setting, for the model operator $\partial_x^2+x$, by means of van der Corput derivative tests (see \cite{doi25}). Obtaining analogous improvements in higher dimensions appears to be considerably more difficult, even for model domains.\\

We start with a lemma which allows us to deal with the parametrix ``behind the wave front''.
\begin{lemma}
\label{derriere}
There exist $c_{0}$ and $T_{0}$ such that , with $\mathcal{B}=\{ 0\leq x\leq a, |y|\leq c_{0} t, 0<h\leq t \}$, 
  \begin{equation}
    \label{eq:2bis}
\forall t\in [0, T_{0}]\,,\quad\quad   \sup_{x,y,t\in \mathcal{B}}  |  \Prond_{S,h,a}(t,x,y)|\leq C h^{-d}  O(( h/ t)^{\infty})\,.
  \end{equation}
\end{lemma}
\begin{proof}
The lemma follows from classical propagation of singularities: for $t\leq a/C_{1}$ with a large $C_{1}$, we apply Melrose-Sjöstrand's theorem in the interior of $\Omega$, as for data localized at small distance $a\ll 1$ from the boundary the time needed to reach the boundary is $~a$, hence the solution to \eqref{schrod} equals the free solution as long as $t\ll a$.

Now, let $T_{0}$ be small enough and consider a given $s\in [h,T_{0}]$. We rescale using $(t',x',y')$ as new variables and $\hbar$ a new parameter: $t=st'$, $x=sx'$, $y=sy'$, $\hbar=h/s$, let $v_{s}(t',x',y')=v(st',sx',sy')$ for any function $v$, then
\begin{equation}
  \label{eq:5}
  (i\hbar\partial_{t'}-\hbar^2\Delta_{g,s}) v_{s} =(ih\partial_t-h^2 \Delta_{g}) v\,\,\text{with }\,\, \Delta_{g,s} : =\partial^{2}_{x'}+s^{2} R(sx', sy', s^{-1} D_{y'})\,.
\end{equation}
Set $b=a/s$, $0< b\leq c_{1}$ and $\Prond_{S,b,s,\hbar}=s^{d} (\Prond_{S,a,h})_{s}$; then one may apply the Melrose-Sj\"ostrand theorem to $\Prond_{S,b,s,\hbar}$ to obtain $\Prond_{S,b,s,\hbar}\in O(\hbar^{\infty})$ for $0\leq x'\leq b$, $|y'|\leq c_{0}$ and $t'=1$, when $s\leq T_{0}$ and $T_0$ is small enough. In fact, $\Prond_{S,b,s,\hbar}$ is a parametrix for $ (i\hbar\partial_{t'}-\hbar^2\Delta_{g,s})$, and $ (i\hbar\partial_{t'}-\hbar^2\Delta_{g,s})$ is smooth in $s$ and for $s$ small enough, $ (i\hbar\partial_{t'}-\hbar^2\Delta_{g,s})$ is close to the usual semiclassical Schrödinger operator {$i\hbar\partial_{t'}-\hbar^2(\partial^{2}_{x'}+\Delta_{y'})$} (recall that we picked geodesic normal coordinates at the boundary, and $R_{0}(y,
\eta)=|\eta|^{2}+O(|y|)$). 
To complete our proof, we need to check the following property of an optical ray $\sigma\in [0,1]\ra (x'(\sigma),y'(\sigma),\xi'(\sigma),\eta'(\sigma))(\sigma)$: if it starts at $\sigma=0$ from $(b,0,\xi_{0},\eta_{0})$ with $\xi^{2}_{0}+R(sb,0;\eta_{0})=1$ and  if $|\eta_{0}|\geq c_{1}>0$ for some constant $c_1$, then there exists $c_0$ such that $|y'(1)|\geq c_{0}>0$ (notice that $\partial_{\sigma}\tau=0$ and $\partial_{\sigma}t'=1$ which implies $t=st'$, $t'=1$).
Note that $(y',\eta')(\sigma)$ are solutions to the Hamilton-Jacobi equations $\partial_{\sigma} y'=\partial_{\eta'} R_{s}$, and $\partial_{\sigma} \eta'=- \partial_{y'} R_{s}$, with {$x'(0)=b$, $y'(0)=0$, $\eta'(0)=\eta_{0}$, $R_{s}(x',y',\eta')=s^{2} R(sx',sy',\eta'/s)$}. For some $c>0$ small, $|y'(\sigma)|\geq 2 c_{1} \sigma-cs\sigma$ for all $\sigma \in [0,1]$, and if $T_{0}$ is small enough, as $s\leq T_0$, we obtain the lower bound $|y'(1)|\geq c_{0}>0$ for some $c_0>0$. See also \cite[Appendix]{doi08} for details on the almost finite speed of propagation for the semiclassical Schrödinger equation (where all the arguments work for strictly convex boundaries).
\end{proof}

\subsection{Localisation of the parametrix depending on the initial direction}\label{sectcardN1}
This dyadic decomposition is internal to a single glancing piece.
Unlike the fixed spatial and conic partition, its cardinality may grow
with $h^{-1}$. Its contribution, including the logarithm from summing
the $h^{1/3}$ terms, is estimated in
Section~\ref{sec:completion-linear-proof}.
We further localize our parametrix $\Prond_{S,h,a}$: let $\phi\in C^{\infty}_{0}(\R)$ be even, $\phi=1$ on $[-1,1]$ and $\phi=0$ outside $(-3/2,3/2)$, set $\kappa=\phi-\phi(2\cdot)$, then define (see \eqref{eq:Kequiv}), for $\gamma =2^j a \leq \ceps_{0}$,
\begin{equation}\label{defKomegakgamma}
K_{S,\omega,\gamma}(f)(t,x,y):=\int e^{iht\tau^2(\omega,\eta)}G(x,y,\eta,\omega)\cutoffchi^{\sharp}(\omega) q^{1/6}(\eta)\varkappa(h\eta)\varkappa(h\tau(\omega,\eta))\kappa\Big(\frac\omega{\gamma q(\eta)^{1/3}}\Big)\hat{f}\Big(\eta,\frac \omega{h^{1/3}}\Big)\,d\eta
\end{equation}
as well as $\Prond_{S,h,a,\gamma}$ by replacing $K_{S,\omega}$ with $K_{S,\omega,\gamma}$ in \eqref{eq:Prond2cut}.
This localization gives $\alpha\sim\gamma$ and
$k\sim\gamma^{3/2}/h$.
The preliminary truncation \eqref{eq:suminN-gamma} is valid
for $\gamma\gtrsim h^{2/3}$.
In the present regime, the contributing pieces satisfy
$\gamma\gtrsim a>h^{2/3-\ceps}$, so that
$\gamma^{3/2}/h\gtrsim h^{-3\ceps/2}$. Then,
$\Prond_{S,h,a}(t,x,y)=\sum_{\gamma}\Prond_{S,h,a,\gamma}(t,x,y)$.
{We group the finitely many cutoffs with $\gamma\sim a$ into a single
central cutoff, still denoted by $\gamma=a$.}
For $\gamma\ll a$, the corresponding $\Prond_{S,h,a,\gamma}$ is irrelevant, as the phase of $g_{h,a}$ is non-stationary. Using  \eqref{eq:Prondcut} or \eqref{eq:AiryPoisson},
\begin{equation}\label{defProndgamma}
\Prond_{S,h,a,\gamma}(t,x,y)= \langle \sum_{N\in \Z} e^{-iN L(\omega)}
    , {\cutoffchi^{\flat}}(h^{2/3}\omega/\ceps_0)K_{S,\omega,\gamma}(g_{h,a})(t,x,y)\rangle_{\omega}.
\end{equation}
By \eqref{eq:suminN-gamma}, the sum over the tail
$|N|\ge C\gamma^{-1/2}$ is negligible under the hypotheses
of Lemma \ref{sommeNfinie}.
This preliminary truncation uses $\gamma$, not $a$; the next lemma
uses the full phase to give the sharper time-dependent bound.
We distinguish $\gamma\sim a$, corresponding to tangent initial
directions, from $\gamma\ge4a$, corresponding to almost transverse
initial directions. The second regime gives smaller contributions,
even after summation in $\gamma$. We write
\begin{align}
  \label{eq:newVNgam}
  W_{N,h,\gamma}(t,x,y) & =\int e^{-iNL(\omega)}\cutoffchi^{\flat}(h^{2/3}\omega/\ceps_{0})K_{S,\omega,\gamma}(g_{h,a})(t,x,y)d\omega\\
  & = \begin{multlined}[t]\frac 1 {2\pi h^{d+1}} \int e^{\frac i h (
    t\tau^2(\alpha,\theta)+\Phi(x,y,\theta,\alpha,\sigma)-\Phi(a,0,\theta,\alpha,s)-Nh L(h^{-2/3}\alpha)
    )}\kappa(\alpha/(q(\theta)^{1/3}\gamma)) \\
  {}\times {\cutoffchi^{\flat}}(\alpha/\ceps_{0}){\cutoffchi^{\sharp}}(\alpha/h^{2/3})\chi(s)p_h(x,y,\theta,\alpha,\sigma)\tilde q_h (\theta,\alpha,s) ds \,d\theta  d{\sigma}d\alpha\,,\end{multlined}\\
    \label{eq:newProndgam}
  \Prond_{S,h,a,\gamma}(t,x,y) & =\sum_{N\in \Z} W_{N,h,\gamma}(t,x,y)\,,
\end{align}
where the symbol of $W_{N,\gamma}$ (the same for every $N$) is of degree $0$ and where we used \eqref{eq:newVN}.
\begin{lemma}\label{lemmaNgam}
 At fixed $|t|\lesssim T_0$, the sum \eqref{eq:newProndgam} defining $\Prond_{S,h,a,\gamma}$ is finite and {$|N|\lesssim 1+|t|\gamma^{-1/2}$.}
\end{lemma}
{On the stationary support, $\sigma,s=O(\sqrt\gamma)$,
$x,a=O(\gamma)$ and $|y|\lesssim t$.
The $\alpha$ equation for the phase in \eqref{eq:newVNgam} therefore
gives $|N|\sqrt\gamma\lesssim t+\sqrt\gamma$.
The additive constant retains the possible first reflection at
$t\ll\sqrt\gamma$.}
The proof follows Proposition \ref{propsommefinie} with
$\alpha\sim\gamma$. For bounded times, the number of possible
indices is $O(\gamma^{-1/2})$, hence $O(a^{-1/2})$ in the tangent
regime $\gamma\sim a$. Proposition \ref{propsommefinie} shows that
at $t=0$ only $N=0$ contributes. 

There is one main difference with the wave equation: if in the case of the wave equation at fixed $t\geq C\gamma^{1/2}$ we could further bound the number of indices $N$ that contribute significantly among the $C|t|\gamma^{-1/2}$ which are left (see \cite[Section 3.1]{ILLP}), in the case of the semiclassical Schrödinger equation the proof of Proposition \ref{propsommefinie} shows that all the oscillatory integrals $W_{N,\gamma}$ with $N\sim t/\sqrt{\gamma}$ provide important contributions.

In what follows, we distinguish the cases $a\geq (ht)^{1/2}$ and $a\leq (ht)^{1/2}$ : in the first case, the sum over the number of reflections $N$ has fewer terms than the number of quasimodes $k$ : we will estimate each $\|W_{N,h,\gamma}\|_{L^{\infty}}$ and then sum the contributions. In the second case the number of terms in the sum over $\omega_k$ in \eqref{eq:Prond2cut} has fewer terms : in this situation we estimate each $K_{S,\omega_k}(g_{h,a})$ and then sum the contributions. The two sums have the same number of terms when $a\sim (ht)^{1/2}$, in which case the two methods provide the same result (with $1/4$ loss in the dispersive bounds).

\subsection{Case $(ht)^{1/2}\leq a \leq \varepsilon_0$} We recall that in this section we also have $\gamma\geq a>h^{2/3-\epsilon}$, which is particularly important as otherwise one would not be able to use stationary phase methods to deal with $W_{N,h,\gamma}$.

Let $\frac 1h(\Phi_{N,a,\gamma}+NhB_L(\alpha^{3/2}/h))$ denote the phase function of $W_{N,h,\gamma}$, defined as follows
\begin{equation}\label{PhiNagamma}
\Phi_{N,a,\gamma}:= t\tau^2(\alpha,\theta)+\Phi(x,y,\theta,\alpha,\sigma)-\Phi(a,0,\theta,\alpha,s)-\frac 43N\alpha^{3/2}.
\end{equation}
Here $\Phi(x,y,\theta,\alpha,\sigma)$ is the phase function of $G(x,y,\theta/h,\omega/h^{2/3})$ defined in \eqref{eq:Phi}. 
We rescale variables:
\begin{gather}\label{gatherchofvar}
 \text{if } \gamma\sim a \text{ we set } s=\sqrt{a}q(\theta)^{1/6}S, {\sigma=\sqrt{a}q(\theta)^{1/6}\Sigma}, \alpha=aq(\theta)^{1/3}A \,;\\
 \text{if } 4a\leq \gamma \text{ we set }
 s=\sqrt{\gamma}q(\theta)^{1/6}S, {\sigma=\sqrt{\gamma}q(\theta)^{1/6}\Sigma}, \alpha=\gamma q(\theta)^{1/3}A\,.
\end{gather}
\begin{rmq}
In this regime, the factor $e^{iNB_L(\gamma^{3/2}A^{3/2}q^{1/2}(\theta)/h)}$ can be moved into the symbol as {$|N|\lesssim 1+t/\sqrt\gamma$} by Lemma \ref{lemmaNgam} and therefore, for $\theta$ on the support of $\varkappa$,
\[
N B_L(\gamma^{3/2}A^{3/2}q^{1/2}(\theta)/h)\sim \frac{N h b_1}{ q(\theta)^{1/2}\gamma^{3/2}}\lesssim {\frac h{\gamma^{3/2}}+\frac{ht}{\gamma^2}}\lesssim 1, \quad \text{ as } \gamma\geq a\geq (ht)^{1/2}.
\]
\end{rmq}
We recall that $\Phi(x,y,\theta,\alpha,\sigma)=\Phi_M(x,y,\theta,\alpha,\sigma)+\tau(\alpha,\theta)(B_{\Gamma}(y,\theta/\tau)+xA_{\Gamma}(x,y,\sigma q^{1/3}(\theta)/\tau,\theta/\tau))$, where $\tau=\tau(\alpha,\theta)=\sqrt{|\theta|^2+\alpha q^{2/3}(\theta)}$ and $\tau(\gamma q(\theta)^{1/3} A,\theta)-|\theta|=\frac{\gamma Aq(\theta)}{|\theta|+\tau_q}=\frac{\gamma Aq(\theta)}{2|\theta|}+O(\gamma^2)$. We let
\begin{align}
\tilde\Phi_{M,\gamma}(x,y,{\theta},A,\Sigma):&=\Phi_M(x,y,\theta,\gamma q(\theta)^{1/3}A,\sqrt{\gamma}q(\theta)^{1/6}\Sigma)\\
&=<y, {\theta}>+\gamma^{3/2}q({\theta})^{1/2}\Big(\frac{\Sigma^3}{3}+\Sigma(\frac{x}{\gamma}-A)
\Big).
\end{align}
Set $\tilde\Phi_{\gamma}(x,y,{\theta},A,\Sigma):=\Phi(x,y,\theta,\gamma q(\theta)^{1/3}A,\sqrt{\gamma}q(\theta)^{1/6}\Sigma)$, then, with $\tau=\tau(\gamma A q(\theta)^{1/3},\theta)$, we have
\begin{gather}
\tilde\Phi_{\gamma}(x,y,{\theta},A,\Sigma) := \begin{multlined}[t]\tilde\Phi_{M,\gamma}(x,y,{\theta},A,\Sigma)+\tau \Big[B_{\Gamma}(y,\theta/ \tau)
+xA_{\Gamma}\Big(x,y,\sqrt{\gamma}\Sigma q^{1/2}(\theta)/\tau,\theta/\tau \Big)\Big]\,,\end{multlined}\\
\tilde\Phi_{\gamma}(a,0,{\theta},A,S) :=\tilde\Phi_{M,\gamma}(a,0,{\theta},A,S)+\tau aA_{\Gamma}\Big(a,0,\sqrt{\gamma}S q^{1/2}(\theta)/\tau,\theta/\tau\Big)\,,
\end{gather}
where in the last line we used that $B_{\Gamma}(0,\theta/\tau)=0$.
We expand $B_\Gamma$ directly using \eqref{formalseriesBA}, with
$\omega_\theta=\theta/|\theta|$ and $\tau^2=|\theta|^2+\gamma Aq(\theta)$.
In particular, $\tau B_0(y,\omega_\theta)$ contributes both
$B_0(y,\theta)$ and $(\tau-|\theta|)B_0(y,\omega_\theta)$.
Since
\begin{equation}\label{eq:radial-tau-expansion}
 \tau-|\theta|
 =\frac{\gamma Aq(\theta)}{2|\theta|}
 -\frac{(\tau-|\theta|)^2}{2|\theta|},
\end{equation}
the boundary contribution to the phase is
\begin{equation}\label{eq:boundary-phase-expansion}
 \tau B_\Gamma(y,\theta/\tau)
 ={} B_0(y,\theta)
 +\frac{\gamma Aq(\theta)}{2|\theta|}(B_0-B_2)(y,\omega_\theta)  +\frac{\gamma^2A^2q(\theta)^2}{8|\theta|^3}
 (2B_4-B_0+B_2)(y,\omega_\theta)+O_{C^m}(\gamma^3)\,.
\end{equation}
These bounds hold for bounded $A$, $|\theta|\sim1$ and $y$ in the fixed
patch, with any fixed number of derivatives in $(y,A,\theta)$.
The remainder is obtained by a finite Taylor expansion; convergence
of the formal series is not needed.
In these variables we have
\begin{gather}
\tau B_{\Gamma}(y,\theta/{\tau})=  
  \label{Bgamhomog}
 B_0(y,\theta)+(\tau-|\theta|)(B_0-B_2)(y,\omega_\theta) {}+\frac{(\tau-|\theta|)^2}{\tau}B_4(y,\omega_\theta)
 +O_{C^m}(\gamma^3)\,,
\\
\tau A_{\Gamma}\Big(x,y,\sqrt{\gamma}\Sigma q^{1/2}(\theta)/\tau,\theta/\tau\Big)=
\begin{multlined}[t]\label{Agamtau}
\sqrt\gamma q^{1/2}(\theta)\Sigma\Lp
 +(\gamma q(\theta)/\tau)(\Sigma^2-A)\mu_\Gamma\\
{}+\alpha_1x\sqrt\gamma q^{1/2}(\theta)\Sigma
 +\mu_1\gamma^{3/2}q^{3/2}(\theta)\Sigma^3/\tau^2+\tau R_A,
\end{multlined}
\end{gather}
{Here the coefficients are evaluated at
$(y,\omega_\theta)$, $\omega_\theta=\theta/|\theta|$, and $R_A$ at
$(x,y,\sqrt\gamma\Sigma q^{1/2}(\theta)/\tau,\theta/\tau)$.
We used \eqref{eq:normalized-AGamma}; the third-order terms are
$O_{C^m}(\gamma^{3/2})$ and $\tau R_A=O_{C^m}(\gamma^2)$
for bounded $x/\gamma,\Sigma,A$ and $|\theta|\sim1$.} We set
\begin{equation}\label{deftildephiNagamma}
\tilde\Phi_{N,a,h,\gamma}(t,x,y,\Sigma,S,A,\theta):=\Phi_{N,a,h}(t,x,y,\sqrt{\gamma}q(\theta)^{1/6}\Sigma,\sqrt{\gamma}q(\theta)^{1/6}S,\gamma q(\theta)^{1/3}A,\theta)\,,
\end{equation}
and from \eqref{PhiNagamma}, $\tilde\Phi_{N,a,h,\gamma}=\tilde\Phi_{0,a,h,\gamma}-\frac 43\gamma^{3/2} q(\theta)^{1/2} N A^{3/2}$, where
\begin{equation}
\label{eqtildePhiN}
\tilde\Phi_{0,a,h,\gamma}(t,x,y,\Sigma,S,A,\theta)=|\theta|^2t+\gamma tA q({\theta})+\tilde\Phi_{\gamma}(x,y,{\theta},A,\Sigma)-\tilde\Phi_{\gamma}(a,0,{\theta},A,S).
\end{equation}
With parameters $(\theta,\omega_{\theta}=\frac{\theta}{|\theta|},\Sigma,S,A)$, the phase function $\tilde\Phi_{N,a,h,\gamma}$ is stationary when
\begin{gather}\label{desclambdaN}
\left\{ \begin{array}{l}
t \gamma q(\theta)+\partial_{A}\Big(\tilde\Phi_{\gamma}(x,y,\theta,A,\Sigma)-\tilde\Phi_{\gamma}(a,0,\theta,A,S)\Big)=2N \gamma^{3/2}q(\theta)^{1/2}A^{1/2}, \\
  \Sigma^{2}+\frac{x}{\gamma} \Big(1+\frac{{\tau(\gamma Aq(\theta)^{1/3},\theta)} }{\sqrt{\gamma}q(\theta)^{1/2}}\partial_{\Sigma}\Big(A_{\Gamma}(x,y,\sqrt{\gamma}\Sigma q^{1/2}(\theta)/\tau,\theta/\tau)\Big)\Big)=A,\\
 S^{2}+\frac{a}{\gamma} \Big(1+\frac{ {\tau(\gamma Aq(\theta)^{1/3},\theta)} }{\sqrt{\gamma}q(\theta)^{1/2}}\partial_{S}\Big(A_{\Gamma}(a,0,\sqrt{\gamma}S q^{1/2}(\theta)/\tau,\theta/\tau)\Big)\Big)=A,\\
2t\theta +\gamma tA\nabla q(\theta)+\nabla(\tilde\Phi_{\gamma}(x,y,\theta,A,\Sigma)-\tilde\Phi_{\gamma}(a,0,\theta,A,S))=\frac 43 \gamma^{3/2}NA^{3/2}\frac{\nabla q(\theta)}{2q(\theta)^{1/2}}.
   \end{array} \right.
\end{gather}
where $\nabla$ is the gradient with respect to $\theta\in\mathbb{R}^{d-1}$.
Using \eqref{Bgamhomog} and \eqref{Agamtau}, we write
\begin{multline}
\tilde\Phi_{0,a,h,\gamma}(t,x,y,\Sigma,S,A,\theta)=|\theta|^2t+\gamma tA q({\theta})+\tilde \Phi_{M,\gamma}(x,y,\theta,A,\Sigma)-\tilde\Phi_{M,\gamma}(a,0,\theta,A,S)\\+ B_0(y,\theta)+{\frac{\gamma A q(\theta)}{2|\theta|}(B_0-B_2)(y,\omega_\theta)}+\gamma^{3/2}\frac{x}{\gamma}q(\theta)^{1/2}\Sigma \Lp (y,\omega_\theta)+\gamma^2(\tilde \Gamma_{\gamma}(x,y,\theta,\sqrt{\gamma}\Sigma,A)-\tilde \Gamma_{\gamma}(a,0,\theta,\sqrt{\gamma}S,A)),
\end{multline}
where we have set, for {$\tau=\tau(\gamma Aq(\theta)^{1/3},\theta)$}, $\omega_{\theta}=\theta/|\theta|$, $\tilde\Gamma_{\gamma}:=\tilde\Gamma_{\gamma}(x,y,\theta,\sqrt{\gamma}\Sigma,A)$ as follows
\begin{multline}\label{deftildeGammagam}
\tilde\Gamma_\gamma:=\gamma^{-2}\left\{
 \tau B_\Gamma(y,\theta/\tau)-B_0(y,\theta)
 -\frac{\gamma Aq(\theta)}{2|\theta|}(B_0-B_2)(y,\omega_\theta)\right.\\
\left.{}+x\left(
 \tau A_\Gamma(x,y,\sqrt\gamma\Sigma q^{1/2}(\theta)/\tau,\theta/\tau)
 -\sqrt\gamma q^{1/2}(\theta)\Sigma\Lp(y,\omega_\theta)
 \right)\right\}.
\end{multline}
From \eqref{Bgamhomog} and \eqref{Agamtau}, $\tilde\Gamma_{\gamma}$ is a smooth bounded function and, as $|\theta|-\tau=\frac{|\theta|^2-\tau^2}{2|\theta|}+\frac{(|\theta|-\tau)^2}{2|\theta|}$,
\begin{multline}\label{Gammatildegam}
   \Xi=\frac{\sqrt\gamma\Sigma q^{1/2}(\theta)}{\tau}\,,\quad\quad
\tilde\Gamma_\gamma
 =\frac{(\tau-|\theta|)^2}{\gamma^2\tau}
 \left\{\frac{\tau}{2|\theta|}(B_2-B_0)(y,\omega_\theta)
 +B_4(y,\omega_\theta)+O_{C^m}(\gamma)\right\}\\
{}+\frac{x}{\gamma}\left\{
 \frac{q(\theta)}{\tau}(\Sigma^2-A)\mu_\Gamma
 +\gamma^{-1}(\tau\alpha_1x\Xi+\tau\mu_1\Xi^3+\tau R_A)
 \right\}\,.
\end{multline}
Here $(|\theta|-\tau)^2/\gamma^2=A^2q^2(\theta)/(|\theta|+\tau)^2$
is bounded. {Since $B_0,B_2=O(|y|^2)$
and $B_{2j}=O(|y|)$ for $j\ge2$, the boundary term in
\eqref{Gammatildegam} is $O_{C^m}(|y|)$ and is independent of
$\Sigma,S$ before substitution of the critical points.}
We use the following bounds where the phase is stationary: the equations provide $x,a=O(\gamma)$ and bounded $\Sigma,S$;
outside this region we integrate by parts in $\Sigma$ or $S$. {After this substitution, a weighted remainder of order $j$
is $O_{C^m}(\gamma^{j/2})$,}
for any fixed number of derivatives in $(A,\Sigma,S,\theta,x/\gamma,a/\gamma)$, with $y$ in our fixed neighborhood.
 Indeed $\Xi=O(\sqrt\gamma)$,
$\rho-1=O(\gamma)$, and differentiating these expressions in the
rescaled variables preserves these factors. Thus
{$\gamma^{-1}
(\tau\alpha_1x\Xi+\tau\mu_1\Xi^3+\tau R_A)
=O_{C^m}(\sqrt\gamma)$; the contribution of $\tau R_A$ alone is
$O_{C^m}(\gamma)$. These are the third-order and higher terms in
\eqref{Gammatildegam}.} The bounds for
$\tilde\Gamma_\gamma$, including its first two $A$ derivatives,
are uniform. Its $\Sigma$ and $S$ derivatives after one $A$
derivative are $O(\sqrt\gamma)$.
For the $x\tau A_\Gamma$ part of the phase, the present normalization
also gives
\begin{equation}\label{eq:normalized-endpoint-derivatives}
 \partial_A(x\tau A_\Gamma)
 =-\gamma x\frac{q(\theta)}{|\theta|}\mu_\Gamma+O_{C^m}(\gamma^3),
 \qquad
 \partial_A^2(x\tau A_\Gamma),\ 
 \partial_A\partial_\Sigma(x\tau A_\Gamma)=O_{C^m}(\gamma^3).
\end{equation}
These derivatives are taken at fixed $(x/\gamma,y,\Sigma,\theta)$.
The term with factor $\alpha_1$ is independent of $A$.
The same estimates hold at the source point, with $(x,y,\Sigma)$
replaced by $(a,0,S)$; see Appendix~\ref{app:canonical-normalization}.
The weaker bounds used below therefore remain valid.

We first deal with the first reflection, which occurs at short times: we separate $N=0,\pm1$ from $|N|\ge2$, as in the wave construction
of \cite{ILLP}. Here the time phase is $t\tau^2$.
Fix $c_*>0$ sufficiently small. Assume
$a\ge h^{2/3-\varepsilon}$, $\gamma\ge a$, $0\le x\le2a$ and
$h<t\le c_*\sqrt\gamma$. In this interval the terms with
$|N|\ge2$ are non-stationary. For positive times the possible
short reflected trajectory belongs to $N=1$; negative times
are dealt with by conjugation.

Indeed, the equations in $\Sigma,S$ give
$|\Sigma|,|S|\le\sqrt A+O(T_0+\sqrt\gamma)$.
The $A$ equation in \eqref{desclambdaN} then excludes $|N|\ge2$
when $t/\sqrt\gamma$ is small. For $N=1$ it selects
$\Sigma<0<S$. The full equations also imply
$(x+a)/\gamma\lesssim c_*$. Outside this region, or for the
other sign choices, the phase is non-stationary.
Thus $|\Sigma|,|S|$ are bounded below on the remaining support.
We first apply stationary phase in these two variables, keeping
$A,\theta$ fixed, with parameter $\lambda_\gamma=\gamma^{3/2}/h$.
The factor is $\lambda_\gamma^{-1}$.

In the model the two critical points are
$\Sigma=-\sqrt{A-x/\gamma}$ and $S=\sqrt{A-a/\gamma}$.
Using \eqref{Bgamhomog}, \eqref{Agamtau} and the bounds following
\eqref{Gammatildegam}, stationary phase in $\Sigma,S$ for the full
phase gives
\begin{equation}\label{eq:first-reflection-reduced-phase}
 t|\theta|^2+y\cdot\theta+B_0(y,\theta)+\gamma tAq(\theta)+\frac23\gamma^{3/2}\sqrt{q(\theta)}
 \left\{(A-x/\gamma)^{3/2}+(A-a/\gamma)^{3/2}-2A^{3/2}\right\}
 +R\,.
\end{equation}
For any fixed required number of derivatives in $A,\theta$,
\begin{equation}\label{eq:first-reflection-remainder}
 R=O_{C^k}\left((x+a)\gamma+(x+a)\sqrt\gamma\,|y|
                  +\gamma|y|^2+\gamma^2|y|\right).
\end{equation}
The terms containing $\mu_\Gamma$ have the first bound; those
containing $\Lp$ have the second. The linear boundary term
containing $B_0-B_2$ and the higher boundary terms give the last
two bounds. Stationary substitution preserves these estimates,
since the two eliminated second derivatives are bounded away
from zero after division by $\gamma^{3/2}$.
The factor $e^{iB_L}$ remains in the symbol, with uniform bounds
because $\lambda_\gamma\gg1$.

For the principal part, the equation in $A$ is
\begin{equation}\label{eq:first-reflection-time}
 t\sqrt{q(\theta)}/\sqrt\gamma
 =2\sqrt A-\sqrt{A-x/\gamma}-\sqrt{A-a/\gamma}.
\end{equation}
The right-hand side is comparable to $(x+a)/\gamma$.
The corrections in \eqref{eq:first-reflection-remainder} are
smaller, so the full equations give
$t\gamma\sim(x+a)\sqrt\gamma$ and $|y|\lesssim t$.
At such a point the second $A$ derivative of the principal part is
\begin{equation}\label{eq:first-reflection-A-hessian}
 \gamma^{3/2}\sqrt{q(\theta)}
 \left\{\frac1{2\sqrt{A-x/\gamma}}
       +\frac1{2\sqrt{A-a/\gamma}}-\frac1{\sqrt A}\right\}
 \sim(x+a)\sqrt\gamma\sim t\gamma.
\end{equation}
All derivatives of the remainder needed here are
$O(t\gamma(T_0+\sqrt\gamma))$ at the critical point.
The full mixed block in $A,\theta$ is $O(t\gamma)$, and the
$\theta$ block is $2t\,Id+O(t(T_0+\gamma))$.
Eliminating $A$ changes this last block by $O(t\gamma)$.
Thus the system is uniformly nondegenerate with scales
$t\gamma$ and $t$, and its normalized derivatives are bounded.
The $A$ scale is $t\gamma$, not $\gamma^{3/2}$.

There is no stationary reflected trajectory when $t\gamma\lesssim h$:
\begin{equation}\label{eq:first-reflection-large-parameter}
 t\gamma\sim(x+a)\sqrt\gamma\ge a\sqrt\gamma
 \ge a^{3/2}\ge h^{1-3\varepsilon/2}\gg h
\end{equation}
would be required. In this case integrations by parts give
$O(h^\infty)$, since $h/(a\sqrt\gamma)\le h^{3\varepsilon/2}$.
Near a stationary point, stationary phase in $A,\theta$ gives
$(h/(t\gamma))^{1/2}(h/t)^{(d-1)/2}$.
Together with the initial Jacobian and the factor from $\Sigma,S$,
\begin{equation}\label{eq:first-reflection-free-factor}
 \frac{\gamma^2}{h^{d+1}}\lambda_\gamma^{-1}
 \left(\frac h{t\gamma}\right)^{1/2}
 \left(\frac ht\right)^{(d-1)/2}
 =h^{-d}\left(\frac ht\right)^{d/2}.
\end{equation}
For fixed $t,x,a$, stationarity requires
$\gamma\sim((x+a)/t)^2$, leaving only a bounded number of dyadic
pieces. On the other dyadic pieces the $A$ derivative is bounded below
by a constant times $t\gamma+(x+a)\sqrt\gamma$, and repeated
integrations by parts give a summable $O(h^\infty)$ remainder.
We have proved
\begin{equation}\label{eq:first-reflection-short-sum}
 \sum_{\substack{\gamma=2^j a\le\varepsilon_0\\t\le c_*\sqrt\gamma}}
 \bigl(|W_{1,h,\gamma}|+|W_{-1,h,\gamma}|\bigr)
 \lesssim h^{-d}(h/t)^{d/2}.
\end{equation}
The unreflected term is estimated after summing in $\gamma$ by
Proposition \ref{propfreeN=0}. Hence the short-time contributions
have free dispersion.

We now restrict to $t\ge c_*\sqrt\gamma$. In this regime the
constant in $1+t/\sqrt\gamma$ is absorbed, and the contributing
nonzero indices satisfy $|N|\sim t/\sqrt\gamma$.
The following reduction in $A,\theta$ and the subsequent
reflection estimates are used in this regime only.
We turn to critical points $A$ and $\theta$ of $\tilde \Phi_{N,a,h,\gamma}$.
\begin{prop}\label{propcritptsalphaeta}
Let $|N|\geq 1$ {and $t\ge c_*\sqrt\gamma$.} The phase function $\tilde\Phi_{N,a,h,\gamma}$ can have at most one critical point $(A_c,\theta_c)$ on the support of the symbol. At critical points in $(A,\theta)$, the stationary phase theorem applies in both $A\sim 1$ and $\theta\in\mathbb{R}^{d-1}$ with $|\theta|\sim 1$ and yields a decay factor $(h/t)^{(d-1)/2}\times{(\gamma^{3/2}N/h)}^{-1/2}$.
\end{prop}
\begin{proof}
We recall from \eqref{eq:newVNgam} that $W_{N,\gamma}$ has symbol supported where $|\theta|\sim 1$ and $\frac{\alpha}{|\theta|\gamma}\sim 1$ : after the change of variable $\alpha=\gamma q(\theta)^{1/3}A$ we obtain $A\sim 1$ on the support of $\kappa(A)$. The critical points with respect to $A,\theta$ satisfy the first and the last equations in the system \eqref{desclambdaN}. We further compute 
\begin{equation}
\label{derivAPhi}
\quad  \partial_{A}\tilde\Phi_{\gamma}(x,y,\theta,A,\Sigma)
  ={\frac{\gamma q(\theta)}{2|\theta|}(B_0-B_2)(y,\omega_\theta)}-\gamma^{3/2}q(\theta)^{1/2}\Sigma+\gamma^{2}\partial_A\tilde\Gamma_{\gamma}(x,y,\theta,\sqrt{\gamma}\Sigma,A).
\end{equation}
{The first equation of \eqref{desclambdaN} therefore reads as follows.
The factor $|\theta|^{-1}$ is retained, since $(B_0-B_2)(y,\omega_\theta)$
is evaluated on the unit sphere.}
{\begin{equation}\label{eq1Ac}
\quad\quad q(\theta)^{\frac 12}{\gamma}^{-\frac 12}\Bigl (t+\frac{(B_0-B_2)(y,\omega_\theta)}{2|\theta|}\Bigr)-(\Sigma-S)+\frac{\sqrt{\gamma} }{q(\theta)^{1/2}}\partial_A(\tilde\Gamma_{\gamma}(x,y,\theta,\sqrt{\gamma}\Sigma,A)-\tilde\Gamma_{\gamma}(a,0,\theta,\sqrt{\gamma}S,A))
=2NA^{1/2}.
\end{equation}}
\begin{rmq}
In the region considered below, $x,a=O(\gamma)$ and $\Sigma,S$
are bounded. At a critical point on the support, equation
\eqref{eq1Ac} and Lemma \ref{lemmaNgam} give
$t\sim |N|\sqrt\gamma$. Moreover, $|y|\lesssim t$ and
$(B_0-B_2)(y,\omega_\theta)=O(|y|^2)$, so that
$t+(B_0-B_2)(y,\omega_\theta)/(2|\theta|)=t(1+O(T_0))$.
For fixed $\theta$, equation \eqref{eq1Ac} has at most one solution
with $A\sim1$: the derivative of its left-hand side minus its
right-hand side is $-NA^{-1/2}+O(\sqrt\gamma)$.
Whenever this solution exists, it depends smoothly on the other
variables. We denote it by $A_c(\theta)$ before solving {the equation for the critical point in $\theta$}. This does not assert that a critical point
lies in the support for every choice of parameters.
\end{rmq}
The second derivative $\partial^2_A\tilde\Phi_{N,a,h,\gamma}$ with respect to $A$ equals
 \[
-N\gamma^{3/2}q(\theta)^{1/2}A^{-1/2}{+}\gamma^2\partial^2_{A}(\tilde\Gamma_{\gamma}(x,y,\theta,\sqrt{\gamma}\Sigma,A)-\tilde\Gamma_{\gamma}(a,0,\theta,\sqrt{\gamma}S,A)),
 \]
and, as on the support of the symbol we must have $A_c\sim1 $, it follows that, at the critical point $A_c$, the second derivative of $\frac ih \tilde\Phi_{N,a,h,\gamma}$ behaves like $ N\gamma^{3/2}/h=N\lambda_{\gamma}$ where we set $\lambda_{\gamma}=\frac{\gamma^{3/2}}{h}$.  We now compute \begin{multline}\label{nablaPhigamma}
\nabla \tilde\Phi_{\gamma}(x,y,\theta,A,\Sigma)=y+\nabla B_0(y,\theta){+\frac{\gamma A}{2}\nabla\left(\frac{q(\theta)}{|\theta|}(B_0-B_2)(y,\omega_\theta)\right)}\\
{}+\gamma^{3/2}\bigg[\frac{\nabla q(\theta)}{2q^{1/2}(\theta)}\bigg(\frac{\Sigma^3}{3}+\Sigma\Big(\frac{x}{\gamma}(1+\Lp(y,\omega_{\theta}))-A\Big)\bigg)+\frac{x}{\gamma} \Sigma q^{1/2}(\theta)\nabla \Lp(y,\omega_{\theta})\bigg]+\gamma^2\nabla \tilde\Gamma_{\gamma}(x,y,\theta,\sqrt{\gamma}\Sigma,A).
\end{multline}
Taking the derivative in $\theta$ in the last line of \eqref{desclambdaN}
 yields the second derivative of $\tilde\Phi_{N,a,h,\gamma}$ as follows
\begin{multline}\label{eq:dervisecPhi}
\nabla_\theta^2\tilde\Phi_{N,a,h,\gamma}
 =2t\,Id+\gamma tA\nabla^2q(\theta)+\nabla^2 B_0(y,\theta)
 +\frac{\gamma A}{2}\nabla^2\left(
 \frac{q(\theta)}{|\theta|}(B_0-B_2)(y,\omega_\theta)\right)\\
{}-\frac43N\gamma^{3/2}A^{3/2}\nabla^2q(\theta)^{1/2}
 +O(\gamma^{3/2})+O(\gamma^2)
 =2t\,Id+{O(|y|^2+\gamma t)}.
\end{multline}
Here $B_0(y,\theta)=O(|y|^2)$ and
$|N|\gamma^{3/2}\sim\gamma t$.
We also need the mixed derivative. Differentiating
\eqref{derivAPhi} in $\theta$, and including the time and reflection
terms, gives
\begin{equation}\label{eq:critical-mixed-block}
 \tilde\Phi_{A\theta}
 =O\big(\gamma t+|N|\gamma^{3/2}
          +\gamma^{3/2}+\gamma|y|^2+\gamma^2\big)
 =O(\gamma t).
\end{equation}
The remainder from the source and observation terms has bounded $A,\theta$ derivatives by
\eqref{Gammatildegam}; the displayed $\gamma^{3/2}$ term comes from
$\Sigma-S$. In the last equality we used
$t\sim |N|\sqrt\gamma$, $|N|\ge1$ and $|y|\lesssim t$.

Substitute $A=A_c(\theta)$ in the phase. Its Hessian in $\theta$ is
\begin{equation}\label{eq:critical-reduced-hessian}
 \left(\tilde\Phi_{\theta\theta}
 -\tilde\Phi_{\theta A}\tilde\Phi_{AA}^{-1}
                         \tilde\Phi_{A\theta}\right)\Big|_{A=A_c(\theta)}
 =2t\big(Id+O(T_0+\gamma)\big).
\end{equation}
Indeed, $|\tilde\Phi_{AA}|\sim |N|\gamma^{3/2}$, so the correction
from eliminating $A$ is $O(\gamma t)$.
On a sufficiently small frequency patch the {equation for the critical point of the reduced phase} has at most one solution. The implicit function theorem
is uniform near any such solution: dividing the rows and columns
of the full Hessian by $(|N|\gamma^{3/2})^{1/2}$ and $t^{1/2}$
leaves uniformly invertible diagonal blocks and a mixed block
$O(\sqrt\gamma)$.
Stationary phase therefore gives
$(|N|\lambda_\gamma)^{-1/2}(h/t)^{(d-1)/2}$, where
$\lambda_\gamma=\gamma^{3/2}/h\ge a^{3/2}/h\ge h^{-3\varepsilon/2}$.
\end{proof}

Using Proposition \ref{propcritptsalphaeta}, we apply the standard stationary phase theorem with respect to $A,\theta$ : in the next two lemmas we provide a description of the critical points $(A_c, \theta_c)$ and their main properties which will be useful later on, in the proof of the dispersive bounds for $W_{N,a,\gamma}$.

\begin{lemma}\label{lemthetac} 
Let $|N|\geq 1$ {and $t\ge c_*\sqrt\gamma$.}
The critical point $\theta_c$ of the phase $\tilde\Phi_{N,a,h,\gamma}(t,x,y,\Sigma,S,A,\theta)$ satisfies
\begin{equation}\label{eq:thetac}
\theta_c(t,x,a,y,\Sigma,S,\gamma,N)=-\frac{y+|y|^2g(t,y)}{2t}+O\Big(\gamma\times \frac{t}{2\sqrt{\gamma}N}, \frac{\gamma^{3/2}}{t}\Sigma,\frac{\gamma^{3/2}}{t}S,\gamma^{3/2}\frac{(\Sigma-S)}{2\sqrt{\gamma}N}\Big),
\end{equation}
where $g$ is a smooth function for $|y|/2t\sim 1$ such that $\theta_g=-\frac{y+|y|^2g(t,y)}{2t}$ satisfies $\theta_g=-\frac{y+\nabla B_0(y,\theta_g)}{2t}$.
Moreover, $\theta_c=\theta_c(t,x,a,y,\Sigma,S,\gamma,N)$ is defined for $\frac{t}{2\sqrt{\gamma}N}\sim 1$ and $\frac{|y|}{2t}\sim 1$ and for $N,\tilde N\sim N_t:=[\frac{t}{2\sqrt{\gamma}}]$,
\begin{equation}\label{dertheta}
\Big|\theta_c(\cdot,N)-\theta_c(\cdot,\tilde N)\Big| \lesssim \gamma N_t|\frac{1}{N}-\frac{1}{\tilde N}|\sim \gamma\frac{|N-\tilde N|}{N_t}, \quad |\partial_{\Sigma}\theta_c|=O(\frac{\gamma^{3/2}}{t}), |\partial_{S}\theta_c|=O(\frac{\gamma^{3/2}}{t}).
\end{equation}
\end{lemma}

\begin{proof}
Using \eqref{eq1Ac} we first obtain the equation satisfied by the critical point $A_c$ for $|N|\geq 1$ as follows:
\begin{equation}\label{Acexpres}
A^{1/2}=\frac{q(\theta)^{1/2}}{2N\sqrt{\gamma}}\Big(t{+\frac{(B_0-B_2)(y,\omega_\theta)}{2|\theta|}}\Big)-\frac{\Sigma-S}{2N}
+\frac{\sqrt{\gamma}}{2Nq(\theta)^{1/2}}\partial_A \Big(\tilde\Gamma_{\gamma}(x,y,\theta,\sqrt{\gamma}\Sigma,A)-\tilde\Gamma_{\gamma}(a,0,\theta,\sqrt{\gamma}S ,A)\Big).
\end{equation}
Using $\partial_A\tilde\Gamma_{\gamma}=O(1)$, it follows from the last equation that $NA_c^{1/2}$ is a function of $(t,x,y,\Sigma,S,\theta,\gamma)$ of the form
\begin{equation}\label{eq:Acsquare}
NA_c^{1/2}=\frac{q(\theta)^{1/2}}{2\sqrt{\gamma}}\Big(t{+\frac{(B_0-B_2)(y,\omega_\theta)}{2|\theta|}}-\frac{\sqrt{\gamma}}{q^{1/2}(\theta)}(\Sigma-S)+\gamma E(x,a,y,\sqrt{\gamma}\Sigma,\sqrt{\gamma}S,\theta)\Big),
\end{equation}
{where $E$ is the smooth bounded correction obtained from \eqref{Acexpres}; the rescaled bounds following \eqref{Gammatildegam} apply to it.} 
In particular, we have $\partial_{\Sigma} E=O(\gamma^{1/2})$, {$\partial_S E=O(\gamma^{1/2})$}. We now turn to the critical point with respect to $\theta$. 
Eliminating $N$ between the first and the last equations of \eqref{desclambdaN} and using \eqref{nablaPhigamma} and \eqref{eq1Ac} yields
\begin{multline}\label{nablaPhiNagamma}
2t\theta +y+\nabla B_0(y,\theta)
+\gamma A\Big(t\nabla q(\theta){+\frac12\nabla\left(\frac{q(\theta)}{|\theta|}(B_0-B_2)(y,\omega_\theta)\right)}\Big)\\
+\gamma^{3/2}\Big[\frac{\nabla q(\theta)}{2q^{1/2}(\theta)}\Big(\frac{\Sigma^3}{3}+\Sigma(\frac{x}{\gamma}(1+\Lp(y,\omega_{\theta}))-A)-\frac{S^3}{3}-S(\frac{a}{\gamma}-A))\Big)+\frac{x}{\gamma}\Sigma q^{1/2}(\theta)\nabla \Lp(y,\omega_{\theta})\Big]\\
+\gamma^2\nabla \Big(\tilde\Gamma_{\gamma}(x,y,\theta,\sqrt{\gamma}\Sigma,A)-\tilde\Gamma_{\gamma}(a,0,\theta,\sqrt{\gamma}S,A)\Big)\\=\frac{\gamma A}{3}\nabla q(\theta)\Big(t{+\frac{(B_0-B_2)(y,\omega_\theta)}{2|\theta|}}-\frac{\sqrt{\gamma}}{q^{1/2}(\theta)} (\Sigma-S)+\frac{\gamma }{q(\theta)}\partial_A(\tilde\Gamma_{\gamma}(x,y,\theta,\sqrt{\gamma}\Sigma,A)-\tilde\Gamma_{\gamma}(a,0,\theta,\sqrt{\gamma}S,A))\Big).
\end{multline}
We need to replace $A$ by $A_c$. Using the expression \eqref{eq:Acsquare} of $A_c^{1/2}$ gives, with $\omega_\theta=\theta/|\theta|$,
\begin{multline}\label{forA}
A_c=\frac{q(\theta)}{(2\sqrt\gamma N)^2}
 \left(t+\frac{(B_0-B_2)(y,\omega_\theta)}{2|\theta|}\right)^2
 \Bigg[1-\frac{2\sqrt\gamma(\Sigma-S)}{q^{1/2}(\theta)
 \left(t+\frac{(B_0-B_2)(y,\omega_\theta)}{2|\theta|}\right)} +\frac{\gamma(\Sigma-S)^2}{q(\theta)
 \left(t+\frac{(B_0-B_2)(y,\omega_\theta)}{2|\theta|}\right)^2}\\
{} +   O\bigg(
 \frac{\gamma}{t+\frac{(B_0-B_2)(y,\omega_\theta)}{2|\theta|}},
 \frac{\gamma^{3/2}}{\left(t+\frac{(B_0-B_2)(y,\omega_\theta)}{2|\theta|}\right)^2}
 \bigg)\Bigg]\,.
\end{multline}
The critical point $A_c$ lies in the support of the symbol $\kappa(A)$ only if $\frac{t{+\frac{(B_0-B_2)(y,\omega_\theta)}{2|\theta|}}}{2N\sqrt{\gamma}}\sim 1$. {As $(B_0-B_2)(y,\omega_\theta)=O(|y|^2)$ and $|y|\lesssim t\leq T_0<1$} and $|\theta|\sim 1$ it follows that, for $T_0$ sufficiently small, we have $t\sim (t{+\frac{(B_0-B_2)(y,\omega_\theta)}{2|\theta|}})\sim 2N\sqrt{\gamma}$ where the phase is stationary. {Consequently,
$\sqrt\gamma/(t+(B_0-B_2)(y,\omega_\theta)/(2|\theta|))\sim1/N$,
and both error terms in \eqref{forA} are $O(\sqrt\gamma/N)$.}

Introducing $A=A_c$ from the last formula in \eqref{nablaPhiNagamma} yields the following equation satisfied by $\theta$ 

\begin{multline}\label{nablaPhiNagamma2}
\theta+\frac{y+\nabla B_0(y,\theta)}{2t}
+\frac{\gamma A_c}{2t}\left\{
 \frac23t\nabla q(\theta)
 +\frac12\nabla\left(\frac{q(\theta)}{|\theta|}(B_0-B_2)(y,\omega_\theta)\right)
 -\frac{(B_0-B_2)(y,\omega_\theta)}{6|\theta|}\nabla q(\theta)
 \right\}\\
{}+\frac{\gamma^{3/2}}{2t}\left\{
 \frac{\nabla q(\theta)}{2q^{1/2}(\theta)}
 \left(\frac{\Sigma^3-S^3}{3}
 +\Sigma\frac{x}{\gamma}(1+\Lp)
 -S\frac{a}{\gamma}-\frac{A_c}{3}(\Sigma-S)\right)
 +\frac{x}{\gamma}\Sigma q^{1/2}(\theta)\nabla\Lp
 \right\}\\
{}+\frac{\gamma^2}{2t}\left(\nabla
 -\frac{A_c\nabla q(\theta)}{3q(\theta)}\partial_A\right)
 \left(\tilde\Gamma_\gamma(x,y,\theta,\sqrt\gamma\Sigma,A)
 -\tilde\Gamma_\gamma(a,0,\theta,\sqrt\gamma S,A)\right)\Big|_{A=A_c}=0.
\end{multline}
{Here $\Lp=\Lp(y,\omega_\theta)$, and all derivatives in the last
line are taken before substituting $A=A_c$.
This identity follows by eliminating $N$ using \eqref{eq1Ac} and then
setting $A=A_c$. Substitution of \eqref{forA} gives the same estimates
for $\theta_c$ and its derivatives: the added boundary coefficients are
$O(|y|^2)$ with their $\theta$ derivatives.}
First solve $\theta_g=-(y+\nabla B_0(y,\theta_g))/(2t)$.
The Jacobian of this equation is $Id+O(T_0)$.
After substituting $A_c(\theta)$, equation \eqref{nablaPhiNagamma2}
is a small smooth perturbation of this equation, with the same
uniform invertibility by \eqref{eq:critical-reduced-hessian}.
The implicit function theorem gives $\theta_c$ and
\eqref{eq:thetac}, locally in the fixed frequency patch.
Since $|\theta_c|\sim1$, we also have $|y|/(2t)\sim1$.

Differentiating {the rescaled equations for the critical points} gives the bounds
in \eqref{dertheta}. All source and observation corrections, including the terms
containing $B_0-B_2$, are retained before differentiation.
The same argument gives uniform bounds for any fixed number of
derivatives in the rescaled variables. We will use these bounds
when differentiating the reduced phase.
\end{proof}

\begin{lemma}\label{lemAc}
Let $|N|\geq 1$, {$t\ge c_*\sqrt\gamma$,} $N\sim \frac{t}{2\sqrt{\gamma}}$ and $\frac{|y|}{2t}\sim 1$.
{The critical point $A_c$ defined above satisfies}
\begin{gather}
  \label{eq:Acdet}
A_{c}=\Ab_0^2-2\Ab_0\Ab_1+\Big(\frac{\Sigma}{2N}-\frac {S} {2N}\Big)^2+ O\Big(\gamma (\frac{\Sigma}{N},\frac{S}{N})^2\Big)\,,\\
\label{eq:defA0}
\quad \Ab_0=\frac{q^{1/2}(\theta)}{2N\sqrt{\gamma}}\big(t{+\frac{(B_0-B_2)(y,\omega_\theta)}{2|\theta|}}+\gamma E_0\big)|_{\theta=\theta_c|_{\Sigma=S=0}},\quad \Ab_1=\frac{\Sigma}{2N}(1-\gamma E_1)-\frac {S} {2N}(1-\gamma E_2),
\end{gather}
where $\Ab_j=\Ab_j(t,x,y,a,\gamma,N)$ $j\in \{1,2\}$. Here
and $E_0=E_0(t,x,y,a,\gamma,\frac{t}{2\sqrt{\gamma}N})=O( |y|, \frac{x}{\gamma})$, $E_{0,1,2}$ are independent of $\Sigma$ and $S$ and $E_{1}=E_1(t,x,y,\gamma,\frac{t}{2\sqrt{\gamma}N})=O(1)$, $E_2=E_2(t,y,a,\gamma,\frac{t}{2\sqrt{\gamma}N})=O(1)$.
\end{lemma}

\begin{proof}
For $|N|\geq 1$, the critical point $A_c$ satisfies \eqref{Acexpres}, hence at $\Sigma=S=0$, $A_c^{1/2}|_{\Sigma=S=0}$ satisfies
\begin{equation}\label{Ab0}
\Ab_0:=A_c^{1/2}|_{\Sigma=S=0}=\frac{q(\theta_c)^{1/2}}{2N\sqrt{\gamma}}\Big(t{+\frac{(B_0-B_2)(y,\omega_{\theta_c})}{2|\theta_c|}}+\gamma E(x,a,y,0,0,\theta_c)\Big)|_{\Sigma=S=0},
\end{equation}
where $\theta_c|_{\Sigma=S=0}$ is the value of $\theta_c$, the critical point with respect to $\theta$, at $\Sigma=S=0$ and where $\omega_{\theta_c}=\theta_c/|\theta_c|$. 
The dependence of $\Ab_0(\cdot)$ on $N$ comes through the factor $\frac{t}{2\sqrt{\gamma}N}$ : this follows from \eqref{eq:thetac} and \eqref{Ab0}. Taking the derivative of \eqref{Acexpres} with respect to $\Sigma,S$ and using \eqref{dertheta} yields, for $N\sim \frac{t}{2\sqrt{\gamma}}$
\begin{equation}
\partial_{\Sigma}(A^{1/2}_c)=-\frac{1}{2N}\Big(1+O(\gamma)\Big),\quad \partial_{S}(A^{1/2}_c)={+}\frac{1}{2N}\Big(1+O(\gamma)\Big).
\end{equation}
This allows us to obtain the explicit form of the derivative of $A_c$ with respect to $\Sigma,S$.
\end{proof}

\subsubsection{Dispersive bounds for $W_{N,a,\gamma}$ for $(ht)^{1/2}\lesssim a\lesssim \gamma\ll 1$, when $a>h^{2/3-\epsilon}$, $\epsilon>0$}\label{secta>}

In this section we prove Theorem \ref{dispintermediaire} when $(ht)^{1/2}\lesssim a$. We first evaluate each $W_{N,a,\gamma}$ and provide sharp dispersive bounds for all $N$,  then we sum all these contributions. Although such sharp bounds for $W_{N,a,\gamma}$ do hold for all $a\geq h^{2/3-\epsilon}$, as in the region $a^{2/3-\epsilon}\leq a<(ht)^{1/2}$ the number of $N$ is very large, the sum over $N$ yields dispersive bounds which are worse than those announced in Theorem \ref{dispintermediaire}; therefore, when $a\leq (ht)^{1/2}$ we deal with $\Prond_{S,h,a,\gamma}$ separately, using its form \eqref{eq:Prond2cut} as a discrete sum. In this section we only use \eqref{eq:Prondcut} and the localisation with respect to $\gamma$ satisfying $a\lesssim \gamma\leq \varepsilon_0\ll 1$ from the beginning of section \ref{sectcardN1}.

Before stating the estimates for $W_{N,h,\gamma}$, when $|N|\geq 1$ and $a\leq  \gamma<\varepsilon_0$, we recall that we have the free space dispersion for $W_{0}:=\sum_{\gamma} W_{0,h,\gamma}$ (see \eqref{eq:newVN} for the definition of $W_0$), assuming that $t>h$ so that the dispersive effect takes over :
\begin{prop}\label{propfreeN=0}
For $a_0$, $\varepsilon_0$ and $T_0$ sufficiently small and for all $h\in (0,1)$, $a\in (h^{2/3-\epsilon},a_0]$ for some $\epsilon>0$ and $t\in (h,T_0]$ 
\begin{equation}\label{eq:370}
 |W_{0}(t,x,y)|:=\Bigg|\sum_{\gamma =2^j a, 0\leq j\leq \log(\frac{\varepsilon_0}{a})}W_{0,\gamma}(t,x,y)\Bigg|\lesssim \frac{1}{h^{d}}\Big(\frac ht\Big)^{d/2}.
\end{equation}
\end{prop}

Proposition \ref{propfreeN=0} follows in the same way as \cite[Prop.2]{doiSF21} in the case of the Friedlander model domain (the phase function $\Phi_{0,a,h}$ has non-degenerate critical points in all variables, still, to avoid complications, one has to pay attention only to the order in which one applies the stationary phase). In what follows, we consider $|N|\geq 1$.

\begin{prop}\label{proptang}
Assume $h\in (0,1)$, $h^{2/3-\epsilon}\leq a\leq a_0$
and let $\lambda=a^{3/2}/h$. Then there exists $C>0$ such that for $t\gtrsim \sqrt a$ and $0\leq x\leq 2a$,
  \begin{equation}
    \label{eq:37ter}
    \|(\Prond_{S,h,a,a}-W_{0,h,a})(t,\cdot)\|_{L^{\infty}(0\leq x\leq 2a, y)}\!\leq\!\sup_{\substack{0\leq x\leq 2a\\y}}\!\sum_{|N|\geq 1}|W_{N,h,a}(t,x,y)| \leq C
   \left\{ \begin{array}{l} 
\frac{1}{h^d}\Big(\frac ht \Big)^{\frac{d-1}2}\Big({\frac{ht}{a}}\Big)^{1/2}, \text{ if } \frac{t}{\sqrt{a}}\gtrsim \lambda^{1/3}, \,\\
\frac{1}{h^d}\Big(\frac ht \Big)^{\frac{d-1}2}\Big(\frac{ha}{t}\Big)^{1/4}, \text{ if } 1\lesssim \frac{t}{\sqrt{a}}\lesssim\lambda^{1/3}.
 \end{array} \right.
  \end{equation}
\end{prop}

\begin{prop}\label{proptransv}
Assume $h\in (0,1)$, $h^{2/3-\epsilon}\leq a\leq a_0$,
$\gamma\geq 4a$ and let $\lambda_{\gamma}=\gamma^{3/2}/h$. Then there exists $C$ such that, for $t\gtrsim \sqrt \gamma$ and {$0\leq x\leq a$}
  \begin{equation}
    \label{eq:37bis}
\|(\Prond_{S,h,a,\gamma}-W_{0,h,\gamma})(t,\cdot)\|_{L^{\infty}({0\leq x\leq a}, y)}\leq \sup_{\substack{0\leq x\leq a\\y}}\sum_{|N|\geq 1}|W_{N,h,\gamma}(t,x,y)| \leq C
   \left\{ \begin{array}{l} 
\frac{1}{h^d}\Big(\frac ht \Big)^{\frac{d-1}2}\Big({\frac{ht}{\gamma}}\Big)^{1/2}, \text{ if } \frac{t}{\sqrt{\gamma}}\gtrsim \lambda_{\gamma}^{1/3}, \,\\
\frac{1}{h^d}\Big(\frac ht \Big)^{\frac{d-1}2}h^{1/3}, \text{ if } 1\lesssim \frac{t}{\sqrt{\gamma}} \lesssim\lambda_{\gamma}^{1/3}.
 \end{array} \right.
  \end{equation}
\end{prop}
Proposition \ref{proptang} will be proved in the following section \ref{sectang}. Proposition \ref{proptransv} is postponed to section \ref{sectransv}. 
\begin{rmq}
Our main estimate for $a\geq (ht)^{1/2}$, \eqref{eq:dispersion-reflections} from Theorem \ref{dispintermediaire}, follows from the previous propositions, in the regime $\max\{h^{2/3-\epsilon}, (ht)^{1/2}\}\leq a\leq a_0$ : the additional condition $a\geq (ht)^{1/2}$ is necessary in order to have $\Big({\frac{ht}{a}}\Big)^{1/2}\leq (ht)^{1/4}\leq (\frac ht)^{1/4}$. 
For $a\ll (ht)^{1/2}$, the first bound in \eqref{eq:37ter} is worse than \eqref{eq:dispersion-reflections}.
\end{rmq}

We recall that for $a\leq  \gamma$, $W_{N,h,\gamma}$ was defined in \eqref{eq:newVNgam}, where the symbol of $W_{N,h,\gamma}$ (the same for every $N$) is of degree $0$. Let $|N|\geq 1$, $a\lesssim \gamma$ {and $t\ge c_*\sqrt\gamma$,} and set as before $\lambda_{\gamma}=\gamma^{3/2}/h$. Applying the (usual) stationary phase with respect to $A,\theta$ (see Proposition \ref{propcritptsalphaeta}) yields, for all $a\lesssim \gamma\ll 1${,} {$W_{N,h,\gamma}(t,x,y)=\frac{1}{h^{d}}\Big(\frac ht\Big)^{\frac{d-1}{2}}V_{N,h,\gamma}(t,x,y),$} \begin{equation}\label{VNdef}
V_{N,h,\gamma}(t,x,y)=\frac{\gamma^2}{h}\frac{1}{\sqrt{\lambda_{\gamma}N}}\int e^{\frac ih \Psi_{N,a,h,\gamma}(t,x,y,\Sigma,S) } \sigma_{V,h,\gamma}(\Sigma,S,t,x,y,h,a,\gamma,1/N) d\Sigma dS,
\end{equation}
where, for $a\sim \gamma$, $\Psi_{N,a,h,a}(t,x,y,\Sigma,S):=\tilde \Phi_{N,a,h,a} (t,x,y,\Sigma,S,A_c,\theta_c)$ and for $\gamma>4a$, $\Psi_{N,a,h,\gamma}(t,x,y,\Sigma,S):=\tilde \Phi_{N,a,h,\gamma} (t,x,y,\Sigma,S,A_c,\theta_c)$. The symbols $\sigma_{V,a,h}$, of order $0$, are obtained from \eqref{eq:newVNgam} after the changes of variables and the stationary phase. They are asymptotic expansions with small parameters {$h/t$} and $1/(\lambda_{\gamma}N)$ and main contribution of the form
\[
\kappa(A_c)\chi(\sqrt{\gamma}q(\theta_c)^{1/6}S) p_h(x,y,\theta_c,\gamma q(\theta_c)^{1/3}A_c,\sqrt{\gamma}q(\theta_c)^{1/6}\Sigma)\tilde q_h (\theta,\gamma q(\theta_c)^{1/3}A_c,\sqrt{\gamma}q(\theta_c)^{1/6}A_c S),
\]
where we have dropped the cutoffs $\cutoffchi^{\flat}(\gamma q(\theta)^{1/3}A_c/\ceps_{0})$ and $\cutoffchi^{\sharp}(\lambda_{\gamma}^{2/3}q(\theta)^{1/3}A_c)$ of \eqref{eq:newVNgam} as they are identically equal to $1$ on the support of $\kappa(A_c)$ and $p_h$ (which contains the factor $\varkappa(\theta)$ supported where $|\theta|\sim 1 $). Using \eqref{eq:thetac} and \eqref{eq:Acdet}, we observe that, when necessary, we can add cutoffs supported where $\frac{t}{2N\sqrt{\gamma}}\sim 1$ and $\frac{|y|}{2t}\sim 1$ without changing the contributions of $V_{N,h,\gamma}$ modulo $O(h^{\infty})$. We now prove Proposition \ref{proptang}.

\subsubsection{Tangent initial directions $\gamma\sim a$, $a\gtrsim h^{2/3-\epsilon}$, $a\gtrsim x\geq 0$}\label{sectang}

First let $\gamma=a$ and set $\lambda=a^{3/2}/h$. 
Let $\Ab_0=\Ab_0(t,x,y,a,\gamma,h,N)$ be defined in Lemma \ref{lemAc}. For $ t/\sqrt{a} \sim N\geq 1$ the following estimates hold true:

\begin{prop}\label{dispNgrand}
Let $N\geq\lambda^{1/3}$, $0\leq x\lesssim a$ and $t\leq T_0$,
with $N\sim t/\sqrt a$ and $T_0$ sufficiently small. Set
$\omega_c^0=\omega_c|_{\Sigma=S=0}$. Then, for either sign of the
two differences,
\begin{equation}\label{eq:1ff}
 |V_{N,a}(t,x,y)|\lesssim
 \frac{h^{1/3}}{(N/\lambda^{1/3})^{1/2}}
\max\left(
 \frac{1}{1+\lambda^{1/3}|\Ab_0^2(\cdot,N)-1|^{1/2}},
 \frac{1}{1+\lambda^{1/3}
 |\Ab_0^2(\cdot,N)-\frac xa(1+\Lp(y,\omega_c^0))|^{1/2}}
 \right)\,.
\end{equation}
\end{prop}
This follows from Lemma \ref{lem:large-reflections} and
\eqref{VNdef}, since
$(a^2/h)(\lambda N)^{-1/2}\lambda^{-2/3}
=h^{1/3}(N/\lambda^{1/3})^{-1/2}$.
We next give the counting and summation estimates for these two weights.

\begin{lemma}\label{rmqNgrand}
Fix $(t,x,y)$ with $0\leq x\lesssim a$, $|y|\sim t\leq T_0$,
and choose $N_t\geq1$ with
$N_t\sim t/(2\sqrt a)$. For either of the two differences
$\Ab_0^2(\cdot,N)-1$ and
$\Ab_0^2(\cdot,N)-\frac xa(1+\Lp(y,\omega_c^0(\cdot,N)))$,
the number of integers $N\sim N_t$ for which its absolute value
is at most $\delta$ is $O(1+N_t\delta)$, uniformly for
$0<\delta\leq1$. This holds also when $N_t\leq\lambda^{1/3}$.
\end{lemma}
\begin{proof}
Regard $N$ as a real variable in {the equations defining}
$\theta_c^0$ and $\Ab_0$, with $N\sim N_t$. From
\eqref{eq:defA0} and \eqref{dertheta},
\begin{equation}\label{Ab0NtildeN}
 \partial_N\Ab_0^2=-\frac{2\Ab_0^2}{N}
       +O\left(\frac{a+a/t}{N_t}\right),\qquad
 \partial_N\left[\frac xa(1+\Lp(y,\omega_c^0))\right]
       =O\left(\frac{a|y|}{N_t}\right).
\end{equation}
Indeed, $\partial_N\theta_c^0=O(a/N_t)$; differentiating the
correction $aE_0$ in \eqref{eq:defA0} gives the term
$O(a/t)$ relative to the main contribution. The terms $B_0-B_2$
are included in the same differentiation. On this band,
$\Ab_0\sim1$ and $a/t\lesssim\sqrt a\lesssim T_0$.
Thus, for small $T_0,a_0$, each difference decreases with
$N$, with derivative of size $1/N_t$. Its inverse image of
$[-\delta,\delta]$ has length $O(N_t\delta)$, and therefore
contains $O(1+N_t\delta)$ integers.
\end{proof}

\begin{lemma}\label{rmqNgrand2}
Under the assumptions of Lemma \ref{rmqNgrand}, for every $L\geq1$,
\begin{equation}\label{eq:two-weight-sum}
 \sum_{N\sim N_t}\left[
 \frac{1}{1+L|\Ab_0^2(\cdot,N)-1|^{1/2}}
 +\frac{1}{1+L
 |\Ab_0^2(\cdot,N)-\frac xa(1+\Lp(y,\omega_c^0))|^{1/2}}
 \right]
 \lesssim 1+\frac{N_t}{L}\,.
\end{equation}
\end{lemma}
\begin{proof}
For either difference, the integers where its absolute value is
at most $L^{-2}$ contribute $O(1+N_t/L^2)$ by
Lemma \ref{rmqNgrand}. On the sets where it lies between
$2^jL^{-2}$ and $2^{j+1}L^{-2}$, up to size $1$, the weight is
$O(2^{-j/2})$ and there are $O(1+N_t2^j/L^2)$ integers.
Summing these bounds gives $O(1+N_t/L)$. The remaining integers
have weight at most $1/L$ and contribute $O(N_t/L)$.
We apply this argument to each of the two differences.
\end{proof}

\begin{prop}
\label{dispNpetitloin}
Fix a sufficiently large constant $C_0$, as in the proof below.
For $1\leq N\leq 16$ and $|\Ab^2_0(\cdot,N)-1|+|\Ab^2_0(\cdot,N)-\frac xa(1+\Lp(y,\omega_c^0))|\geq C_0/N^2$, $x\lesssim a$, we have 
  \begin{equation}
    \label{eq:21ff}
       \left| V_{N,a}(t,x,y)\right| \lesssim h^{1/3}\,.
  \end{equation}

For $16\leq N\leq \lambda^{1/3}$ and $|\Ab^2_0(\cdot,N)-1|+|\Ab^2_0(\cdot,N)-\frac xa(1+\Lp(y,\omega_c^0))|\geq C_0/N^2$, $x\lesssim a$, we have
  \begin{equation}
    \label{eq:2ff}
       \left| V_{N,a}(t,x,y)\right| \lesssim \max\Big\{\frac{h^{1/3}}{1+N|\Ab^2_0(\cdot,N)-1|^{1/2}}, \frac{h^{1/3}}{1+N|\Ab^2_0(\cdot,N)-\frac xa(1+\Lp(y,\omega_c^0))|^{1/2}}\Big\}\,.
  \end{equation}
\end{prop}

\begin{prop}
\label{dispNpetitpres}
For $1\leq N\leq \lambda^{1/3}$ and $|\Ab^2_0(\cdot,N)-1|+|\Ab^2_0(\cdot,N)-\frac xa(1+\Lp(y,\omega_c^0))|\leq C_0/N^2$, $x\lesssim a$,
we have
  \begin{equation}
    \label{eq:2hh}
       \left| V_{N,a}(t,x,y)\right| \lesssim \frac{h^{1/3}}{(N/\lambda^{1/3})^{1/4}}\,.
  \end{equation}
\end{prop}

For the bound \eqref{eq:2hh} (corresponding to the worst degeneracy), only a bounded number of
reflection indices need to be retained.
\begin{lemma}\label{lemNpetit}
Let $N_t\sim t/(2\sqrt a)$, with $1\leq N_t\lesssim\lambda^{1/3}$.
There are $O(1)$ integers $N\sim N_t$ satisfying
\[
 |\Ab_0^2(\cdot,N)-1|
 +|\Ab_0^2(\cdot,N)-\tfrac xa(1+\Lp(y,\omega_c^0))|
 \leq C_0/N^2\,,
\]
where $C_0$ is the sufficiently large constant fixed in
Proposition~\ref{dispNpetitloin}; the $O(1) $number
of relevant $N$'s may depend on $C_0$, but is uniform in $h,a,t,x,y$.
\end{lemma}
\begin{proof}
Such indices satisfy $|\Ab_0^2(\cdot,N)-1|\lesssim C_0/N_t^2$.
Lemma \ref{rmqNgrand} bounds their number by
$O(1+C_0/N_t)=O(1)$ when $N_t$ is large. For bounded $N_t$,
the whole band contains only $O(1)$ integers.
\end{proof}

\begin{rmq}\label{rmqsummarizeminim}
For fixed $(t,x,y)$, the smaller of the two absolute differences
in \eqref{eq:1ff} and \eqref{eq:2ff} may change with $N$.
The ordering of $1$ and $\frac xa(1+\Lp(y,\omega_c^0))$ alone
does not determine it. We sum the two weights separately, using
Lemma \ref{rmqNgrand2}, after bounding their maximum by their sum.

There is one useful simplification when the two reference values
are close. If
$|1-\frac xa(1+\Lp(y,\omega_c^0(\cdot,N_t)))|\lesssim aT_0$
for some $N_t\sim t/(2\sqrt a)$, then \eqref{dertheta} gives the
same bound, with a possibly larger constant, for all $N\sim N_t$.
Put $L=\lambda^{1/3}$ when $N\geq\lambda^{1/3}$, and $L=N$
when $N\leq\lambda^{1/3}$. In both cases $aL^2\lesssim t^2$,
and the triangle inequality gives
\begin{equation}\label{eq:nearby-reference-weights}
 1+L|\Ab_0^2(\cdot,N)-1|^{1/2}
 \sim 1+L
 |\Ab_0^2(\cdot,N)-\tfrac xa(1+\Lp(y,\omega_c^0))|^{1/2}.
\end{equation}
Indeed the two denominators differ by at most
$L|1-\frac xa(1+\Lp(y,\omega_c^0))|^{1/2}\lesssim T_0^{3/2}$.
This compares the regularized weights even when one of the
underlying differences vanishes; no comparison of the unregularized
differences is needed. The summation below works without this
extra assumption.
\end{rmq}

\begin{rmq}
In \cite[Section 3.1.3]{doiSF21}, we proved that the (worst) bound
\eqref{eq:2hh} is attained in the Friedlander model at $x=a$,
for suitable values of $y$ in the regime $N\ll\lambda^{1/3}$.
At these points, the contribution near such a swallowtail point
dominates the sum of the other reflection terms. This gives
sharpness of the quarter-loss in that model.

For the upper bound in the general case, Lemma \ref{lemNpetit}
leaves only a bounded number of indices to which we apply
\eqref{eq:2hh}. We use \eqref{eq:2ff} for the remaining indices,
and sum the two weights separately. This proves \eqref{eq:37ter}
without locating a contribution near the swallowtail that dominates the full sum.
\end{rmq}

The proof of Proposition \ref{dispNgrand} is postponed to section \ref{sectNlarge}. Propositions \ref{dispNpetitloin} and \ref{dispNpetitpres} will be proved in section \ref{sectmodN}.
In the next subsection we deal with Proposition \ref{proptang}.

\subsubsection{Proof of Proposition \ref{proptang}}
Fix $(t,x,y)$ and choose $N_t\geq1$ with $N_t\sim t/(2\sqrt a)$.
By the cutoffs in \eqref{VNdef}, only $N\sim N_t$ contribute,
modulo $O(h^\infty)$. We sum \eqref{eq:1ff} and \eqref{eq:2ff}
by replacing each maximum by the sum of its two terms.

For $N\geq\lambda^{1/3}$, Lemma \ref{rmqNgrand2} with
$L=\lambda^{1/3}$ gives
\begin{equation}\label{eq:large-reflection-summation}
 \sum_{\substack{N\sim N_t\\N\geq\lambda^{1/3}}}
 |V_{N,a}(t,x,y)|
 \lesssim h^{1/3}\lambda^{1/6}N_t^{-1/2}
       (1+N_t\lambda^{-1/3})
 \lesssim h^{1/3}\lambda^{-1/6}N_t^{1/2}
 \sim (ht/a)^{1/2}\,.
\end{equation}
Here a nonempty sum requires $N_t\gtrsim\lambda^{1/3}$, so the
first term is absorbed by the second. Both signs of the differences are included in this summation. For the remaining indices, we prove the following lemma.
\begin{lemma}\label{lem:moderate-reflection-sum}
If $1\leq N_t\lesssim\lambda^{1/3}$ and $N_t\sim t/(2\sqrt a)$,
then
\begin{equation}\label{eq:moderate-reflection-summation}
 \sum_{\substack{N\sim N_t\\1\leq N\leq\lambda^{1/3}}}
 |V_{N,a}(t,x,y)|\lesssim (ha/t)^{1/4}.
\end{equation}
\end{lemma}
\begin{proof}
By Lemma \ref{lemNpetit}, at most $O(1)$ indices satisfy the
hypothesis of Proposition \ref{dispNpetitpres}. Their contribution
is bounded by
$h^{1/3}(N_t/\lambda^{1/3})^{-1/4}\sim(ha/t)^{1/4}$.
For all other indices $N\geq16$, use Proposition
\ref{dispNpetitloin} and \eqref{eq:two-weight-sum} with $L=N_t$;
since $N\sim N_t$, their sum is $O(h^{1/3})$.
The at most $16$ remaining indices are bounded by \eqref{eq:21ff}.
Finally $h^{1/3}\lesssim(ha/t)^{1/4}$, since
$N_t\lesssim\lambda^{1/3}$.
\end{proof}
Combining \eqref{eq:large-reflection-summation} and
\eqref{eq:moderate-reflection-summation} proves
\eqref{eq:37ter}. Bands meeting $N=\lambda^{1/3}$ are covered by
both arguments with the same constants.

\subsubsection{Proof of Proposition \ref{dispNgrand} (large number of reflections $N\geq \lambda^{1/3}$)}\label{sectNlarge}
We first prove a uniform estimate for the integral in \eqref{VNdef}.
\begin{lemma}\label{lem:large-reflections}
Let $0\leq x\lesssim a$, $\lambda=a^{3/2}/h\geq1$ and
$|N|\geq\lambda^{1/3}$, with $|N|\sim |t|/\sqrt a$ and
$|t|\leq T_0$, for $T_0$ sufficiently small. Then
\begin{equation}\label{eq:large-reflections-away}
\left|\int e^{\frac ih\Psi_{N,a,h,a}}
 \sigma_{V,h,a}(\Sigma,S,\cdot)\,d\Sigma\,dS\right|
\lesssim\lambda^{-2/3}
\max\left(
 \frac{1}{1+\lambda^{1/3}|\Ab_0^2-1|^{1/2}},
 \frac{1}{1+\lambda^{1/3}
 |\Ab_0^2-\frac xa(1+\Lp(y,\omega_c^0))|^{1/2}}
\right)\,.
\end{equation}
The implicit constant is independent of $h,a,t,N,x,y$ in the
above regime. No sign condition on the two differences is required.
\end{lemma}

\begin{proof}
We write the proof for $t,N>0$; negative times follow by conjugation.
We use the rescaling from \cite[Proposition 4]{doiSF21} and
\cite[Lemma 3.21]{ILLP}, keeping the source and observation corrections in the full phase.
Using $\tilde\Phi_{0,a,h,a}$ as provided in \eqref{eqtildePhiN} for $N=0$ and $\gamma$ replaced by $a$, we compute
\begin{multline}\label{derivSigmaphizero1}
\partial_{\Sigma}\Psi_{N,a,h,a}
=\partial_{\Sigma}\tilde\Phi_{N,a,h,a}|_{\theta_c,A_c}
=a^{3/2}q(\theta_c)^{1/2}\Big[\Sigma^2+\frac xa-A_c\\
+\frac xa\partial_{\Xi}A_{\Gamma}
\Big(x,y,\frac{\sqrt a\Sigma q(\theta_c)^{1/2}}
{\tau(aA_cq(\theta_c)^{1/3},\theta_c)},
\frac{\theta_c}{\tau(aA_cq(\theta_c)^{1/3},\theta_c)}\Big)\Big].
\end{multline}
The derivatives of $A_c,\theta_c$ cancel in this first derivative, since
$\Psi$ is the critical value of $\tilde\Phi$. They must be included when
this identity is differentiated again.
In the same way we have, with $A_c$ given in \eqref{eq:Acdet},
\begin{equation}
\partial_{S} \Psi_{N,a,h,a}=\partial_{S}\tilde\Phi_{N,a,h,a}|_{\theta_c,A_c}
= -a^{3/2}q(\theta_c)^{1/2}\Big[S^2+1-A_c+\partial_{\Xi}A_{\Gamma}(a,0,\Xi,\frac{\theta_{c}}{{\tau(a A_cq(\theta_{c})^{1/3},\theta_{c})}})|_{\Xi=\frac{\sqrt{a}S q^{1/2}(\theta_{c})}{ {\tau(a A_cq(\theta_{c})^{1/3},\theta_{c})}}}\Big],
\end{equation}
With the normalization $\beta_1=0$, the expansion used here is
\[
\partial_{\Xi}A_{\Gamma}
=\Lp+2\mu_\Gamma\Xi+\alpha_1x+3\mu_1\Xi^2+\partial_{\Xi}R_A,
\qquad R_A\in\mathcal H_{\ge4},\quad \Lp(0,\cdot)=0.
\]
All coefficients are evaluated at $(y,\Theta/|\Theta|)$. In particular,
the weight-two terms give both a constant term and a quadratic term
after rescaling; they cannot all be included in a remainder independent
of the rescaled normal variables. The estimates below hold after
composition with $(A_c,\theta_c)$, with all derivatives of any fixed
order in the variables used in each rescaling. To check them we
repeatedly use the chain rule, including the derivatives of $A_c$ and
$\theta_c$ from Lemmas~\ref{lemAc} and~\ref{lemthetac}.
The bounds obtained by differentiating {the equations for the critical points},
\eqref{eq:full-phase-critical-chain}, also apply here.
The constant source and observation corrections have size
$\lambda^{2/3}a\lesssim t^2$, since $N\ge\lambda^{1/3}$.
We treat this case separately: as $\lambda^{1/3}/N\to0$,
the two third-order terms decouple and the limiting Hessian can vanish.
Lemma \ref{lem:refined-quartic-family}, which requires rank one,
does not apply to this limit.

We rescale variables $(\Sigma,S)=(\lambda^{-1/3} \xx,\lambda^{-1/3}\yy)$. The symbol bounds give, for any multi-index $\nu$,
\begin{equation}
  \label{eq:66}
|\partial^{\nu}_{(\xx,\yy)} \sigma_{V,h,a}(\lambda^{-1/3}\xx,\lambda^{-1/3}\yy,\rho)|\leq C_{\nu} (1+|\xx|+|\yy|)^{-|\nu|}\,.
\end{equation}
Set
$\Af \lambda^{-2/3}=\Ab_0^2-\frac{x}{a} (1+\Lp(y,\omega^0_c))$, $\Bf\lambda^{-2/3}=\Ab_0^2-1$.
We prove, uniformly in $(\Af,\Bf)$,
\begin{equation}
  \label{eq:65}
  \left|   \int e^{\frac i h \Psi_{N,a,h,a}(\lambda^{-1/3}\xx,\lambda^{-1/3}\yy,\rho)} \sigma_{V,h,a}(\lambda^{-1/3}\xx,\lambda^{-1/3}\yy,\cdot) \,d\xx d\yy\right|  \lesssim \max\{\frac{1}{1+|\Bf|^{1/2}},\frac{1}{1+|\Af|^{1/2}}\}\,.
\end{equation}
In the new variables, $A_c$ from \eqref{eq:Acdet} reads as follows
\begin{multline}
A_{c}=\Ab_0^2-2\Ab_0\Big(\frac{\Sigma}{2N}(1-a E_1)-\frac {S} {2N}(1-a E_2)\Big)+\Big(\frac{\Sigma}{2N}-\frac {S} {2N}\Big)^2+ O\Big(a (\frac{\Sigma}{N},\frac{S}{N})^2\Big)\\
=\Ab_0^2-\Ab_0\frac{\lambda^{-1/3}}{N}(\xx(1-a E_1)- \yy(1-a E_2))+\frac{\lambda^{-2/3}}{4N^2}(\xx-\yy)^2+a\frac{\lambda^{-2/3}}{N^2}O(\xx^2,\yy^2).
\end{multline}
Also 
\[
\partial_{\Xi}A_{\Gamma}(x,y,\Xi,\Theta)|_{\Xi=\frac{\sqrt{a}\Sigma q^{1/2}(\theta_{c})}{ {\tau(a A_cq(\theta_{c})^{1/3},\theta_{c})}}}=\Lp(y,\omega_c)+2\frac{\sqrt{a}\lambda^{-1/3}\xx q^{1/2}(\theta_{c})}{ {\tau(a A_cq(\theta_{c})^{1/3},\theta_{c})}}{\mu_\Gamma}(y,\omega_c)+\mathcal{H}_{j\geq 2}.
\]
Introducing these formulas in \eqref{derivSigmaphizero1} gives
\begin{align*}
\frac{\partial_{\xx}\Psi_{N,a,h,a}}{h\sqrt{q(\theta_c)}}
&\begin{multlined}[t]
=\xx^2-\Af+\frac{\lambda^{1/3}\Ab_0}{N}
\big(\xx(1-aE_1)-\yy(1-aE_2)\big)-\frac{(\xx-\yy)^2}{4N^2}\\
+2\sqrt a\lambda^{1/3}\frac xa\mu_\Gamma(y,\omega_c)
\frac{\sqrt{q(\theta_c)}}{\tau(aA_cq(\theta_c)^{1/3},\theta_c)}\xx
+\lambda^{2/3}\frac xa\alpha_1(y,\omega_c)x\\
+3a\frac xa\mu_1(y,\omega_c)
\frac{q(\theta_c)}{\tau(aA_cq(\theta_c)^{1/3},\theta_c)^2}\xx^2
+O(a(\xx^2+\yy^2))+O(a^{3/2}\lambda^{2/3}),
\end{multlined}\\
-\frac{\partial_{\yy}\Psi_{N,a,h,a}}{h\sqrt{q(\theta_c)}}
&\begin{multlined}[t]
=\yy^2-\Bf+\frac{\lambda^{1/3}\Ab_0}{N}
\big(\xx(1-aE_1)-\yy(1-aE_2)\big)-\frac{(\xx-\yy)^2}{4N^2}\\
+2\sqrt a\lambda^{1/3}\mu_\Gamma(0,\omega_c)
\frac{\sqrt{q(\theta_c)}}{\tau(aA_cq(\theta_c)^{1/3},\theta_c)}\yy
+\lambda^{2/3}\alpha_1(0,\omega_c)a\\
+3a\mu_1(0,\omega_c)
\frac{q(\theta_c)}{\tau(aA_cq(\theta_c)^{1/3},\theta_c)^2}\yy^2
+O(a(\xx^2+\yy^2))+O(a^{3/2}\lambda^{2/3}).
\end{multlined}
\end{align*}
Here the coefficient of $\xx$ coming from $A_\Gamma$ occurs only at the
observation point, and that of $\yy$ only at the source point.
The variation of $\Lp(y,\omega_c)$ from its value at
$\xx=\yy=0$ is included in the remainders. On bounded sets these are
bounds in $C^m$ for every fixed $m$. On the support of the rescaled
symbol, the weight-three remainder is $O(a^{3/2}\lambda^{2/3})$;
its derivatives are estimated by the same chain rule. In particular,
the constant terms have size $O(a\lambda^{2/3})$, whereas the
quadratic terms have size $O(a(\xx^2+\yy^2))$.

Write $r^2=\Af^2+\Bf^2$. If $\lambda$ is bounded, the rescaled
symbol has uniformly bounded support and \eqref{eq:65} follows directly.
We may therefore take $\lambda$, and hence $N$, sufficiently large.
For bounded $r$, the integral in \eqref{eq:65} is $O(1)$: this is
immediate on bounded sets, and integration by parts applies outside
a fixed ball, where the quadratic terms in the gradient dominate. We henceforth take
$r_0<r\lesssim\lambda^{2/3}$, for $r_0$ sufficiently large.
The relations $N\sim t/\sqrt a$ and $N\geq\lambda^{1/3}$ give
\[
 \lambda^{1/3}/N\lesssim1,\qquad
 N^{-2}\lesssim r^{-1},\qquad
 a\lesssim T_0^2/r,\qquad a\lambda^{2/3}\lesssim T_0^2.
\]
Set $(\xx,\yy)=r^{1/2}(\xx',\yy')$ and
$\Psi_{N,a,h,a}/h=r^{3/2}\phi_{N,a,h}$, and denote the rescaled
symbol by $\tilde\sigma_{V,h,a}$. On every fixed bounded set,
the preceding expansions give
\begin{align}\label{eq:74}
 \partial_{\xx'}\phi_{N,a,h}
 &=q(\theta_c)^{1/2}
 \bigl(\xx'^2-\Af/r+r^{-1/2}O(1)\bigr),\\
 \partial_{\yy'}\phi_{N,a,h}
 &=-q(\theta_c)^{1/2}
 \bigl(\yy'^2-\Bf/r+r^{-1/2}O(1)\bigr).
\end{align}
Here the remainders are bounded in $C^m$ for every fixed $m$.
This includes differentiation of $(A_c,\theta_c)$ through {the equations defining them}, not just differentiation before stationary phase.
Outside a sufficiently large fixed ball in $(\xx',\yy')$, the
quadratic terms dominate; integration by parts, with a dyadic partition,
gives $O(r^{-M})$ in the integral of \eqref{eq:65}, for any fixed $M$.
We may therefore work on this ball.

At least one of $|\Af|,|\Bf|$ is comparable to $r$.
If that quantity is negative, the corresponding derivative in
\eqref{eq:74} does not vanish, and integration by parts proves
\eqref{eq:65}. Otherwise, exchanging the two integration variables
if necessary, we may assume $\Af\sim r>0$ and $|\Bf|\leq\Af$.
Stationary phase in $\xx'$ applies near
$\xx'_\pm=\pm\sqrt{\Af/r}+O(r^{-1/2})$, where
$|\partial_{\xx'}^2\phi_{N,a,h}|\gtrsim1$; its complement is
nonstationary. Thus, to any fixed order,
\begin{equation}\label{eq:75}
 \int e^{ir^{3/2}\phi_{N,a,h}}\tilde\sigma_{V,h,a}
 \,d\xx'\,d\yy'
 =r^{-3/4}\sum_{\pm}\int
 e^{ir^{3/2}\phi_{N,a,h,\pm}(\yy')}
 \tilde\sigma_{V,h,a,\pm}(\yy')\,d\yy'+O(r^{-M}).
\end{equation}
The phases on the right are the critical values in $\xx'$.
{To compute these derivatives we must also differentiate the equation defining the eliminated critical point.} In the variable $\yy=r^{1/2}\yy'$, these derivatives satisfy
\begin{align*}
 \partial_{\yy}\bigl(r^{3/2}\phi_{N,a,h,\pm}(r^{-1/2}\yy)\bigr)
 &=-q(\theta_c^0)^{1/2}\bigl(\yy^2-\Bf+O(\sqrt r)\bigr),\\
 \partial_{\yy}^2\bigl(r^{3/2}\phi_{N,a,h,\pm}(r^{-1/2}\yy)\bigr)
 &=-2q(\theta_c^0)^{1/2}\yy+O(1),\\
 \partial_{\yy}^3\bigl(r^{3/2}\phi_{N,a,h,\pm}(r^{-1/2}\yy)\bigr)
 &=-2q(\theta_c^0)^{1/2}+O(r^{-1/2}).
\end{align*}
These bounds hold uniformly for $|\yy|\lesssim\sqrt r$.
They follow either by differentiating {the equation defining the eliminated critical point} or by the
Schur complement formula for the Hessian; its denominator is bounded
away from zero in the primed variables. The symbols after stationary
phase have uniformly bounded supremum and total variation in $\yy$,
and uniform symbol bounds after each rescaling below.
In particular, the third derivative is bounded away from zero.

Fix $K$ sufficiently large. If $|\Bf|\leq K\sqrt r$, Van der Corput's
third derivative estimate gives
\begin{equation}\label{eq:76}
 \left|\int e^{ir^{3/2}\phi_{N,a,h,\pm}(r^{-1/2}\yy)}
 \tilde\sigma_{V,h,a,\pm}(r^{-1/2}\yy)\,d\yy\right|\lesssim1.
\end{equation}
Multiplying by the Jacobian $r$ and the factors $r^{-3/4}$ and
$r^{-1/2}$ in \eqref{eq:75}, we obtain $O(r^{-1/4})$ for
\eqref{eq:65}. Since $|\Bf|\leq K\sqrt r$, this is bounded by
$C/(1+|\Bf|^{1/2})$.

If $\Bf>K\sqrt r$, the reduced critical points have
$\yy_\pm=\pm\sqrt{\Bf}+O(\sqrt r/\sqrt{\Bf})$, and the second
derivative there has size $\sqrt{\Bf}$. Rescaling
$\yy=\sqrt{\Bf}\,\yy''$, stationary phase with parameter
$\Bf^{3/2}$ bounds each reduced integral by $O(\Bf^{-1/4})$.
Outside fixed neighborhoods of these two points, integration by parts
has the same or better bound. Hence the integral in \eqref{eq:65} is
\[
 O(r^{-1/4}\Bf^{-1/4})
 \lesssim\frac{1}{1+\Bf^{1/2}},
 \qquad \Bf\leq\Af\sim r.
\]
Finally, if $\Bf<-K\sqrt r$, the reduced first derivative has
absolute value at least $c(|\Bf|+\yy^2)$. The same rescaling
$\yy=\sqrt{|\Bf|}\,\yy''$ and integration by parts with parameter
$|\Bf|^{3/2}$ bound each reduced integral by $O(|\Bf|^{-1})$.
A dyadic partition gives this bound also for large $|\yy''|$.
Thus the integral in \eqref{eq:65} is
$O(r^{-1/4}|\Bf|^{-1})$, which is again bounded by
$C/(1+|\Bf|^{1/2})$.

This proves \eqref{eq:65} for both signs of $\Bf$; exchanging the
variables gives the remaining cases. The Jacobian $\lambda^{-2/3}$
of the first rescaling proves the lemma. The transition
between the Airy and nonstationary regimes occurs at
$|\Bf|\sim\sqrt r$, not at a fixed negative value of $\Bf$.
\end{proof}
\subsubsection{Proof of Propositions \ref{dispNpetitloin} and \ref{dispNpetitpres} (moderately large $1\leq |N|\leq \lambda^{1/3}$)}\label{sectmodN}

We start in the same way as the proof of Lemma \ref{lem:large-reflections}, replacing $\lambda^{1/3}$ by $|N|$ : we rescale variables with $(\Sigma,S)=(\frac{\xx}{|N|},\frac{\yy}{|N|})$. 
Here we need all the first three terms in the formula \eqref{eq:Acdet}. In our new variables we have, with $\Lambda:=\lambda/N^3$
  \begin{align}
\nonumber
 \frac 1h  \partial_{\xx}\Psi_{N,a,h,a} & \! \begin{multlined}[t] = \Lambda q^{1/2}(\theta_c)\Big(\xx^{2}+N^2(\frac xa (1+\Lp(y,\omega_c))-\Ab_0^2)+\Ab_0\Big(\xx(1-aE_1)-\yy(1-aE_2)\Big)\\+2N\sqrt{a}\xx\frac xa{\mu_\Gamma}(y,\omega_c)\frac{q^{1/2}(\theta_c)}{{\tau(aA_cq(\theta_c)^{1/3},\theta_c)}} -\frac{1}{4N^2}(\xx-\yy)^2+O(a\frac{\yy^2}{N^2},a\frac{\xx^2}{N^2})+N^2O(a)
 \Big),
\end{multlined}\\
\nonumber
\frac 1h  \partial_{\yy}\Psi_{N,a,h,a} & \! \begin{multlined}[t] = -\Lambda q^{1/2}(\theta_c)\Big(\yy^{2}+N^2(1-\Ab_0^2)+\Ab_0\Big(\xx(1-aE_1)-\yy(1-aE_2)\Big)\\{+2N\sqrt{a}\yy\mu_\Gamma}(0,\omega_c)\frac{q^{1/2}(\theta_c)}{{\tau(aA_cq(\theta_c)^{1/3},\theta_c)}}-\frac{1}{4N^2}(\xx-\yy)^2+O(a\frac{\yy^2}{N^2},a\frac{\xx^2}{N^2})+N^2O(a)
\Big),
\end{multlined}
      \end{align}
where the terms $O(a\frac{\yy^2}{N^2},a\frac{\xx^2}{N^2})$ come from the formula \eqref{eq:Acdet} of $A_c$ while the terms $N^2O(a)$
in the last two formulas are homogeneous terms of order $j\geq 2$ in the expression of $\partial_{\Xi}A_{\Gamma}$. We recall that $\partial_{\Xi}A_{\Gamma}(x,y,\Xi,\Theta)=\Lp(y,\omega_c)+2\Xi{\mu_\Gamma}(y,\omega_c)+\mathcal{H}_{j\geq 2}$ where {the weight-two term is
$\alpha_1x+3\mu_1\Xi^2$, and the remaining terms are
$\partial_\Xi R_A\in\mathcal H_{\ge3}$}, with
$\Xi=\sqrt{a}\Sigma q^{1/2}/\tau$.
In fact, we infer that we can write
 \begin{gather}\label{eq:derivePhiN2}
 \frac 1h  \partial_{\xx}\Psi_{N,a,h,a}=\Lambda  q^{1/2}(\theta_c)\Big(\xx^{2}-\Af+(b \xx-d \yy)-\frac{1}{4N^2}(\xx-\yy)^2+O(a\xx^2)+O(a\yy^2)+O(N^2a)\Big)\\
 \frac 1h  \partial_{\yy}\Psi_{N,a,h,a}=-\Lambda  q^{1/2}(\theta_c)\Big(\yy^{2}-\Bf+(d \xx-c\yy)-\frac{1}{4N^2}(\xx-\yy)^2+O(a\xx^2)+O(a\yy^2)+O(N^2a)\Big),
\end{gather}
The terms $O(N^2a)$ include the constant contribution of $\alpha_1x$
at the observation point and of $\alpha_1a$ at the source point. The contribution of
$3\mu_1\Xi^2$ is included in $O(a\xx^2,a\yy^2)$.
These are smooth remainders after composition with the critical points,
with the corresponding derivative bounds; in particular, the notation
$O(N^2a)$ does not assert independence of $\xx,\yy$.
We use the following coefficients, with
$\omega_c^0=\omega_{\theta_c}|_{\xx=\yy=0}$ and
$\theta_c^0=\theta_c|_{\xx=\yy=0}$:
\begin{gather*}
-\frac{\Af}{N^2}=\frac xa(1+\Lp(y,\omega_c^0))-\Ab_0^2,
\qquad -\frac{\Bf}{N^2}=1-\Ab_0^2,\\
\begin{aligned}
b&=\Ab_0(1+O(a))+2N\sqrt a\frac xa\mu_\Gamma(y,\omega_c^0)
\frac{\sqrt{q(\theta_c^0)}}{\tau(aA_cq(\theta_c^0)^{1/3},\theta_c^0)}
+O(Na^{3/2}),\\
c&=\Ab_0(1+O(a))-2N\sqrt a\mu_\Gamma(0,\omega_c^0)
\frac{\sqrt{q(\theta_c^0)}}{\tau(aA_cq(\theta_c^0)^{1/3},\theta_c^0)}
+O(Na^{3/2}),\\
d&=\Ab_0(1+O(a))+O(Na^{3/2}).
\end{aligned}
\end{gather*}
The quotients in these formulas are evaluated at $\xx=\yy=0$.
The cutoff gives $|\theta_c^0|\sim1$, not $|\theta_c^0|=1$,
so their denominators must be retained.
Since $N^2a\lesssim t^2$, the constant corrections are small.
On bounded sets the nonconstant corrections are small in $C^m$ for
any fixed $m$; after the radial rescaling below the quadratic
corrections are $O(a)$ and the constant corrections are $O(t^2/r)$.

This follows easily from writing the Taylor expansion for the first order derivatives of $\Psi_{N,a,h,a}$ with respect to $\xx,\yy$ and use that $\partial^2_{\xx,\yy}\Psi_{N,a,h,a}=\partial^2_{\yy,\xx}\Psi_{N,a,h,a}$. The  terms $O(Na^{3/2})$ come from the contribution of $N^2\frac xa \Lp(y,\omega_{\theta_c})-N^2\frac xa \Lp(y,\omega_{\theta_c}^0)$, $\omega^0_c=\omega_c|_{\xx=\yy=0}$ that we write as follows using \eqref{dertheta} and $\Lp(y,\cdot)=O(|y|)$
\[
N^2\frac xa (O(< \nabla \Lp(y,\omega_c),\partial_{\xx}\theta_c>)+O(< \nabla \Lp(y,\omega_c),\partial_{\yy}\theta_c>))=N^2O\Big(|y|\times\frac{1}{N}\times \frac{a^{3/2}}{|y|}\Big)=O(Na^{3/2})=O(ta),
\]
where the factor $1/|N|$ comes from $\partial\Sigma/\partial\xx=1/|N|$.
Next, we compute the second derivatives : notice that \eqref{dertheta} yields $|\nabla_{\Sigma,S}\theta_c|\sim O(a^{3/2}/t)\sim O(a/N)$ since $|\theta_c|\sim 1$, $|y|\sim 2t$ and $t/\sqrt{a}\sim 2N$.
We obtain, at $\nabla_{\xx}\Psi_{N,a,h,a}=\nabla_{\yy}\Psi_{N,a,h,a}=0$
\begin{align}\label{eq:160bis}
\frac1h\partial_{\xx}^2\Psi
&=\Lambda\sqrt{q(\theta_c)}
\Big(2\xx+b-\frac{\xx-\yy}{2N^2}+O(a(1+|\xx|+|\yy|))\Big),\\
\frac1h\partial_{\yy}^2\Psi
&=\Lambda\sqrt{q(\theta_c)}
\Big(-2\yy+c-\frac{\xx-\yy}{2N^2}+O(a(1+|\xx|+|\yy|))\Big),\\
\frac1h\partial_{\xx}\partial_{\yy}\Psi
&=\Lambda\sqrt{q(\theta_c)}
\Big(-d+\frac{\xx-\yy}{2N^2}+O(a(1+|\xx|+|\yy|))\Big).
\end{align}
Here $\Psi=\Psi_{N,a,h,a}$. We denote this Hessian by
$\Lambda\sqrt{q(\theta_c)}M(\xx,\yy)$; thus $M$ is the normalized
matrix. Its determinant at a critical point is
\begin{equation}\label{eq:80}
\det M=-\Big(4\xx\yy+2(b\yy-c\xx)+d^2-bc
+\frac{(\xx-\yy)^2}{N^2}
+\frac{(\xx-\yy)(b+c-2d)}{2N^2}\Big)
+O\big(a(1+|\xx|+|\yy|)^2\big).
\end{equation}
The determinant of the Hessian itself is
$\Lambda^2q(\theta_c)\det M$.

We first establish uniform bounds for the reduced phase.
We work with $t>0$, $N>0$; the other signs follow by conjugation.
In the definition below, $\Psi_{N,a,h,a}$ is evaluated in its
original variables $(\Sigma,S)$. Put $\Lambda=\lambda/N^3\ge1$
and define
\begin{equation}\label{eq:moderate-normalized-full-phase}
 \Phi^{\mathrm{nor}}_{N,a}(\xx,\yy)
 =\frac{\Psi_{N,a,h,a}(\xx/N,\yy/N)
             -\Psi_{N,a,h,a}(0,0)}
 {h\Lambda\sqrt{q(\theta_c^0)}}.
\end{equation}
Subtracting a constant does not change the estimates, nor does
the factor $\sqrt{q(\theta_c^0)}\sim1$.
We omit the indices of $\Phi^{\mathrm{nor}}$ when they are fixed.
On any fixed compact set in the integration variables,
\eqref{eq:derivePhiN2} and \eqref{eq:Acdet} give
\begin{equation}\label{eq:moderate-full-phase-expansion}
 \Phi^{\mathrm{nor}}
 =\frac{\xx^3-\yy^3}{3}+\frac{\Ab_0}{2}(\xx-\yy)^2
 -\frac{(\xx-\yy)^3}{12N^2}-\Af\xx+\Bf\yy+R_{N,a},
 \qquad \|R_{N,a}\|_{C^m}\lesssim T_0.
\end{equation}
Here $m$ is fixed and sufficiently large to apply the lemmas
in Appendix~\ref{app:refined-quartic-estimate}. The bound also holds
for the parameter derivatives used below, and $R_{N,a}(0,0)=0$.
The remainder depends on the integration variables.

To check uniformity, also for the other rescalings in this section,
we differentiate {the equations for the critical points}. If $U=(A,\theta)$ and
$\Psi$ is the critical value of $\tilde\Phi$, then
\begin{equation}\label{eq:full-phase-critical-chain}
 \Psi_i=\tilde\Phi_i|_c,\qquad
 \Psi_{ij}=\left(\tilde\Phi_{ij}
       -\tilde\Phi_{iU}\tilde\Phi_{UU}^{-1}\tilde\Phi_{Uj}
                \right)\Big|_c,\qquad i,j\in\{\Sigma,S\}.
\end{equation}
The two scales of $\tilde\Phi_{UU}$ are $Na^{3/2}$ and $t$,
where $t\sim N\sqrt a$, and its mixed entries are $O(at)$.
Dividing its rows and columns by $(Na^{3/2})^{1/2}$ and $t^{1/2}$
leaves a uniformly invertible matrix with mixed block $O(\sqrt a)$.
This follows from \eqref{eq1Ac}, \eqref{eq:dervisecPhi} and
\eqref{Acexpres}. Repeated implicit differentiation of this
normalized system gives the required bounds for $A_c,\theta_c$;
in particular,
$\partial_{\Sigma,S}\sqrt{A_c}=O(N^{-1})$ and
$\partial_{\Sigma,S}\theta_c=O(a/N)$.
We differentiate the expansion of the source and observation terms after substitution:
$\Xi=\sqrt a\Sigma\sqrt q/\tau=O(\sqrt a)$,
$\rho-1=O(a)$ and
$\partial_\Xi R_A=O_{C^m}(a^{3/2})$.
These factors are preserved under differentiation in the rescaled variables. The term
$aAq(B_0-B_2)/(2|\theta|)$ and the quadratic $A$ term in
\eqref{Gammatildegam} change the critical points and are included
in this system.

On a fixed compact set, \eqref{eq:derivePhiN2} gives
$b,c,d=\Ab_0+O(N\sqrt a)$; the constant corrections are
$O(N^2a)$ and the quadratic ones are $O(a)$. The ratio $\sqrt{q(\theta_c)}/\sqrt{q(\theta_c^0)}$
adds a small smooth correction as well. Since
$N\sqrt a\lesssim t\le T_0$, these observations and
\eqref{eq:full-phase-critical-chain} prove
\eqref{eq:moderate-full-phase-expansion} with any fixed required
number of derivatives. Uniformity in $N$ uses the bounded scale
parameters $N^{-1}$ and $N\sqrt a$.

To locate the swallowtail point, we first consider bounded
$r=(\Af^2+\Bf^2)^{1/2}\le r_0$, with $r_0$ fixed.
Fix $t,a,N$ and the direction of $y$. Near a reference point,
we vary $x$ and the length of $y$ by setting, with $z=(z_{1},z_{2})$
\begin{equation}\label{eq:moderate-physical-parameters}
 x=x_*+\frac a{N^2}z_1,\qquad y=(1+z_2/N^2)y_*.
\end{equation}
We check that the map $z\mapsto(\Af,\Bf)$ has a Jacobian
determinant bounded away from zero, so that $(\Af,\Bf)$
can be used as local parameters. The preceding derivative bounds remain uniform for bounded $z$.
Using \eqref{eq:defA0},
$\theta_g=-(y+\nabla B_0(y,\theta_g))/(2t)$ and the homogeneity
of $q^{1/2}$, with $B_0,B_2=O(|y|^2)$, gives
\begin{equation}\label{eq:moderate-physical-Jacobian}
 D_z(\Af,\Bf)=
 \begin{pmatrix}-(1+\Lp)&2\Ab_0^2\\0&2\Ab_0^2\end{pmatrix}
 +O(T_0),\qquad
 \det D_z(\Af,\Bf)=-2\Ab_0^2(1+\Lp)+O(T_0).
\end{equation}
For example, the $x$ derivative of the correction $aE_0$
in \eqref{eq:defA0} is bounded; the scale $a/N^2$ in
\eqref{eq:moderate-physical-parameters} gives
$\partial_{z_1}\Ab_0=O(\sqrt a/N^3)$.
Differentiating $\Lp(y,\omega_c^0)$ in the radial $y$ direction
adds $O(T_0)$ to the second column. The other terms coming from $A_c,\theta_c$ are controlled by
the same normalized system.
Thus the determinant is bounded away from zero, and we may use
$(\Af,\Bf)$ as local parameters. This check uses a radial variation
of $y$: in the principal approximation,
$\theta_c^0=-y/(2t)$ and
$\Ab_0=\sqrt{q(y)}/(4N\sqrt a)$, so varying $t$ at fixed $y$
would not give the required second parameter.

For the principal part of \eqref{eq:moderate-full-phase-expansion},
rotate to
$\xi_1=(\xx-\yy)/\sqrt2$, $\xi_2=(\xx+\yy)/\sqrt2$.
It becomes
\begin{equation}\label{eq:moderate-rotated-principal-phase}
 \Ab_0\xi_1^2+\frac{\xi_1\xi_2^2}{\sqrt2}
 +\frac{1-N^{-2}}{3\sqrt2}\xi_1^3+P\xi_1+Q\xi_2,
 \qquad P=-\frac{\Af+\Bf}{\sqrt2},\quad
 Q=\frac{\Bf-\Af}{\sqrt2}.
\end{equation}
Let $\psi$ be the critical value of this principal phase after
eliminating $\xi_1$; from now $\xi_2$ is the remaining, one dimensional, integration variable, and we denote with (multiple) $'$ derivatives in $\xi_{2}$. The nondegenerate second derivative is
$2\Ab_0$, the mixed third derivative is $\sqrt2$, and the
reduced fourth derivative after eliminating $\xi_1$ is
$-3/\Ab_0$. At its central parameter value,
\begin{equation}\label{eq:moderate-central-expansion}
 D_{\Af,\Bf}(\psi',\psi'')=
 \begin{pmatrix}-1/\sqrt2&1/\sqrt2\\
                 1/(2\Ab_0)&1/(2\Ab_0)\end{pmatrix}.
\end{equation}
Since $\Ab_0$ multiplies $\xi_1^2$, its parameter derivatives
contribute zero at the origin. The determinant has absolute value
$1/(\sqrt2\Ab_0)$, and the full phase changes these derivatives
by $O(T_0)$.
Hence the implicit function theorem applied to
$\psi' =\psi'' =\psi''' =0$, with unknowns
$(\Af,\Bf,\xi_{2})$, gives a unique nearby swallowtail point for the full
phase, with displacement $O(T_0)$ in both integration and
parameter variables. The fourth derivative of the reduced phase remains
$-3/\Ab_0+O(T_0)$, bounded away from zero. Then, Lemma \ref{lem:refined-quartic-family} applies uniformly.
To put the phase in the coordinates of its statement, we may
diagonalize the full Hessian at the displaced swallowtail point
by an additional (small !) $O(T_0)$ rotation.
The parameter condition \eqref{eq:quartic-parameter-jacobian}
is preserved, since its determinant is bounded away from zero.
This auxiliary rotation is used only to apply the lemma;
we keep the initial $45^\circ$ coordinates in our calculations.
The parameters $(\widehat P,\widehat Q)$ are defined by
\eqref{eq:quartic-adapted-parameters} for the full reduced phase.

Denote by $I_{\swallowtail}$ the contribution to the integral
in \eqref{eq:85} below, localized by a smooth cutoff near
the swallowtail point of the full phase. Lemma \ref{lem:refined-quartic-family}, with oscillatory parameter
$\Lambda\sqrt{q(\theta_c^0)}\sim\Lambda$, now gives
\begin{equation}\label{eq:moderate-full-refined-bound}
 |I_{\swallowtail}|\lesssim
 \frac{\Lambda^{-5/6}}
 {\Lambda^{-1/12}+|\widehat P|^{1/6}+|\widehat Q|^{1/9}}.
\end{equation}
Once the branches $\widehat Q=\widehat Q_\pm(\widehat P)$
are separated, \eqref{eq:refined-quartic-branch-bound} applies near
either one. Since the coefficient of the mixed third-order term is
$\sqrt2+O(T_0)$, Section~\ref{app:cubic-quartic-matching} also
applies when we rescale or change the eliminated variable.
We keep the coordinates centered at the full swallowtail point:
when the branches are very close, an $O(T_0)$ rotation of the
old coordinates is not negligible.

We next consider the other critical points on the fixed compact set.
Outside that common small neighborhood of the swallowtail point, the principal
phase has only nondegenerate or Airy-type critical points.
Indeed, on $\xi_1=0$ its Hessian is singular only at
$\xi_2=0$, which we have just dealt with. For $\xi_1\ne0$ we eliminate
$\xi_2$, and use $\psi$ again for the remaining phase and $'$ to denote its derivatives w.r.t. $\xi_{1}$, obtaining
\begin{equation}\label{eq:moderate-exterior-reduced-phase}
 \psi=\Ab_0\xi_1^2+\frac{1-N^{-2}}{3\sqrt2}\xi_1^3
       +P\xi_1-\frac{Q^2}{2\sqrt2\xi_1},\qquad
 \psi'''=\sqrt2(1-N^{-2})+\frac{3Q^2}{\sqrt2\xi_1^4}.
\end{equation}
For $N\ge2$ the third derivative is uniformly positive; for
$N=1$ it can vanish only at $Q=0$, where $\psi''=2\Ab_0$.
A finite cover of the remaining compact therefore satisfies
Lemma \ref{lem:uniform-cubic-continuation}, uniformly also in the
limit $N^{-2}=0$. The small $C^m$ perturbation preserves these
critical-point conditions.

For the estimate transverse to the branches, we must also
differentiate $\Ab_0^2=1+\Bf/N^2$. For the principal phase,
\begin{equation}\label{eq:moderate-exterior-transversality}
 (\partial_{\Af}+\partial_{\Bf})\psi'
 =\frac{\xi_1}{N^2\Ab_0}-\sqrt2
 =-\frac{\sqrt2}{\Ab_0}
    \left(\Ab_0-\frac{\xi_1}{\sqrt2N^2}\right).
\end{equation}
The factor in parentheses is the principal part of the positive
square root of $A_c$, from \eqref{Acexpres} before squaring to
obtain \eqref{eq:Acdet}. It is bounded away from zero on the
symbol support. The $O(T_0)$ correction from the full phase
preserves this bound, so Lemma \ref{lem:uniform-cubic-continuation}
gives the estimate across each branch, matching
stationary phase away from the branch.

For $r\le r_0$ all relevant critical points lie in a fixed ball
$|(\xx,\yy)|\le C\sqrt{1+r_0}$. For $N\ge2$ this follows from
the full gradient \eqref{eq:derivePhiN2}: its leading quadratic
part stays away from zero outside a sufficiently large fixed
ball. For $N=1$ the original symbol has fixed compact support.
On the complement we integrate by parts, keeping the ball fixed
even when $r$ is small. Combining the worst estimate with the finite covering of the
other critical points gives, for the phase defined in \eqref{eq:moderate-normalized-full-phase} but after we rescaled its variables $(\Sigma,S)=(\xx/N,\yy/N)$,
\begin{equation}\label{eq:85}
 \left|\int e^{i\Lambda\sqrt{q(\theta_c^0)}\Phi^{\mathrm{nor}}}
       \sigma_{V,h,a}(\xx/N,\yy/N,\cdot)\,d\xx\,d\yy\right|
 \lesssim\Lambda^{-3/4},\qquad r\le r_0\,.
\end{equation}
On fixed bounded parameter sets separated from zero by the
constant in the hypothesis of Proposition \ref{dispNpetitloin},
the parameter value corresponding to the swallowtail is excluded once $T_0$ is small.
The same cover then gives $O(\Lambda^{-5/6})$.
Near the parameter value corresponding to the swallowtail we only need the uniform bound;
no comparison of $(\widehat P,\widehat Q)$ with the coefficients
of the principal phase is used.

We now consider $r\ge r_0\gg1$. Denote by $I_{\mathrm{large}\,r}$
the integral in \eqref{eq:85} in this regime, before
multiplication by the prefactor in \eqref{VNdef}. Write
$(\Af,\Bf)=r(\cos\tf,\sin\tf)$ and
$(\xx,\yy)=\sqrt r(\xx',\yy')$.
We divide $\Phi^{\mathrm{nor}}$ by $r^{3/2}$; its oscillatory
parameter is $\Lambda r^{3/2}\sqrt{q(\theta_c^0)}$.
The full normalized gradients in \eqref{eq:derivePhiN2} have
constant terms $-\cos\tf,\sin\tf$, third-order coefficients $1,N^{-2}$,
linear terms $O(r^{-1/2})$ and remainder
$O_{C^m}(a)+O_{C^m}(N^2a/r)$.
On the original symbol support $\sqrt r/N$ is bounded. In
particular,
\begin{equation}\label{eq:moderate-radial-scale-bounds}
 r\lesssim N^2,\qquad N^{-2}\lesssim r^{-1},\qquad
 a\lesssim T_0^2/N^2\lesssim T_0^2/r.
\end{equation}
For $N\ge2$ the critical points remain in a common annulus.
The limiting homogeneous polynomial of degree three is
\eqref{eq:moderate-rotated-principal-phase} with its quadratic
term removed, together with its normalized linear terms.
On that annulus its degenerate critical points have a
nondegenerate $\xi_2$ direction and reduced third derivative
at least $3\sqrt2/4$. Large $r_0$ and small $T_0,a_0$ preserve
these bounds for the full phase. Lemma \ref{lem:uniform-cubic-continuation} gives
\begin{equation}\label{eq:83}
 |I_{\mathrm{large}\ r}|\lesssim
 r(\Lambda r^{3/2})^{-5/6}
 =\Lambda^{-5/6}r^{-1/4}.
\end{equation}
We integrate by parts on the inner and outer regions without
critical points. For $N=1$ and large $r$, the phase has no
critical point on the symbol support, so integration by
parts suffices. Varying the coefficients of the normalized linear terms
gives a nonzero parameter derivative of the first derivative
of the reduced phase.
The dependence of $\Ab_0/\sqrt r$ contributes only
$O(\sqrt r/N^2)=O(r^{-1/2})$ by
\eqref{eq:moderate-radial-scale-bounds}.
Thus the transversality condition of
Lemma~\ref{lem:uniform-cubic-continuation} holds uniformly, and we get our estimate across both branches, matching stationary phase
away from them.

For the stronger bound in \eqref{eq:2ff}, take $N\ge16$ and,
exchanging the source and observation points if necessary, $\Af\ge\Bf$, so that
$\Af\sim r$ whenever there is a critical point. Fix a large
constant $K$. If $|\Bf|\le K\sqrt r$, the $\xx'$ direction is
nondegenerate, with critical points
$\xx'_c=\pm\sqrt{\cos\tf}+O(r^{-1/2})$.
Its reduced third derivative in $\yy'$ stays bounded away from
zero, so \eqref{eq:83} applies. If $\Bf<-K\sqrt r$, the full
$\yy'$ derivative has no zero on the relevant support, and we
integrate by parts. If $\Bf>K\sqrt r$, the stationary points
satisfy $|\xx'_c|\sim1$ and
$|\yy'_c|\gtrsim\sqrt{\Bf/r}$.
Using \eqref{eq:80}, the normalized Hessian determinant is bounded
below by a constant times $\sqrt{\Af\Bf}/r$.
Indeed the product $|\xx'_c\yy'_c|$ dominates the linear terms
of size $r^{-1/2}$ once $K$ is large, the $N^{-2}$ terms by
\eqref{eq:moderate-radial-scale-bounds}, and the smooth remainder
once $T_0,a_0$ are small. Stationary phase, including
the Jacobian $r$, gives
\begin{equation}\label{bounddetNsl}
 |I_{\mathrm{large}\ r}|\lesssim
 \Lambda^{-1}(\Af\Bf)^{-1/4}.
\end{equation}

Finally the prefactor in \eqref{VNdef}, including the Jacobian
$N^{-2}$, is
$h^{1/3}\lambda^{5/6}N^{-5/2}=h^{1/3}\Lambda^{5/6}$.
Thus \eqref{eq:85} gives
$|V_{N,a}|\lesssim h^{1/3}\Lambda^{1/12}$, proving
\eqref{eq:2hh}. The Airy bound away from the swallowtail on bounded parameter
sets gives $|V_{N,a}|\lesssim h^{1/3}$ there.
For large $r$, \eqref{eq:83} gives
$h^{1/3}r^{-1/4}$; on $|\Bf|\le K\sqrt r$ this is at most a
constant times $h^{1/3}/(1+|\Bf|^{1/2})$.
For $\Bf>K\sqrt r$, \eqref{bounddetNsl} gives the same bound
because $\Lambda\ge1$ and $\Af\ge\Bf$.
The nonstationary part is smaller. Exchanging the source and observation points gives
\eqref{eq:2ff}; for $1\le N\le16$, \eqref{eq:83} and the
Airy bound on bounded parameter sets give \eqref{eq:21ff}.
This proves Propositions \ref{dispNpetitloin} and
\ref{dispNpetitpres}.
\subsubsection{Almost transverse initial directions $\gamma=2^j a\geq 4a$, $a\gtrsim h^{2/3-\epsilon}$, $0\leq x\leq a$. Proof of Proposition \ref{proptransv}}\label{sectransv}
Let $V_{N,h,\gamma}$ be as in \eqref{VNdef}.
Proposition \ref{proptransv} follows from the following lemma :
\begin{lemma}\label{lem:transverse-directions}

{Let $0\le x\le a$. There exists $C$,
independent of $N$, such that the following estimate holds for
$|N|\ge1$, with $t/\sqrt\gamma\ge c_*>0$ when $|N|=1$;
in the latter case $C$ may depend on $c_*$.}
  \begin{equation}
    \label{eq:transverse-stationary-phase}
\left|      \int e^{\frac i h \Psi_{N,a,h,\gamma}} \sigma_{V,h,a}(\Sigma,S,\cdot) d\Sigma dS  \right|  \leq C\lambda_{\gamma}^{-1} \text{ hence }  |V_{N,h,\gamma}(t,x,y)|\lesssim \frac{\gamma^2}{h}\frac{1}{\sqrt{\lambda_{\gamma}N}}\frac{1}{\lambda_{\gamma}}.
  \end{equation}
\end{lemma}
Summing up over $N\sim t/\sqrt{\gamma}$ yields
\eqref{eq:37bis} (where, for $\frac{t}{\sqrt{\gamma}}\lesssim \lambda_{\gamma}^{1/3}$, we have $\Big(\frac{ht}{\gamma}\Big)^{1/2}\lesssim h^{1/3}$). 

\begin{proof}(of Lemma \ref{lem:transverse-directions})
{We differentiate the reduced phase $\Psi_{N,a,h,\gamma}$ defined in \eqref{VNdef}. Its first derivatives are}
\begin{align}\label{derivSigmaphizero}
\partial_{\Sigma} \Psi_{N,a,h,\gamma} &
= \gamma^{3/2}q(\theta_c)^{1/2}\Big[\Sigma^2+\frac{x}{\gamma}-A_c+\frac{x}{\gamma}\partial_{\Xi}A_{\Gamma}(x,y,\Xi,\frac{\theta_c}{ {\tau(\gamma A_cq(\theta_{c})^{1/3},\theta_{c})}})|_{\Xi=\frac{\sqrt{\gamma}\Sigma q^{1/2}(\theta_{c})}{{\tau(\gamma A_cq(\theta_{c})^{1/3},\theta_{c})}}}\Big],\\
\partial_{S} \Psi_{N,a,h,\gamma} &
= -\gamma^{3/2}q(\theta_c)^{1/2}\Big[S^2+\frac{a}{\gamma}-A_c+\frac{a}{\gamma} \partial_{\Xi}A_{\Gamma}(a,0,\Xi,\frac{\theta_{c}}{{\tau(\gamma A_cq(\theta_{c})^{1/3},\theta_{c})}})|_{\Xi=\frac{\sqrt{\gamma}S q^{1/2}(\theta_{c})}{ {\tau(\gamma A_cq(\theta_{c})^{1/3},\theta_{c})}}}\Big].
\end{align}
Here $\omega_c=\omega_{\theta_c}$ and $A_c$ is given by \eqref{eq:Acdet}.
The critical points (which will be denoted $\Sigma_{\pm},S_{\pm}$) satisfy
\begin{equation}
\Sigma^2+\frac{x}{\gamma}(1+\Lp(y,\omega_c)+O(\sqrt{\gamma}))=A_c,\quad
S^2+\frac{a}{\gamma}(1+O(\sqrt{\gamma}))=A_c.
\end{equation} 
{The cutoff $\kappa=\phi-\phi(2\cdot)$ is supported in
$1/2\le A\le3/2$. We retain this support.
Since $0\le x\le a$ and $\gamma\ge4a$, both $x/\gamma$ and
$a/\gamma$ are at most $1/4$. With $\Lp=O(|y|)$ and
$|y|\lesssim t\le T_0$, the two equations give
$|\Sigma_\pm|^2,|S_\pm|^2\ge1/8$ for small $T_0,\varepsilon_0$.}

Writing $\Psi=\Psi_{N,a,h,\gamma}$, we differentiate
\eqref{derivSigmaphizero}, including the dependence of $A_c,\theta_c$
on $\Sigma,S$:
\begin{equation}\label{secderivSigSig}
\partial_\Sigma^2\Psi
=\gamma^{3/2}\sqrt{q(\theta_c)}\Big[2\Sigma-\partial_\Sigma A_c
+\frac x\gamma\partial_\Sigma(\partial_\Xi A_\Gamma)\Big]
+\frac{\langle\nabla q(\theta_c),\partial_\Sigma\theta_c\rangle}
{2q(\theta_c)}\partial_\Sigma\Psi.
\end{equation}
In this formula $\partial_\Sigma(\partial_\Xi A_\Gamma)$ is a total
derivative, evaluated at the arguments in \eqref{derivSigmaphizero}.
Indeed, before evaluation at the critical points,
\begin{multline}\label{formderivXiAGam}
\partial_\Sigma(\partial_\Xi A_\Gamma)
=\left.\partial_\Sigma(\partial_\Xi A_\Gamma)\right|_{A,\theta}
+\partial_A(\partial_\Xi A_\Gamma)\partial_\Sigma A_c
+\nabla_\theta(\partial_\Xi A_\Gamma)\cdot\partial_\Sigma\theta_c\\
=2\sqrt\gamma\mu_\Gamma(y,\omega_c)
\frac{\sqrt{q(\theta_c)}}{\tau(\gamma A_cq(\theta_c)^{1/3},\theta_c)}
+6\gamma\Sigma\mu_1(y,\omega_c)
\frac{q(\theta_c)}{\tau(\gamma A_cq(\theta_c)^{1/3},\theta_c)^2}
+O(\gamma^{3/2}).
\end{multline}
Here $\partial_A(\partial_\Xi A_\Gamma)=O(\gamma^{3/2})$ with
$\Sigma,\theta$ fixed. This follows from $\beta_1=0$, the expansion
$\partial_\Xi A_\Gamma=\Lp+2\mu_\Gamma\Xi+\alpha_1x+
3\mu_1\Xi^2+\partial_\Xi R_A$, and
$\tau^2=|\theta|^2+\gamma Aq(\theta)$.
The terms containing $\partial_\Sigma\theta_c$ are
$O(\gamma^{3/2})$, since $\partial_\Sigma\theta_c=O(\gamma^{3/2}/t)$,
$\Lp(y,\cdot)=O(|y|)$ and $|y|\sim t$.
These bounds also hold with any fixed number of derivatives on bounded
sets. The $S$ derivative of $\partial_\Xi A_\Gamma$ at the
observation point is $O(\gamma^{3/2})$; the source point is
handled in the same way.

By \eqref{eq:Acdet},
\[
-\partial_\Sigma A_c=\frac{\Ab_0}{N}
-\frac{\Sigma-S}{2N^2}+O(\gamma/N),\qquad
-\partial_S A_c=-\frac{\Ab_0}{N}
+\frac{\Sigma-S}{2N^2}+O(\gamma/N).
\]
Since $\sqrt{A_c}=\Ab_0-(\Sigma-S)/(2N)+O(\gamma)$, at a critical
point of $\Psi_{N,a,h,\gamma}$ we obtain the full Hessian:
\begin{align}\label{eq:secderivSTprime}
\partial_\Sigma^2\Psi_{N,a,h,\gamma}
&=\gamma^{3/2}\sqrt{q(\theta_c)}
\big(2\Sigma+\sqrt{A_c}/N+O(\sqrt\gamma)\big),\\
\partial_S^2\Psi_{N,a,h,\gamma}
&=\gamma^{3/2}\sqrt{q(\theta_c)}
\big(-2S+\sqrt{A_c}/N+O(\sqrt\gamma)\big),\\
\partial_\Sigma\partial_S\Psi_{N,a,h,\gamma}
&=\gamma^{3/2}\sqrt{q(\theta_c)}
\big(-\sqrt{A_c}/N+O(\sqrt\gamma)\big).
\end{align}
Consequently,
\[
\frac{\det\nabla^2_{\Sigma,S}\Psi_{N,a,h,\gamma}}
{\gamma^3q(\theta_c)}
=-4\Sigma S+\frac{2\sqrt{A_c}}N(\Sigma-S)+O(\sqrt\gamma).
\]
{We have $x/\gamma,a/\gamma\le1/4$ and $A_c\ge1/2$.
The equations for the critical points give
$|\Sigma|/\sqrt{A_c},|S|/\sqrt{A_c}
\ge1/\sqrt2-O(T_0+\sqrt\gamma)$.}
For $|N|\geq2$, irrespective of the signs of $\Sigma,S$,
\begin{equation*}
\left|-4\Sigma S+\frac{2\sqrt{A_c}}N(\Sigma-S)\right|
\geq A_c\left(4\frac{|\Sigma|}{\sqrt{A_c}}
\frac{|S|}{\sqrt{A_c}}-
\frac{|\Sigma|+|S|}{\sqrt{A_c}}\right)
\geq A_c\left({2-\sqrt2}
-O(T_0+\sqrt\gamma)\right).
\end{equation*}

The last constant is positive. Taking $T_0,\varepsilon_0$ small
proves the determinant bound for $|N|\ge2$.
For $N=1$ put $m=|\Sigma|/\sqrt{A_c}$,
$n=|S|/\sqrt{A_c}$. Cancellation can only occur for the sign choice
$\Sigma=-m\sqrt{A_c}$, $S=n\sqrt{A_c}$; the other three
give a fixed positive lower bound immediately. For this choice
the principal determinant divided by $A_c$ is
$-2(m+n-2mn)$. The equation in $A$, \eqref{Acexpres}, gives
\begin{equation}\label{eq:transverse-one-reflection-time}
 \frac{t\sqrt{q(\theta_c)}}{\sqrt\gamma\sqrt{A_c}}
 =2-m-n+O(T_0+\sqrt\gamma).
\end{equation}
The $B_0-B_2$ term is a relative $O(T_0)$ correction and the
source and observation terms contribute $O(\sqrt\gamma)$.
If $t/\sqrt\gamma\ge c_*>0$, as in Proposition
\ref{proptransv}, we obtain $2-m-n\ge c>0$.
{The equations for the critical points give} $m,n\le1+O(T_0+\sqrt\gamma)$,
as well as the fixed positive lower bounds above. Hence
\begin{equation*}
 m+n-2mn=n(1-m)+m(1-n)
 \ge\min(m,n)(2-m-n)-O(T_0+\sqrt\gamma)\ge c'>0,
\end{equation*}
after decreasing $T_0,\varepsilon_0$ if necessary, with constants
depending on $c_*$. The full determinant is therefore bounded
away from zero after division by $\gamma^3q(\theta_c)$ for
$N=1$ as well. Conjugation treats $N=-1$.
Stationary phase proves
\eqref{eq:transverse-stationary-phase} in all these cases.
Here $t/\sqrt\gamma$ is bounded below, as required in the proposition.
\end{proof}

\subsection{Case $(ht)^{1/2}\geq a\geq h^{2/3-\varepsilon}$}\label{seca<}
We use the data $g_{h,a}$ obtained in Lemma \ref{lemp}, (as everything done before works as long as $a\geq h^{2/3-\varepsilon}$ except for the sum over $N$). Then $\mathcal{P}_{S,h,a}$ is given by both formulas \eqref{eq:Prond} and \eqref{eq:Prond2}. Write
\[
\mathcal{P}_{S,h,a}(t,x,y)=\sum_{a\leq\gamma< (ht)^{1/2}}\mathcal{P}_{S,h,a,\gamma}(t,x,y) +\sum_{\varepsilon_0>\gamma\gtrsim  (ht)^{1/2}}\mathcal{P}_{S,h,a,\gamma}(t,x,y).
\]

\subsubsection{The sum for $\varepsilon_0\geq \gamma={2^j} a\geq (ht)^{1/2}$} \label{seca>(ht)12}

 For $\gamma\geq \max\{h^{2/3-\varepsilon}, (ht)^{1/2}\}$ we use the estimates from Proposition \ref{proptransv} which are available for all $a\geq h^{2/3-\varepsilon}$ and all $\gamma\geq 4a$. As in \eqref{eq:37bis} we find
  \begin{equation}
    \label{eq:37bis<}
    \|(\Prond_{S,h,a,\gamma}-W_{0,h,\gamma})(t,\cdot)\|_{L^{\infty}({0\leq x\leq a}, y)}\leq \sup_{\substack{0\leq x\leq a\\y}}\sum_{|N|\geq 1}|W_{N,h,\gamma}(t,x,y)| \lesssim 
   \left\{ \begin{array}{l} 
\frac{1}{h^d}\Big(\frac ht \Big)^{\frac{d-1}2}\Big({\frac{ht}{\gamma}}\Big)^{1/2}, \text{ if } \frac{t}{\sqrt{\gamma}}\gtrsim \lambda_{\gamma}^{1/3}, \,\\
\frac{1}{h^d}\Big(\frac ht \Big)^{\frac{d-1}2}h^{1/3}, \text{ if } 1\lesssim \frac{t}{\sqrt{\gamma}}\lesssim\lambda_{\gamma}^{1/3},
 \end{array} \right.
  \end{equation}
hence, using that $\frac{t}{\sqrt{\gamma}}\gtrsim \lambda_{\gamma}^{1/3}$ when $\gamma<h^{1/3}t$, we get
\begin{align}
\sum_{\varepsilon_0>\gamma=2^j a \gtrsim  (ht)^{1/2}}|\mathcal{P}_{S,h,a,\gamma}(t,x,y)|= & \sum_{(ht)^{1/2}\leq \gamma\leq h^{1/3}t}|\mathcal{P}_{S,h,a,\gamma}(t,x,y)|+\sum_{\varepsilon_0>\gamma\geq h^{1/3}t}|\mathcal{P}_{S,h,a,\gamma}(t,x,y)|\\
\lesssim & \frac{1}{h^d}\Big(\frac ht \Big)^{\frac{d-1}2}\Big(\sum_{(ht)^{1/2}\leq \gamma={2^j} a \leq h^{1/3}t}\Big({\frac{ht}{\gamma}}\Big)^{1/2}+\sum_{\varepsilon_0>\gamma ={2^j} a \geq h^{1/3}t}h^{1/3}\Big)\\
\lesssim &\frac{1}{h^d}\Big(\frac ht \Big)^{\frac{d-1}2}\Big((ht)^{1/2-1/4}+h^{1/3}\log(\varepsilon_0/a)\Big) \sim \frac{1}{h^d}\Big(\frac ht \Big)^{\frac{d-1}2}(ht)^{1/4}\,.
\end{align}
In the last line we have used the fact that in the regime $(ht)^{1/2}\geq a\geq h^{2/3-\varepsilon}$ we have 
\[
(ht)^{1/4}\geq h^{1/3-\varepsilon/2}\gg h^{1/3}\log(1/h)\gtrsim h^{1/3}\log(\varepsilon_0/a).
\]

\subsubsection{The sum for $a\leq \gamma\leq (ht)^{1/2}$}

Applying \eqref{eq:AiryPoisson} in the formula \eqref{defProndgamma} of $\mathcal{P}_{S,h,a,\gamma}$ yields
\[
\sum_{a\leq\gamma< (ht)^{1/2}}\mathcal{P}_{S,h,a,\gamma}(t,x,y)=2\pi\sum_{\lambda\lesssim k\lesssim \lambda_{\gamma},\\a\leq \gamma={2^j}a\leq (ht)^{1/2}}\frac{1}{L'(\omega_k)}K_{S,\omega_k,\gamma}(g_{h,a})(t,x,y),
\]
where we can drop the localisation ${\cutoffchi^{\flat}}(h^{2/3}\omega_k/\ceps_0)$ as this cutoff is identically equal to $1$ for $\omega_k\lesssim \lambda^{2/3}_{\gamma}$ as here $\gamma\leq (ht)^{1/2}$ which implies $h^{2/3}\omega_k\lesssim h^{2/3}\lambda^{2/3}_{\gamma}=\gamma\leq \gamma_{max} := (ht)^{1/2}\ll \ceps_0$. We could in fact introduce a cutoff ${\cutoffchi^{\flat}}(\omega_k/\omega_{\lambda_{\gamma_{max}}})$ in the previous sum to keep better track of the size of the indices. We use the expression of $K_{S,\omega}(g_{h,a})$ provided in \eqref{eq:Kequiv} with $\hat{f}$ replaced by $\hat{g}_{h,a}$ given in \eqref{ghaF}, which allows us to write $K_{S,\omega_k}(g_{h,a})$ as follows 
\begin{multline}\label{Kghak2}
K_{S,\omega_k}(g_{h,a})(t,x,y)=\frac{h^{2/3}}{2\pi h^{d+1}}{\int e^{\frac ih t\tau^2(h^{2/3}\omega_k,\theta)}}q(\theta/h)^{1/6}G(x,y,\theta/h,\omega_k)\chi^{\#}(\omega_k)\varkappa(\theta)\varkappa(\tau(h^{2/3}\omega_k,\theta))\\
\times q(\theta)^{-1/6}\int {e^{-\frac ih\Phi(a,0,\theta,h^{2/3}\omega_k,s)}}\chi(s) \tilde q_h(\theta,\alpha,s) ds d\theta,
\end{multline}
{By Lemma \ref{lemp}, the asymptotic expansion of
$\tilde q_h(\theta,h^{2/3}\omega_k,s)$ has main contribution
\[
 q_h(a,0,\theta,h^{2/3}\omega_k,s)
 \chi_0(\partial_x\Phi(a,0,\theta,h^{2/3}\omega_k,s))
 \chi(\nabla_y\Phi(a,0,\theta,h^{2/3}\omega_k,s)),
\]
where}
\begin{equation}\label{symbKgha}
q(\theta/h)^{1/6}q^{-1/6}(\theta) q_h(a,0,\theta,h^{2/3}\omega_k,s)=h^{-1/3}\varkappa(\theta)\varkappa(\tau(h^{2/3}\omega_k,\theta))\chi^{\#}(\omega_k) h^{1/3}(\sum_{m\geq 0} h^mq_{h,m}),
\end{equation}
where $\chi^{\#}(\omega_k)=1$ for all $k\geq 1$ (as $\omega_k\geq \omega_1>2$). We now use \cite[Thm. 7.7. 18]{Hormander}, based on Malgrange preparation theorem \cite[Thm.7.5.13]{Hormander}, to rewrite the integrals with respect to $s$ and $\sigma$ as follows:
\begin{thm}\label{thmmalg}(\cite[Thm.7.5.13]{Hormander})
Let $f(T,z)$ be a real valued $C^{\infty}$ function near $0$ in $\R^{1+n}$ such that $\partial_T f=\partial^2_T f=0$ but $\partial^3_T f\neq 0$ at $T=0$. Then {there exist} $C^{\infty}$ real valued functions $a_{1,2}(z)$ near $0$ such that $a_1(0)=0$, $a_2(0)=f(0,0)$ and 
\begin{equation}\label{eqmalg}
\int q(T,z)e^{\frac ih f(T,z)}dT \sim e^{\frac ih a_2(z)}\Big(h^{1/3}Ai(a_1(z)/h^{2/3})\sum_{m\geq 0} q^{(1)}_{m}(z)h^{m}+i h^{2/3}Ai'(a_1(z)/h^{2/3})\sum_{m\geq 0} q^{(2)}_{m}(z)h^{m}\Big),
\end{equation}
provided that $q\in C^{\infty}_0$ and supp$(q)$ is sufficiently close to $0$. Here $q^{(1,2)}_m\in C^{\infty}$. 
\end{thm}
Using Theorem \ref{thmMelrose}, there exists a smooth change of coordinates $s\rightarrow T$, $\frac{ds}{dT}\sim 1$, which transforms the phase function $\Phi(a,0,\theta,h^{2/3}\omega_k,s)$ into $f(T,z)=T^3/3-T\zeta(z)+\psi(z)$, where $\zeta, \psi$ are introduced in Theorem \ref{thmMelrose}. Here $z=(a,0,\theta,h^{2/3}\omega_k)$ and $T\sim s$ remains close to $0$ on the support of $\chi(s)$. After this change of coordinates the symbol \eqref{symbKgha} remains elliptic, supported where $T$ {is near} $0$. Recall that $\zeta(x,y,\theta,\alpha)=\alpha-x|\theta|^{2/3} e_0(x,y,\theta/|\theta|,\alpha/|\theta|^{2/3})$. {Homogeneity of degree $2/3$ gives
$\zeta(x,y,\theta,h^{2/3}\omega_k)/h^{2/3}
 =\zeta(x,y,\theta/h,\omega_k)$.}

Taking $a_1(z):=-\zeta(z)$, $a_2(z)=\psi(z)$ transforms the integral in $s$ of $K_{S,\omega_k}(g_{h,a})(t,x,y)$ into an integral involving {a linear combination} of an Airy function and its derivative as in \eqref{eqmalg}, with symbols $q_{0,1}(z,h)$ obtained from $\tilde q_h$ after the change of variables $s\rightarrow T$ and after applying Theorem \ref{thmmalg}, with $q_{0}(z,h)$ elliptic and with 
{$a_1(z)/h^{2/3}=-\zeta(a,0,\theta/h,\omega_k)$, where
$-\zeta(a,0,\theta/h,\omega_k)=\allowbreak
 a(|\theta|/h)^{2/3}e_0(a,0,\theta/|\theta|,h^{2/3}\omega_k/|\theta|^{2/3})-\omega_k$.}
Applying twice Theorem \ref{thmmalg} and using \eqref{eq:defG} yields $K_{S,\omega_k}(x,y,t)$ of the form 
\begin{multline}\label{Kghak}
K_{S,\omega_k}(x,y,t)=\frac{h^{1/3}}{2\pi h^{d}}\int e^{\frac ih (t\tau^2(h^{2/3}\omega_k,\theta)+\psi(x,y,\theta,h^{2/3}\omega_k)-\psi(a,0,\theta,h^{2/3}\omega_k))} \\
\times \Big(p_0Ai(-\zeta(x,y,\theta/h,\omega_k))+ih^{1/3}xp_1q^{-1/6}(\theta)Ai'(-\zeta(x,y,\theta/h,\omega_k))\Big)\varkappa(\theta)\varkappa(\tau(h^{2/3}\omega_k,\theta))\\
\times  \Big(q_{0}(z,h)Ai(-\zeta(a,0,\theta/h,\omega_k))+i h^{1/3}q_{1}(z,h)Ai'(-\zeta(a,0,\theta/h,\omega_k))\Big) d\theta \Big|_{z=(a,0,\theta, h^{2/3}\omega_k)}+O(h^{\infty}).
\end{multline}
In what follows, we use \eqref{Kghak} in order to estimate the sum of $\frac{1}{L'(\omega_k)} K_{S,\omega_k}(x,y,t)$ for $\gamma\leq (ht)^{1/2}$.  
\begin{prop}\label{propdispsumk} For $t\geq h^{1/3-2\varepsilon}$, the following dispersive estimates hold
\begin{equation}\label{estimdispK}
\Big|\sum_{\lambda\lesssim k\lesssim \lambda_{\gamma},\\a\leq \gamma={2^j}a\leq (ht)^{1/2}}\frac{1}{L'(\omega_k)}K_{S,\omega_k,\gamma}(g_{h,a})(t,x,y)\Big|\lesssim \frac{1}{h^d}\Big(\frac ht\Big)^{(d-1)/2} (ht)^{1/4},
\end{equation}
where the condition $t\geq h^{1/3-2\varepsilon}$ ensures that we indeed have $(ht)^{1/2}\geq \gamma\geq a\geq h^{2/3-\varepsilon}$.
\end{prop} 
\begin{proof}
The idea of the proof is the following : we first apply the standard stationary phase theorem with respect to $\theta$ in \eqref{Kghak} while keeping the Airy factors (of the second and third line of \eqref{Kghak}) as part of the symbol. This will provide a factor $(h/t)^{(d-1)/2}$ ; the remaining sum in $k$ will only involve Airy factors and their derivatives and the corresponding sums will be evaluated using well known properties of Airy functions.

{The equivalence of \eqref{eq:defG} and \eqref{eq:Gosc} follows from the Airy normal form \cite{ChFrUr}, as detailed in Lemma~\ref{lemphaseG} of Appendix~\ref{app:phase-equivalence}. There the normal phase is written as $s^3/3-s\zeta+\Upsilon$, which defines the smooth function $\Upsilon(x,y,\theta,\alpha)$. Corollary~\ref{corUpsi} gives}
\[
 \psi(x,y,\theta,\alpha)
 =y\cdot\theta+\tau(\alpha,\theta)
 B_\Gamma(y,\theta/\tau)
 +\Upsilon(x,y,\theta,\alpha),
\]
where
\begin{equation}\label{eq:Ups1}
 \Upsilon
 =
 x^2\tau\,{\mathcal F}
 -\frac{x^2\mu_\Gamma(y,\theta/|\theta|)q(\theta)}
        {\tau}
  (1+\Lp(y,\theta/|\theta|))
 +x\mathcal H_{\ge4}.
\end{equation}
{With our normalization we take ${\mathcal F}=0$ in this expansion; higher-order contributions are included in $x\mathcal H_{\ge4}$.}
Here the weight-five contribution to the mean critical value vanishes,
and the first undetermined contribution has weighted degree at least
six.
In the same way, the phase function in the second line of \eqref{Kghak2} is the phase of $G(a,0,\theta/h,\omega_k)$ and the change of variables $s\rightarrow T$ above allowed to transform the phase of $G$ into $T^3/3-T\zeta(a,0,\theta,h^{2/3}\omega_k)+\psi(a,0,\theta,h^{2/3}\omega_k)$. As such, after bringing all the Airy factors (the second and third lines of \eqref{Kghak}) into the symbol, the remaining phase of $K_{S,\omega_k}(g_{h,a})(x,y,t)$ reads as $t\tau^2(h^{2/3}\omega_k,\theta)+\psi(x,y,\theta,h^{2/3}\omega_k)-\psi(a,0,\theta,h^{2/3}\omega_k)$, with
\begin{equation}\label{phasestat}
\psi(x,y,\theta,h^{2/3}\omega_k)=<y,\theta>+\tau(h^{2/3}\omega_k,\theta)B_{\Gamma}(y,\theta/\tau)+\Upsilon(x,y,\theta,h^{2/3}\omega_k).
\end{equation}
\begin{lemma}\label{lemphasestatksmall1}
Let $t\geq h^{1/3-3\varepsilon}$. The stationary phase theorem with respect to $\theta$ applies in \eqref{Kghak} with large parameter $t/h$, 
with phase function $t\tau^2(h^{2/3}\omega_k,\theta)+\psi(x,y,\theta,h^{2/3}\omega_k)-\psi(a,0,\theta,h^{2/3}\omega_k)$ and symbol
\begin{multline}\label{symbstat1}
\sigma_{h,k}(x,y,a,\theta,h^{2/3}\omega_k):=\Big(p_0Ai(-\zeta(x,y,\theta/h,\omega_k))+ih^{1/3}xp_1q(\theta)^{-1/6}Ai'(-\zeta(x,y,\theta/h,\omega_k))\Big)\\\times \varkappa(\theta)\varkappa(\tau(h^{2/3}\omega_k,\theta))   
\Big(q_0Ai(-\zeta(a,0,\theta/h,\omega_k))+ih^{1/3}q_1Ai'(-\zeta(a,0,\theta/h,\omega_k))\Big).
\end{multline}
Here $p_0, p_1$ are as in \eqref{eq:defG} (and the factor $h^{1/3}$ comes from $q^{-1/6}(\theta/h)=h^{1/3}q^{-1/6}(\theta)$) and $q_0,q_1$ are the smooth symbols obtained by applying Theorem \ref{thmmalg} in the second line of \eqref{Kghak2}.
\end{lemma}
\begin{proof}(of Lemma \ref{lemphasestatksmall1})
The critical points of the phase in the first line of \eqref{Kghak} satisfy
\begin{multline}\label{eqcritthetahkcrit}
\nabla_{\theta}\Big(t(|\theta|^2+h^{2/3}\omega_kq^{2/3}(\theta))+ <y,\theta>+\tau(h^{2/3}\omega_k,\theta)B_{\Gamma}(y,\theta/\tau)+\Upsilon(x,y,\theta,h^{2/3}\omega_k)\\
-\tau(h^{2/3}\omega_k,\theta)B_{\Gamma}(0,\theta/\tau)-\Upsilon(a,0,\theta,h^{2/3}\omega_k)\Big)
=0.
\end{multline}
As we have $\omega_k\lesssim \lambda^{2/3}_{\gamma}$ and $\gamma\leq (ht)^{1/2}$, we already saw that $h^{2/3}\omega_k\leq \gamma_{max}=(ht)^{1/2}$.

We now check directly that the corrected expansion
\eqref{eq:Ups1} gives exactly the size needed for stationary phase.
Set $\delta_k=h^{2/3}\omega_k$. On the present support,
\(
 x\lesssim\gamma\), \(
 \delta_k\lesssim\gamma\), \(
 \gamma\le (ht)^{1/2}\), \( |\theta|\mathrel{\sim}1\). The coefficients $\tau,{\mathcal F},\mu_\Gamma,q,\Lp$ are smooth symbols of
$\theta$ on this fixed conic set.  The two explicit terms in
\eqref{eq:Ups1} therefore satisfy, for every fixed multi-index
$\beta$,
\begin{equation}\label{eq:Schr-Upsilon-explicit-theta}
 \left|
 \partial_\theta^\beta
 \left(
 x^2\tau{\mathcal F}
 -\frac{x^2\mu_\Gamma q}{\tau}(1+\Lp)
 \right)\right|
 \le C_\beta x^2.
\end{equation}

For the remainder we use the weighted information proved in the
Appendix.  Since the weight-five contribution to the mean critical
value cancels, the non-model remainder begins at weight six and is
divisible by $x$.  With $\deg x=\deg\alpha=2$, and
$\alpha=\delta_k$ here, its monomials are bounded by $x^3$,
$x^2\delta_k$ and $x\delta_k^2$, up to smooth symbol factors.
Differentiating in $\theta$ does not
change these powers: $|\theta|\mathrel{\sim}1$ and
$\partial_\theta^\beta\delta_k=0$ when $\delta_k=h^{2/3}\omega_k$ is
viewed as the fixed spectral parameter, while all occurrences through
$\tau$ and the homogeneous coefficients remain smooth symbols.
Hence
\begin{equation}\label{eq:Schr-Upsilon-rem-theta}
 |\partial_\theta^\beta(x\mathcal H_{\ge4})|
 \le C_\beta
 (x^3+x^2\delta_k+x\delta_k^2).
\end{equation}
Because $x,\delta_k\lesssim\gamma\ll1$, the right hand side is
$O_\beta(\gamma^3)$ and in particular $O_\beta(\gamma^2)$.
Combining \eqref{eq:Schr-Upsilon-explicit-theta} and
\eqref{eq:Schr-Upsilon-rem-theta} gives the uniform estimate
\begin{equation}\label{eq:Schr-Upsilon-theta-final}
 |\partial_\theta^\beta\Upsilon|
 \le C_\beta\gamma^2
 \le C_\beta ht
 \le C_\beta h,
 \qquad |\beta|\ge1,
\end{equation}
where the last inequality uses the fixed local time interval
$t\le T_0$.  The same estimate holds at the source point $x=a$ since
$a\lesssim\gamma$.
{In particular, both $\nabla_\theta\Upsilon$ and
$\nabla_\theta^2\Upsilon$ are $O(h)$, as required below.}
Using the expression of $B_{\Gamma}$ (and the fact that $B_{0,2}(y,\cdot)=O(|y|^2)$, $B_{2j}(y,\cdot)=O(|y|)$), we also have {$\nabla_\theta(\tau B_\Gamma(y,\theta/\tau))
=\nabla_\theta B_0(y,\theta)+O(h^{2/3}\omega_k|y|)
=O(|y|^2,h^{2/3}\omega_k|y|)$}. From \eqref{eqcritthetahkcrit} we obtain {a unique critical point} ${\theta_c=-(y+O(|y|^2;h^{2/3}\omega_k;(h^{2/3}\omega_k)^2))/(2t)}$. The matrix of second derivatives is of the form ${2t}\times \Big(Id_{d-1}+O(h^{2/3}\omega_k)+\frac{1}{t}O(|y|^2,h^{2/3}\omega_k|y|,(h^{2/3}\omega_k)^2)\Big)$ and as $|y|\lesssim t\leq T_0$ with $T_0$ small enough, $h^{2/3}\omega_k\leq (ht)^{1/2} \ll 1$ and $(h^{2/3}\omega_k)^2/t\lesssim (ht)/t=h$, it follows that its determinant behaves like $t^{d-1}$. 
In order to apply the stationary phase with symbol $\sigma_{h,k}$ it remains to check that there exists some $\nu>0$ such that $\forall j\geq 1$ and $\forall\alpha$ with $|\alpha|=j$,
\[
|\nabla^{\alpha}_{\theta}\sigma_{h,k}|\leq C_j\Big(\frac th\Big)^{j(1-2\nu)/2}.
\]
To see that, it will be enough to check that there exists some $\nu>0$ such that $\forall j\geq 1$ and $\forall\alpha$ with $|\alpha|=j$,
\[
|\nabla^{\alpha}_{\theta} Ai(-\zeta(x,y,\theta/h,\omega_k) |\leq C_j\Big(\frac th\Big)^{j(1-2\nu)/2},
\]
where we recall that $-\zeta(x,y,\theta/h,\omega_k)=\frac{x}{h^{2/3}} |\theta|^{2/3}e_0(x,y,\theta/|\theta|,h^{2/3}\omega_k/|\theta|)-\omega_k$. Recall that, if $-\zeta>0$, then $Ai(-\zeta)$ is exponentially decreasing : as $e_0$ is an elliptic, bounded symbol, this implies that, for $\gamma/h^{2/3}\gg \omega_k$, the previous inequality becomes trivial. Let $\gamma/h^{2/3}=\lambda_{\gamma}^{2/3}\gtrsim \omega_k$ with $\gamma\leq (ht)^{1/2}$. As one has, for all $l\geq 0$, $\sup_{b\geq 0}\Big|b^l Ai^{(l)}(b-\omega_k)\Big|\leq C_l \omega_k^{3l/2}$,
it is sufficient to check that for $t\geq h^{1/3-2\epsilon}$ and $\omega_k\lesssim (ht)^{1/2}/h^{2/3}$ the following holds
\begin{equation}\label{toprove1}
\omega_k^{3/2}\lesssim \Big(\frac th\Big)^{(1-2\nu)/2}.
\end{equation}
As for $\omega_k\lesssim (ht)^{1/2}/h^{2/3}$ we have $\omega_k^{3/2}\lesssim (ht)^{3/4}/h=t^{3/4}/h^{1/4}\ll (t/h)^{1/2}$, \eqref{toprove1} holds for all $0<\nu<1/4$ and the stationary phase theorem applies and yields a factor $(h/t)^{\frac{d-1}{2}}$. The condition on $t$ is only used to ensure that $(ht)^{1/2}\gtrsim \gamma\geq a\geq h^{2/3-\varepsilon}$.
\end{proof}
We now complete the proof of Proposition \ref{propdispsumk}. After applying stationary phase in $\theta$, the symbol becomes an asymptotic expansion with main contribution $\sigma_{h,k}$ computed at $\theta_c$ and small parameter $(h/t)$. Moreover, as the new symbol is obtained by taking derivatives on $\sigma_{h,k}$, and, since $Ai^{''}(z)=zAi(z)$, all the contributions in the asymptotic expansions come with factors $Ai^{(')}(-\zeta(x,y,\theta/h,\omega_k) )$ or $Ai^{(')}(-\zeta(a,0,\theta/h,\omega_k))$. We need the following lemma.
\begin{lemma}( \cite[Lemma 3.5]{ILLP})
There exists $C_0$ such that for $L\geq 1$,
\begin{equation}\label{estairy21}
\sup_{b\in\R}\sum_{1\leq k\leq L}
\frac{Ai^2(b-\omega_k)}{L'(\omega_k)} \leq C_0L^{1/3}\,, \quad\quad\sup_{b\in\R_+}\sum_{1\leq k\leq L}
\frac{Ai'^2(b-\omega_k)}{L'(\omega_k)} \leq C_0L\,.
\end{equation}
\end{lemma}
We first bound the main contribution of \eqref{Kghak} after stationary phase, by applying the {Cauchy--Schwarz} inequality to separate the Airy terms in $\sigma_{h,k}$ {depending on} $(x,y)$ and those depending on $(a,0)$. We get
\[
\frac{h^{1/3}}{h^d}\Big(\frac ht\Big)^{(d-1)/2} 
\Big(\sum_{k\leq \lambda_{\gamma}, \gamma\leq (ht)^{1/2}} \frac{1}{L'(\omega_k)}{Ai^2(-\zeta(x,y,\theta/h,\omega_k))}\Big)^{1/2} \Big(\sum_{k\leq \lambda_{\gamma}, \gamma\leq (ht)^{1/2}} \frac{1}{L'(\omega_k)}{Ai^2(-\zeta(a,0,\theta/h,\omega_k))}\Big)^{1/2}\,,
\]where $-\zeta(x,y,\theta/h, \omega_k)=x(|\theta|/h)^{2/3}e_0(x,y,\theta/|\theta|,h^{2/3}\omega_k/|\theta|^{2/3})-\omega_k$ and where $h^{-(d-1)}(h/t)^{(d-1)/2}$ comes from the change of variables followed by the stationary phase in $\theta$.
Using \eqref{estairy21} with $L:= \lambda_{\gamma_{\max}}=((ht)^{1/2})^{3/2}/h$, the product of the two factors in brackets is bounded by $L^{1/3}=(ht)^{1/4}/h^{1/3}$, hence $h^{1/3}L^{1/3}\leq (ht)^{1/4}$ and therefore the main contribution (involving Airy terms only) satisfies \eqref{estimdispK}. Proceeding in the same way with the terms in $\sigma_{h,k}$ involving $Ai'$, which come with an additional factor $h^{1/3}$ allows us to bound the sum of the main contributions (from $\sigma_{h,k}$) of $\frac{2\pi}{L'(\omega_k)}K_{S,\omega_k}(g_{h,a})(t,x,y)$ by
\[
\frac{C}{h^{d}}\Big(\frac ht\Big)^{\frac{d-1}{2}}h^{1/3}(L^{1/3}+L^{1/6}\times h^{1/3}L^{1/2}+h^{2/3}L)=\frac{C}{h^{d}}\Big(\frac ht\Big)^{\frac{d-1}{2}}h^{1/3}\Big(L^{1/3}+h^{1/3}L^{2/3}+h^{2/3}L\Big),
\]
where we have used \eqref{estairy21}. As $h^{1/3}L^{1/3}=(ht)^{1/4}\gg (hL)^{2/3}\gg (hL)$, we obtain \eqref{estimdispK}. As all the lower order terms in the asymptotic expansion of the symbol obtained after the stationary phase involve Airy factors and their derivatives (which can be seen using the equation satisfied by the Airy function), they will all satisfy similar estimates; as they come with factors that are powers of $(h/t)$, we can sum all these contributions and end up with an estimate of the form \eqref{estimdispK}.
\end{proof}
\subsection{Case $a>h^{2/3-\varepsilon}$ and $t<h^{1/3-2\varepsilon}$}\label{seca<t<}
Here $(ht)^{1/2}<h^{2/3-\varepsilon}\le a$.
For the dyadic pieces with $t\le c_*\sqrt\gamma$, use
\eqref{eq:first-reflection-short-sum}, together with the unreflected
bound in Proposition \ref{propfreeN=0}. A first reflected trajectory
may reach the boundary at time $a/\sqrt\gamma$, so it must be retained
even when $t\ll\sqrt a$.
For the remaining dyadic pieces, $\gamma\lesssim t^2<h^{2/3-4\varepsilon}$
and the preceding reflection estimates apply.
After factoring out $h^{-d}(h/t)^{(d-1)/2}$, the large-reflection
sum is bounded by $(ht/a)^{1/2}\le h^{1/3-\varepsilon/2}$.
For the moderate-reflection sum, $a\lesssim t^2$ gives
$(ha/t)^{1/4}\lesssim(ht)^{1/4}\le h^{1/3-\varepsilon/2}$.
The transverse $h^{1/3}$ terms contribute at most
$O(h^{1/3}(1+|\log h|))$. All these factors are bounded by
$C h^{1/3-2\varepsilon}$, proving the stated transition bound.

\section{Dispersion for small $a\lesssim h^{2/3-\ceps}$}
\label{anciennelast}
We now obtain dispersion for $\mathcal{P}_{S,h,a}$ for small $a\in (0,h^{2/3-\ceps}]$ with a small $\ceps>0$. The only difference with respect to the previous situations when $a>h^{2/3-\varepsilon}$ is that we cannot obtain $g_{h,a}$ anymore using Lemma \ref{lemp} but as a sum of two contributions, $g_{h,a}=g_{h,a,1}+g_{h,a,2}$ with $g_{h,a,j}$ defined in Propositions \ref{propdataapetitpetit1}, \ref{propdataapetitpetit2}, \ref{propparama<}. Write
$
\Prond_{S,h,a}(t,x,y):=\Prond^{1}_{S,h,a}+\Prond^{2}_{S,h,a}\,,
$
where $\Prond^{j}_{h,a}$ is given by \eqref{eq:Prond2cut} with $g_{h,a}$ replaced by $g_{h,a,j}$, $j\in\{1,2\}$. We treat $\Prond^1_{S,h,a}$ through the gallery-mode sum.
For the reflected contribution $\Prond^2_{S,h,a}$,
the cutoff in \eqref{eq:g1} gives
$\gamma\gtrsim h^{2/3-2\ceps}$.
The rescaled parameter $\gamma^{3/2}/h$ is therefore large
even in the present small-$a$ regime.
\begin{prop}\label{prop:dispersion-small-a}
Let $0<a\le h^{2/3-\ceps}$, $0<\ceps<1/24$ and $h<t\le T_0\le1$.
For the first contribution,
\begin{equation}\label{P1disp}
\|\Prond^{1}_{S,h,a}(t,.)\|_{L^{\infty}_{x>0,y}}\lesssim h^{-d}\Big(\frac ht\Big)^{\frac{d-1}{2}}((ht)^{1/4}+h^{1/3-2\ceps}).
\end{equation}
For the second contribution on $0\le x\le a$, the same bound holds
when $t\ge h^{1/3-2\ceps}$. For all $h<t\le T_0$ in this region,
\begin{equation}\label{eq:small-a-second-all-times}
 |\Prond^2_{S,h,a}(t,x,y)|
 \lesssim h^{-d}(h/t)^{(d-1)/2}
 \bigl((ht)^{1/4}+(h/t)^{1/2}\bigr),\qquad h<t\le T_0.
\end{equation}
\end{prop}
\begin{proof}

In what follows, we evaluate the $L^{\infty}$ norm of $\Prond^{1}_{S,h,a}(t,.)$ : recall first from section \ref{sectproofs}) that 
\[
\Prond^{1}_{S,h,a}(t,.)=\sum_{k\geq 1}\frac{2\pi}{L'(\omega_k)}\cutoffchi^{\flat}(\omega_k/\omega_{K_\epsilon})K_{S,\omega_k}(g_{h,a,1})(t,\cdot),
\]
where $\cutoffchi^{\flat}$ is compactly supported in a fixed neighborhood of $0$ and $K_{\epsilon}\leq h^{-1/4+\epsilon}$ for some $\epsilon>0$ small.

We need the following lemma, relying on Corollary \ref{lemgomegak} and the notations therein :
\begin{lemma}\label{lemformgapetit}(see \cite[Lemma 4.2]{ILLP})
{There exist symbols} $\sigma(\eta,\omega_k)$ 
and $r_j(a,\eta,\omega_k)$, of degree $0$, $\sigma, r_0$ elliptic such that 
 \begin{multline}\label{K<} 
K_{S,\omega_k}(g_{h,a,1})(t,x,y)=\int e^{iht\tau^2(\omega_k,\eta)}G(x,y,\eta,\omega_k)q^{\frac 16}(\eta)\varkappa(h\eta)\varkappa(h\tau(\omega_k,\eta))r(\eta,\omega_k)\sigma(\eta,\omega_k)\\
 \cutoffchi^{\flat}\Big(\frac{\omega_k}{\omega_{K_{\ceps}}}\Big) q^{\frac 16}(\eta)\Big(r_0Ai(-\zeta(a,0,\eta, \omega_k))+i q^{-\frac 16}(\eta)r_1Ai'(-\zeta(a,0,\eta,\omega_k))\Big)d\eta\,,
\end{multline}
where $\zeta$ is the phase introduced in Theorem \ref{thmMelrose}.
\end{lemma}
This lemma is the Schr\"odinger analogue of \cite[Lemma 4.2]{ILLP},
which was proved there for the wave equation using the properties of
$e(\cdot,\omega_k)$ detailed in \cite[Section 2.3.2]{ILLP}. The only
difference between the two settings lies in the time phase: in the wave
case the corresponding factor is
$e^{it\tau(\omega_k,\eta)}$, whereas here it is
$e^{iht\tau^2(\omega_k,\eta)}$. Since this factor is carried unchanged
throughout the computation of $K_{S,\omega_k}(g_{h,a,1})$, this
difference does not affect the argument. The proof therefore follows
exactly as in \cite[Lemma 4.2]{ILLP}. We nevertheless provide it here
for completeness.

\begin{proof} (of Lemma \ref{lemformgapetit}) 
Replacing $\hat{g}_{h,a,1}$ given by Corollary \ref{lemgomegak} in formula \eqref{eq:Kequiv} yields
\[
K_{S,\omega_k}(g_{h,a,1})=\int e^{iht\tau^2_{q}(\omega_k,\eta)}G(x,y,\eta,\omega_k)\chi^{\#}(\omega_k)q^{\frac 16}(\eta)\varkappa(h\eta)\varkappa(h\tau_{q}(\omega_k,\eta))r(\eta,\omega_k)\frac{\sqrt{L'(\omega_k)}}{\sqrt{2\pi}}I_a(\eta,\omega_k)d\eta\,,
\]
with $I_a(\eta,\omega_k)$ defined in \eqref{Io}. To compute $I_a$ use \eqref{defEmathcalM} and $\overline{e( x, y,\eta,\omega_k)}=e^{-i|\eta|B_0}\overline{\tilde e( x, y,\eta,\omega_k)}$
\begin{equation}\label{I<}
 I_a  = \int e^{-i(\tilde y\cdot (\eta-\tilde \eta)+|\eta| B_{0}(\tilde y,\frac{\eta}{|\eta|}))}\varkappa(h\tilde\eta)\cutoffchi^{\flat}\Big(\frac{\omega_k}{\omega_{K_{\ceps}}}\Big) e_k(a,\tilde \eta)\left[\int_{\R_{+}} \overline{\tilde e( \tilde x, \tilde y,\eta,\omega_k)}e_k(\tilde x,\tilde \eta) d\tilde x\right] 
d\tilde \eta d\tilde y\,.
\end{equation}
As in the proof of Lemma \ref{lemestimderivekk}, the bracket term in \eqref{I<} behaves like a symbol. For small $a\lesssim h^{2/3-\ceps}$ and $k\leq K_{\ceps}$, $e_k(a,\tilde\eta)$ can also be included in the symbol, and
 with $\eta=\theta/h$ and $\tilde \eta=\tilde\theta/h$, stationary phase applies in $\tilde y,\tilde \theta$ for the phase $-(\tilde y\cdot\theta+|\theta|B_{0}(\tilde y,\vartheta))+\tilde y\cdot \tilde \theta$, with large parameter $\frac 1h$ and symbol 
\begin{equation}\label{symbK<}
\frac{1}{h^{d-1}}\cutoffchi^{\flat}\Big(\frac{\omega_k}{\omega_{K_{\ceps}}}\Big)\varkappa(\tilde\theta)\chi_0(h^{2/3}\omega_kq^{2/3}(\tilde\theta))e_k(a,\tilde\theta/h)
\Big(\int_0^{\infty}  \overline{\tilde e( \tilde x, \tilde y,\theta/h,\omega_k)}e_k(\tilde x,\tilde \theta/h) d\tilde x\Big)\,.
\end{equation}
Stationary points are such that $\tilde y=0$, $\tilde \theta=\theta+|\theta|\partial_{\tilde y} B_{0}(\tilde y,\vartheta)=\theta$ (as $\partial_y B_0(0,\vartheta)=0$). All derivatives with respect to $\tilde y$ land on $ \overline{\tilde e( \tilde x, \tilde y,\theta/h,\omega_k)}$ and we can use Lemma \ref{lemestimderivekk}.
Derivatives with respect to $\tilde \theta$ land either on cutoffs, $e_k(\tilde x,\tilde \theta/h)$ or $e_k(a,\tilde \theta/h)$ : using \eqref{estimderivekk} for $e_k$,
\[
\Big|\partial^{\beta_1}_{\tilde y}\partial^{\beta_2}_{\tilde \eta}\Big(\int_0^{\infty}  \overline{\tilde e( \tilde x, \tilde y,\eta,\omega_k)}e_k(\tilde x,\tilde \eta) d\tilde x\Big)\Big|\leq \|\partial^{\beta_1}_{\tilde y}\tilde e(.,\omega_k) \|_{L^{2}_{x\geq 0}} \|\partial^{\beta_2}_{\tilde \eta} e_k \|_{L^{2}_{x\geq 0}}\lesssim (\omega_k|\eta|^{1/3})^{|\beta_1|+|\beta_2|}\,.
\]
As $|\eta|\sim 1/h$ and $k\lesssim K_{\ceps}< h^{-1/4+\ceps}$ we find $\omega_k|\eta|^{1/3} \leq h^{-1/2+2\ceps/3}$. As one derivative on $e_k(a,\tilde\theta/h)$ yields at most $\frac ah\sqrt{h^{2/3}\omega_k}$ and we have $\frac ah\sqrt{h^{2/3}\omega_k}\leq h^{-1/12-2\ceps/3}\ll h^{-1/2}$ for $\ceps<5/8$, stationary phase applies in $\tilde y,\tilde\theta$. After stationary phase, $e_k(a,\tilde\theta/h)$ transforms into a linear combination of $e_k$ and its derivative : indeed, repeated derivatives in $\tilde\theta$ lead to $Ai(aq^{1/3}(\tilde\theta/h)-\omega_k)$ and $Ai'(aq^{1/3}(\tilde \theta/h)-\omega_k)$, and we denote $r_0, r_1$ the corresponding asymptotic expansions; $r_0, r_1$ are functions of $(a,\theta/h,\omega_k)$ and $r_0$ is elliptic with main contribution equal to $1$. Taking the difference between phase functions of $e_k(a,\tilde\theta/h)$ and $e(a,0,\theta/h,\omega_k)$ yields $a\tau_{q}(h^{2/3}\omega_k,\theta)A_{\Gamma}(a,0,s q^{1/3}(\theta)/\tau_{q},\theta/\tau_{q})$, which behaves (at worst) like $a(s^2-h^{2/3}\omega_kq^{2/3}(\theta)+O(s^3,s(h^{2/3}\omega_k))$ (where $s^2\lesssim h^{2/3}\omega_k$);
as $ah^{2/3}\omega_k\leq h^{7/6-\ceps/3}$, $e^{\frac ih a\tau_{q}(h^{2/3}\omega_k,\theta)A_{\Gamma}(a,0,s q^{1/3}(\theta)/\tau_{q},\theta/\tau_{q})}$ does not oscillate and can be brought into the symbol.
The $L^2_{x>0}$ product in \eqref{I<} becomes a new symbol $\sigma(\theta/h,\omega_k)$, with main contribution $\sigma_0$
\begin{multline*}
\sigma_{0}=\int_0^{\infty}  \overline{\tilde e( \tilde x, 0,\theta/h,\omega_k)}e_k(\tilde x, \theta/h) d\tilde x=h^{-2/3}\int_0^{\infty}e^{\frac ih (-\tilde s^3/3-\tilde s(\tilde xq^{1/3}(\theta)-h^{2/3}\omega_k)+s^3/3+s(\tilde x q^{1/3}(\theta)-h^{2/3}\omega_k))}\\
\times \frac{2\pi}{L'(\omega_k)}  \frac{q^{1/3}(\theta)}{h^{2/3}}   \tilde p_h(\tilde x,\theta, h^{2/3}\omega_k,\tilde s, h/t) e^{-\frac ih \tilde x \tau_{q}(h^{2/3}\omega_k,\theta)A_{\Gamma}(\tilde x,0,\sigma q^{1/3}(\theta)/\tau_{q},\theta/\tau_{q})} d\tilde s ds d \tilde x,
\end{multline*}
for a symbol $\tilde p_h$ obtained from $p_h$ in \eqref{eq:qh}. For $\tilde x\gtrsim h^{2/3}\omega_k$, repeated integrations by parts in $s$ yield an $O(h^{\infty})$ contribution. For $\tilde x\lesssim h^{2/3}\omega_k$, apply stationary phase in $s,\tilde x$: as $\tilde xq^{1/3}(\theta)=h^{2/3}\omega_k-s^2$ and $s q^{1/3}(\theta)=\tilde s q^{1/3}(\theta)+\tau_{q} (A_{\Gamma}+\tilde x\partial_{\tilde x}A_{\Gamma})$, it remains to consider an integral in $\tilde s$ with phase 
\[
e^{\frac ih (h^{2/3}\omega_k-\tilde s^2)q^{1/3}(-\theta)\tau_{q}(h^{2/3}\omega_k,\theta)A_{\Gamma}((h^{2/3}\omega_k-\tilde s^2)q^{-1/3}(\theta),0,\tilde s q^{1/3}(\theta)/\tau_{q},\theta/\tau_{q})}.
\]
One derivative with respect to $\theta$ on this phase yields at most $\frac 1h (h^{2/3}\omega_k)^2\lesssim  h^{4\ceps/3}$ for $k\leq K_{\ceps}\leq h^{-1/4+\epsilon}$ and therefore $\sigma_0$ is of order $0$. Rewriting $K_{S,\omega_{k}}(g_{h,a,1})$ after stationary phase, relabelling $r_0$ and $r_1$, we obtain \eqref{K<}.
\end{proof}
{To complete the proof} of \eqref{P1disp} we can now proceed in the same way as sections \ref{seca<} and \ref{seca<t<}, using the form of $K_{S,\omega_k}(g_{h,a,1})$ obtained in Lemma \ref{lemformgapetit} exactly as we did with  $K_{S,\omega_k}(g_{h,a})$ written in the form \eqref{Kghak}, for both $t>h^{1/3-2\varepsilon}$ and $t\leq h^{1/3-2\varepsilon}$.

We now estimate the second contribution. The source point is
$(a,0)$ at time $0$, and the observation point is $(x,y)$ at time $t$.
We work with $0\le x\le a$ and $0<a\le h^{2/3-\varepsilon}$,
where $0<\varepsilon<1/24$.
We prove \eqref{eq:small-a-second-all-times}.
For $t\ge h^{1/3-2\varepsilon}$, the free factor is at most
$h^{1/3+\varepsilon}$, and this gives the right-hand side of
\eqref{P1disp} for $j=2$ on $0\le x\le a$.
The free term in \eqref{eq:small-a-second-all-times} is necessary:
when $a/h\to\infty$ and $t\sim h$, a ray with fixed nonzero normal
frequency can remain inside $0<x\le a$ without reaching the boundary.
Its localized kernel has size $h^{-d}$.

We first give the representation obtained from the inversion in
Proposition~\ref{propapetitg2}.
We insert dyadic cutoffs in $\alpha=h^{2/3}\omega$, with
$\alpha\sim\gamma$ on the fixed tangential frequency support. Then
\begin{equation}\label{eq:small-a-transverse-support}
 \gamma\gtrsim h^{2/3-2\varepsilon},\qquad
 \lambda_\gamma=\gamma^{3/2}/h\gtrsim h^{-3\varepsilon},\qquad
 a/\gamma\lesssim h^\varepsilon.
\end{equation}
Thus $x/\gamma,a/\gamma\le1/4$ for $h$ sufficiently small.
The equivalent localization $\alpha=\gamma q(\theta)^{1/3}A$
keeps $A$ in a fixed compact subset of $(0,\infty)$; we subdivide
and rescale the dyadic cutoffs by fixed constants if necessary.

On each piece, the functions $g_{h,a,2,\pm}$ admit the representation
\begin{equation}\label{eq:small-a-inverted-source}
 \widehat g_{h,a,2,\pm}(\theta/h,\alpha/h)
 =h^{-1}q(\theta)^{-1/6}
 \int e^{-i\Phi(a,0,\theta,\alpha,s)/h}
 b_{\gamma,\pm}(a,\theta,\alpha,s;h)\,ds
 +O(h^\infty).
\end{equation}
Here $s=\sqrt\gamma q(\theta)^{1/6}S$ and
$\alpha=\gamma q(\theta)^{1/3}A$.
The amplitudes have uniform bounds with any fixed number of
$(\theta,A,S)$ derivatives, and are supported near the source
critical point with $\pm S>0$ and $|S|\sim1$.
The bounds are uniform as $a\to0$.

To verify \eqref{eq:small-a-inverted-source}, we use the elliptic
operators with phases $\psi\pm\frac23\zeta^{3/2}$ which are inverted
in Proposition~\ref{propapetitg2}.
In \eqref{EM1}--\eqref{EM2}, the direct model term has phase
$y\cdot\eta\pm\frac23\{\zeta_M(x,\eta,\omega)^{3/2}-\zeta_M(a,\eta,\omega)^{3/2}\}$, where $\zeta_M(x,\eta,\omega)=\omega-x|\eta|^{2/3}$.
With $\eta=\theta/h$ and $\omega=\alpha/h^{2/3}$,
the direct model phase is
\[
 \frac1h\bigl(y\cdot\theta+(x-a)\xi_\pm\bigr),
 \qquad
 \xi_\pm=\mp|\theta|^{2/3}
 \int_0^1
 \sqrt{\alpha-(a+u(x-a))|\theta|^{2/3}}\,du.
\]
Since $\alpha\sim\gamma\gg x,a$, we have
$|\xi_\pm|\sim\sqrt\gamma$ and
$|\partial_\alpha\xi_\pm|\sim\gamma^{-1/2}$.
We therefore change variables from $\alpha$ to $\xi_\pm$.
The direct term becomes a pseudodifferential operator applied
to $\delta_{(a,0)}$, with amplitude bounds uniform in the
rescaled variables. The reflected term has the opposite sign
of the normal frequency and gives a non-stationary composition
with the chosen inverse.
The same is true of the unwanted general reflected term.
These estimates remain uniform up to $x=0$: the inverses are
constructed in coordinates extended across the boundary,
and the two normal frequencies have opposite signs and
remain separated by at least $c\sqrt\gamma/h$.
For the auxiliary cutoff on the strip
$x\lesssim h^{2/3-\varepsilon}$, integration by parts gains at least
$h/(h^{2/3-\varepsilon}\sqrt\gamma)\lesssim h^{2\varepsilon}$.

Applying the chosen inverse to the direct model term gives
the source phase $-\Phi(a,0,\theta,\alpha,s)$ in
\eqref{eq:small-a-inverted-source}.
We compute this composition by stationary phase in the
intermediate variables. After rescaling
$s=\sqrt\gamma q(\theta)^{1/6}S$ and the normal spatial
variables by $\gamma$, the Hessian is uniformly invertible.
The phase and amplitude derivatives are uniformly bounded
for $a/\gamma\le1/4$ and $|S|\sim1$.
Stationary phase therefore gives
\eqref{eq:small-a-inverted-source}, with expansion parameter
$\lambda_\gamma^{-1}$.
The model cutoff becomes
part of $b_{\gamma,\pm}$; its support meets only a fixed number of
neighboring $\gamma$ pieces, which we regroup.
This retains $B_\Gamma$ and $aA_\Gamma$ in the phase and does not
require their exponentials to be symbols.

Substitution in the parametrix gives the full phase
\begin{equation}\label{eq:small-a-full-phase}
 t\tau^2(\alpha,\theta)
 +\Phi(x,y,\theta,\alpha,\sigma)
 -\Phi(a,0,\theta,\alpha,s)
 -NhL(\alpha h^{-2/3}).
\end{equation}
The amplitudes have the bounds just established. In particular,
the spectral derivatives of $\tau A_\Gamma$ are retained at the
source and observation points. The source and observation terms
in the rescaled $A$ equation are
$O(x/\sqrt\gamma)+O(a/\sqrt\gamma)=O(\sqrt\gamma)$,
with the differentiated bounds from the normalized expansion.
The phase reductions of the transverse argument therefore remain
uniform without a lower bound on $a$.

We sum the quasimodes first when $\gamma\lesssim(ht)^{1/2}$.
Both Airy arguments satisfy $\zeta\sim\omega_k$, since $x,a\ll\gamma$,
so we use their two oscillatory expansions and retain their phases.
At fixed $\omega_k$, their leading terms
$\pm\frac23h\omega_k^{3/2}$ are independent of $\theta$.
The $\theta$ derivatives of the remaining normal phases are
$O((x+a)\sqrt\gamma)$, and those of the boundary phase are
$O(|y|^2)$. On the stationary region $|y|\lesssim t$ the Hessian in
$\theta$ is consequently
\[
 2t\,Id+O\bigl(t^2+t\gamma+(x+a)\sqrt\gamma\bigr).
\]
Moreover,
$a\sqrt\gamma/t\lesssim h^{1/6-\varepsilon}$ for $t\ge h$ and
$\gamma\lesssim(ht)^{1/2}$.
Thus stationary phase in all $d-1$ tangential frequencies gives
$(h/t)^{(d-1)/2}$, with uniform amplitude bounds. Outside this
region the tangential phase is non-stationary.
Each Airy factor is $O(\omega_k^{-1/4})$ and
$L'(\omega_k)\sim\omega_k^{1/2}$. With the normalization of
\eqref{eq:Prond2cut}, each term is therefore bounded by
$C h^{-d}(h/t)^{(d-1)/2}h^{1/3}\omega_k^{-1}$.
There are $O(\lambda_\gamma)$ terms with
$\omega_k\sim\gamma h^{-2/3}$, and hence
\begin{equation}\label{eq:small-a-mode-piece}
 |\Prond^2_{S,h,a,\gamma}|
 \lesssim h^{-d}(h/t)^{(d-1)/2}\sqrt\gamma.
\end{equation}
The differentiated Airy terms are included in these amplitudes:
the observation phase contributes terms of size
$x\sqrt\gamma=O(1)$ to its $\theta$ derivatives, and the
source amplitude in \eqref{eq:small-a-inverted-source} has uniform
bounds before stationary phase.
Summing \eqref{eq:small-a-mode-piece} over
$\gamma\lesssim(ht)^{1/2}$ gives the factor $(ht)^{1/4}$.

For $\gamma\gtrsim(ht)^{1/2}$ and $t\ge c_*\sqrt\gamma$, we use
\eqref{eq:small-a-full-phase} as a sum over reflections.
The $N=0$ term is non-stationary here: its $A$ equation would give
$t/\sqrt\gamma=O(a/\gamma+\sqrt\gamma)+O(T_0t/\sqrt\gamma)$,
which is incompatible with $t/\sqrt\gamma\ge c_*$ for small
$h,T_0$ and the fixed upper bound on $\gamma$.
Lemma~\ref{lem:transverse-directions} applies: its proof only uses
$x/\gamma,a/\gamma\le1/4$, a large $\lambda_\gamma$, and the uniform
phase and amplitude bounds above. Its reduced integral is
$O(\lambda_\gamma^{-1})$. Together with the prefactor in
\eqref{VNdef}, this yields
\[
 |W^2_{N,h,\gamma}|
 \lesssim h^{-d}(h/t)^{(d-1)/2}
 \frac{\gamma^2}{h\lambda_\gamma^{3/2}\sqrt{|N|}}.
\]
Since $|N|\sim t/\sqrt\gamma$, summation gives
$C h^{-d}(h/t)^{(d-1)/2}(ht/\gamma)^{1/2}$.
We keep this bound rather than replace it by $h^{1/3}$ on part of
the range. Its geometric sum over $\gamma\gtrsim(ht)^{1/2}$ is
$O((ht)^{1/4})$, after removing the common prefactor.

It remains to treat $t<c_*\sqrt\gamma$ on these latter pieces.
Only $N=0,\pm1$ can contribute. We use the short-time reduction
leading to \eqref{eq:first-reflection-free-factor}, but do not exclude $t\gamma\lesssim h$.
Stationary phase in $\Sigma,S$ is still uniform with parameter
$\lambda_\gamma$; with tangential stationary phase its prefactor is
$h^{-d}(h/t)^{(d-1)/2}\sqrt\gamma$.
The reflected $A$ equation requires
$t\gamma\sim(x+a)\sqrt\gamma$, and its second derivative is
comparable to $t\gamma$. Thus integration in $A$ gives
$\min\{1,(h/(t\gamma))^{1/2}\}$.
For the direct term we obtain the same bounds, with the difference $x-a$ between the observation and source
normal coordinates instead of $x+a$.
The remainders in these statements are relatively
$O(T_0+\sqrt\gamma)$ on the stationary region, as in the short-time
calculation. Here we use the hypothesis $\gamma\gtrsim (ht)^{1/2}$ in the final estimate to obtain $C h^{-d}(h/t)^{(d-1)/2}(ht/\gamma)^{1/2}\lesssim C h^{-d}(h/t)^{(d-1)/2}(ht)^{1/4}$.

The pieces with $t\gamma\le h$ sum by their absolute bounds:
$\sum\sqrt\gamma\lesssim\sqrt{h/t}$.
When $t\gamma>h$, a stationary point selects only a bounded number
of dyadic pieces through
$\gamma\sim((x+a)/t)^2$, or
$\gamma\sim((x-a)/t)^2$ for the direct term.
Their contribution is again $O(\sqrt{h/t})$.
Away from these pieces, integrations by parts in $A$ give powers
of $h/(t\gamma+|x\pm a|\sqrt\gamma)$ and a geometric summation.
On the part where $|x\pm a|\sqrt\gamma$ is not small relative to
$t$, these integrations give arbitrary powers of $h/t$; on the
remaining part the tangential Hessian is uniformly comparable to
$t\,Id$. This gives the same bound for the non-stationary terms.
Thus the whole short-time contribution is bounded by
$C h^{-d}(h/t)^{d/2}$, uniformly even when the source and observation points approach
the boundary. Combining the three sums proves
\eqref{eq:small-a-second-all-times}.
\end{proof}

\subsection{Completion of the global linear estimates}
\label{sec:completion-linear-proof}

\begin{proof}[Proof of the dispersive bound in Theorem~\ref{thmdisperS}]
Decrease $a_0,T_0,h_0$ so that they are at most one, and fix
$0<\varepsilon<1/24$ small enough for all the local constructions.
The constants are uniform over the fixed finite atlas of coordinate patches and finite covering of the relevant frequency directions.
We first sum the estimates in a glancing patch and then sum the
finitely many terms of the global partition.

Start with the summation within a glancing patch. For $h<t\leq T_0$, factor out
$h^{-d}(h/t)^{(d-1)/2}$. The unreflected term, already summed in
$\gamma$ in Proposition~\ref{propfreeN=0}, has remaining factor
$(h/t)^{1/2}$.
{For $a\ge h^{2/3-\varepsilon}$, first separate the dyadic pieces for which
  $t\le c_*\sqrt\gamma$. Their first reflections have the same free
bound by \eqref{eq:first-reflection-short-sum}; all higher reflection
indices are non-stationary. In the following sums we retain only
$t\ge c_*\sqrt\gamma$.}
We first consider the case where $a\geq h^{2/3-\varepsilon}$ and
$a\geq(ht)^{1/2}$. The {central cutoff obtained by grouping the dyadic pieces with $\gamma\sim a$ is}
estimated by Proposition~\ref{proptang}, whose proof uses
Propositions~\ref{dispNgrand}, \ref{dispNpetitloin}, and
\ref{dispNpetitpres}. Their remaining factors are bounded by
\[
 (ha/t)^{1/4}\leq(h/t)^{1/4},\qquad
 (ht/a)^{1/2}\leq(ht)^{1/4}\leq(h/t)^{1/4}.
\]
For $\gamma\geq4a$, Proposition~\ref{proptransv} applies.
The threshold $t/\sqrt\gamma\sim\lambda_\gamma^{1/3}$ is
$\gamma\sim h^{1/3}t$. The terms with remaining factor
$(ht/\gamma)^{1/2}$ form a geometric sum, bounded by its first
term. The terms with remaining factor $h^{1/3}$ contribute at most
$C h^{1/3}(1+|\log h|)$, since
$a\geq h^{2/3-\varepsilon}$.

If $h^{2/3-\varepsilon}\leq a<(ht)^{1/2}$, split at
$\gamma\sim(ht)^{1/2}$ as in Section~\ref{seca<}.
For the dyadic pieces with larger $\gamma$, Section~\ref{seca>(ht)12} gives
\[
 \sum_{\substack{\gamma\gtrsim(ht)^{1/2}\\
                  \gamma\lesssim h^{1/3}t}}
      (ht/\gamma)^{1/2}
       \lesssim (ht)^{1/4},
 \qquad
 \sum_{\substack{\gamma\gtrsim(ht)^{1/2}\\
                  \gamma\gtrsim h^{1/3}t}}h^{1/3}
       \lesssim h^{1/3}(1+|\log h|).
\]
Empty sums are understood to vanish. For the dyadic pieces with smaller $\gamma$ and
$t\geq h^{1/3-2\varepsilon}$, use the already summed quasimode
estimate in Proposition~\ref{propdispsumk}, whose remaining factor
is $(ht)^{1/4}$. For $t<h^{1/3-2\varepsilon}$ use
Section~\ref{seca<t<}: the nondegenerate short-time contribution
has free decay, and the transition contribution has remaining
factor at most $h^{1/3-2\varepsilon}$.
We thus sum the dyadic pieces in their respective regimes and retain
the logarithmic term.
For $0<a\leq h^{2/3-\varepsilon}$, \eqref{P1disp} and
\eqref{eq:small-a-second-all-times} give the remaining factors
$(ht)^{1/4}+h^{1/3-2\varepsilon}$ and
$(ht)^{1/4}+(h/t)^{1/2}$, respectively, uniformly as $a\to0$.
The additional free factor satisfies $(h/t)^{1/2}\le(h/t)^{1/4}$.

All these terms are bounded by $C(h/t)^{1/4}$. In particular, for the
logarithmic term and the small-$a$ contribution we use
\begin{equation}\label{eq:absorption-dyadic-losses}
 h^{1/3}(1+|\log h|)(h/t)^{-1/4}
   =t^{1/4}h^{1/12}(1+|\log h|)\lesssim 1\,,\quad\quad  h^{1/3-2\varepsilon}(h/t)^{-1/4}
   =t^{1/4}h^{1/12-2\varepsilon}\lesssim 1\,.
\end{equation}
Thus the glancing contribution is bounded by
$C h^{-d}(h/t)^{d/2-1/4}$.
The same bounds hold after inserting the additional spectral
cutoffs chosen above:
{ in the Airy representation these cutoffs
multiply the amplitude by \eqref{eq:airy-global-energy-factor},
with the uniform symbol bounds recorded there.}

We now proceed with the finite reconstruction and explain how the local estimates with {$x\leq a$} give a bound
for both variables of the kernel. For a real metric and real $\psi$, the kernel of $\mathcal K_h$,
defined in \eqref{eq:global-localized-kernel}, satisfies
\[
 \mathcal K_h(t;z,z')=\mathcal K_h(t;z',z),\qquad
 \mathcal K_h(-t;z,z')=\overline{\mathcal K_h(t;z,z')}.
\]
These identities follow by choosing a real Dirichlet eigenbasis.
Before making the local microlocal decomposition, we exchange the two
points if necessary so that the source has the larger distance to the
boundary. In boundary normal coordinates the observation then satisfies
$x\leq a$. This requires no spatial cutoff depending on $h$.
We apply the preceding estimates with this ordering and
tangential coordinates centered at the source footpoint. Uniformity
over that compact family of coordinate systems is part of
Lemma~\ref{lem:direct-glancing-matching}. Negative times follow from
the second identity.

We apply \eqref{eq:global-operator-partition}, estimate the interior and
mixed terms by \eqref{eq:interior-piece-dispersion}, and use the
one-sided boundary partition \eqref{eq:global-kernel-partition}.
The hyperbolic terms satisfy \eqref{eq:hyperbolic-kernel-bound};
the elliptic terms are $O(h^\infty)$ by
Lemma~\ref{lem:elliptic-spectral-removal}. For the glancing terms,
Lemmas~\ref{lem:boundary-locality} and
\ref{lem:direct-glancing-matching} identify the local and global
kernels. {The tangential and time-frequency cutoffs} preserve their bounds
by \eqref{eq:tangential-time-cutoff-kernel}--\eqref{eq:weighted-time-convolution}.
The normal variable is fixed by these cutoffs, so the ordering $x\leq a$
is preserved. Physical multipliers and coordinate density factors
are uniformly bounded. Summing the fixed finite number of pieces
therefore gives the same bound for the global kernel. The remainders are uniform in the source distance and are absorbed
by taking their order sufficiently large.

For $|t|\leq h$ we use the uniform bound $Ch^{-d}$ for the globally
localized kernel. Indeed, if
$-\Delta_g e_m=\lambda_m e_m$ with an orthonormal eigenbasis,
Cauchy--Schwarz gives
\[
 |\mathcal K_h(t;z,z')|
 \leq\left(\sum_m\psi(h^2\lambda_m)^2|e_m(z)|^2\right)^{1/2}
      \left(\sum_m\psi(h^2\lambda_m)^2|e_m(z')|^2\right)^{1/2}
 \lesssim h^{-d}.
\]
For completeness, each diagonal sum is bounded by a constant times
the Dirichlet heat kernel at time $h^2$, since
$\psi(h^2\lambda)^2\lesssim  e^{-h^2\lambda}$.
The uniform heat-kernel bound $C(h^2)^{-d/2}$ follows from the
Sobolev (or Nash) inequality for the Dirichlet form on the fixed smooth
compact domain; it holds up to the boundary. This gives the small-time bound, including the bounded continuation
through zero used in {the convolutions in time above}. At these times one may
use the exact spectrally localized kernel in the matching lemma:
{the nested time-frequency and tangential cutoffs} have uniformly integrable
kernels, so they preserve its $Ch^{-d}$ bound. The local comparison
then gives the same bound for the corresponding cut-off parametrix.
Carrying out the local
estimates on a slightly larger time interval before applying
$\psi^2(hD_t)$ gives
precisely the time margins of Section~\ref{sec:boundary-spectral-matching}.

We have proved $Ch^{-d}\min\{1,(h/|t|)^{d/2-1/4}\}$.
For $h<|t|\leq T_0$ the other two terms in the bracket of
\eqref{dispco} are dominated by its first term, since
$(h/|t|)^{1/2}\leq(h/|t|)^{1/4}$ and
$h^{1/3}(h/|t|)^{-1/4}=h^{1/12}|t|^{1/4}\lesssim 1$.
This proves \eqref{dispco}; substituting
$t=t_{\mathrm{phys}}/h$ proves
\eqref{eq:dispersion-physical-time}.
\end{proof}

\begin{proof}[Proof of Theorem~\ref{thmStrichartz}]
We use the exact global family $\mathcal U_h(t)\Pi_h$.
Its $L^2$ norm is bounded uniformly, and its $TT^*$ kernel is
$\mathcal K_h(t-t')$. We choose the time interval short enough for all differences of times
to lie in the dispersive range.
For a pair on the admissible line,
$1/q=(d/2-1/4)(1/2-1/r)$.
The $TT^*$ argument, including the endpoint theorem of
\cite{KeTa98}, gives the power {$h^{-(d-(d/2-1/4))(1/2-1/r)} =h^{-d(1/2-1/r)+1/q}=h^{-s-1/q}.$} Here the equality of the two powers uses the admissible-line relation.
For a pair below that line and $q<\infty$, choose $r_0\leq r$
on the line with the same $q$ and apply spectral Bernstein from
$L^{r_0}$ to $L^r$, with a larger cutoff equal to one on
$\supp\psi$. Its factor $h^{-d(1/r_0-1/r)}$ gives exactly the
stated exponent. For $q=\infty$, the assertion is the $L^2$
energy bound followed by spectral Bernstein. The change of time preceding \eqref{eq:strichartz-physical-time}
then gives the estimate in physical time.
\end{proof}

\begin{rmq}
For the Strichartz estimates one may also localize the solution
before applying $TT^*$. For the physical partition above, setting
$T_{j,h}(t)=\chi_j\mathcal U_h(t)\Pi_h$ gives
\[
 T_{j,h}(t)T_{j,h}(s)^*
 =\chi_j\mathcal U_h(t-s)\Pi_h^2\chi_j.
\]
The finite reconstruction $\mathcal U_h(t)\Pi_h=\sum_j\chi_jT_{j,h}(t)$
requires estimates only for these diagonal families. Microlocal
refinements can likewise be chosen on the observation side, with their
adjoints appearing automatically on the source side in $TT^*$.
This is in the spirit of the Isozaki--Kitada construction
\cite{IsozakiKitada85} and its use for Strichartz estimates by
Bouclet--Tzvetkov \cite[Section~4, especially Lemma~4.3]{BoucletTzvetkov07}:
the adjoint symmetry of the $TT^*$ kernel allows one to use the time
direction in which the local parametrix gives the required estimates. Such a proof would
still require the corresponding local dispersive estimates and the
$L^2$ bounds for the localized families. We mention this only for Strichartz estimates: Theorem~\ref{thmdisperS}
still uses the bound for the full kernel proved above.
\end{rmq}

\section{The defocusing cubic nonlinear Schr\"odinger equation}
\label{sec:defoc-cubic-nonl}

Throughout this section $t$ is physical time and the linear propagator is
$e^{-it\Delta_g}$. We work on the compact three-dimensional manifold
$(\Omega,g)$ of Theorem~\ref{gwp}, with its Dirichlet Laplacian and
Riemannian volume measure. {All frequency cutoffs in this section are global spectral multipliers
and commute with the linear flow.}

To define the global spectral cutoffs, we choose a nonincreasing function $\phi\in C^\infty_0([0,+\infty))$,
equal to one on $[0,1]$ and zero on $[4,+\infty)$, and set
$\rho(\lambda)=\phi(\lambda/4)-\phi(\lambda)$.
Thus $\rho$ is smooth, nonnegative, and supported in $[1,16]$.
Define the global Dirichlet Littlewood--Paley operators by
\[
 \Delta_{-1}=\phi(-\Delta_g),\qquad
 \Delta_j=\rho(-2^{-2j}\Delta_g),\quad j\geq0.
\]
Indeed, the telescoping identity
\[
 \phi(\lambda)+\sum_{j=0}^{J}\rho(2^{-2j}\lambda)
       =\phi(2^{-2(J+1)}\lambda)
\]
gives $\sum_{j\geq-1}\Delta_j=\mathrm{Id}$, with strong convergence
on $L^2(\Omega)$ and on the spectral Sobolev spaces used below.
For $j\geq0$ and $h=2^{-j}$, the operator $\Delta_j=\rho(-h^2\Delta_g)$
is precisely the construction of $\Pi_h$ in
\eqref{eq:global-spectral-cutoffs}, with the annular profile $\psi$
replaced by $\rho$. These are different operators when the cutoff
profiles differ, but the same construction applies. The preceding
linear estimates apply to this fixed annulus
by a finite subdivision and rescaling of the spectral window.
An enlarged profile equal to one on $\supp\rho$ plays the same role as
$\widetilde\psi$ in that construction. The block $\Delta_{-1}$ is dealt with
separately as a compact low-frequency spectral multiplier.
Spectral Sobolev norms are denoted here by
$\|f\|_{H_D^s}=\|(1-\Delta_g)^{s/2}f\|_2$; in particular
$H_D^1=H^1_0(\Omega)$ and $H_D^2=H^2(\Omega)\cap H^1_0(\Omega)$.
At higher regularity we will keep track of the Dirichlet boundary
conditions; we do not identify the spectral and ordinary Sobolev
spaces without these conditions.

We use the Dirichlet multiplier and square-function theory of
\cite[Theorems~1.1--1.2 and Proposition~1.7]{IPsquare}.
The heat-flow proof applies to smooth uniformly elliptic coefficients,
as explained there, and hence to the present Laplace--Beltrami operator
in a finite smooth atlas of coordinate charts. In particular, smooth Littlewood-Paley operators
are uniformly bounded on each fixed $L^p$, $1<p<\infty$, and
\begin{equation}\label{eq:nonlinear-spectral-tools}
 \begin{aligned}
 \|\Delta_j f\|_{L^p}
   &\leq C\,2^{3j(1/10-1/p)}\|\Delta_j f\|_{L^{10}},
                       &&10\leq p\leq\infty,\\
 \sum_{j\geq-1}\|\Delta_j f\|_{H^1}^2
   &\leq C\|f\|_{H^1}^2, && f\in H^1_0(\Omega).
 \end{aligned}
\end{equation}
The first constant is uniform in $p$: we factor an enlarged cutoff
through $e^{h^2\Delta_g}$, $h=2^{-j}$, use its $L^{10}\to L^{10}$
and $L^{10}\to L^\infty$ bounds, and interpolate the output norm.
The remaining compact spectral multiplier acts on the fixed space
$L^{10}$. The second bound follows directly from the spectral theorem
and finite overlap. In the uniqueness argument the square-function estimate is used at
fixed exponent; its constant need not remain bounded as $p\to\infty$.

We next record the consequence of Theorem~\ref{thmStrichartz}
in physical time on an arbitrary fixed interval, with space dimension $d\geq 2$.
\begin{thm}\label{thmStriSob}
Let $2\leq q\leq\infty$, $2\leq r<\infty$ satisfy the admissibility
condition of Theorem~\ref{thmStrichartz}, and put
$s=d/2-2/q-d/r$. For every $T_*>0$, the homogeneous linear solution
$v(t)=e^{-it\Delta_g}v_0$ satisfies
\begin{equation}\label{eq:Hs}
 \|v\|_{L^q([0,T_*];L^r(\Omega))}
 \leq C_{T_*}\|v_0\|_{H_D^{s+1/q}}.
\end{equation}
The constant is bounded for $0<T_*\leq1$.
\end{thm}
\begin{proof}
For a block at frequency $h^{-1}$, partition $[0,T_*]$ into at most
$C(1+T_*/h)$ intervals of length at most $ch$ and use
\eqref{eq:strichartz-physical-time}, together with unitarity.
This gives $C h^{-s}(1+T_*/h)^{1/q}$, not
$C T_*^{1/q}h^{-s-1/q}$ uniformly for all $T_*>0$.
For fixed $T_*$ it is bounded by $C_{T_*}h^{-s-1/q}$.
The square-function inequality, Minkowski's inequality for $q,r\geq2$,
and spectral orthogonality allow summation in $\ell^2$.
The low-frequency term is finite dimensional. This gives the stated estimate for the homogeneous linear equation.
\end{proof}

We also need the double-endpoint estimate for the Duhamel term,
with integration from the initial time to $t$.
In dimension three the endpoint is $(2,10)$ and
$\sigma=3/2-1-3/10=1/5$.
\begin{prop}\label{inStri}
Let $I=[a,b]$, $|I|\leq ch$, where $c$ depends only on the linear
time window, and let
\begin{equation}\label{schrodinhomo}
 -i\partial_tu+\Delta_g u=f,\qquad u(a)=0,
 \qquad u|_{\partial\Omega}=0.
\end{equation}
For every fixed smooth annular spectral cutoff, in particular
$\Delta_j$ with $h=2^{-j}$, one has
\begin{equation}\label{eq:inhom-endpoint-global}
 \|\Delta_j u\|_{L^2(I;L^{10})}
 \leq C h^{-2/5}\|\Delta_j f\|_{L^2(I;L^{10/9})}.
\end{equation}
The analogous homogeneous estimate on $I$ has factor $h^{-1/5}$.
\end{prop}
\begin{proof}
We apply \cite[Theorem~1.2, estimate~(7)]{KeTa98} to the exact spectrally
localized propagator, with an auxiliary cutoff equal to one on the
frequency support in question. Its decay exponent is $5/4>1$;
both pairs in the inhomogeneous estimate may therefore be $(2,10)$.
We restrict the family to $I$ and extend the source by zero.
All time differences lie in the range of the global dispersive theorem.
In semiclassical time the homogeneous factor is $h^{-7/10}$ and the
Duhamel integral without the equation's prefactor has factor
$h^{-7/5}$. The substitution $t=h\widetilde t$ contributes a factor
$h$ to the latter operator norm, giving $h^{-2/5}$ in physical time.
Equivalently, the two homogeneous physical-time factors are
$h^{-1/5}$. The auxiliary cutoff equals one on the support of $\Delta_j f$,
so its square acts as the identity there. We apply this estimate directly to the Duhamel term, without a Christ--Kiselev argument or a spatial
commutator.
\end{proof}
We are interested in $d=3$: for $q=2$ we have $r=10$, $s+1/q=7/10<1$. The crucial point is that $r<+\infty$: in \cite{bss12}, for $q=2$, $r=+\infty$, $s+1/q=1$ and one ends up with two successive logarithmic losses that force to consider only lower order nonlinearities. In \cite{FP}, Strichartz estimates are bypassed and replaced by bilinear estimates: logarithmic losses turn out to be more manageable in that setting and one can implement a suitable version of the Brezis-Gallouet argument, obtaining global existence of solutions to the cubic nonlinear Schr\"odinger equation but only for $H^{s}(\Omega)$ data, with $1<s\leq 3$.
\subsection{Global well-posedness in the energy space: proof of Theorem \ref{gwp}}
\label{sec:glob-wellp-ness}
\leavevmode

In this section we closely follow the elegant argument from \cite{bgt04}, replacing their double endpoint Strichartz estimate by the estimate from Proposition \ref{inStri}. The double endpoint estimate need not be lossless with respect to factors of $h^{-1}$, but must lie on the scale invariance line and have a finite spatial exponent. The rest of the argument then requires routine modifications in the numerology (which is, again, tied by scaling considerations).

We consider $-i\partial_t v+\Delta_gv=|v|^2v$, with
$v_0\in H^1_0(\Omega)$, and use {$E(v)=\tfrac12\|\nabla_gv\|_2^2+\tfrac14\|v\|_4^4.$} Finite-dimensional Dirichlet spectral Galerkin approximations conserve
mass and energy. Their uniform $L^\infty H^1_0$ bound and the equation
give compactness in $C_tL^2$ on bounded time intervals: the time
derivatives are bounded in $L^\infty H^{-1}$, and $H^1_0\Subset L^2$.
Passing to the limit gives a global distributional solution in
$L^\infty_{\mathrm{loc}}(\mathbb R;H^1_0)\cap C(\mathbb R;L^2)$,
with conserved mass and $E(v(t))\leq E(v_0)$ for every $t$.
The strong $L^4$ convergence needed for the potential energy follows
by interpolation between strong $L^2$ convergence and boundedness
in $L^6$. The same construction is available with any initial time.
We first prove uniqueness for the solutions so obtained and then
deduce energy conservation and strong $H^1$ continuity.

For the cubic term we use the following estimate.
For $j\geq0$, $h=2^{-j}$, and $F\in W^{1,10/9}_0(\Omega)$,
\begin{equation}\label{eq:dirichlet-product-block}
 \|\Delta_j F\|_{10/9}\leq Ch\|\nabla_g F\|_{10/9}.
\end{equation}
We give a direct proof of this estimate.
For a test function $G\in L^{10}$, integration by parts gives
\[
 |\langle\Delta_jF,G\rangle|
 =|\langle\nabla_gF,
       \nabla_g(-\Delta_g)^{-1}\Delta_jG\rangle|
 \lesssim h\|\nabla_gF\|_{10/9}\|G\|_{10}.
\]
There is no boundary term because $F$ has zero trace. The operator
bound follows by writing
\[
 \nabla_g(-\Delta_g)^{-1}\Delta_j
   =h^2\nabla_g e^{h^2\Delta_g}m(-h^2\Delta_g),
 \qquad m(\lambda)=e^\lambda\rho(\lambda)/\lambda,
\]
and using the fixed-$L^{10}$ gradient heat estimate and multiplier
bound in \cite{IPsquare}. Thus, for $v\in H^1_0(\Omega)$,
\begin{equation}\label{eq:cubic-product-energy}
 \|\Delta_j(|v|^2v)\|_{10/9}
 \lesssim h\|v\|_5^2\|\nabla_gv\|_2
 \lesssim h\|v\|_{H^1}^3.
\end{equation}
Indeed, $1/(10/9)=2/5+1/2$. Approximation by compactly supported
smooth functions in $H^1_0$ proves that $|v|^2v$ belongs to
$W^{1,10/9}_0$ and justifies the preceding calculation for energy
solutions. This is the product estimate we shall use in the uniqueness argument.

We next localize in time to track the dependence on $p$.
Let $I=[a,b]$, $\ell=b-a\leq1$, and set
$M=1+\|v\|_{L^\infty(I;H^1)}$.
Time cutoffs on intervals of length $ch$, followed by
Proposition~\ref{inStri}, give
\begin{equation}\label{tameH1}
 h^{1/5}\|\Delta_jv\|_{L^2(I;L^{10})} \lesssim \big(\|\Delta_jv(a)\|_2+\|\Delta_jv(b)\|_2 +h^{1/2}\|\Delta_jv\|_{L^2(I;H^1)} +h^{-1/5}\|\Delta_j(|v|^2v)\|_{L^2(I;L^{10/9})}\big).
\end{equation}
On an interior interval multiply $\Delta_jv$ by a
smooth cutoff with derivative $O(h^{-1})$ and vanishing initial value.
{The term produced by differentiating the time cutoff is estimated in
$L^1_tL^2_x$ by the homogeneous}
endpoint bound and Minkowski, giving
$C h^{-1/5-1/2}\|\Delta_jv\|_{L^2_tL^2_x}$ before multiplication
by $h^{1/5}$. Frequency localization converts this to the third
term in \eqref{tameH1}. The cubic source term is controlled by the double-endpoint
Duhamel estimate. We sum the squares over the bounded-overlap intervals;
the two end intervals are handled from $v(a)$ and $v(b)$,
respectively, using forward and backward Duhamel. If $\ell<ch$,
one interval suffices. This proves \eqref{tameH1} also for weak
energy solutions, since each spectral block satisfies Duhamel's
formula and the cubic term belongs to $L^\infty_tL^2_x$.

Combining \eqref{eq:cubic-product-energy}, Bernstein, and
$\|\Delta_jv(a)\|_2+\|\Delta_jv(b)\|_2\lesssim hM$ gives, uniformly
for $10\leq p<\infty$,
\begin{equation}\label{eq:derdesder}
 \|\Delta_jv\|_{L^2(I;L^p)}
 \lesssim \left[
 h^{3/p}\|\Delta_jv\|_{L^2(I;H^1)}
 +h^{1/2+3/p}M
 +h^{3/10+3/p}\sqrt\ell\,M^3\right].
\end{equation}
The exponent $3/10=4/5-(1/5+3/10)$ in the
last term is positive and independent of $p$.
We now sum by the triangle inequality in $L^2_tL^p_x$.
For the first term, we use \eqref{eq:nonlinear-spectral-tools} and
\[
 \sum_{j\geq0}2^{-3j/p}\|\Delta_jv\|_{L^2(I;H^1)}
 \leq\left(\sum_{j\geq0}2^{-6j/p}\right)^{1/2}
       \left(\sum_{j\geq0}\|\Delta_jv\|_{L^2(I;H^1)}^2\right)^{1/2}
 \lesssim \sqrt{p\ell}\,M.
\]
The other two sums are geometric with constants independent of $p$.
The low block contributes at most $C\sqrt\ell M$, uniformly in $p$,
by its fixed $L^2\to L^\infty$ bound. Consequently, for $ 10\leq p<\infty$,
\begin{equation}\label{eq:yudovich-large-p}
 \|v\|_{L^2(I;L^p)}\leq C_M(1+\sqrt{p\ell})\,,
\end{equation}
where $C_{M}$ is independent of both $I$ and $p$.

We now follow the Yudovich argument in
\cite[proof of Theorem~3]{bgt04}, using our $(2,10)$ endpoint.
Let $v_1,v_2$ be two energy solutions with the same value at time $a$.
On a fixed bounded time interval, we choose
$M\geq1+\|v_1\|_{L^\infty H^1}+\|v_2\|_{L^\infty H^1}$, put
$w=v_1-v_2$, and let $Y(t)=\|w(t)\|_2^2$. Then
\[
 -i\partial_tw+\Delta_gw=|v_1|^2v_1-|v_2|^2v_2,
 \qquad w(a)=0,\qquad w|_{\partial\Omega}=0.
\]
The equation gives $\partial_tw\in L^\infty H^{-1}$.
The $H^1_0,H^{-1}$ duality identity for the derivative of the
$L^2$ norm (or a time regularization) shows that $Y$ is absolutely
continuous and, for almost every $t$, {$|Y'(t)|\lesssim \int_\Omega(|v_1|^2+|v_2|^2)|w|^2.$} For any finite $p\geq10$, H\"older and interpolation between $L^2$
and $L^6$ yield
\begin{equation}\label{eq:yudovich-differential}
 |Y'(t)|\lesssim  M^{6/p}
    \big(\|v_1(t)\|_p^2+\|v_2(t)\|_p^2\big)
    Y(t)^{1-3/p},
\end{equation}
since
$\|w\|_{2p/(p-2)}\leq\|w\|_2^{1-3/p}\|w\|_6^{3/p}$.
We first integrate with $Y$ replaced by $Y+\delta$ and let
$\delta\downarrow0$. From \eqref{eq:yudovich-large-p} on $[a,t]$,
\begin{equation}\label{eq:yudovich-vanishing}
 Y(t)^{3/p}\leq \frac{C_M}{p}\big(1+p(t-a)\big),
 \qquad
 Y(t)\leq\big[C_M(t-a+p^{-1})\big]^{p/3}.
\end{equation}
We choose $\ell_0>0$ so that $C_M\ell_0<1/2$, independently of $p$.
For $a\leq t\leq a+\ell_0$, let $p\to\infty$ in
\eqref{eq:yudovich-vanishing}: the right-hand side tends to zero.
Hence $w=0$ on this interval. We restart at its right endpoint, using
$C_tL^2$ continuity, and iterate finitely many times on every bounded
time interval. The backward argument is identical. Thus the solution is unique in
$L^\infty_{\mathrm{loc}}H^1_0\cap C_tL^2$.

We now prove conservation, strong continuity and continuous dependence on the data.
Fix any time $a$. The compactness construction with datum $v(a)$ gives
an energy-inequality solution both forward and backward from $a$;
uniqueness identifies it with $v$. Thus $E(v(t))\leq E(v(a))$ for
all $t$. Exchanging $a,t$ proves equality, without having assumed
energy conservation in the uniqueness argument.
The $C_tL^2$ continuity and boundedness in $H^1_0$ give weak $H^1_0$
continuity and strong $L^4$ continuity. Energy conservation then gives
continuity of $\|\nabla_gv(t)\|_2$, hence strong continuity in $H^1_0$.

The same compactness and uniqueness argument gives continuous
dependence on the initial data. If $v_{0,n}\to v_0$ in $H^1_0$, the solutions
have a common energy bound; any subsequential compactness limit in
$C_tL^2$ is the unique solution with datum $v_0$. For any $t_n\to t$,
strong $L^4$ convergence and conservation of the convergent initial
energies give convergence of the gradient norms, while the gradients
converge weakly. Therefore $v_n(t_n)\to v(t)$ strongly in $H^1_0$.
A contradiction argument on each compact time interval yields
$v_n\to v$ in $C_tH^1_0$. This proves the defocusing assertion of
Theorem~\ref{gwp}.

For the focusing equation the estimates and the uniqueness argument
are the same. We obtain the required uniform energy bound as follows. Write $m=\|v_0\|_2$, $X=\|\nabla_gv\|_2$ and
$E_-(v)=X^2/2-\|v\|_4^4/4$. The Dirichlet Sobolev inequality and
interpolation give {$E_-(v)\geq \tfrac12X^2-CmX^3.$} For $m>0$, let $X_*=(4Cm)^{-1}$. If $X(0)<X_*/4$, then
$E_-(v_0)\leq X(0)^2/2<X_*^2/8$, whereas any first crossing of
$X=X_*$ would force $E_-(v)\geq X_*^2/4$, a contradiction.
This trapping argument applies uniformly to the Galerkin approximations
under the hypothesis
$\|v_0\|_2\|\nabla_gv_0\|_2<c_\Omega$, with
$c_\Omega=(16C)^{-1}$. Indeed, this is precisely $X(0)<X_*/4$,
and the spectral truncations do not increase either factor. It gives global
compactness, uniqueness, and conservation as above; $m=0$ is trivial.
Indeed, $E_-(v)\geq X^2/4$ for $X\leq X_*$, hence
$X(t)\leq\sqrt{2}X(0)<X_*/2$ for the approximations and their limit.
Restarting the construction from $v(a)$ preserves the energy bound
below $X_*^2/4$, so the conservation argument above applies.
The smallness assumption in Theorem~\ref{gwp} is therefore a
condition on the product of the initial mass and gradient norms. It is less
restrictive than smallness of the whole $H^1_0$ norm. The quantity has
the critical $\dot H^{1/2}$ scaling of the three-dimensional cubic
equation, and spectral interpolation gives {$\|(-\Delta_g)^{1/4}v_0\|_2^2 \leq \|v_0\|_2\|\nabla_gv_0\|_2.$} We use only the displayed implication; the converse is not needed.
For preservation of $H^{2}$ regularity, one may proceed exactly as in \cite{bgt04}, using the Brezis-Gallouet argument. One should notice, however, that the resulting bound on the Sobolev norm is a double exponential (to be compared to the triple exponential from \cite{FP}). We address this issue in the next subsection.
\subsection{Exponential growth for higher Sobolev norms $H^{m}$, $m>1$. Proof of Theorem \ref{growthHs}}
\label{sec:growth-norm-Sobolev}
The double exponential growth for higher Sobolev norms was reduced to a single exponential, for solutions on a generic compact, boundaryless manifold, in \cite{PTV}, using modified energy methods. We now check that the careful treatment of modified energies in \cite{PTV} is not spoiled if we replace the endpoint Strichartz estimate from \cite{bgt04} by our endpoint estimate.

In \cite{PTV}, on a compact $d$-dimensional Riemannian manifold $M^d$
without boundary, bounds on the $H^{2k}$ norms follow from computing
the time derivative of
\begin{equation}
  \label{E2k}
   {\mathcal E}_{2k}(v)=
 \|\partial_t^k v\|_{L^2(M^d)}^2-\frac{1}2 \int_{M^d} |\partial^{k-1}_{t} \nabla_g (|v|^2)|_g^2 - \int_{M^d}|\partial^{k-1}_{t} (|v|^{2} v)|^2 
\end{equation}
All integrals and norms are taken with respect to the Riemannian volume form. One gets 
\begin{multline}
  \label{dtE2k}
\frac d{dt} {\mathcal E}_{2k}(v)=2  \int_{M^d} \partial^{k}_{t} (|v|^{2})\partial^{k-1}_{t}
(|\nabla_g v|_g^2)
+\Re \sum_{j=0}^{k-2}
c_j   \int_{M^d} \partial^{k}_{t} (|v|^{2}) \partial_t^j (\Delta_g v) \partial_t^{k-1-j} \bar v 
\\
{}+ \Re \sum_{j=0}^{k-1} c_j  \int_{M^d}  
\partial^{j}_{t} (|v|^{2})  \partial_t^{k-j} v \partial_{t}^{k-1 }(|v|^{2}\bar v) 
 +\Im
\sum_{j=1}^{k-1} 
 c_j \int_{M^d} \partial_t^j (|v|^{2}) \partial_t^{k-j} v \partial^{k}_{t}\bar v \,,
\end{multline}
where $c_{j}$ are harmless numerical constants.
For the Dirichlet problem on $\Omega$, we use this identity for
$k=1,2$. At the $H^2$ level the initial datum belongs to
$H^2(\Omega)\cap H^1_0(\Omega)$. At the $H^4$ level the natural
assumption is
\[
 v_0\in D\bigl((-\Delta_{g,D})^2\bigr)
 =\{u\in H^4(\Omega):
       u|_{\partial\Omega}=0,\quad
       (\Delta_gu)|_{\partial\Omega}=0\},
\]
where $-\Delta_{g,D}$ denotes the Dirichlet realization of
$-\Delta_g$. These are the only compatibility conditions required at this regularity. Indeed, for a solution with {$v\in C(I;H^4),\quad \partial_tv\in C(I;H^2),\quad \partial_t^2v\in C(I;L^2),$} differentiating the Dirichlet condition once in time gives
$\partial_tv|_{\partial\Omega}=0$. Since
$i\partial_tv=\Delta_gv-|v|^2v$ and $v|_{\partial\Omega}=0$,
this is equivalent to $(\Delta_gv)|_{\partial\Omega}=0$.
No boundary trace is imposed on $\partial_t^2v$, which is only
$L^2$ at this regularity. Thus we do not impose a boundary trace
on $\Delta_g^2v_0$.

The cubic nonlinearity preserves the above domain: if
$v\in D((-\Delta_{g,D})^2)$, then $|v|^2v\in H^4$ and {$(|v|^2v)|_{\partial\Omega}=0,\quad \Delta_g(|v|^2v)|_{\partial\Omega}=0.$} The second equality holds because every term obtained by applying
two spatial derivatives to the cubic product still contains an
undifferentiated factor $v$ or $\bar v$. Thus no additional
nonlinear compatibility condition occurs at the $H^4$ level.
The integrations by parts leading to \eqref{dtE2k}, for $k=1,2$,
are first justified for smooth compatible solutions and then by
approximation at the stated regularity; the Dirichlet traces of
$v$ and $\partial_tv$ are the ones used in the $H^4$ argument.

An extension to $H^m$ with $m>4$ would require tracking the higher
compatibility conditions obtained by differentiating
$v|_{\partial\Omega}=0$ in time and using the nonlinear equation,
whenever the corresponding traces are defined, as just requiring the solution to be in the domain of higher powers of the linear Dirichlet Laplacian is no longer sufficient. Indeed, from three spatial
derivatives onward, derivatives of $|v|^2v$ may contain products in
which every factor is differentiated, whose boundary traces need
not vanish. Nonlinear compatibility conditions can therefore appear at
higher regularity, although this does not mean that a new independent
condition occurs for every $m>4$. We will not pursue these higher
orders here.

We use the following identities for $v$:
\[
 \|\partial_tv\|_2^2=\|\Delta_gv-|v|^2v\|_2^2\,,
 \qquad
 \|\partial_t^2v\|_2^2
   =\|\Delta_g\partial_tv-\partial_t(|v|^2v)\|_2^2\,.
\]
All norms here are over $\Omega$. Elliptic regularity and product
estimates compare these quantities with the corresponding Sobolev
norms, up to lower-order terms which we control below. We use
$\mathcal E_2(v)$ and $\mathcal E_4(v)$ defined by \eqref{E2k}
{with $M^d$ replaced by $\Omega$. We first recall the estimate of
\cite{PTV} on $M^d$, then adapt it to $\Omega$.}
The modified energy bound in \cite{PTV} proceeds with
\begin{equation}
  \label{modE2}
\frac{d}{dt} \mathcal{E}_{ 2}(v) \lesssim \int_{M^{d}} |\nabla ^{2}v| |\nabla v|^{2} |v |\lesssim \|v\|_{H^{2}}\|\nabla v\|^{2}_{L^{6}}\|v\|_{L^{6}}\,,
\end{equation}
using the equation to eliminate all time derivatives and ignoring lower order terms; one then uses a suitable version of Proposition \ref{inStri}, involving the endpoint Strichartz estimate, for the gradient term, which is therefore bounded by (the square of) an $H^{3/2}$ with an (important) additional power of time; by interpolation one recovers another $H^{2}$ norm and this leads to exponential growth for the $H^{2}$ norm by standard arguments.

From the computation of the time derivative \eqref{dtE2k} (with $M^{d}$ replaced by $\Omega$), we similarly get
\begin{equation}
  \label{modE2Omega}
  \frac{d}{dt} \mathcal{E}_{ 2}(v) \lesssim \int_{\Omega} |\nabla ^{2}v| |\nabla v|^{2} |v |
  \,,
\end{equation}
We leave aside the lower-order terms for the moment. To estimate
the displayed integral, let $p>14$ be fixed and let $I=[0,T]$,
$0<T\leq1$, recall {$M=1+\|v\|_{L^\infty(I;H^1)}$ , and define $ B=\|v\|_{L^\infty(I;H^{3/2})}.$} The constant $M$ is bounded in terms of the conserved defocusing
energy. Constants below may depend on the fixed exponent $p$ and on
the geometry, but not on $T$ or the dyadic index.

We will use the elementary inequality
\begin{equation}\label{eq:dyadic-minimum-sum}
 \sum_{j\geq0}\min\{U2^{-\mu j},V2^{\nu j}\}
 \leq C_{\mu,\nu}
 U^{\nu/(\mu+\nu)}V^{\mu/(\mu+\nu)},
 \qquad \mu,\nu>0,\quad U,V\geq0.
\end{equation}
Indeed, for $U,V>0$, one may split the sum where $2^{(\mu+\nu)j}$ crosses $U/V$;
both resulting sums are then geometric. If the crossing lies below $j=0$,
the decreasing bound alone gives the same inequality.

We first sum the bounds for the undifferentiated factor.
Set {$A_v(t)=\sup_{j\geq0}2^{3j/p}\|\Delta_jv(t)\|_{L^p}.$} Applying \eqref{eq:derdesder} with $h=2^{-j}$ and $\ell=T$ gives
\[
 2^{3j/p}\|\Delta_jv\|_{L^2(I;L^p)}
 \lesssim \left(
 \|\Delta_jv\|_{L^2(I;H^1)}
 +2^{-j/2}M+2^{-3j/10}\sqrt T\,M^3\right).
\]
Taking the square sum in $j$, using the spectral theorem and finite
overlap for the first term, yields
\begin{equation}\label{eq:Av-time-bound}
 \|A_v\|_{L^2(I)} \leq\left(\sum_{j\geq0}2^{6j/p} \|\Delta_jv\|_{L^2(I;L^p)}^2\right)^{1/2} \lesssim \bigl(\sqrt T\,M+M+\sqrt T\,M^3\bigr) \leq C_M.
\end{equation}
Thus we have obtained an $L^2(I)$ estimate for $A_v$ directly
{from the dyadic square sum.}

On the other hand, Bernstein's inequality and the energy bound imply
\[
 \|\Delta_jv(t)\|_{L^p}
 \leq C_p2^{3j(1/2-1/p)}\|\Delta_jv(t)\|_2
 \leq C_p M2^{j(1/2-3/p)}.
\]
Combining this with the definition of $A_v(t)$ and applying
\eqref{eq:dyadic-minimum-sum} with $U=A_{v}(t)$, $V=C_p M$, $\mu=3/p$ and
$\nu=1/2-3/p$, we obtain {$\sum_{j\geq0}\|\Delta_jv(t)\|_{L^p} \leq C_p A_v(t)^{1-6/p}M^{6/p}.$} Thus the dyadic series is absolutely summable in $L^p$ for almost
every $t$. The low block satisfies
$\|\Delta_{-1}v\|_{L^\infty(I;L^p)}\leq C_pM$.
Integrating in time and using \eqref{eq:Av-time-bound}, we conclude
\begin{equation}\label{eq:growth-v-mixed-norm}
 \|v\|_{L^{2/(1-6/p)}(I;L^p)}
 \leq C_p\left(M^{6/p}\|A_v\|_{L^2(I)}^{1-6/p}
                  +T^{(1-6/p)/2}M\right)
 \leq C_{p,M}.
\end{equation}

For the gradient factors, we choose {$q=\frac{4p}{p-2},\quad r=\frac{4p}{p+6},\quad \theta=\frac2r=\frac{p+6}{2p}.$} Then $4<q<14/3$, $2<r<4$, and {$\frac1p+\frac2q=\frac12,\quad \frac12\left(1-\frac6p\right)+\frac2r=1, \quad \frac3q+\frac1r=1.$} These relations allow us to estimate $|v||\nabla v|^2$ in
$L^1(I;L^2)$ by H\"older's inequality.

We next prove the fractional source-term bound
\begin{equation}\label{eq:fractional-cubic-dyadic}
 \|\Delta_j(|v|^2v)\|_{L^{10/9}}
 \lesssim 2^{-3j/2}\|v\|_{H^{3/2}}\|v\|_{H^1}^2,
 \qquad j\geq0,
\end{equation}
for functions satisfying the Dirichlet condition. We deduce it from product estimates at orders one and two by
smoothing at the spectral scale. Write $F(v)=|v|^2v$.
The chain rule, H\"older's inequality and the three-dimensional
Sobolev embeddings give
\begin{align}
 \|F(v)-F(w)\|_{W^{1,10/9}}
 &\lesssim (\|v\|_{H^1}+\|w\|_{H^1})^2\|v-w\|_{H^1},
 \label{eq:cubic-integer-difference}\\
 \|F(w)\|_{W^{2,10/9}}
 &\lesssim \|w\|_{H^1}^2\|w\|_{H^2}.
 \label{eq:cubic-integer-two}
\end{align}
For the first estimate, we place the differentiated factor in $L^2$
and the two remaining factors in $L^5$; the terms without derivatives
are controlled by $H^1\hookrightarrow L^{10/3}$.
For the second estimate, the terms with two derivatives on one factor
are bounded in the same way. For the other second-order terms, we use
\[
 \|w\,|\nabla w|^2\|_{10/9}
 \leq \|w\|_{30/7}\|\nabla w\|_3^2
 \lesssim \|w\|_{H^1}^2\|w\|_{H^2},
\]
since $H^1\hookrightarrow L^{30/7}$ and
$\|\nabla w\|_3^2\leq\|\nabla w\|_2\|\nabla w\|_6$.

Set $h=2^{-j}$ and $w=e^{h^2\Delta_g}v$, with the Dirichlet heat
semigroup. The spectral theorem and elliptic regularity imply
\[
 \|w\|_{H^1}\lesssim\|v\|_{H^1},\qquad
 \|v-w\|_{H^1}\lesssim h^{1/2}\|v\|_{H_D^{3/2}},\qquad
 \|w\|_{H^2}\lesssim h^{-1/2}\|v\|_{H_D^{3/2}}.
\]
Indeed, these follow from the scalar bounds for
$1-e^{-h^2\lambda}$ and $e^{-h^2\lambda}$ in the Dirichlet spectral
resolution. The first-order spectral estimate used in
\eqref{eq:dirichlet-product-block} gives
\[
 \|\Delta_j G\|_{10/9}\lesssim h\|G\|_{W^{1,10/9}}
 \quad\text{for }G\in W^{1,10/9}_0(\Omega).
\]
For $G\in W^{2,10/9}(\Omega)\cap W^{1,10/9}_0(\Omega)$,
we also have
\[
 \Delta_jG=h^2\left(\frac{\rho(\lambda)}{\lambda}\right)
                (-h^2\Delta_g)(-\Delta_g)G,
 \qquad
 \|\Delta_jG\|_{10/9}\lesssim h^2\|G\|_{W^{2,10/9}},
\]
by the uniform $L^{10/9}$ boundedness of smooth spectral multipliers
supported away from zero. Both $F(v)-F(w)$ and $F(w)$ have zero
boundary trace. Splitting $F(v)$ accordingly, we obtain
{\[
 \|\Delta_jF(v)\|_{10/9}
 \lesssim h\|F(v)-F(w)\|_{W^{1,10/9}}
            +h^2\|F(w)\|_{W^{2,10/9}} 
 \lesssim h^{3/2}\|v\|_{H^1}^2\|v\|_{H_D^{3/2}}.
\]}
At this regularity the Dirichlet spectral norm is equivalent to the
ordinary $H^{3/2}$ norm with zero boundary trace, proving
\eqref{eq:fractional-cubic-dyadic}. The calculation may first be made
for smooth Dirichlet functions and then extended by density.

To estimate the gradient blocks in time, we apply \eqref{tameH1} at the dyadic scale, multiply by $2^j$, and use
the spectrally localized derivative estimate
$\|\nabla\Delta_jv\|_{10}\lesssim2^j\|\Delta_jv\|_{10}$.
With \eqref{eq:fractional-cubic-dyadic}, this gives
\begin{equation}\label{tameH32}
 2^{-j/5}\|\nabla\Delta_jv\|_{L^2(I;L^{10})} \lesssim \bigl(2^j\|\Delta_jv(0)\|_2 +2^j\|\Delta_jv(T)\|_2 +\|\Delta_jv\|_{L^2(I;H_D^{3/2})} +2^{-3j/10}\sqrt T\,BM^2\bigr).
\end{equation}
The source exponent is $1+1/5-3/2=-3/10$.
For Dirichlet functions at this regularity, the spectral
$H_D^{3/2}$ norm is equivalent to the $H^{3/2}$ norm with zero
boundary trace. The endpoint terms in \eqref{tameH32} have square
sum bounded by $CM$, the third term by $C\sqrt T\,B$, and the last
term has geometric decay. Consequently, with {$A_{\nabla v}(t) =\sup_{j\geq0}2^{-j/5}\|\nabla\Delta_jv(t)\|_{L^{10}},$} we obtain the stronger form of the dyadic bound needed below:
\begin{equation}\label{eq:6}
 \|A_{\nabla v}\|_{L^2(I)} \leq\left(\sum_{j\geq0}2^{-2j/5} \|\nabla\Delta_jv\|_{L^2(I;L^{10})}^2\right)^{1/2} \lesssim \bigl(M+\sqrt T(1+M^2)B\bigr).
\end{equation}

We now sum in the final spatial norm.
At each time and for $j\geq0$, write $g_j=\nabla\Delta_jv$. We have {$\|g_j(t)\|_{10}\leq A_{\nabla v}(t)2^{j/5},\quad \|g_j(t)\|_2\lesssim B2^{-j/2}.$} Interpolate spatially, with {$\frac1q=\frac{\vartheta}{10}+\frac{1-\vartheta}{2}, \quad \vartheta=\frac{5(p+2)}{8p}.$} This gives a decreasing bound in the final norm:
\[
 \|g_j(t)\|_q
 \lesssim  A_{\nabla v}(t)^{\vartheta}B^{1-\vartheta}2^{-a_pj},
 \qquad a_p=\frac{1-\vartheta}{2}-\frac{\vartheta}{5}
           =\frac{p-14}{16p}>0.
\]
The complementary increasing bound follows from Bernstein and the
$H^{3/2}$ bound:
\[
 \|g_j(t)\|_q\lesssim 2^{3j(1/2-1/q)}\|g_j(t)\|_2
       \lesssim B2^{b_pj},\qquad
 b_p=1-\frac3q=\frac{p+6}{4p}>0.
\]
The derivative version of Bernstein here is applied to the
spectrally localized scalar function $\Delta_jv$; it follows from
the corresponding localized gradient multiplier bound.
Since {$a_p+b_p=\frac{\vartheta}{2},\quad \frac{\vartheta b_p}{a_p+b_p}=2b_p=\theta,$} the geometric summation \eqref{eq:dyadic-minimum-sum} yields {$\sum_{j\geq0}\|\nabla\Delta_jv(t)\|_q \leq C_p A_{\nabla v}(t)^\theta B^{1-\theta}.$} As $\theta r=2$, integration in time and \eqref{eq:6} imply
\begin{equation}\label{eq:growth-gradient-mixed-norm}
 \|\nabla v\|_{L^r(I;L^q)}
 \leq C_p\bigl(M+\sqrt T(1+M^2)B\bigr)^\theta B^{1-\theta}.
\end{equation}
The low block is included in this bound: its contribution is at most
$C_pT^{1/r}\min\{M,B\}\leq C_pM^\theta B^{1-\theta}$.
Since $\theta>1/2$ and $T\leq1$, we may also write
\[
 \|\nabla v\|_{L^r(I;L^q)}
       \leq C_{p,M}\bigl(B^{1-\theta}+T^{1/4}B\bigr).
\]
We now combine this estimate with \eqref{eq:growth-v-mixed-norm}
and the H\"older relations above.
Gathering all our estimates, we obtain the following bound for the principal contribution to the increment of the modified energy on $0<T<1$:
\begin{align}
  \int_{0}^{T} \int_{\Omega} |\nabla_g^{2}v|_g|\nabla_g v|_g^{2}|v|\,dV_g\,dt & \lesssim \|v\|_{L^{\infty}_{T}H^{2}}\|v |\nabla_g v|^{2}\|_{L^{1}_{T}L^{2}}\\
   & \lesssim \|v\|_{L^{\infty}_{T}H^{2}}\|\nabla_g v\|^{2}_{L^{r}_{T}L^{q}}\|v\|_{L^{2/(1-6/p)}_{T}L^{p}}\\
 &    \lesssim C(E) \big(\sup_{0\leq t\leq T}\|v\|_{H^{2}}^{2-\theta}+ T^{1/2} \sup_{0\leq t \leq T}\|v \|^{2}_{H^{2}}\big)
\end{align}
As in \cite{PTV}, we pass from the modified energy $\mathcal E_2(v)$ to the squared $H^2$ norm using elliptic regularity and absorbing the lower-order correction terms, with constants depending on the uniform $H^1$ bound; the remaining lower-order contributions to the energy increment are dealt with in the same way, and iteration on sufficiently short time intervals yields exponential growth.
  This concludes the proof of Theorem \ref{growthHs}  for $m=2$. The case $m=4$ is handled similarly, following \cite{PTV} and using the numerology established above for the case $m=2$.

\appendix

\section{Equivalence of phase functions for $G(x,y,\eta,\omega)$}
\label{app:phase-equivalence}
\subsection{Choice of the canonical normalization}
\label{app:canonical-normalization}

We first specify the representative of $\Gamma$ used throughout
the paper; its role in the generating function of $\canonchi_M$
and its decomposition \eqref{eq:GAB} are given in
Proposition~\ref{lemgamma}.
The weights are one for $\Xi$ and two for $x$ and $\rho-1$,
where $\rho=|\Theta|$; the coefficient functions depend on
$(y,\Theta/|\Theta|)$.
The coefficient $\alpha_\Gamma$ introduced in \eqref{AGam}
is set to zero by the canonical transformation following
Proposition~\ref{propimpformgamma}.
With this normalization, the third-order term in $A_\Gamma$ is
$\alpha_1x\Xi+\beta_1(\rho-1)\Xi+\mu_1\Xi^3$.
There is still a canonical freedom at this order.
Define the Hamiltonian $F_5$ by
\begin{equation}\label{eq:residual-canonical-Hamiltonian}
 F_5=X_M(p_M-1)\Xi b_1(Y_M,\Theta),\qquad
 p_M=\Xi^2+|\Theta|^2+X_Mq(\Theta)\,,
\end{equation}
where the subscript 5 denotes its leading weighted degree (two for $X_M$ and $p_M-1$ and one for $\Xi$).
We compose on the model side with the canonical transformation given by
the Hamilton flow generated by $F_5$. This flow preserves the boundary
and the model characteristic set and
fixes their glancing intersection.

Let $\Gamma^{\rm new}$ denote the function obtained after the preceding canonical transformation: this composition leaves the terms of weighted degree below five
unchanged; the degree-five term of $\Gamma^{\rm new}-\Gamma$,
after expanding as in the Appendix of \cite{ILLP}, is
\begin{equation}\label{eq:residual-canonical-action}
(\Gamma^{\rm new}-\Gamma)_5=(1+\Lp)b_1x\Xi
 \{\Xi^2+2(\rho-1)+q(1+\Lp)x\}
 \pmod{\mathcal H_{\ge6}}.
\end{equation}
Thus
$\mu_1^{\rm new}-\mu_1=(1+\Lp)b_1$,
$\beta_1^{\rm new}-\beta_1=2(1+\Lp)b_1$, and
$\alpha_1^{\rm new}-\alpha_1=q(1+\Lp)^2b_1$.
The subsequent terms of the Hamilton flow have weighted order at least
six in $\Gamma$, so do not affect these identities. The coefficients
on the right-hand sides are those before this transformation. We choose
$b_1=-\beta_1/[2(1+\Lp)]$. The transformed coefficients satisfy
\begin{equation}\label{eq:normalized-cubic-coefficients}
 \beta_1^{\rm new}=\beta_1+2(1+\Lp)b_1=0,\qquad
 \mu_1^{\rm new}=\mu_1-\frac{\beta_1}{2},\qquad
 \alpha_1^{\rm new}=\alpha_1-\frac{q(1+\Lp)\beta_1}{2}.
\end{equation}
This does not change $\alpha_\Gamma=0$, $\mu_\Gamma$, $\Lp$ or
$B_\Gamma$. We now drop the superscripts $\rm new$ and keep the same notations
for the transformed coefficients. Both $\alpha_1$ and $\mu_1$ have
changed along with $\beta_1$; the old $\beta_1$ term is not simply
deleted with the other coefficients left unchanged.
The homogeneous quadratic term is
$\mu_\Gamma\{\Xi^2+2(\rho-1)\}$.
Replacing $2(\rho-1)$ by $\rho^2-1$ changes it by a term of weight
four, absorbed in $R_A$. We therefore use
\eqref{eq:normalized-AGamma} with $\beta_1=0$.

For the rescaling in \eqref{deftildephiNagamma}, $\tau^2=|\theta|^2+\gamma A q(\theta)$,
$\Xi=\sqrt\gamma\Sigma q^{1/2}(\theta)/\tau$ and $\rho=|\theta|/\tau$.
Equation~\eqref{Agamtau} follows directly, and gives
\begin{equation}\label{eq:normalized-endpoint-expansion}
\begin{split}
 x\tau A_\Gamma={}&\gamma^{3/2}\frac{x}{\gamma}
 q^{1/2}(\theta)\Sigma\Lp
 +\gamma^2\frac{x}{\gamma}\frac{q(\theta)}{\tau}
 (\Sigma^2-A)\mu_\Gamma\\
 &+\gamma^{5/2}\left\{
 \left(\frac{x}{\gamma}\right)^2\alpha_1q^{1/2}(\theta)\Sigma
 +\frac{x}{\gamma}\mu_1\frac{q^{3/2}(\theta)}{\tau^2}\Sigma^3
 \right\}+O_{C^m}(\gamma^3).
\end{split}
\end{equation}
The bounds are uniform for bounded $x/\gamma,\Sigma,A$, $|\theta|\sim1$
and $y$ in the fixed patch, for any fixed $m$.
Differentiation in these rescaled variables preserves the remainder
bound. The $\alpha_1$ term is independent of $A$, and differentiating
$\tau^{-1}$ or $\tau^{-2}$ in $A$ gains a factor $\gamma$.
This proves \eqref{eq:normalized-endpoint-derivatives}.
The same reasoning applies with $(x,y,\Sigma)=(a,0,S)$, uniformly for
$0\le a\lesssim\gamma$; the terms at the source point carry the factor
$a/\gamma$ (or its square).
\subsection{Equivalence of the two phase functions}
We return here to the phase functions used in the paper. The two
representations of $G$ come from \cite{ILLP}; we derive directly
their equivalence and the expansion of $\Upsilon$.

\begin{samepage}
Both $\psi+s^3/3-s\zeta$ (from \eqref{eq:defG}) and
\[
 y\cdot\eta+\sigma^3/3
 +\sigma(xq^{1/3}(\eta)-\omega)
 +\tau_q\Gamma\left(
 x,y,\frac{\sigma q^{1/3}(\eta)}{\tau_q},
 \frac{\eta}{\tau_q}\right)
\]
(from \eqref{eq:Gosc}) generate the same local Lagrangian.  We now compare the two normal variables.
\end{samepage}

\begin{lemma}\label{lemphaseG}
Let $w=(x,y,\theta,\alpha)$ and
\[
 \phi(w,\sigma)
 =
 \frac{\sigma^3}{3}
 +\sigma(xq^{1/3}(\theta)-\alpha)
 +x\tau_q
 A_\Gamma\left(
 x,y,\frac{\sigma q^{1/3}(\theta)}{\tau_q},
 \frac{\theta}{\tau_q}\right).
\]
After shrinking the glancing neighborhood, there exist smooth
functions $\zeta(w)$ and $\Upsilon(w)$ and a smooth local change
$\sigma\mapsto s=s(w,\sigma)$ with $\partial_\sigma s\ne0$ such that
\begin{equation}\label{eq:Schr-Airy-normal}
 \phi(w,\sigma)=\frac{s^3}{3}-s\zeta(w)+\Upsilon(w).
\end{equation}
On the side $\zeta>0$, let $\sigma_\pm(w)$ correspond to
$s_\pm=\pm\sqrt{\zeta(w)}$.  Then
\begin{equation}\label{formephaseUpsilon}
 \frac23\zeta(w)^{3/2}=\frac 12\{\phi(w,\sigma_-)-\phi(w,\sigma_+)\}
 ,\qquad
 \Upsilon(w)
 =\frac12\{\phi(w,\sigma_+)+\phi(w,\sigma_-)\}.
\end{equation}
The functions defined by these identities extend smoothly through
$\zeta=0$.

If $\sigma_0(w)$ is the unique zero of
$\phi_{\sigma\sigma}(w,\cdot)$ and
$c_3(w)=\frac12\phi_{\sigma\sigma\sigma}(w,\sigma_0(w))$,
then the inverse normal change has
\begin{equation}\label{eq:Schr-inverse-normal}
 \sigma=\sigma_0(w)+c_3(w)^{-1/3}s+O(s^2),
 \qquad c_3(w)=1+O(x).
\end{equation}
In particular, the change is uniformly non-degenerate, although
its linear coefficient is not necessarily one.
\end{lemma}

\begin{proof}
At the model glancing point $\phi_{\sigma\sigma\sigma}=2$; the {Taylor expansion of $A_\Gamma$ established above gives}
$\phi_{\sigma\sigma\sigma}=2+O(x)$.
Hence $\phi_{\sigma\sigma}$ is strictly monotone in $\sigma$ and has
a unique zero $\sigma_0(w)$, smooth in $w$.  Put
$r=\sigma-\sigma_0(w)$.  Taylor's formula gives
\begin{equation}\label{eq:Schr-centered-A2}
 \phi(w,\sigma_0+r)
 =
 C(w)+A(w)r+\frac{c_3(w)}3r^3+r^4R(w,r),
 \qquad c_3(w)=1+O(x).
\end{equation}
The one-dimensional $A_2$ normal-form theorem with parameters
applied to \eqref{eq:Schr-centered-A2} yields
\eqref{eq:Schr-Airy-normal}. Matching the third-order terms gives
$s=c_3(w)^{1/3}r+O(r^2)$; inversion proves
\eqref{eq:Schr-inverse-normal}.

The two critical values of the normal form are
\[
 \left(\frac{s^3}{3}-s\zeta+\Upsilon\right)_{s=\mp\sqrt\zeta}
 =\Upsilon\pm\frac23\zeta^{3/2}.
\]
Taking their difference and mean gives \eqref{formephaseUpsilon};
in particular, the coefficient of the full difference is $4/3$.  These
{identities are invariant under the normal change of variable. The
normal-form construction shows that these functions extend smoothly
through $\zeta=0$.}
\end{proof}

\begin{cor}\label{corUpsi}
We may choose
\[
 \psi(x,y,\theta,\alpha)
 =
 y\cdot\theta
 +\tau_q(\alpha,\theta)
 B_\Gamma(y,\theta/\tau_q)
 +\Upsilon(x,y,\theta,\alpha).
\]
With the weighted convention used in the local calculation,
$\deg\sigma=1$ and $\deg x=\deg\alpha=2$, the mean critical value has
the expansion
\begin{equation}\label{eq:Ups}
 \Upsilon
 =
 x^2\tau_q\,{\mathcal F}
 -\frac{x^2\mu_\Gamma(y,\theta/|\theta|)q(\theta)}
        {\tau_q}
  (1+\Lp(y,\theta/|\theta|))
 +x\mathcal H_{\ge4}.
\end{equation}
The weight-five contribution to the mean critical value vanishes; the
first coefficient not determined by {the expansions used here} has weighted
degree at least six.
\end{cor}

\begin{proof}
{Our normalization of the generating function $\Gamma$ allows ${\mathcal F}=0$ in the following expansion.} Let
\[
 b=\frac{x\mu_\Gamma(y,\theta/|\theta|)q(\theta)^{2/3}}{\tau_q},
 \qquad
 D=\alpha-xq(\theta)^{1/3}
 (1+\Lp(y,\theta/|\theta|)),
\]
and denote by $C=x^2\tau_q{\mathcal F}-b\alpha$ the constant part
through weight four. The corresponding {third-order polynomial expansion} is
\begin{equation}\label{eq:Schr-cubic-expansion}
 \phi_{\rm cub}(\sigma)
 =
 \frac{\sigma^3}{3}+b\sigma^2-D\sigma+C.
\end{equation}
The exact translation $\sigma=s-b$ gives
\begin{equation}\label{eq:Schr-cubic-translation}
 \phi_{\rm cub}(s-b)
 =
 \frac{s^3}{3}-(D+b^2)s+C+bD+\frac23b^3.
\end{equation}
Thus
\[
 \zeta=D+b^2,\qquad
 \Upsilon=C+bD+\frac23b^3
\]
at this order.  Substituting $C,D,b$ yields
\[
 \Upsilon
 =
 x^2\tau_q{\mathcal F}
 -\frac{x^2\mu_\Gamma q}{\tau_q}(1+\Lp)
 +\frac23b^3.
\]
The last term has weighted degree six.

In our normalization, the two weight-five terms of the normal phase are
\begin{equation}\label{eq:normalized-odd-phase-terms}
 \alpha_1x^2\sigma q^{1/3}(\theta)
 +\mu_1x\sigma^3q(\theta)/\tau_q^2.
\end{equation}
They are odd in $\sigma$. Since $b$ has weight two, replacing
$\sigma$ by $s-b$ changes these terms only at weight six or higher.
Their contributions to the critical values are opposite through weight
five and therefore cancel in the mean. They can change $\zeta$ and
the higher terms of $\Upsilon$, but not the displayed coefficient in
\eqref{eq:Ups}. In particular, $\alpha_1$ has not been set to zero.
The cancellation at the preceding order can also be seen directly.
After centering, the contributions of weights three, four and five
are respectively odd, even and odd in the normal variable. The
weight-four even term translates the
two critical points by the same amount, whereas the weight-five
correction is antisymmetric.  Consequently the weight-five
contribution cancels in the mean of the two critical values.  The
remaining non-model contribution is divisible by $x$, since at
$x=0$ the normal phase is the exact Airy model phase.  This gives the
notation $x\mathcal H_{\ge4}$ in \eqref{eq:Ups}.
\end{proof}

\makeatletter
\let\SchrSavedSection\section
\def\section{\@startsection{section}{1}  \z@{.7\linespacing}{.5\linespacing}  {\normalfont\scshape\centering}}
\makeatother
\section{Localization and reduction to the glancing region}
\label{app:localization-reduction}
\let\section\SchrSavedSection
We prove Lemma~\ref{lem:boundary-locality}, estimate the
interior,
hyperbolic and elliptic terms, and match the remaining
glancing part to our parametrix.

\subsection{A common localization argument}
\label{app:common-localization}
For either type of patch, we choose a smooth compact extension of the
common coordinate neighborhood, retaining the physical boundary and
Dirichlet condition when present. The cutoffs stay a fixed distance
from where the two problems differ. Finite propagation for the wave equation, with the Dirichlet condition
when a physical boundary is present, gives
$s_0>0$ such that
\begin{equation}\label{eq:common-wave-kernel}
 \chi_k\cos(s\sqrt{-\Delta_g})\chi_j
 =\chi_k\cos(s\sqrt{-\Delta_{\mathrm{loc}}})
                      \chi_j,\qquad |s|<s_0 .
\end{equation}
Here $s_0$ depends on the distance to the artificial edge, not to
the physical boundary, as in the wave reduction of
\cite[Section~2]{ILLP}.

Fourier inversion of the even function
$e^{it\tau^2/h}f(\tau^2)$ gives
\begin{align}
 c_h(t,s)&=\frac{1}{2\pi h}\int_{\mathbb R}
       e^{i(t\tau^2-s\tau)/h}f(\tau^2)\,d\tau,
       \label{eq:wave-transfer-weight}\\
 \mathcal U_h(t)f(-h^2\Delta_g)&=\int_{\mathbb R}c_h(t,s)
                         \cos(s\sqrt{-\Delta_g})\,ds .
       \label{eq:wave-transfer}
\end{align}
The same formula holds for the auxiliary problem. We choose a smooth
even cutoff $\zeta_0(s)$ supported in $(-s_0,s_0)$ and equal to one
for $|s|\leq s_0/2$. We choose $T_0$ so small that
$2T_0\sup\{|\tau|:\tau^2\in\supp f\}<s_0/4$.
On the support of $1-\zeta_0$, the derivative
$2t\tau-s$ of the phase with respect to $\tau$ is bounded away
from zero and has size comparable to $1+|s|$.
Repeated integration by parts in $\tau$ therefore gives, for all
nonnegative integers $N,L,p,q$,
\begin{equation}\label{eq:wave-tail-symbol}
 \left|\partial_t^p\partial_s^q
       \bigl[(1-\zeta_0(s))c_h(t,s)\bigr]\right|
 \leq C_{N,L,p,q}h^N(1+|s|)^{-L}.
\end{equation}
Further integrations by parts absorb the factors $h^{-1}$ from derivatives.

For either Laplacian, denote the tail operator by
$\mathcal T_h(t)=\int(1-\zeta_0)c_h(t,s)
\cos(s\sqrt{-\Delta})\,ds$.
The wave equation gives
$(1-\partial_s^2)\cos(s\sqrt{-\Delta})
=(1-\Delta)\cos(s\sqrt{-\Delta})$.
Integration by parts in $s$ yields
\begin{equation}\label{eq:wave-tail-smoothing}
 \big\|(1-\Delta)^m\partial_t^p\mathcal T_h(t)
                  (1-\Delta)^m\big\|_{L^2\to L^2}
 \leq
 \big\|(1-\partial_s^2)^{2m}
       \partial_t^p[(1-\zeta_0)c_h(t,\cdot)]\big\|_{L^1}
 =O(h^N)
\end{equation}
for every $m,p,N$. The identity holds first on finite spectral sums,
then by extension. Elliptic regularity, including the Dirichlet
condition, and Sobolev embedding for arbitrarily large $m$ give the
kernel bounds up to the boundary; physical cutoffs preserve them.
The parts with $\zeta_0$ agree by
\eqref{eq:common-wave-kernel}, and the two tails are negligible.
The same argument at $t=0$ gives the spectral comparison.
Taking the smallest time margin over the finite atlas gives uniformity.

\subsection{The strictly interior contribution}
\label{app:strictly-interior-contribution}

Fix an interior cutoff $\chi_j^{\mathrm{int}}$ and a larger cutoff
$\widehat\chi_j^{\mathrm{int}}$ in a coordinate neighborhood compactly
contained in $\Omega$. We choose a smooth compact boundaryless extension
of the metric, with Laplacian $\widetilde\Delta_g$.
Lemma~\ref{lem:boundary-locality}, at $t=0$, gives
\begin{equation}\label{eq:interior-spectral-comparison}
 \chi_j^{\mathrm{int}}f(-h^2\Delta_g)\widehat\chi_j^{\mathrm{int}}
 =\chi_j^{\mathrm{int}}f(-h^2\widetilde\Delta_g)
   \widehat\chi_j^{\mathrm{int}}+O(h^\infty).
\end{equation}
For $|t|\le T_0$, the same lemma with $f=\psi^2$ identifies the
localized propagators. Decrease $T_0$ so that the {rays with frequencies in the chosen spectral window and issuing
from $\supp\chi_j^{\mathrm{int}}$ remain} in this enlarged interior
neighborhood. {The kernel with its other variable outside this enlarged neighborhood
is $O(h^\infty)$ by the short-time locality used in
Section~\ref{sec:global-microlocal-partition}.}
We may therefore estimate the kernel localized at the source by
the corresponding boundaryless kernel.

The spectral pseudodifferential localization and local WKB estimates of
\cite[Proposition~2.1, Corollary~2.4, and Lemmas~2.5 and~2.7]{bgt04}
therefore apply. With physical time $s=ht$, they give
$C(h|t|)^{-d/2}=Ch^{-d}(h/|t|)^{d/2}$; for $|t|\le h$, the compact
semiclassical amplitude gives $Ch^{-d}$. This proves
\eqref{eq:interior-piece-dispersion}, uniformly over the finite atlas.
By taking the adjoint, we obtain the same bound when the interior
cutoff is at the observation point. The physical partition then gives
\begin{equation}\label{eq:interior-eliminated-kernel}
 \mathcal K_h(t)=
 \sum_{k,k'\in\mathcal J_\partial}
 (\chi_k^\partial)^2\mathcal K_h(t)(\chi_{k'}^\partial)^2
 +\mathcal K_h^{\mathrm{int}}(t).
\end{equation}
Here $\mathcal K_h^{\mathrm{int}}$ is the finite sum of the interior
and mixed terms. Its free decay is stronger than the required glancing
bound. We now divide the boundary terms into hyperbolic, elliptic and
glancing pieces. This reduction to \cite{bgt04} for interior terms
is also used in \cite[Section~2]{bss12}.

\subsection{The uniformly hyperbolic boundary contribution}
\label{app:hyperbolic-contribution}

On the hyperbolic set \eqref{eq:hyperbolic-region}, the normal
characteristic roots at the boundary are {$\xi=\pm\sqrt{\tau^2-R_0(y,\eta)},\quad |\xi|\geq\sqrt{\delta}.$} Here the frequency variables are understood at semiclassical scale.
The fixed number $\delta$ does not depend on $h$ or on the source
distance $a$.  The two roots are therefore real and uniformly separated.
The geometrical-optics construction of \cite[Section~1.3]{meta}
applies to this transverse region.  The beginning of
\cite[Section~7.5]{meta} likewise separates the classical Fourier
integral contribution of rays avoiding glancing from the glancing
analysis.  Although that section concerns a diffractive boundary, we
use here only the transverse construction of Section~1.3, which
does not require that sign of the curvature.

We now prove Proposition~\ref{prop:hyperbolic-dispersion}.
\begin{proof}
On our {compact spectral range}, uniform
transversality gives a positive lower bound, depending on $\delta$,
for the interval between transverse reflections: the inward normal
speed cannot reverse immediately. Taking $T_0$ below this bound
leaves only the direct and once-reflected branches. No lower bound
on the initial source-to-boundary travel time is used.

For these branches, we compose the transverse geometrical optics
parametrix with {the smooth spectral cutoffs}.  {Passing from the wave-frequency parameter $\tau$ to the Schrödinger spectral parameter $\tau^2$}, or equivalently solving
the reflected Schrödinger eikonal and transport equations directly,
gives a real-phase WKB representation.  On each member of a finite
refinement, its kernel has the form
\begin{equation}\label{eq:hyperbolic-wkb}
\begin{split}
 \mathcal K_{r,h}^{\mathrm{hyp}}(t,x,y;a,y')
 ={}&(2\pi h)^{-d}\int
 e^{\frac{i}{h}(\Phi_r^{\mathrm{hyp}}(t,x,y,\xi,\eta)-a\xi-y'\cdot\eta)}
 \\[-2mm]
 &\hspace{8mm}\times
 c_r^{\mathrm{hyp}}(t,x,y,a,y',\xi,\eta;h)\,d\xi\,d\eta
 +O(h^\infty).
\end{split}
\end{equation}
Here $r$ labels a direct or reflected branch and its coordinate
patch.  The amplitudes have order zero and compact semiclassical
frequency support, with uniform symbol bounds.
The incident and reflected amplitudes satisfy the Dirichlet matching
condition, and  construction to arbitrary order, followed by the energy
estimate, controls the localized remainder in arbitrary kernel
seminorms.

We now explain why such a Fourier integral representation yields the
Schrödinger decay.  Freeze the metric and flatten the boundary at a
reflection point, using orthonormal coordinates there.  For a flat
boundary the reflected ray is obtained by changing the sign of the
normal component; the frequency Hessian of the Schrödinger phase is
$2t\,\mathrm{Id}$, just as for the direct ray, with the sign dictated
by \eqref{SE}.  For a transverse reflected ray in the variable
geometry, both its legs have length $O(|t|)$.

The reflection time is determined by the vanishing of the incident
ray's normal coordinate. Differentiating this equation with respect
to the initial frequencies divides by the corresponding velocity
component, whose absolute value is bounded below by $c\sqrt{\delta}$;
the frequency derivatives of the reflection time and point are
therefore bounded by $C_\delta|t|$.
Compared with reflection at a flat boundary, the variation of the
boundary normal contributes $O_\delta(|t|)$ to the frequency
derivatives of the reflected velocity. Over the remaining time
$O(|t|)$, this gives an $O_\delta(t^2)$ correction to the frequency
derivatives of the final position.

Smooth variation of the metric produces an error of the same order.
The mixed position--frequency derivatives remain invertible, so
the corresponding generating phase satisfies, near its critical set,
in these normalized coordinates,
\begin{equation}\label{eq:hyperbolic-hessian}
 \partial^2_{(\xi,\eta)}\Phi_r^{\mathrm{hyp}}
       =2t\,\mathrm{Id}+O_\delta(t^2),\qquad
 \left|\det\partial^2_{(\xi,\eta)}\Phi_r^{\mathrm{hyp}}\right|
       \geq c_\delta |t|^d .
\end{equation}
Coordinate changes replace the identity by a uniformly invertible
matrix and preserve the determinant bound.  Shrinking $T_0$ makes
this uniform over the finite atlas.  For the dispersion estimate we use this short-time
nonfocusing property of the phase, in addition to the classical
Fourier integral representation.

Stationary phase in the $d$ variables $(\xi,\eta)$, with large
parameter $|t|/h$, now gives $C h^{-d}(h/|t|)^{d/2}$ for $|t|\geq h$,
as in \cite[Lemma~2.7]{bgt04}.  Nonstationary pieces are handled by
integration by parts.  For $|t|\leq h$, the absolute value of the
frequency integral is bounded by $C h^{-d}$.  The finite branch sum
proves \eqref{eq:hyperbolic-kernel-bound}.
\end{proof}

By continuity of the flow and reflection map, shrink the boundary
neighborhood and $T_0$ so that rays from the compact hyperbolic
microsupport stay uniformly transverse at each encounter. Thus the
proposition also handles distinct source and observation patches on
such a ray, and \eqref{eq:hyperbolic-kernel-bound} is stronger than
\eqref{dispco}.

The localization arguments of \cite[Section~2]{bss08} and
\cite[Sections~2--3]{bss12} give a modern wave-packet treatment.
The second assertion of \cite[Theorem~2.3]{bss12}, for example,
assumes $|\xi_n|\geq |\xi|/20$. Its Strichartz proof provides a
localization framework from which a pointwise bound could also be
extracted; here we use transverse geometrical optics and stationary
phase for \eqref{eq:hyperbolic-kernel-bound}.

This region is separated from glancing; the transverse regime
$\gamma\geq4a$ of Section~\ref{sectransv} lies within a glancing patch
and uses the Airy analysis. Extending that analysis to large $\gamma$
could recover the hyperbolic regime, as in \cite{OanaKG} for
large-time asymptotics on the Friedlander model.
\subsection{The elliptic boundary contribution}
\label{app:elliptic-contribution}

We next remove the uniformly elliptic pieces of the boundary partition.
Here $\tau$ retains its wave-frequency meaning, and the corresponding
{semiclassical spectral parameter is $\tau^2$.}
We will use elliptic regularity for the homogeneous Dirichlet problem.

{We first subdivide the spectral interval.}
The condition $R_0(y,\eta)-\tau^2>\delta$ is {relative to the spectral parameter $\tau^2$ of the piece under consideration}. It does not imply that $R_0(y,\eta)$
lies above the entire support of $\widetilde\psi$.
We refine the compact elliptic support by finitely many product
neighborhoods in $(y,\eta)$ and $\tau^2$, {using smooth spectral cutoffs}
$f_\ell$ supported in intervals of width smaller than $\delta/4$.
The number of these neighborhoods is independent of $h$.
Shrinking their tangential supports and the neighborhood of the boundary if necessary,
each associated tangential cutoff
$\mathcal A_{k,\ell,h}^{\mathrm{ell}}$ can be chosen so that
\begin{equation}\label{eq:elliptic-window-gap}
 R(x,y,\eta)-\tau^2\geq\delta/2
 \quad\hbox{on its symbol support, whenever }\tau^2\in\supp f_\ell \,,
\end{equation}
using nested, properly supported cutoffs whose enlarged supports still
satisfy a fixed positive gap. The spatial cutoffs are smooth up to
$x=0$, act by multiplication in the normal variable, and hence
preserve the zero Dirichlet trace. Such product neighborhoods cover
the original elliptic support. A subordinate finite partition allows
the elliptic part to be quantized using pieces
$\mathcal A_{k,\ell,h}^{\mathrm{ell}}f_\ell(-h^2\Delta_g)$,
with symbol corrections supported in the same enlarged elliptic
sets.
We now prove Lemma~\ref{lem:elliptic-spectral-removal}.

\begin{proof}
We consider a normalized Dirichlet eigenfunction with eigenvalue
$\tau^2$ for $-h^2\Delta_g$, where $\tau^2\in\supp f_\ell$.
The principal normal equation in the region
\eqref{eq:elliptic-window-gap} has roots
$\xi=\pm i\sqrt{R(x,y,\eta)-\tau^2}$, uniformly separated from
the real axis. The elliptic Dirichlet parametrix, applied with
nested cutoffs, therefore makes this eigenfunction microlocally
$O(h^\infty)$ there, in all semiclassical Sobolev norms.
More explicitly, the local estimate controls a cutoff by the
localized equation, its Dirichlet trace, and $h$ times a norm with
a larger cutoff, with a residual smoothing error. The first two
terms vanish. Iteration through arbitrarily many nested cutoffs
gives any prescribed power of $h$; conversion to ordinary
$H^M$ norms costs only a fixed power.

The direct boundary construction is the elliptic Poisson parametrix
in \cite[Section~1.3, equations~(1.3.26)--(1.3.27)]{meta}.
Its decaying normal root handles the boundary, while ordinary
interior elliptic regularity handles the rest of the {neighborhood} of $\partial\Omega$.
The exponential decay of the Poisson factor does not by itself give
smallness uniformly up to $x=0$; the homogeneous Dirichlet condition
removes the boundary-data term. One may find all the necessary 
microlocal boundary framework in \cite{melrosesjostrand};
the uniform semiclassical and kernel estimates used here are the
consequences of elliptic regularity detailed in this proof.

To pass from individual eigenfunctions to the operator statement,
expand the smooth spectral window in an orthonormal Dirichlet
eigenbasis. There are $O(h^{-d})$ eigenvalues in a fixed compact
{semiclassical spectral interval}. Normalized eigenfunctions have
at most polynomial growth in each fixed Sobolev or $C^m$ norm,
by global elliptic regularity. The localized eigenfunction factor
has an arbitrarily large power of $h$, uniformly over the window.
This absorbs both the number of terms and all fixed derivative
losses, proving \eqref{eq:elliptic-spectral-smoothing}.
In the kernel expansion, differentiate both the localized factor
and the unlocalized eigenfunction factor. The same argument gives
$O(h^\infty)$ in both kernel variables, including at the boundary.
The factors supplied by $\mathcal U_h(t)$ have modulus one;
any fixed number of time derivatives costs only a further power
of $h^{-1}$. The source-side assertion follows by taking adjoints
and replacing $t$ by $-t$.
\end{proof}

Thus every refined elliptic piece of \eqref{eq:global-kernel-partition}
is negligible in the kernel norm of \eqref{dispco}, on either side.
Composition with the fixed semiclassical cutoffs preserves this
conclusion: their fixed Sobolev norms grow at most polynomially in
$h^{-1}$. The finite spectral subdivision and nested symbol
corrections introduce no loss. 

\subsection{Spectral matching of the boundary parametrices}
\label{app:boundary-spectral-matching}

We first realize the spectral cutoff exactly by a Fourier
multiplier in time.
We recall that $D_t=-i\partial_t$ and  $\mathcal U_h(t)=\exp(-ith\Delta_g)$. The spectral theorem gives
\begin{equation}\label{eq:temporal-spectral-matching}
 \psi(hD_t)\mathcal U_h(t)=\mathcal U_h(t)\Pi_h,\qquad
 \psi^2(hD_t)\mathcal U_h(t)=\mathcal K_h(t).
\end{equation}
These identities hold as operator-valued distributions in $t$.
Since the physical cutoffs are independent of time,
\begin{equation}\label{eq:temporal-physical-matching}
 \chi_k^\partial\mathcal K_h(t)\chi_j^\partial
 =\psi^2(hD_t)
   \bigl[\chi_k^\partial\mathcal U_h(t)
                    \widetilde\Pi_h\chi_j^\partial\bigr].
\end{equation}
The larger spectral window is kept in this formula to provide smooth
kernels with polynomial bounds in $h^{-1}$ {before time-frequency localization}.
The equality follows from $\psi^2\widetilde\psi=\psi^2$;
it does not require moving $\widetilde\Pi_h$ past $\chi_j^\partial$.
The same observation applies to time-independent microlocal operators
placed at the two ends. {For tangential operators depending on the spectral parameter, we realize them as operators in $(t,y)$ with coefficients independent of time:}
the scalar multiplier $\psi^2(hD_t)$ commutes with these operators.

We choose {a smooth time cutoff} equal to one on a
neighborhood of the closed time interval under consideration, and
supported where the local parametrix is defined. The kernel of $\psi^2(hD_t)$ is
$h^{-1}k((t-s)/h)$ with $k$ Schwartz. For $t$ in this interval and $s$ outside the neighborhood where
the time cutoff equals one, $|t-s|$ is bounded below. The rapid decay
of $k$ absorbs any polynomial
$h^{-1}$ growth of the kernel and its fixed-order derivatives, giving
$O(h^\infty)$ in every fixed kernel seminorm. For the exact
propagator with $\widetilde\Pi_h$, spectral expansion and elliptic
regularity supply these polynomial bounds uniformly in time, since
the evolution factors have unit modulus. Applying $\psi^2(hD_t)$ to a local $O(h^\infty)$ remainder
preserves this bound, including derivatives in both spatial
variables of the kernel. Hence comparison
on the enlarged local interval suffices, without a global-in-time
parametrix or an unlocalized high-frequency kernel.

We next apply this multiplier in the Airy representation.
In \eqref{eq:Kequiv} {the time-dependent oscillatory factor is}
$\exp(iht\tau^2(\omega,\eta))$. Applying $\psi^2(hD_t)$ therefore inserts the scalar factor
\begin{equation}\label{eq:airy-global-energy-factor}
 \psi^2\!\bigl(h^2\tau^2(\omega,\eta)\bigr)
   =\psi^2\!\bigl(\tau^2(\alpha,\theta)\bigr),
 \qquad \alpha=h^{2/3}\omega,\quad \theta=h\eta .
\end{equation}
The same factor is inserted in each summand of
\eqref{eq:Prond2cut}, or inside the pairing in
\eqref{eq:Prondcut}; it does not alter the Airy--Poisson identity.
For a single global spectral cutoff the factor is $\psi$, rather
than $\psi^2$. {All auxiliary spectral cutoffs}, including the factor
$\varkappa(h\tau(\omega,\eta))$, are chosen equal to one on the
{spectral support needed for this matching}, with a fixed margin.
The tangential factor $\varkappa(h\eta)$ and the glancing cutoffs
in $\alpha$ have separate roles and are retained.

On the compact parameter support, $\psi^2(\tau^2(\alpha,\theta))$
is an order-zero symbol with uniform bounds; it changes neither
the phase nor its critical set. With $\alpha=\gamma u$, derivatives
in $u$ bring factors $\gamma$, not inverse powers. In the discrete
sum the factor is evaluated at $\alpha=h^{2/3}\omega_k$, so the
oscillatory symbol bounds persist. Spatial and conic compositions
are dealt with separately below.

For the glancing refinement, we choose semiclassical operators
$\mathcal Q_{\nu,h}^{\mathrm{gl}}$ in time and the tangential variables,
acting at fixed $x$, with compact frequency support and symbol bounds
uniform up to $x=0$. The time-frequency variable associated with
$hD_t$ represents $\tau^2$. We refine the spectral and conic cover so
that these supports lie strictly inside the region where the
corresponding multipliers of Definition~\ref{dfnparametrix} equal one;
{on each spectral subinterval, we replace $h$ by a fixed multiple to obtain
these supports. We choose physical supports and time intervals strictly
inside the patch and time interval where that definition gives the comparison.} We now prove Lemma~\ref{lem:direct-glancing-matching}
using Propositions~\ref{propsommefinie} and~\ref{propparama<}.

\begin{proof}
We use Lemma~\ref{lem:boundary-locality} to replace the exact global
kernel by the exact local one on a slightly enlarged time
interval. For a source $(a,0)$ the local comparison is precisely
\begin{equation}\label{eq:local-comparison-used}
 \varkappa(hD_t)\varkappa(hD_y)
 \bigl(\mathcal P_{S,h,a}
       -\mathcal G_S(h\,\cdot,\cdot,a)\bigr)
       =O_{C^\infty_{t,x,y}}(h^\infty),
       \qquad 0<a\leq a_0,
\end{equation}
with uniform constants. Here $\mathcal G_S$ denotes the kernel
of the auxiliary local problem. Multiplication by the physical
cutoffs followed by $\mathcal Q_{\nu,h}^{\mathrm{gl}}$ preserves
{this comparison: because the larger frequency cutoffs equal one on the
chosen supports, inserting them produces an $O(h^\infty)$ error when
applied to these spectrally localized kernels. The larger spatial and
time cutoffs similarly control terms outside the chosen supports.
The kernels and their fixed-order derivatives have at most polynomial
growth in $h^{-1}$, so these errors remain $O(h^\infty)$; keep the larger
scalar spectral window throughout, as in
\eqref{eq:temporal-physical-matching}.}
We do not commute these cutoffs through the Laplacian.

We then apply $\psi^2(hD_t)$. By the preceding time-cutoff
argument, contributions from outside the local comparison
interval are $O(h^\infty)$. For the exact kernel this is the {global spectral localization}; for the parametrix it inserts
\eqref{eq:airy-global-energy-factor} and the result follows.
For a general source point, center the tangential coordinates at
its boundary projection point. The constructions and cutoff choices
are made uniformly over the compact family of source points
and the finite conic atlas. Only a uniform bound in the source
is required for the $L^1\to L^\infty$ kernel estimate; no source
derivative estimate is inferred solely from
\eqref{eq:local-comparison-used}. The two propositions cover all
$0<a\leq a_0$, so the argument imposes no additional lower bound
on $a$.
\end{proof}

We finally check that the preceding cutoffs preserve the
pointwise dispersive bound.
At fixed $x$, the kernel of $\mathcal Q_{\nu,h}^{\mathrm{gl}}$
satisfies, for every $N$,
\begin{equation}\label{eq:tangential-time-cutoff-kernel}
 |K_{\mathcal Q}(t,y;s,z;x)|
 \leq C_N h^{-d}
       \left(1+\frac{|t-s|+|y-z|}{h}\right)^{-N}.
\end{equation}
This follows by integration by parts in {its $d$ time and tangential frequency variables}. Integrating in $z$ gives an integrable
majorant of the form $C_Nh^{-1}(1+|t-s|/h)^{-N}$.
Consequently the bound
$h^{-d}(1+|t|/h)^{-d/2+1/4}$ is preserved, since for sufficiently
large $N$,
\begin{equation}\label{eq:weighted-time-convolution}
 \int_{\mathbb R}(1+|u|)^{-N}
       (1+|t/h-u|)^{-d/2+1/4}\,du
 \lesssim (1+|t|/h)^{-d/2+1/4}.
\end{equation}
For example, we use
$1+|t|/h\leq(1+|u|)(1+|t/h-u|)$ and then integrate.
The same argument applies to {the scalar time-frequency cutoff}.
Cutoffs in physical space and the smooth coordinate-density
factors preserve the estimate as well. Thus no transfer of an
$L^1\to L^\infty$ bound through an arbitrary Fourier integral
operator is involved.

Here we may microlocalize only at the observation point, while
keeping the physical partition in both kernel variables. The source is already dealt with
uniformly by the local kernel estimates. An additional fixed
source-side tangential localization is possible using its
uniformly integrable kernel, but is not required for this
matching lemma.

\section{A refined estimate near the swallowtail and cusp critical points}
\label{app:refined-quartic-estimate}
We refine the analysis of degenerate stationary phases from \cite[Lemma 2.21]{ilp12} to describe
the two branches with cusp-like singularities and the estimate across each branch, as well as how they degenerate when they coalesce into a swallowtail.
The phase may depend nonlinearly on the parameters. The estimates
are uniform under the derivative bounds specified below.

\subsection{The swallowtail point and its two Airy branches}
\label{app:quartic-branches}
Let $\Phi(\xi,z,\nu)$ be a smooth real phase, where
$\xi=(\xi_1,\xi_2)\in\R^2$ are the integration variables,
$z\in\R^2$ are two parameters to be chosen, and $\nu$ denotes
the other parameters. After
translation and an orthogonal change of the integration variables,
assume that, at $(\xi,z)=(0,0)$,
\begin{equation}\label{eq:quartic-center-expansion}
 \Phi_1= \Phi_2=0,\qquad \Phi_{11}\ne0,\qquad
 \Phi_{12}=\Phi_{22}=\Phi_{222}=0,\qquad
 k_4:=\Phi_{2222}-3\frac{\Phi_{122}^2}{\Phi_{11}}\ne0\,,
\end{equation}
where subscripts $1,2$ denote derivatives in $\xi_1,\xi_2$;
all derivatives in \eqref{eq:quartic-center-expansion} are evaluated
at zero, with $\nu$ fixed. The implicit function theorem gives
$\xi_1=\xi_{1,c}(z,s,\nu)$ solving $\Phi_1=0$, where $s=\xi_2$.
Set $\psi(z,s,\nu)=\Phi(\xi_{1,c}(z,s,\nu),s,z,\nu)$.
At the center, with the same subscript notation for derivatives,
\begin{equation}\label{eq:quartic-reduced-expansion}
 \psi_s=\psi_{ss}=\psi_{sss}=0,\qquad \psi_{ssss}=k_4.
\end{equation}
We assume that the two parameters satisfy
\begin{equation}\label{eq:quartic-parameter-jacobian}
 \det D_z(\psi_s,\psi_{ss})(0,0,\nu)\ne0.
\end{equation}
On a fixed neighborhood assume uniform $C^m$ bounds for the
phase, with $m$ sufficiently large. At zero assume positive
uniform lower bounds for $|\Phi_{11}|$, $|k_4|$ and the absolute
determinant in \eqref{eq:quartic-parameter-jacobian}. By continuity,
the arguments below hold on a smaller fixed neighborhood.
The amplitudes have uniform $C^m$ bounds and are supported in a fixed
sufficiently small neighborhood of $\xi=0$, also when they depend on
$z,\nu,\Lambda$.

\begin{lemma}\label{lem:refined-quartic-family}
Under \eqref{eq:quartic-center-expansion} and
\eqref{eq:quartic-parameter-jacobian}, there are smooth local
parameter coordinates $(\widehat P,\widehat Q)$, vanishing at
$z=0$, such that, for $z$ sufficiently small and $\Lambda\ge1$,
\begin{equation}\label{eq:refined-quartic-family-bound}
 \left|I(z,\nu,\Lambda)\right|
 :=\left|\int e^{i\Lambda\Phi(\xi,z,\nu)}
                  \sigma_I(\xi,z,\nu,\Lambda)\,d\xi\right|
 \lesssim
 \frac{\Lambda^{-5/6}}
 {\Lambda^{-1/12}+|\widehat P|^{1/6}+|\widehat Q|^{1/9}}.
\end{equation}
In particular,
\begin{equation}\label{eq:refined-quartic-one-weight}
 |I|\lesssim\Lambda^{-1/2}
 \min\{\Lambda^{-1/4},\Lambda^{-1/3}|\widehat P|^{-1/6}\},
\end{equation}
where the second term is omitted when $\widehat P=0$.
Degenerate critical points occur on two parameter branches on the side
$k_4\widehat P<0$, with
\begin{equation}\label{eq:refined-quartic-branches}
 \widehat Q_\pm(\widehat P,\nu)
 =\pm\frac{2^{3/2}}{3\sqrt{|k_4|}}|\widehat P|^{3/2}
   +O(|\widehat P|^2).
\end{equation}
Away from the swallowtail point, each branch has an Airy-type critical
point, with third derivative of size
$|\widehat P|^{1/2}$. Write
$\delta_\pm=\widehat Q-\widehat Q_\pm(\widehat P,\nu)$.
For $|\widehat P|\ge C_0\Lambda^{-1/2}$, the contribution $I_\pm$
localized near either branch satisfies
\begin{equation}\label{eq:refined-quartic-branch-bound}
 |I_\pm|\lesssim\Lambda^{-5/6}|\widehat P|^{-1/6}
 \left(1+\Lambda^{2/3}
       \frac{|\delta_\pm|}{|\widehat P|^{1/6}}\right)^{-1/4},
 \qquad |\delta_\pm|\le c_1|\widehat P|^{3/2}.
\end{equation}
The reduced variable is localized within distance
$c|\widehat P|^{1/2}$ of the corresponding Airy-type critical point, with
uniform cutoff bounds after rescaling. The constants and
neighborhoods are uniform under the preceding assumptions.
\end{lemma}

\begin{proof}
Since $\psi_{ssss}\ne0$, there is a unique smooth $s_0(z,\nu)$
near zero with $\psi_{sss}(z,s_0,\nu)=0$. Define
\begin{equation}\label{eq:quartic-adapted-parameters}
 \widehat P=\psi_{ss}(z,s_0,\nu),\qquad
 \widehat Q=\psi_s(z,s_0,\nu).
\end{equation}
At zero, \eqref{eq:quartic-reduced-expansion} cancels the terms containing
derivatives of $s_0$. Hence \eqref{eq:quartic-parameter-jacobian}
gives a uniformly invertible change of parameters. Taylor's formula, with $s$ now measured from
$s_0$, gives, up to an additive constant,
\begin{equation}\label{eq:quartic-centered-phase}
 \psi=\widehat Qs+\frac{\widehat P}{2}s^2
       +\frac{k(\widehat P,\widehat Q,\nu)}{24}s^4
       +s^5R(\widehat P,\widehat Q,s,\nu),\qquad k(0,0,\nu)=k_4.
\end{equation}
The functions $k,R$ have uniform smooth bounds and $|k|$ is
bounded below. We work with this expansion, including the remainder.

Stationary phase in $\xi_1$ reduces $I$ to a finite sum of
one-dimensional integrals with phase
\eqref{eq:quartic-centered-phase}, uniformly smooth amplitudes and
prefactors $\Lambda^{-1/2-j}$, $j\ge0$. The remainder is
$O(\Lambda^{-M})$ for a fixed sufficiently large $M$.
It suffices to estimate the leading term. Set
\begin{equation}\label{eq:quartic-parameter-scale}
 r_* =\max\{|\widehat P|^{1/2},|\widehat Q|^{1/3}\}.
\end{equation}
For $r_*\le\Lambda^{-1/4}$, the fourth derivative of the reduced
phase is bounded away from zero on the small fixed support.
Van der Corput of order four gives $|I|\lesssim\Lambda^{-3/4}$.

For $r_*>\Lambda^{-1/4}$, set $s=r_*v$. We divide the phase by
$r_*^4$. It becomes
\begin{equation}\label{eq:quartic-scaled-phase}
 \frac{\widehat Q}{r_*^3}v+\frac{\widehat P}{2r_*^2}v^2
 +\frac{k}{24}v^4+r_*v^5R(\widehat P,\widehat Q,r_*v,\nu),
\end{equation}
with oscillatory parameter $\Lambda r_*^4>1$.
On every fixed interval $|v|\le C$, the polynomial part cannot
have its first three derivatives vanish simultaneously: its third
derivative vanishes only at zero, while
$\max\{|\widehat P|/r_*^2,|\widehat Q|/r_*^3\}=1$.
By compactness this remains true, with a uniform lower bound,
after adding the small $C^4$ remainder. A finite partition, followed by integration by
parts or Van der Corput of order two or three, bounds this part of
the reduced integral by
$r_*(\Lambda r_*^4)^{-1/3}=\Lambda^{-1/3}r_*^{-1/3}$.

The rescaled support need not be bounded. On $|s|\ge Cr_*$,
with $C$ large and the original support small, we have
$|\psi_{ss}|=|\widehat P+ks^2/2+O(s^3)|\gtrsim s^2$.
On $|s|\sim2^jr_*$ the second derivative estimate is
$O(\Lambda^{-1/2}(2^jr_*)^{-1})$, with uniformly bounded
amplitude variation. Summing gives
$O(\Lambda^{-1/2}r_*^{-1})$, which is at most
$O(\Lambda^{-1/3}r_*^{-1/3})$ because $\Lambda r_*^4\ge1$.
Multiplying by the stationary phase factor gives
\begin{equation}\label{eq:quartic-two-scales}
 |I|\lesssim
 \begin{cases}
  \Lambda^{-3/4},& r_*\le\Lambda^{-1/4},\\
  \Lambda^{-5/6}r_*^{-1/3},& r_*>\Lambda^{-1/4}.
 \end{cases}
\end{equation}
This is \eqref{eq:refined-quartic-family-bound}; omitting the
$|\widehat Q|^{1/9}$ term gives
\eqref{eq:refined-quartic-one-weight}.

Degenerate critical points solve $\psi_s=\psi_{ss}=0$.
Solving for the two adapted parameters as functions of $s$ gives
\begin{equation}\label{eq:quartic-branch-parametrization}
 \widehat P=-k_4s^2/2+O(s^3),\qquad
 \widehat Q=k_4s^3/3+O(s^4).
\end{equation}
The parameter Jacobian at zero is invertible. Taking the two
signs of $s$ gives
\eqref{eq:refined-quartic-branches}, and
$|\psi_{sss}|\sim|s|\sim|\widehat P|^{1/2}$ on either branch.

For the localized estimate, put $r_*=|\widehat P|^{1/2}$ and
$s=s_\pm+r_*v$, where $s_\pm$ denotes the Airy-type critical point on the
chosen branch. After division by $r_*^4$, the phase has uniformly
bounded derivatives, with the coefficient of the third-order term bounded away from
zero on a small fixed interval. Off the branch, we solve $\psi_{ss}=0$ again. At this new point,
the first derivative is comparable to
$\delta_\pm$, uniformly for $|\delta_\pm|\le c_1r_*^3$:
its derivative with respect to $\widehat Q$ is
$1+O(r_*^3)$ in \eqref{eq:quartic-centered-phase}, since
the derivative of the new center is multiplied by $\psi_{ss}=0$.
After rescaling, the linear coefficient is therefore comparable to
$\delta_\pm/r_*^3$.
The third derivative estimate, combined with the transverse Airy
estimate proved in \ref{app:cubic-continuation}, gives
\begin{equation}\label{eq:quartic-local-cubic-proof}
 |I_\pm|\lesssim\Lambda^{-1/2}r_*
 (\Lambda r_*^4)^{-1/3}
 \left(1+(\Lambda r_*^4)^{2/3}
             \frac{|\delta_\pm|}{r_*^3}\right)^{-1/4}\,,
\end{equation}
which is \eqref{eq:refined-quartic-branch-bound}.
\end{proof}

We can treat the two branches separately when
$|\widehat P|\gtrsim\Lambda^{-1/2}$: their separation is
$\sim|\widehat P|^{3/2}$, while the Airy scale across either is
$\Lambda^{-2/3}|\widehat P|^{1/6}$.
At $|\widehat P|\sim\Lambda^{-1/2}$ both are of size
$\Lambda^{-3/4}$, and the bound matches the estimate near the swallowtail.
Below this scale we use the estimate near the swallowtail.
\begin{rmq}
Lemma~\ref{lem:refined-quartic-family} may also be formulated
in coordinates diagonalizing the Hessian of the model phase.
For the perturbed phase, we eliminate the same nondegenerate
variable and locate the displaced swallowtail point by the
vanishing of the first three derivatives of the reduced phase.
The lower bounds for its fourth derivative and for the
parameter Jacobian in \eqref{eq:quartic-parameter-jacobian}
persist under sufficiently small smooth perturbations,
with uniform constants. Exact diagonalization of the
perturbed Hessian is therefore unnecessary.
\end{rmq}
\subsection{The geometric hypotheses of \cite[Lemma~2.21]{ilp12}}
\label{app:quartic-geometric-hypotheses}
For comparison with \cite[Lemma 2.21]{ilp12}, let
$\Phi(\xi,P,Q)=P\xi_1+Q\xi_2-H(\xi)$, with
$H(0)=0$ and $\nabla_\xi H(0)=0$. Throughout this subsection,
$H_{i_1\cdots i_k}$ denotes
$\partial_{\xi_{i_1}}\cdots\partial_{\xi_{i_k}}H(0)$.
Assume that $D_\xi^2H(0)$ has rank one,
$\nabla_\xi(\det D_\xi^2H)(0)\ne0$, and, for a regular parametrization
$\xi(s)$ of $\det D_\xi^2H=0$ through zero, that
$\Xi(s)=\nabla_\xi H(\xi(s))$ satisfies $\Xi'(0)=0$, $\Xi''(0)\ne0$.
Primes on $\Xi$ denote derivatives with respect to $s$.
We choose orthogonal coordinates with $H_{11}\ne0$ and
$H_{12}=H_{22}=0$ at zero. The tangent to this curve is the kernel
direction. Since
$\nabla_\xi(\det D_\xi^2H)(0)=(H_{11}H_{122},H_{11}H_{222})$,
we have $H_{222}=0$ and $H_{122}\ne0$.
We parametrize the curve by $\xi(s)=(\xi_1(s),s)$, which gives
\begin{equation}\label{eq:quartic-geometric-expansion}
 \xi_1(s)=\left(\frac{H_{122}}{H_{11}}
                  -\frac{H_{2222}}{2H_{122}}\right)s^2+O(s^3),
 \qquad
 \Xi''(0)=\left(-\frac{H_{11}}{H_{122}}
                \left(H_{2222}-3\frac{H_{122}^2}{H_{11}}\right),0\right).
\end{equation}
Thus the fourth derivative of the reduced phase is
$k_4=-H_{2222}+3H_{122}^2/H_{11}\ne0$. In the normalized coordinates used in the proof of
\cite[Lemma~2.21]{ilp12}, $H_{11}=1$, and the Taylor
coefficients used there are such that
$H_{111}=6a$, $H_{112}=2b$, $H_{122}=2c$,
$H_{222}=6d$, and $H_{2222}=24e$.
Here $d=0$, $c\ne0$, and $k_4=12(c^2-2e)$,
so $k_4\ne0$ is precisely the condition $c^2\ne2e$
obtained there.
After eliminating $\xi_1$ and using
\eqref{eq:quartic-adapted-parameters}, the parameter Jacobian is
\begin{equation}\label{eq:quartic-linear-parameter-jacobian}
 D_{P,Q}(\widehat P,\widehat Q)(0)
 =\begin{pmatrix}-H_{122}/H_{11}&0\\0&1\end{pmatrix}.
\end{equation}
Lemma \ref{lem:refined-quartic-family} therefore applies: in this
linear family, the lower bound on $|H_{122}|$ gives the parameter
condition. For nonlinear dependence we check
\eqref{eq:quartic-parameter-jacobian} directly, without requiring
a separate lower bound on $|\Phi_{122}|$.

Equivalently, let $u,v$ be unit eigenvectors of the Hessian at
zero, with $v$ in its kernel. The coefficient of the fourth-order term in the reduced phase is
\begin{equation}\label{eq:quartic-intrinsic-coefficient}
 k_4=D^4_\xi\Phi(v,v,v,v)
     -3\frac{D^3_\xi\Phi(u,v,v)^2}{D^2_\xi\Phi(u,u)}.
\end{equation}
Together with $D^3_\xi\Phi(v,v,v)=0$, this gives
\eqref{eq:quartic-center-expansion}, independently of the signs of $u,v$.
It suffices to choose these eigenvectors locally in the parameters.

\subsection{Continuation along the branches and the transverse Airy estimate}
\label{app:cubic-continuation}
The preceding conditions apply near the swallowtail point. On a larger
compact set we also check the other critical points, allowing a
different nondegenerate variable to be eliminated near each one.

\begin{lemma}\label{lem:uniform-cubic-continuation}
Assume uniform smooth bounds for $\Phi(\xi,z,\nu)$ and the
amplitudes, with fixed compact support for the latter, and
compactness of the family, including its limits. At each critical
point on the supports, assume that $\Phi_{\xi\xi}$ is either
nondegenerate, or has rank one with
\begin{equation}\label{eq:cubic-continuation-hypothesis}
 D^3_\xi\Phi(v,v,v)\ne0,
\end{equation}
where $v$ is a unit vector in its kernel. Then
$|I|\lesssim\Lambda^{-5/6}$ uniformly for $\Lambda\ge1$.
One may instead assume a finite covering on which the
nondegenerate or third derivatives of the reduced phase have uniform
positive lower bounds.

If, in addition, after eliminating the nondegenerate variable,
the reduced phase satisfies $D_z\psi_s\ne0$ at each Airy-type critical point,
with a uniform lower bound, the Airy-type critical points occur on a
smooth regular curve in the parameter plane. Locally, define
\begin{equation}\label{eq:cubic-transverse-coordinate}
 \psi_{ss}(z,s_0(z,\nu),\nu)=0,\qquad
 \delta(z,\nu)=\psi_s(z,s_0(z,\nu),\nu).
\end{equation}
Here $\delta$ measures the deviation from the curve of Airy-type critical
points, defined by $\delta=0$. The contribution localized near this curve
satisfies
\begin{equation}\label{eq:cubic-transverse-estimate}
 |I_{\mathrm{Airy}}|\lesssim
 \Lambda^{-5/6}\mathfrak b^{-1/2}\mathfrak c^{-1/3}
 \left(1+\Lambda^{2/3}\frac{|\delta|}{\mathfrak c^{1/3}}
 \right)^{-1/4}\,.
\end{equation}
Here $\mathfrak b$ is the absolute second derivative in the
eliminated direction and $\mathfrak c=|\psi_{sss}(z,s_0,\nu)|$;
we restrict to a neighborhood where these derivatives remain
comparable to their values at the Airy-type critical point. On fixed
neighborhoods both coefficients have positive uniform upper and
lower bounds. The additional assumption defines the degeneracy curve
and a transverse coordinate.
\end{lemma}

\begin{proof}
Stationary phase gives $O(\Lambda^{-1})$ near a
nondegenerate critical point. At a rank-one point, we choose orthonormal coordinates diagonalizing
the Hessian and eliminate the nondegenerate variable. Since the mixed Hessian vanishes at the center,
the reduced third derivative there is
$D^3_\xi\Phi(v,v,v)$. Stationary phase and Van der Corput of order
three give $O(\Lambda^{-5/6})$. These bounds persist on small neighborhoods. By compactness,
a finite covering suffices; on the complement we integrate by
parts since the gradient is bounded away from zero.

For the more precise bound, $\psi_{sss}\ne0$ gives $s_0$ in
\eqref{eq:cubic-transverse-coordinate}. On the curve,
$D_z\delta=D_z\psi_s$, since $\psi_{ss}=0$; hence $\delta$ is a
regular coordinate transverse to the curve. Write $v=s-s_0$. Taylor's formula
gives
\begin{equation}\label{eq:cubic-transverse-Taylor}
 \psi_s=\delta+v^2g(z,\nu,v),\qquad
 |g|\sim\mathfrak c,\qquad
 |\psi_{ss}|\sim\mathfrak c|v|.
\end{equation}
The coefficient of the third-order term has constant sign on a small neighborhood. Stationary phase contributes $(\Lambda\mathfrak b)^{-1/2}$.
For $|\delta|\le\Lambda^{-2/3}\mathfrak c^{1/3}$, Van der Corput
of order three gives $(\Lambda\mathfrak c)^{-1/3}$ for the remaining
integral. Otherwise, we split the integral at
$|v|\sim(|\delta|/\mathfrak c)^{1/2}$, with a small fixed constant
for the central interval. There $|\psi_s|\gtrsim|\delta|$, and
integration by parts gives $O((\Lambda|\delta|)^{-1})$.
On the two outer intervals,
$|\psi_{ss}|\gtrsim(\mathfrak c|\delta|)^{1/2}$, and the second
derivative test gives
$O(\Lambda^{-1/2}\mathfrak c^{-1/4}|\delta|^{-1/4})$.
The central estimate is no larger in this regime. The amplitudes
have bounded variation under this subdivision, proving
\eqref{eq:cubic-transverse-estimate}.
\end{proof}

The Airy scale across the curve is
$|\delta|\sim\Lambda^{-2/3}\mathfrak c^{1/3}$.
On the side where $\delta$ and $\psi_{sss}(s_0)$ have opposite
signs, there are two nearby critical points, at distances
comparable to $(|\delta|/\mathfrak c)^{1/2}$ from $s_0$. At either,
\begin{equation}\label{eq:cubic-stationary-Hessian}
 |\psi_{ss}|\sim(\mathfrak c|\delta|)^{1/2},\qquad
 |\det\Phi_{\xi\xi}|\sim
          \mathfrak b\mathfrak c^{1/2}|\delta|^{1/2}.
\end{equation}
Outside the Airy scale, the right hand side of
\eqref{eq:cubic-transverse-estimate} is therefore comparable to
$\Lambda^{-1}\mathfrak b^{-1/2}\mathfrak c^{-1/4}|\delta|^{-1/4}$,
the usual stationary-phase bound. On the other side there is no nearby critical point, and
integrations by parts give a smaller contribution.

\subsection{Rescaling and changing the eliminated variable}
\label{app:cubic-quartic-matching}
When the coefficient of the third-order term tends to zero, we rescale the two
variables to keep the preceding estimate uniform. Let $\ell$ be the
width of the reduced variable and $w$ that of the eliminated
variable, with $\mathfrak b w^2\sim\mathfrak c\ell^3$. Assume that
the phase divided by $\mathfrak c\ell^3$, after these changes of
variables and subtraction of a constant, has uniform smooth bounds,
and that its nondegenerate second derivative and reduced third
derivative are bounded away from zero. The rescaled amplitudes
have uniform smooth bounds as well. For
$\Lambda\mathfrak c\ell^3\gtrsim1$, the same proof, including
the Jacobian $w\ell$, gives
\eqref{eq:cubic-transverse-estimate}. The derivative bounds are required after rescaling.

Once the two branches in Lemma~\ref{lem:refined-quartic-family} are separated, we work near either one and set $r_*=|\widehat P|^{1/2}$.
In the original nondegenerate direction
$\mathfrak b\sim1$, $\mathfrak c\sim r_*$, and we may take
$\ell\sim r_*$, $w\sim r_*^2$, centering the eliminated variable
at its critical value. The oscillatory parameter in both rescaled
directions is $\Lambda r_*^4$. The transverse coordinate in
\eqref{eq:cubic-transverse-coordinate} is comparable to
$\delta_\pm$ on $|\delta_\pm|\le c_1r_*^3$. Thus
\eqref{eq:cubic-transverse-estimate} gives
\begin{equation}\label{eq:cubic-quartic-matched-bound}
 \Lambda^{-5/6}r_*^{-1/3}
 \left(1+\Lambda^{2/3}\frac{|\delta_\pm|}{r_*^{1/3}}
 \right)^{-1/4},
\end{equation}
which is \eqref{eq:refined-quartic-branch-bound}. At
$r_*\sim\Lambda^{-1/4}$ the prefactor and transverse width are
both $\Lambda^{-3/4}$, matching the swallowtail region.

We now check directly that either nondegenerate variable can be
eliminated. At an Airy-type critical point
where $\Phi_{11}\ne0$ and $\Phi_{12}\ne0$, the rank-one condition
gives $\Phi_{22}=\Phi_{12}^2/\Phi_{11}\ne0$.
Let $\psi^{(1)}$ be the critical value after eliminating $\xi_1$
and $\psi^{(2)}$ that after eliminating $\xi_2$. The corresponding vectors spanning the Hessian kernel are
\begin{equation}\label{eq:cubic-kernel-change}
 v_1=(-\Phi_{12}/\Phi_{11},1),\qquad
 v_2=(1,-\Phi_{11}/\Phi_{12})
       =-\frac{\Phi_{11}}{\Phi_{12}}v_1.
\end{equation}
If $\mathfrak b_j,\mathfrak c_j$ and $\delta^{(j)}$ denote
the coefficients and transverse coordinates for these two
eliminations, then, on their common degenerate critical curve,
\begin{equation}\label{eq:cubic-elimination-identities}
 \mathfrak b_2=\frac{|\Phi_{12}|^2}{\mathfrak b_1},\qquad
 \mathfrak c_2=\left|\frac{\Phi_{11}}{\Phi_{12}}\right|^3
                    \mathfrak c_1,\qquad
 D_z\delta^{(2)}=-\frac{\Phi_{11}}{\Phi_{12}}
                         D_z\delta^{(1)}.
\end{equation}
Indeed, differentiating the critical value three times gives
$D^3_\xi\Phi(v_j,v_j,v_j)$: the other terms vanish since
$\Phi_\xi=0$ and $\Phi_{\xi\xi}v_j=0$. Similarly,
$D_z\delta^{(j)}=v_j\cdot\Phi_{\xi z}$ on the curve.
It follows that the prefactors agree there,
$\mathfrak b_2^{-1/2}\mathfrak c_2^{-1/3}
=\mathfrak b_1^{-1/2}\mathfrak c_1^{-1/3}$, and the normalized
transverse differentials agree in absolute value after division
by $\mathfrak c_j^{1/3}$. These are identities for the full phase.

In the linear family above, $H_{122}\ne0$ gives, on either branch
near zero and in the coordinates of the swallowtail point,
$|\Phi_{11}|\sim1$, $|\Phi_{12}|\sim r_*$ and
$|\Phi_{22}|\sim r_*^2$. Consequently the second elimination has
$\mathfrak b_2\sim r_*^2$, $\mathfrak c_2\sim r_*^{-2}$,
reduced width $\ell\sim r_*^2$ and eliminated width $w\sim r_*$.
Its oscillatory parameter is again $\Lambda r_*^4$.
With uniform bounds in these rescaled variables, the two
transverse coordinates satisfy
$|\delta^{(2)}|\sim r_*^{-1}|\delta^{(1)}|$ for $|\delta^{(1)}|\le cr_*^3$, with $c$ small enough.
This follows from \eqref{eq:cubic-elimination-identities} on
the curve and the mean value theorem after rescaling.
Both choices therefore yield \eqref{eq:cubic-quartic-matched-bound},
including its transverse width. For a perturbed phase we repeat the argument at its swallowtail point,
in coordinates diagonalizing its Hessian and along its Airy branches.

Finally, combine the estimate near the swallowtail point with the
finite covering of the other Airy-type and nondegenerate critical points. If a bounded nonnegative parameter weight $W$ is
comparable to $|\widehat P|$ near the swallowtail point, the sum satisfies
\begin{equation}\label{eq:quartic-global-compact-weight}
 |I_{\mathrm{total}}|\lesssim
 \frac{\Lambda^{-5/6}}{\Lambda^{-1/12}+W^{1/6}}.
\end{equation}
The central contribution follows from
\eqref{eq:refined-quartic-one-weight}; the exterior
$\Lambda^{-5/6}$ bound implies the same estimate because $W$ is
bounded. Any comparison with a physical weight must be checked separately:
the parameters \eqref{eq:quartic-adapted-parameters} are computed
at the swallowtail point of the full phase, which may differ from
that of its principal part.

\end{document}